\documentclass{amsart}
\usepackage{amssymb,amsmath,amsthm}
\usepackage[shortalphabetic]{amsrefs}
\usepackage{mathrsfs}
\usepackage{array}
\usepackage{caption}
\usepackage{pdflscape}
\usepackage[dvipsnames]{xcolor}
\usepackage{tikz-cd}
\usetikzlibrary{patterns}
\usepackage{todonotes}
\usepackage{longtable,booktabs}
\usepackage{spectralsequences}
\usepackage[colorlinks=true,linkcolor=blue,citecolor=mblue]{hyperref}
\usepackage[nameinlink,capitalise,noabbrev]{cleveref}

\definecolor{mblue}{RGB}{65,105,225}

\crefname{equation}{}{}

\newtheorem{thm}{Theorem}[section]
\newtheorem{prop}[thm]{Proposition}
\newtheorem{lem}[thm]{Lemma}
\newtheorem{lemma}[thm]{Lemma}
\newtheorem{conj}[thm]{Conjecture}

\theoremstyle{definition}
\newtheorem{defn}[thm]{Definition}
\newtheorem{eg}[thm]{Example}
\newtheorem{notn}[thm]{Notation}
\newtheorem{rmk}[thm]{Remark}

\makeatletter
\let\c@equation\c@thm
\makeatother

\newenvironment{pf}{\begin{proof}}{\end{proof}}

\newcolumntype{L}{>{$}l<{$}}

\numberwithin{equation}{section}
\numberwithin{figure}{section}
\numberwithin{table}{section}

\newcommand{\cA}{\ensuremath{\mathcal{A}}}
\newcommand{\C}{\ensuremath{\mathbb{C}}}
\newcommand{\F}{\ensuremath{\mathbb{F}}}
\newcommand{\rH}{\ensuremath{\mathrm{H}}}
\newcommand{\bM}{\ensuremath{\mathbb{M}_2}}
\newcommand{\Q}{\ensuremath{\mathbb{Q}}}
\newcommand{\R}{\ensuremath{\mathbb{R}}}
\newcommand{\Z}{\ensuremath{\mathbb{Z}}}

\newcommand{\smsh}{\wedge}
\newcommand{\into}{\hookrightarrow}
\newcommand{\onto}{\twoheadrightarrow}
\newcommand{\rtarr}{\longrightarrow}
\newcommand{\xrtarr}[1]{\xrightarrow{#1}}
\newcommand{\xltarr}[1]{\xleftarrow{#1}}
\newcommand{\iso}{\cong}
\newcommand{\ga}{\gamma}
\newcommand{\bracket}[1]{ \left\langle #1 \right\rangle}
\newcommand{\hsf}{\mathsf{h}}
\newcommand{\odelta}{\overline{\delta}}
\newcommand{\ds}{\displaystyle}
\newcommand{\cl}{\mathrm{cl}}

\DeclareMathOperator{\Ext}{Ext}
\DeclareMathOperator*{\colim}{colim}
\DeclareMathOperator{\coker}{coker}

\newcommand{\bF}{\F_2}
\newcommand{\bMCC}{\bM^{\C}}
\newcommand{\bMR}{\bM^\R}
\newcommand{\bMC}{\bM^{C_2}}
\newcommand{\cAcl}{\cA^\cl}
\newcommand{\cACC}{\cA^\C}
\newcommand{\cAR}{\cA^\R}
\newcommand{\cAC}{\cA^{C_2}}
\newcommand{\ExtCT}{\Ext_{C_2}}
\newcommand{\piC}{\pi^{C_2}}
\newcommand{\piR}{\pi^{\R}}
\newcommand{\picl}{\pi^{\cl}}

\newcommand{\rhocolor}{gray}

\newcommand{\rhodiv}{
\class(\lastx+1,\lasty)
\structline[\rhocolor]
}

\newcommand{\rhomult}{
  \class(\lastx-1,\lasty)
  \structline[\rhocolor]
}

\newcommand{\rhotower}[1]{
    \foreach \i in {2,...,#1} {
        \rhomult
    }
}

\newcommand{\rhocotower}[1]{
    \foreach \i in {2,...,#1} {
        \rhodiv
    }
}

\NewSseqGroup{\vertrholoc}{}{
  \draw[\rhocolor,->,>=stealth,semithick](0,0)--(0.0,0.7);
}

\NewSseqGroup{\vertrhocoloc}{}{
  \draw[\rhocolor,->,>=stealth,semithick](0,0)--(0.0,-0.7);
}

\sseqset{
    classes= {fill, inner sep = 0pt, minimum size = 0.25em},
    class labels={below=0.3pt,font=\scriptsize}, 
    differentials={>=stealth,semithick},
    struct lines=semithick,
    class pattern=linear, 
    class placement transform = { rotate = 45, scale = 1 },
    grid color = gray!40!white,
}

\begin{document}

\title{$C_2$-equivariant stable stems}
\author{Bertrand J. Guillou}
\address{Department of Mathematics\\ University of Kentucky\\
Lexington, KY 40506, USA}
\email{bertguillou@uky.edu}

\author{Daniel C. Isaksen}
\address{Department of Mathematics\\ Wayne State University\\
Detroit, MI 48202, USA}
\email{isaksen@wayne.edu}
\thanks{The first author was supported by NSF grants DMS-1710379 and DMS-2003204.
The second author was supported by NSF grants DMS-1904241
and DMS-2202267.}



\begin{abstract}
We compute the 
$2$-primary
$C_2$-equivariant stable homotopy groups
$\piC_{s,c}$ 
for stems between $0$ and $25$ (i.e., $0 \leq s \leq 25$)
and for coweights between $-1$ and $7$
(i.e., $-1 \leq c \leq 7)$.
Our results, combined with periodicity
isomorphisms and sufficiently extensive $\R$-motivic computations,
would determine all 
of the $C_2$-equivariant stable homotopy
groups for 
all stems up to $20$.
We also compute the forgetful map $\piC_{s,c} \rightarrow \picl_s$
to the classical stable homotopy groups in the same range.
\end{abstract}

\date{\today}

\maketitle

\setcounter{tocdepth}{1}
\tableofcontents

\section{Introduction}

In any stable homotopy theory, the graded endomorphisms of the unit object
play a central computational role.  These endomorphisms control the
construction of finite complexes because they serve as attaching maps for cells.
From another perspective, they are the universal operations for generalized
cohomology theories.

The goal of this manuscript is to study 
$C_2$-equivariant stable homotopy groups.
We analyze the Adams spectral sequence in a range
and compute some $2$-primary $C_2$-equivariant stable homotopy groups.
To ensure that the Adams spectral sequence converges, we assume
without further discussion that everything is appropriately
completed.  In other words, we are computing the homotopy groups
of the $2$-completed $C_2$-equivariant sphere spectrum.

There are two essentially different ways to study homotopy groups in the
$G$-equivariant context: 
$\Z$-graded homotopy Mackey functors
and $RO(G)$-graded homotopy groups.
In the Mackey functor perspective, we consider the fixed-points 
of $G$-spectra
with respect to various subgroups of $G$.  The homotopy groups
of these (non-equivariant) fixed-point spectra assemble into a Mackey functor.
On the other hand, the $RO(G)$-graded homotopy groups are a family
of abelian groups indexed by the virtual representations of $G$.

The Mackey functors
are a full set of
$G$-equivariant homotopical invariants, in the sense that they
detect equivalences.  In general, the $RO(G)$-graded homotopy groups
are not as powerful, in the sense that they do not detect all equivalences.
See \cref{ex:Clover} for an explicit example of a $C_3$-equivariant map
that induces an isomorphism on $RO(C_3)$-graded homotopy groups but is not
an isomorphism on Mackey functors.

However, the $C_2$-equivariant case is somewhat special.
It turns out that the $RO(C_2)$-graded homotopy groups do detect
$C_2$-equivariant weak equivalences.
See \cref{subsctn:Mackey-ROG} for further discussion
of this phenomenon.
In theory, it is therefore possible to translate between
Mackey functor computations and $RO(C_2)$-graded information.

In this manuscript, we focus exclusively on the $RO(C_2)$-graded
homotopy groups.  The most important reason for this choice is that our preferred tool, the Adams spectral sequence, converges
to the $RO(C_2)$-graded homotopy groups.  In other words, we have chosen
to compute what we know how to compute!

Another reason to study $RO(C_2)$-graded homotopy groups is that they 
are compatible with $\R$-motivic homotopy groups.
Betti realization is a functor from
$\R$-motivic homotopy theory to $C_2$-equivariant homotopy theory.
The $\R$-motivic homotopy groups are a bigraded family of abelian groups,
and Betti realization maps these $\R$-motivic homotopy groups to their
corresponding $RO(C_2)$-graded homotopy groups.

Much is known about the comparison between
$\R$-motivic and $C_2$-equivariant homotopy groups
\cite{BGI} \cite{DI-comparison}.
In practice, the $\R$-motivic Adams spectral sequence
\cite{BI} is much more manageable than the more complicated
$C_2$-equivariant Adams spectral sequence.
Therefore, we prefer to compute $\R$-motivically whenever possible.
The Betti realization map $\piR_{s,c} \rtarr \piC_{s,c}$ 
is an isomorphism if $c \geq \frac12s - 2$ \cite{BGI},
where $s$ is the stem and $c$ is the coweight
(see \cref{sec:notn} for more details on grading conventions).
Therefore, we concentrate our $C_2$-equivariant efforts
in the complementary range $c < \frac{1}{2} s - 2$.

On the other hand, the $C_2$-equivariant homotopy groups 
exhibit a periodicity
\cite{B}, 
\cite{L}*{Proposition 6.1},
\cite{AI}*{Section 3},
to be described in more detail later in \cref{subsctn:tau-periodicity}.
For now, we observe that 
for $c\leq -2$, the group
$\piC_{s,c}$ is isomorphic to the 
$\rho$-power torsion subgroup of $\piR_{s,c+k}$,
where $k$ is a certain value depending only on $s$.

\newcommand{\leftboundary}{-6}
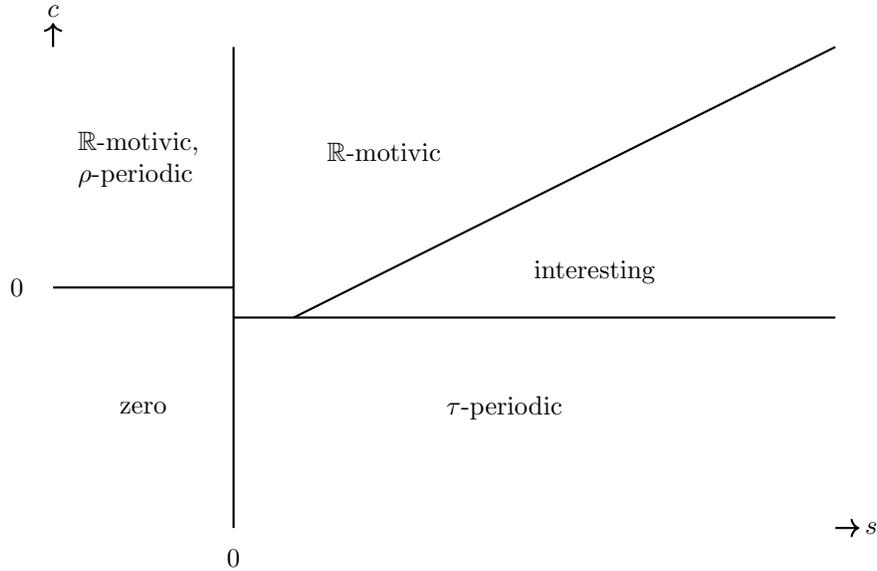
\begin{figure}[h!]
\caption{The $C_2$-equivariant stable stems $\piC_{s,c}$}
\label{fig:InterestingRegion}
\begin{tikzpicture}[scale=0.4]
\draw[thick] (\leftboundary,0) -- (0,0);
\draw[thick] (0,-8) -- (0,8);
\draw[thick] (0,-1) -- (20,-1);
\draw[thick] (2,-1) -- (20,8);
\draw[->,thick] (\leftboundary,8) -- (\leftboundary,8+0.8);
\draw[->,thick,yshift={0ex}] (20,-8) -- (20+0.8,-8);
\node at (-3,-4) {zero};
\node at (-2,4.25) {\parbox{1in}{$\R$-motivic, \\ $\rho$-periodic}};
\node at (9,-4) {$\tau$-periodic};
\node at (5,4.5) {$\R$-motivic};
\node at (12,0.5) { interesting};
\node at (0,-9) { 0};
\node at (\leftboundary-1.2,0) {0};
\node at (21.2,-8) {$s$};
\node at (\leftboundary,9.2) {$c$};
\end{tikzpicture}
\end{figure}

\cref{fig:InterestingRegion} breaks the $(s,c)$-plane into regions
depending on the relationship between $\R$-motivic and $C_2$-equivariant
stable homotopy groups.
This is similar to Figure 1 of \cite{BS} and Figure 5 of \cite{BGI}, using a diffferent choice of grading.
In the  regions labelled ``$\R$-motivic", the $\R$-motivic
and $C_2$-equivariant stable homotopy groups are isomorphic \cite{BGI}.
In the  region labelled ``$\tau$-periodic", 
periodicity allows one to deduce
$C_2$-equivariant stable homotopy groups from $\R$-motivic computations.
The purpose of this manuscript is to study the 
region labelled ``interesting", in which $\R$-motivic computations
do not immediately determine the $C_2$-equivariant stable homotopy groups.

The ``interesting'' region of the $C_2$-equivariant stable stems, as displayed in \cref{fig:InterestingRegion}, 
 consists of the groups $\piC_{s,c}$ with 
\[ -1 \leq c \leq \frac{s-5}2.\] 

In this article, we determine the groups $\piC_{s,c}$  in 
an initial portion of the ``interesting" region.
More precisely, we determine the
$C_2$-equivariant stable homotopy groups in 
stems $s \leq 25$ and coweights $-1\leq c \leq 7$. 
This captures the entire ``interesting'' region for $s \leq 20$.

Along the way, we also compute 
$C_2$-equivariant stable homotopy groups in stems $s \leq 7$
and coweights $-9 \leq c \leq -2$.  This computation is not the main
point of our work, but it provides some data that supports our other
more extensive computations.

Our philosophy is that one should not just compute $C_2$-equivariantly,
nor just $\R$-motivically and $C_2$-equivariantly.
Rather, one should also compute $\C$-motivically and classically
as well, with information passing amongst the four contexts in all directions.

In this general direction,
we compute many values of the underlying homomorphism
from $C_2$-equivariant stable homotopy groups to classical stable
homotopy groups.  
As explained in \cites{BG,C2MW0}, these values are crucial input 
for the calculation of Mahowald invariants.
Our computation of $C_2$-equivariant stable stems goes up 
to coweight 7, so our work recovers the Mahowald 
invariants of all stable homotopy elements up to the 7-stem.
All such Mahowald invariants are already well-known
\cite{Behrens07} \cite{Mahowald-Ravenel93}.

In this range, the Mahowald invariant is 
almost entirely calculated already by $\R$-motivic data \cite{BI}.
The sole exception is the Mahowald invariant of $8 \sigma$.
Our $C_2$-equivariant computations immediately lead to the conclusion
that the Mahowald invariant of $8\sigma$ is $\eta^2 \eta_4$ in the $18$-stem.
The main point is that the $C_2$-equivariant Adams $E_\infty$ chart in 
coweight $7$
shows that the element $h_1^3 h_4$ is $\rho$-periodic but not $\rho$-divisible.

Independently of the relevance to Mahowald invariants, we believe that
the values of the underlying map 
carry deep information about both $C_2$-equivariant and classical stable
homotopy.  At the very least, it is tightly linked to the homotopical
$\rho$-Bockstein spectral sequence that starts with classical stable
homotopy groups and converges to $C_2$-equivariant stable homotopy groups.
We leave a further exploration of these ideas to future work.

\subsection{Related work}

Bredon initiated the computation of $C_2$-equivariant
stable homotopy groups \cite{B} \cite{B68}, 
mostly focusing on coweight $0$.
Landweber \cite{L} carried forward Bredon's work.
Araki and Iriye \cite{AI} \cite{Iriye82} computed 
$C_2$-equivariant stable homotopy groups
roughly through the range $s \leq 13$
using EHP-style techniques.
Their approach determines the abelian group structure of
each $C_2$-equivariant stable homotopy group in their range.  However,
their computations can be hard to interpret because they 
are relatively unstructured.  
Our work extends the range of the Araki-Iriye computations, but it also
independently verifies their computations.  Moreover, we provide
more structure that makes it easier to understand how the various groups fit
together.
We also mention \cite{GHIR} and \cite{C2MW0}, which can be viewed as warmups
to this manuscript.

By the Segal Conjecture (which now has many proofs), the $2$-complete $C_2$-equivariant spectrum 
$S^{0,0}$ is equivalent to its Borel completion $F({EC_2}_+,S^{0,0})$, so the Borel $C_2$-equivariant 
Adams spectral sequence also computes the same $C_2$-equivariant stable 
homotopy groups. 
Recent work  of Sihao Ma \cite{Ma} compares these two approaches, although
it has no direct bearing on our work.

\subsection{Future work}

Our methods are far from exhausted.  In order to make the project manageable, we chose an essentially arbitrary range in which to compute.  We strongly suspect that the ideas set forth in this manuscript can be applied in a significantly larger range.  Further work would be aided significantly by machine computation.

Our philosophy is to separate the $\tau$-periodic phenomena from the
non-periodic phenomena, to the extent that is possible.  Computation would
be greatly aided by some method that independently computes the $\tau$-periodic
$C_2$-equivariant stable homotopy groups.  Having $\tau$-periodic
information in hand would simplify the non-periodic part of the computation.
It is known that the $\tau$-periodic $C_2$-equivariant
stable homotopy groups are isomorphic to certain stable homotopy
groups of stunted projective spaces \cite{L}.  However, that isomorphism in
itself does not provide a tool for extensive computations.

It is natural to ask whether our methods can be generalized to compute
$G$-equivariant stable homotopy groups for finite groups $G$ other than $C_2$.
It is conceivable that such a generalization can be carried out,
although there are real difficulties arising from the more
complicated structure of the $G$-equivariant homology of a point
and the $G$-equivariant dual Steenrod algebra \cites{HKSZ,SW}.

From a computational perspective, $(C_2 \times C_2)$-equivariant
stable homotopy groups are particularly intriguing because the 
appropriate analogue of the second Hopf map $\nu$ is not nilpotent.
This means that periodicity phenomena in $(C_2 \times C_2)$-equivariant 
homotopy theory are both rich and not well-understood.
Similarly, the third Hopf map $\sigma$ is not nilpotent
in $(C_2 \times C_2 \times C_2)$-equivariant homotopy theory.

We also draw attention to \cref{Type3diffs}, which speculates on a 
relationship between $\R$-motivic Bockstein differentials and
$C_2$-equivariant Bockstein differentials.

\subsection{Outline}

For readers who are interested in our
main computation and wish to skip the technical details,
the charts on pages \pageref{Einfstart}--\pageref{Einfend} display the $C_2$-equivariant Adams $E_\infty$-page, sorted by coweight.  The $C_2$-equivariant
stable homotopy groups (in a range) can be read from these charts in the usual
way.  See \cref{subsctn:Einfty-chart-key} for a specific key for the symbols and colors that we use.

We now give a slightly more detailed outline of our manuscript.
We compute equivariant stable homotopy groups via the $C_2$-equivariant Adams spectral sequence, whose input is given by $\Ext$ groups over the $C_2$-equivariant dual Steenrod algebra $\cAC_*$.
In turn, these $\Ext$ groups are computed by a $\rho$-Bockstein spectral sequence, as in \cites{GHIR,C2MW0}. We rely heavily on the computation of the $\R$-motivic Bockstein spectral sequence and Adams spectral sequence \cite{BI}.

The exposition of our computation
relies heavily on detailed notation, 
which is described in \cref{sec:notn}.  The reader is encouraged to cross-reference this section frequently.

\cref{background} contains some miscellaneous background information about
$C_2$-equi\-var\-iant stable homotopy theory.  Our discussion is not meant to
be exhaustive.  We simply touch on several points that are particularly
important for our computation.

Our main program begins in \cref{sec:BockSS}
where we discuss the
$\rho$-Bockstein spectral sequence that converges to the 
$C_2$-equivariant $\Ext$ groups.
As discussed in \cite{GHIR}*{Section 2}, the $C_2$-equivariant $\Ext$ groups split as a sum of the $\R$-motivic $\Ext$ groups as well as 
another piece arising from the ``negative cone".
This splitting arises from a splitting of 
the $C_2$-equivariant $\F_2$-homology of a point,
as displayed in \cref{fig:M2C}.
As a result, given the input of \cite{BI}, 
determination of the $C_2$-equivariant $\Ext$ groups reduces to a computation
of the $\rho$-Bockstein spectral sequence for the negative cone.

\cref{sec:NCPeriodic} and \cref{sec:E1extensions} dig deeper into the
structure of the $E_1$-page of the $\rho$-Bockstein spectral sequence
for the negative cone.
We explain how this $E_1$-page can be described completely in terms of
$\C$-motivic $\Ext$ groups, which are known in a large range
\cite{IWX23}.  We use a short exact sequence to compute the $E_1$-page,
but then we must resolve some extension problems associated to this
sequence.

Then \cref{sec:PCtoNC} describes some methods for computing
$\rho$-Bockstein differentials.
In particular, Bockstein differentials in the negative cone are intimately tied to Bockstein differentials in the positive cone, i.e. $\R$-motivic Bockstein differentials.  This connection suggests deeper structure that we have not
yet made precise.

\cref{sec:etaperiodic} is an interlude on $\eta$-periodic computations.
Experience shows that the $\eta$-periodic (or $h_1$-periodic in an algebraic
context) computations are much easier, but they also detect significant
phenomena about the unperiodicized computations.  Our later computations
are made somewhat easier by having this periodic information in advance.

The $\eta$-periodic $\R$-motivic stable homotopy groups \cite{etaR} exhibit non-trivial
yet fully described structure that is similar to the structure in the
classical image of $J$.  Interestingly, this type of phenomenon does
not occur $C_2$-equivariantly.  Rather, $C_2$-equivariant $\eta$-periodic
homotopy is trivial (except in degree $0$).
At first glance, this vanishing result suggests that
$\eta$-periodic homotopy carries no useful information for us.
However, one can deduce certain $\eta$-periodic Adams differentials from the fact
that the $\eta$-periodic homotopy must vanish.  In turn, the
$\eta$-periodic Adams differentials provide information about
unperiodicized Adams differentials.

\cref{sctn:Bockstein-diff} starts the main thrust of stemwise computation
in a range.  We begin with an exhaustive accounting of Bockstein differentials.
From these differentials, we obtain an additive description of
the $C_2$-equivariant $\Ext$ groups.  We study some of the multiplicative
structure of $\Ext$ in \cref{sctn:hidden}.

Having obtained the $C_2$-equivariant Adams $E_2$-page,
we then proceed to the Adams differentials in \cref{sec:Adams}.
This gives us the $E_\infty$-page of the $C_2$-equivariant Adams
spectral sequence.

The final step is to determine some of the multiplicative structure
in the $C_2$-equivariant stable homotopy groups.  In
\cref{sec:cofibrho} and \cref{sec:hiddenAdams}, 
we resolve hidden extensions by $\rho$, by $\hsf$, and by $\eta$.
Here $\hsf$ is the zeroth Hopf map that is detected by $h_0$.
Beware that $\hsf$ is not equal to $2$; the latter is detected by
$h_0 + \rho h_1$.

\cref{sec:cofibrho} also provides extensive data on the value of the
underlying map on $C_2$-equivariant stable homotopy groups.
We draw particular attention to the sequence
\eqref{eq:SEScofibrho}.  Although simple in appearance, the sequence
is unexpectedly powerful for deducing structure in $C_2$-equivariant
stable homotopy groups.

\cref{sec:charts} presents our main computational results in the standard
visual format of Adams charts.  We provide a detailed key for each chart.

%%%%%%%%%%%%%%%%%%%%%%%%%%%%%%%%%%%%%%%%%%%%%%%%%%%%%%%%%%%%%%%
%%%%%%%%%%%%%%%%%%%%%%%%%%%%%%%%%%%%%%%%%%%%%%%%%%%%%%%%%%%%%%%
\section{Notation}
\label{sec:notn}

We index $C_2$-equivariant stable stems in the form $(s,c)$, 
where $s$ is the stem (i.e. the underlying topological dimension) and
$c$ is the coweight (i.e. the number of trivial $1$-dimensional
representations in a virtual representation, also known as the stem minus
the weight).
Similarly, all homology groups will be indexed as $(s,c)$.

We index our $\Ext$ groups in the form $(s,f,c)$, where $s$
is the stem (i.e., the internal degree minus the homological degree), 
$f$ is the Adams filtration (i.e., the homological degree), and 
$c$ is the coweight.
 
The degrees $s$ and $f$ correspond to
the Cartesian coordinates of a standard Adams chart.
Our grading conventions are guided by practical considerations.
We find that these conventions allow us to systematically organize 
our computations into manageable pieces.
However, beware that the related articles \cites{LowMW,GHIR,BI}
instead grade $\Ext$ groups over $(s,f,w)$, where $w$ is the weight.

Equivariant stable homotopy groups are often graded over the 
real representation ring $RO(G)$.
The  representation ring $RO(C_2)$ of $C_2$ is isomorphic
to $\Z[\sigma]/(\sigma^2 - 1)$,
where $\sigma$ is the $1$-dimensional sign representation.
The main point of translation is that
the representation $p + q \sigma$ is expressed in our 
convention as $(p+q, p)$.
Thus $S^{1,1}$ is the circle with trivial action, whereas 
$S^{1,0}$ is $S^\sigma$. We warn the reader that this is the opposite 
of the convention in \cites{BI,BS,LowMW,GHIR}.

Aside from the change in grading,
we employ the same notation as in \cites{GHIR,C2MW0}. In particular, this includes
notation as follows:
\begin{enumerate}
\item
$\bMCC=\F_2[\tau]$ 
is the motivic homology 
of $\C$ with $\F_2$ coefficients, where $\tau$ has bidegree {$(0,1)$}.
\item
$\bMR=\F_2[\tau,\rho]$ 
is the motivic homology of $\R$ with $\F_2$ coefficients, where $\tau$ and $\rho$ have bidegrees {$(0,1)$ and $(-1,0)$,} respectively.
\item
$\bMC$ is the bigraded equivariant homology of a point with coefficients in the constant Mackey functor $\underline{\F}_2$. See  \cref{sec:BockSS} for a description of this algebra. 
\item \label{NCnotation}
$NC$ is the ``negative cone" part of $\bMC$.  
See  \cref{sec:BockSS} for a precise description.
We use the notation $\frac{\gamma}{\rho^j \tau^k}$ to denote
the unique non-zero element of $NC$ in {degree $(j, -1 - k)$.}
We require that $j \geq 0$ and $k \geq 1$.
\item
$NC_{\rho^n}$ is the $\rho^n$-torsion submodule of $NC$.
In concrete terms, it consists of elements of the form
$\frac{\gamma}{\rho^j \tau^k}$ such that $0 \leq j \leq n-1$
and $k \geq 1$.
\item
$\cAcl$,
$\cACC$, $\cAR$, and $\cAC$ are the 
classical, 
$\C$-motivic, $\R$-motivic, and $C_2$-equivariant mod 2 Steenrod algebras.
Their duals are $\cAcl_*$, $\cACC_*$, $\cAR_*$, and $\cAC_*$.
\item
$\Ext_\cl$ is the bigraded ring $\Ext_{\cAcl_*}(\bF,\bF)$,
i.e., the cohomology of $\cAcl$.
\item
$\Ext_\C$ is the trigraded ring $\Ext_{\cACC_*}(\bMCC,\bMCC)$,
i.e., the cohomology of $\cACC$.
\item
$\Ext_\R$ is the trigraded ring $\Ext_{\cAR_*}(\bMR,\bMR)$,
i.e., the cohomology of $\cAR$.
\item
$\Ext_{C_2}$ is the trigraded ring $\Ext_{\cAC_*}(\bMC,\bMC)$,
i.e., the cohomology of $\cAC$.
\item For any $\cAC_*$-comodule $M$, we write
$\Ext_{C_2}(M)$ for the trigraded $\Ext_{C_2}$-module $\Ext_{\cAC_*}(\bMC,M)$. 
We use similar abbreviations in the $\R$-motivic and $\C$-motivic contexts.
\item
For any $\cAC_*$-comodule $M$, 
let $\Sigma^{t,c} M$ be the shift of $M$ in which
the internal degree is increased by $t$, 
and the coweight is
increased by $c$.
\item
$\Ext_{NC}$ is the $\Ext_{\R}$-module
$\Ext_{\cAR_*}(\bMR,NC)$.
\item
$E^+$ is the $\R$-motivic $\rho$-Bockstein spectral sequence
\[
\Ext_\C[\rho] \Rightarrow \Ext_\R.
\]
See \cref{sec:BockSS}.
\item
$E^-$ is the $\rho$-Bockstein spectral sequence
that converges to $\Ext_{NC}$ (see \cref{sec:BockSS}).
Also, $E^-_{1,\rho^n}$ is the $\rho^n$-torsion submodule of $E^-_1$
(see \cref{E1minus-SES}).
\item
For any $\tau$-torsion class $y$ in $\Ext_\C$,
we define $Qy$, up to some indeterminacy, to be a particular
element of $E_1^-$.
See \cref{Notn:Qclasses} for the exact definition.
\item
$\frac{\F_2[x]}{x^\infty} \{ y \}$ is the infinitely $x$-divisible module
$\colim_n \F_2[x]/x^n$.
In practice, we use this notation only with $x$ equal to $\tau$ or 
equal to $\rho$.
The module $\frac{\F_2[\tau]}{\tau^\infty} \{ \gamma \}$ consists of elements
of the form $\frac{\gamma}{\tau^k}$, with $k \geq 1$, whereas 
$\frac{\F_2[\rho]}{\rho^\infty} \{ \gamma \}$ consists of elements
$\frac{\gamma}{\rho^j}$, with $j \geq 0$.
See \eqref{NCnotation} and \cref{rmk:NCnotation}.
\item 
$S^{0,0}$ is the $C_2$-equivariant sphere spectrum.
\item 
$\rH^{C_2}_{*,*}(X)$ is the $C_2$-equivariant homology of a 
$C_2$-equivariant spectrum $X$ with coefficients in the constant Mackey
functor $\underline{\F}_2$.
\item
$\piC_{*,*}(X)$ is the bigraded $C_2$-equivariant stable homotopy groups of a $C_2$-spectrum, 
completed at $2$
so that the equivariant Adams spectral sequence converges.
In the case of $X=S^{0,0}$, we abbreviate this to  $\piC_{*,*} = \piC_{*,*}(S^{0,0})$. 
\item
$(\ker \rho)_{*,*}$ is the subobject of $\piC_{*,*}$
consisting of elements that are annihilated by $\rho$, and
$(\coker \rho)_{*,*}$ is the quotient of $\piC_{*,*}$
by the image of $\rho$.
\item
$\piR_{*,*}$ are the bigraded 
$\R$-motivic stable homotopy groups, completed at $2$ 
so that the $\R$-motivic Adams spectral sequence converges.
\item
$\picl_*$ are the classical stable homotopy groups, completed at $2$ 
so that the  Adams spectral sequence converges.
\end{enumerate}

\begin{rmk}
\label{rmk:NCnotation}
Beware that our notation for elements of $NC$ is slightly different
from the notation used in other manuscripts.
The relations $\tau^b \cdot \frac{\gamma}{\tau^b} = 0$ 
are the fundamental structure in $NC$.  
The symbol $\gamma$, which does not correspond to an actual element,
has {degree $(0,-1)$}
 and ``represents" these relations.
In practice, our convention makes it easier to state
formulas that describe the behavior of the various
elements $\frac{\gamma}{\tau^b}$.

Our element $\frac{\gamma}{\rho^a \tau^b}$ is often denoted by
$\frac{\theta}{\rho^a \tau^{b-1}}$ elsewhere,
such as in \cite{DI-comparison} \cite{CMay}.
\end{rmk}

The Betti realization map from $\R$-motivic to $C_2$-equvariant homotopy theory
induces a homomorphism from $\bMR$ to $\bMC$.
In our notation, the elements $\tau$ and $\rho$ of $\bMR$ map to the
elements of $\bMC$ of the same name.
The symbols $u$ and $a$ are more traditionally used for
$\tau$ and $\rho$ in the equivariant context
\cite{HHR}*{Definitions 3.11 and 3.12}.

We use
the symbol $\rho$ in four distinct but related contexts: for an element of $\bMR$; for an element
of $\bMC$; for an element of {$\piR_{-1,0}$} that is detected by $\rho$;
and for an element of {$\piC_{-1,0}$} that is detected by $\rho$.

We use the symbol $\hsf$ for 
an element of $\piC_{0,0}$ that is detected
by $h_0$.  This notation reflects that $\hsf$ is the zeroth Hopf map.
The group $\piC_{0,0}$ is the (completion of) the Burnside ring $A(C_2)$, which is a free 
$\Z$-module whose generators are the two $C_2$-orbits, namely the trivial orbit and the free orbit. 
The element $\hsf$ corresponds to the free orbit $C_2/e$.
From the motivic perspective, $\piR_{0,0}$
is the Grothendieck-Witt ring of quadratic forms over $\R$, and $\hsf$
corresponds to the hyperbolic plane.
Beware that $\hsf$ does not equal $2$; the 
latter homotopy element is detected by $h_0 + \rho h_1$.

%%%%%%%%%%%%%%%%%%%%%%%%%%%%%%%%%%%%%%%%%%%%%%%%%%%%%%%%%%%%%%%
%%%%%%%%%%%%%%%%%%%%%%%%%%%%%%%%%%%%%%%%%%%%%%%%%%%%%%%%%%%%%%%

\section{Background on $C_2$-equivariant stable homotopy theory}
\label{background}

\subsection{Mackey functor homotopy groups versus $RO(G)$-graded 
homotopy groups}
\label{subsctn:Mackey-ROG}

We begin with a comparison of two perspectives on equivariant
stable homotopy groups: 
$\Z$-graded homotopy Mackey functors
and $RO(G)$-graded
homotopy groups.  This issue arose in the introduction, and we provide
additional detail here.  Strictly speaking, the rest of the manuscript
does not depend on this discussion, but it adds important conceptual
background to the study of equivariant stable homotopy groups.

First we provide an explicit example that demonstrates that
$RO(G)$-graded homotopy groups are weaker than
$\Z$-graded homotopy Mackey functors.

\begin{eg}
\label{ex:Clover}
Consider the
$C_3$-equivariant Eilenberg--Mac~Lane spectrum $H_{C_3}\underline{\Q}$,
where $\underline{\Q}$ is the constant $C_3$-Mackey functor with value $\Q$. 
The underlying spectrum of $H_{C_3}\underline{\Q}$ is $H\Q$, so there is a corresponding map ${C_3} _+ \smsh H\Q \to H_{C_3} \underline{\Q}$. This map induces an isomorphism on $RO(C_3)$-graded homotopy groups, yet cannot be 
a $C_3$-equivariant
equivalence because the source and target do not agree as underlying spectra.
We are indebted to Clover May for this example.
\end{eg}

However, the $C_2$-equivariant case is somewhat special
because the $RO(C_2)$-graded homotopy groups do detect $C_2$-equivariant weak equivalences. 

\begin{prop}
Let $f: X \xrtarr{} Y$ be a map of $C_2$-equivariant spectra.
The map $f$ induces an isomorphism on 
$\Z$-graded homotopy Mackey functors
if and only if it induces an isomorphism on $RO(C_2)$-graded
homotopy groups.
\end{prop}

\begin{proof}
The definition of $C_2$-equivariant weak equivalences \cite{MM}*{Definition~III.3.2, Theorem~III.6.1} says that if  $f$ induces an isomorphism on
$\Z$-graded Mackey functors,  then $f$ is a
$C_2$-equivariant equivalence. It follows that $f$ induces an isomorphism on $RO(C_2)$-graded homotopy groups.

Now assume that $f$ induces an isomorphism on $RO(C_2)$-graded
homotopy groups.
There is a 
$C_2$-equivariant cofiber sequence
${C_2}_+ \to S^{0,0} \to S^{1,0}$ which describes a $C_2$-CW structure on 
$S^{1,0}$.
(Beware of our grading conventions discussed in \cref{sec:notn};
the sphere associated to the trivial representation is $S^{1,1}$,
and the sphere associated to the sign representation is $S^{1,0}$.)

By mapping 
the $\Sigma^{n,n}$-suspension of
this cofiber sequence into 
$X$ and $Y$,
we obtain a diagram
\[
\begin{tikzcd}
\cdots \ar[r] & \piC_{n+1, n} X \ar[r] \ar[d,"f_*"] & \piC_{n,n} X \ar[r] \ar[d,"f_*"] & {} [{C_2}_+, X]_{n,n} \ar[r] \ar[d,"f_*"] & \cdots \\
\cdots \ar[r] & \piC_{n+1, n} Y \ar[r] & \piC_{n,n} Y \ar[r] & {} [{C_2}_+, Y]_{n,n} \ar[r] & \cdots
\end{tikzcd}
\]
in which the rows are long exact sequences.
The left and middle vertical maps are isomorphisms since they are part
of the data of $RO(C_2)$-graded homotopy groups.
By the five lemma, the right vertical map is also an isomorphism.

The free-forgetful adjunction implies that the right vertical map
is an isomorphism $\picl_n UX \to \picl_n UY$, where $U$ is the forgetful functor.
On the other hand, the trivial-action-fixed-point adjunction implies
that the middle vertical map
is an isomorphism $\picl_n (X^{C_2}) \to \picl_n(Y^{C_2})$, 
where $X^{C_2}$ and $Y^{C_2}$ are the fixed-point spectra of $X$ and $Y$ respectively.
These maps are precisely 
the components of the natural transformation 
$f_*\colon\underline{\pi}_n(X) \to \underline{\pi}_n(Y)$
of Mackey functors,
so $f$ induces an isomorphism of Mackey functor homotopy groups.
\end{proof}

\subsection{The underlying homomorphism and the cofiber of $\rho$}
\label{subsctn:underlying}

We write $U$ for the forgetful functor from $C_2$-equivariant
stable homotopy theory to classical stable homotopy theory.
Note that $U S^{s,c}$ equals $S^s$ for all $s$ and $c$, so
$U$ induces a homomorphism $\piC_{s,c} \xrtarr{} \picl_s$.
First we restate a well-known proposition about this homomorphism
for later use.

\begin{prop}
\label{prop:rho-forget}
There is a long exact sequence
\[ 
\cdots \xrtarr{} \piC_{s+1,c} \xrtarr{\rho} \piC_{s,c} \xrtarr{U} \picl_s  
\xrtarr{} \piC_{s,c-1} \xrtarr{\rho} \piC_{s-1,c-1} \xrtarr{} \cdots.
\]
In particular, an element
 $\alpha$ of $\piC_{s,c}$, 
is divisible by $\rho$ if and only if 
its underlying class $U \alpha $ in $\picl_s$ is zero.
\end{prop}

Recall the $C_2$-equivariant cofiber sequence
\begin{equation}
\label{eq:cofiber rho}
S^{0,-1} \xrtarr{} {C_2}_+ \xrtarr{} S^{0,0} \xrtarr{\rho} S^{1,0}
\end{equation}
that displays ${C_2}_+$ as a model for (a suspension of) the cofiber of $\rho$.
The free-forgetful adjunction implies that the
$C_2$-equivariant cohomotopy
$[{C_2}_+, S^{0,0}]^{C_2}_{*,*}$ of ${C_2}_+$ 
is isomorphic to $\picl[\tau^{\pm 1}]$.
Here we are considering $\picl[\tau^{\pm 1}]$
as a bigraded object by adjoining a formal parameter $\tau$ in
coweight $1$. 

\begin{rmk}
The Wirthm\"{u}ller isomorphism implies that ${C_2}_+$ is its own Spanier-Whitehead
dual.  Consequently, its cohomotopy is isomoprhic to its more familiar homotopy.
However, the cohomotopy of ${C_2}_+$ is the object that is naturally
relevant here.
\end{rmk}

\begin{prop}
\label{prop:Adams-C2-cohtpy}
The $C_2$-equivariant Adams spectral sequence for
$[{C_2}_+, S^{0,0}]^{C_2}_{*,*}$ is isomorphic to the
classical Adams spectral sequence for $\picl_*$ tensored
with $\Z[\tau^{\pm 1}]$.
\end{prop}

\begin{proof}
Let $\overline{H}$ be the fiber of the $C_2$-equivariant
map $S^{0,0} \xrtarr{} H$, where $H$ is the
$C_2$-equivariant Eilenberg-Mac Lane object that represents cohomology
with coefficients in the constant Mackey functor $\underline{\mathbb{F}}_2$.
Then we can build a $C_2$-equivariant Adams resolution $R_* \xrtarr{} S^{0,0}$
in which $R_k$ is $\overline{H}^{\wedge k}$.

The functor $U$ preserves cofiber sequences, respects smash products,
and takes $C_2$-equivariant
Eilenberg-Mac Lane objects to classical Eilenberg-Mac Lane objects.
Consequently,
$UR_* \xrtarr{} S^0$ is the classical Adams resolution in which
$UR_k$ is $\overline{H_{\cl}}^{\wedge k}$.  Here
$\overline{H_{\cl}}$ is the fiber of the classical map
$S^0 \xrtarr{} H_{\cl}$, and
$H_{\cl}$ is the classical Eilenberg-Mac Lane object that
represents ordinary $\mathbb{F}_2$-cohomology.

The $C_2$-equivariant Adams spectral sequence for
$[{C_2}_+, S^{0,0}]^{C_2}_{*,*}$ derives from an exact couple
$\oplus [{C_2}_+, R_*]_{*,*}$.  By the free-forgetful adjunction,
this exact couple is isomorphic to the exact couple
$\oplus [S^0, UR_*]_{*,*}$, from which the classical Adams
spectral sequence is derived.
\end{proof}

Let $(\ker \rho)_{*,*}$ be the subobject of $\piC_{*,*}$
consisting of elements that are annihilated by $\rho$.
Let $(\coker \rho)_{*,*}$ be the quotient of $\piC_{*,*}$
by the image of $\rho$.

\begin{prop}
\label{prop:rho-SES}
There is a short exact sequence
\[
0 \rightarrow (\coker \rho)_{*,*} \xrightarrow{U} \picl_*[\tau^{\pm 1}] 
\rightarrow (\ker \rho)_{*,*-1} \rightarrow 0
\]
of $\piC_{*,*}$-modules, in which both maps preserve the Adams filtration.
\end{prop}

By ``preserve'', we mean that the maps in the short exact sequence may increase, but cannot decrease, the Adams filtration.

\begin{proof}
The sequence \cref{eq:cofiber rho} contravariantly induces a long exact
sequence
\[
\label{eq:les-rho}
\cdots \xrtarr{}
\piC_{*-1,*-1} \xltarr{\rho} 
\piC_{*,*-1} \xltarr{} 
[{C_2}_+, S^{0,0}]_{*,*} \xltarr{} 
\piC_{*,*} \xltarr{\rho}
\piC_{*+1,*} \xltarr{}
\cdots.
\]
Combined with our description of 
$[{C_2}_+, S^{0,0}]$, we obtain short exact sequences
\[
(\coker \rho)_{*,*} \xrtarr{} 
\picl_{*} \xrtarr{} 
(\ker \rho)_{*,*-1}.
\]

The claim about Adams filtrations follows from contravariant
functoriality of the
$C_2$-equivariant Adams spectral sequence applied to the maps
$S^{0,0} \xrtarr{} {C_2}_+$ and
${C_2}_+ \xrtarr{} S^{0,1}$ in 
sequence \cref{eq:cofiber rho}.
\end{proof}

We will rely on \cref{prop:rho-SES} many times in \cref{sec:Adams} 
to establish Adams differentials in the $C_2$-equivariant Adams
spectral sequence and also in \cref{sec:cofibrho,sec:hiddenAdams} to establish
multiplications in $\piC_{*,*}$ that are hidden in the
$C_2$-equivariant Adams spectral sequence.

\subsection{$\tau$-periodicity}
\label{subsctn:tau-periodicity}
We next turn to $\tau$-periodicity, as displayed in \cref{fig:InterestingRegion}.
Bredon showed \cite{B} 
that the $C_2$-equivariant stable stems exhibit a periodicity in the coweight direction (see also \cite{L}*{Proposition 6.1}).
We introduce the following notation in order to state the periodicity results.

\begin{notn}
\label{notn:ps}
For any $s\geq 0$, let $p_s=2^{\varphi(s+1)}$, 
 where $\varphi(s+1)$ is the number of positive integers $j \leq s+1$ such that $j\equiv 0,1,2, \text{or } 4\pmod8$.
\end{notn}

\begin{center}
\begin{tabular}{c|ccccccccccccccccc}
$s$ & 0 & 1 & 2 & 3 & 4 & 5 & 6 & 7 & 8 & 9 & 10 & 11 & 12 & 13 & 14 & 15 & 16 
\\ \hline
$p_s$ & 2 & 4 & 4 & 8 & 8 & 8 & 8 & 16 & 32 & 64 & 64 & 128 & 128 & 128 & 128 & 256 & 512 
\end{tabular}
\end{center}

Landweber
demonstrated a periodicity in the elements that are annihilated
by a power of $\rho$.
In negative coweight, all classes are $\rho$-power torsion, and
the periodicity result takes the form:

\begin{thm}\cite{B}\label{Landw1} Let $s\geq 0$ and $c\leq -2$. 
There is a periodicity isomorphism
\[ T^{-p_s}\colon \piC_{s,c} \iso \piC_{s,c-p_s},\]
where $p_s$ is as in \cref{notn:ps}.
\end{thm}

For example, for $c \leq -2$ we have periodicity operators
\[ 
	T^{-2}\colon \piC_{0,c} \iso \piC_{0,c-2},
	\qquad
	T^{-4}\colon \piC_{1,c} \iso \piC_{1,c-4},
\] 
\[ 
	T^{-4}\colon \piC_{2,c} \iso \piC_{2,c-4},
	\qquad
	T^{-8}\colon \piC_{3,c} \iso \piC_{3,c-8}.
\]

In positive coweights, the periodicity result takes the form:

\begin{thm}\cite{B}\label{Landw2} Let $s\geq 0$ and 
$c \geq \frac{s-1}2$.
There is a periodicity isomorphism
$T^{p_s}$ from the $\rho$-power torsion subgroup of $\piC_{s,c}$ to
the $\rho$-power torsion subgroup of $\piC_{s,c+p_s}$,
where $p_s$ is as in \cref{notn:ps}.
\footnote{In \cite{BS}*{Theorem~7.4}, this is stated in the smaller region $s>2w$, which corresponds to $c > \frac{s}2$.}
\end{thm}

Moreover, the periodicity identifies periodic stems in negative coweight with the periodic stems in positive coweight:

\begin{prop}
\label{prop:gapbridging}
Let $s \geq 0$ and $c \leq -2$. If the integer $n$ satisfies $c + n p_s \geq \frac{s-1}2$, then there is a periodicity isomorphism
from 
$\piC_{s,c}$ to the $\rho$-power torsion subgroup of $\piC_{s,c+n p_s}$.
\end{prop}

\begin{pf}[Proof of \cref{Landw1}, \cref{Landw2}, and \cref{prop:gapbridging}]
The cofiber sequence 
\[(EC_2)_+ \rtarr S^{0,0} \rtarr \widetilde{EC_2}\]
gives rise to a long exact sequence
\[\cdots \rtarr \picl_{c+1} \rtarr \piC_{s,c}((EC_2)_+) \rtarr \piC_{s,c} \xrtarr{\Phi} \picl_{c} \rtarr \dots\]
Bredon established (\cite{B}, but see also \cite{L}*{Proposition~6.1}) the periodicity of the groups $\piC_{s,c}((EC_2)_+)$.
If $c\leq -2$, then both outside groups vanish, giving an isomorphism $\piC_{s,c}((EC_2)_+) \iso \piC_{s,c}$. 
On the other hand, if 
$c \geq \frac{s-1}2$,
then the geometric fixed points homomorphism 
$\Phi\colon \piC_{s+1,c+1} \rtarr \picl_{c+1}$ is surjective by \cite{B}.
In computational terms, $\Phi$ has the effect of
inverting $\rho$, 
so $\piC_{s,c}((EC_2)_+) \rtarr \piC_{s,c}$ is an identification with the $\rho$-power torsion.
\end{pf}

We use \cref{prop:gapbridging} later in the proof of \cref{prop:GArhotauPermCycles} to establish that certain classes are permanent cycles in the $C_2$-equivariant Adams spectral sequence.

%%%%%%%%%%%%%%%%%%%%%%%%%%%%%%%%%%%%%%%%%%%%%%%%%%%%%%%%%%%%%%%
%%%%%%%%%%%%%%%%%%%%%%%%%%%%%%%%%%%%%%%%%%%%%%%%%%%%%%%%%%%%%%%
\section{The equivariant Bockstein spectral sequence}
\label{sec:BockSS}

Recall that $\bMR \iso \bF[\tau,\rho]$ and that 
$\bMC \iso \bMR \oplus NC$, where
$NC$ is the negative cone. 
\cref{fig:M2C} displays $\bMC$.
Elements of $NC$ are of the form $\frac{\ga}{\rho^j\tau^k}$, where $j\geq 0$ and $k\geq 1$.

The equivariant dual Steenrod algebra $\cAC_*$ is a Hopf algebroid over $\bMC$.
As an $\bMC$-algebra, it has generators $\tau_0, \tau_1, \ldots$
and $\xi_1, \xi_2, \ldots$, subject to the relations
\[
\tau_i^2 = \tau \xi_{i+1} + \rho \tau_{i+1} + \rho \tau_0 \xi_{i+1}.
\]
The right unit is given by
\[ \eta_R(\rho) = \rho, \qquad \eta_R(\tau) = \tau + \rho \tau_0,\]
and
\begin{equation}
\label{eq:etaR}
\eta_R \left( \frac{\ga}{\rho^j \tau^k} \right) = \frac{\ga}{\rho^j \tau^k}  \cdot \left[  \sum_{i\geq 0} \left( \frac{\rho}\tau \tau_0 \right)^i \right]^k
\end{equation}
\cite{GHIR}*{(2-3)}, \cite{HK}.

\begin{figure}
\caption{{The $C_2$-equivariant homology of a point}}
\label{fig:M2C}
\begin{center}
\begin{tikzpicture}[scale=0.5]
\draw[->,thick,xshift={-1.85ex},yshift={-1.85ex}] (-4,-4) -- (-4,4+0.8);
\draw[->,thick,xshift={-1.85ex},yshift={-1.85ex}] (-4,-4) -- (3+0.8,-4);
\filldraw[color=red!40, thick, rounded corners, xshift=1.7ex] (0, -0.2) -- (0, 4.2) -- (-4.45,4.2) -- (-4.45,-0.2) -- cycle;
\filldraw[color=red!40, thick, xshift=1.7ex] (0, 1) -- (0, 4.2) -- (-4.45,4.2) -- (-4.45,-0.2) -- (-2,-0.2) -- cycle;
\draw[color=red!60, thick,  xshift=1.7ex] (0, -0) -- (0, 4.2) -- (-4.45,4.2) -- (-4.45,-0.2) -- (-0.25,-0.2) to [bend right=30] cycle;
\filldraw[color=green!40, thick, rounded corners, xshift=-1.5ex] (0, -1.8) -- (3.45, -1.8) -- (3.45,-4.2) -- (0,-4.2) -- cycle;
\filldraw[color=green!40, thick,  xshift=-1.5ex] (1, -1.8) -- (3.45, -1.8) -- (3.45,-4.2) -- (0,-4.2) -- (0,-2) -- cycle;
\draw[color=olive!100,thick,  xshift=-1.5ex] (0,-2) -- (0,-4.2) -- (3.45,-4.2) -- (3.45,-1.8) -- (0.2,-1.8) to [bend right=30] (0,-2);
\draw[gray] (-4,-4) grid (3,4);

\foreach \y in {-4,...,4} {
    \node at (-4.8,\y) {\y};
}
\foreach \x in {-4,...,3} {
    \node at (\x,-4.7) {\x};
}
\node at (3.75,-4.3) {$s$};
\node at (-4.3,4.75) {$c$};

\foreach \x in {0,...,-4} {
  \foreach \y in {0,...,4} {
    \filldraw (\x,\y) circle (2.5pt);
  }
}
\node at (1.5,2.5) {$\bMR$};
\node at (-2,-3) {$NC$};
\node[xshift={7pt}] at (0,0) {$1$};
\node[xshift={7pt}] at (0,1) {$\tau$};
\node[xshift={-2.5pt}, yshift={-7.5pt}] at (-1,0) {$\rho$};
\node[xshift={-7pt}] at (0,-2) {$\frac{\ga}{\tau}$};
\node[xshift={-3pt},yshift={10pt}] at (1,-2) {$\frac{\ga}{\rho\tau}$};
\node[xshift={-9pt}] at (0,-3) {$\frac{\ga}{\tau^2}$};

\foreach \x in {0,...,3} {
  \foreach \y in {-2,...,-4} {
    \filldraw (\x,\y) circle (2.5pt);
  }
}

\end{tikzpicture}
\end{center}
\end{figure}
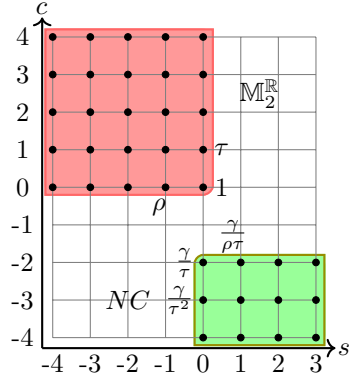

As discussed in \cite{GHIR}*{Proposition~3.1}, there is a $\rho$-Bockstein spectral sequence converging to the groups $\Ext_{C_2}$.
The splitting of $\bMC$ extends to a splitting as $\cAR_*$-comodules and 
gives rise to a splitting $\Ext_{C_2} \iso \Ext_\R \oplus \Ext_{NC}$ \cite{GHIR}*{Proposition~2.2}, as well as a corresponding splitting of the Bockstein $E_1$-page 
\[ E_1 \iso E_1^+ \oplus E_1^- \Rightarrow \Ext_\R \oplus \Ext_{NC}.\]
The $E_1^+$-page is the $E_1$-page of the $\R$-motivic Bockstein spectral sequence, as discussed in \cite{LowMW} and \cite{BI}.

The $E_1^-$-page consists entirely of families of
infinitely $\rho$-divisible elements.  Every element of
$E_1^-$ is $\rho^k$-torsion for some $k$.
In less formal terms, $\rho$-towers extend infinitely in one direction,
but they always terminate in the other direction.

\begin{defn}
For $k \leq -1$, the $k$th stage of the $\rho$-filtration on $E_1^-$
consists of all elements $a$ such that $\rho^{-k} a$ equals zero.
\end{defn}

We frequently refer to the $\rho$-filtration of a specific element
$a$ in $E_1^-$.  In this context, we mean the 
least value of
$k$ such that $a$ belongs to the $k$th stage of the filtration.

\begin{rmk}
The filtration grading of $E_1^-$ is concentrated in negative degrees.
We will use the phrases ``higher filtration" and ``lower filtration"
in the strict numerical sense, e.g. filtration $-1$ is higher than
filtration $-2$.
The $\rho$-filtration is decreasing in the sense that
the $(k+1)$st stage is a subset of the $k$th stage.
This is compatible with the $\rho$-filtration on $E_1^+$, 
where the $k$th stage consists of all multiples of $\rho^k$.
As $k$ tends to $-\infty$, the stages of the filtration on $E_1^-$ 
get larger, and their union is equal to $E_1^-$.
\end{rmk}

\begin{prop}
\label{prop:E-minus-structure}
\mbox{}
\begin{enumerate}
\item
If $a$ is nonzero and annihilated by $\rho$ in $E_r^-$,
then the $\rho$-filtration of $a$ is $-1$.
\item
If an element of $E_r^-$ 
has $\rho$-filtration less than or equal to $-r$,
then it is $\rho$-divisible.
\end{enumerate}
\end{prop}

We combine the distinct statements of \cref{prop:E-minus-structure}
into one result because we prove them simultaneously by induction.

\begin{proof}
Write (1:$r$) and (2:$r$) for
each of the two statements applied to the $E_r^-$ page.

Statement (2:$r$) follows from (1:$r-1$) and (2:$r-1$).
Suppose that $a$ has $\rho$-filtration less than or equal to $-r$
in $E_r^-$.  By (2:$r-1$), $a = \rho a'$ in $E_{r-1}^-$ for some $a'$.
Then $a$ is also $\rho$-divisible in $E_r^-$, unless there is a 
non-zero differential $d_{r-1}(a') = b$ such that $\rho b = 0$.
In that case, (1:$r-1$) would imply that
the $\rho$-filtration of $b$ must be $-1$, so
the $\rho$-filtration of $a'$ would be $-r$, and the $\rho$-filtration
of $a$ would be $-r+1$.  By contradiction, $d_{r-1}(a')$ must be zero, and
$a$ is $\rho$-divisible in $E_r^-$.

Statement (1:$r$) follows from (1:$r-1$) and (2:$r-1$).
Suppose that $\rho a = 0$ in $E_r^-$.
If the $\rho$-filtration of $a$ is less than $-1$, then
$\rho a$ is non-zero in $E_{r-1}^-$ by (1:$r-1$).
Therefore, there exists a non-zero differential
$d_{r-1}(b) = \rho a$ in $E_{r-1}^-$.
The $\rho$-filtration of $b$ is at most $-r$, so $b = \rho b'$
by (2:$r-1$).  Then $d_{r-1}(b') - a$ is annihilated by $\rho$
in $E_{r-1}^-$.  Again by (1:$r-1$), we conclude that
$d_{r-1}(b') = a$, so $a = 0$ in $E_r^-$.
\end{proof}

\begin{prop}
\label{prop:E-minus-structure3}
If a nonzero element of $E_r^-$ is $\rho$-divisible, then it is uniquely
$\rho$-divisible.
\end{prop}

\begin{proof}
Suppose that $a$ has filtration $f$, so $f \leq -1$.
If $a = \rho b$
and $a = \rho b'$, then $b - b'$ is annihilated by $\rho$
and has filtration $f-1 \leq -2$.
According to \cref{prop:E-minus-structure}(1), $b - b'$ must be zero.
\end{proof}

\begin{prop}
\label{prop:E-minus-structure4}
If $d_r(a)$ is non-zero in $E_r^-$, then
$a$ and $d_r(a)$ are both $\rho$-divisible.
\end{prop}

\begin{proof}
If $d_r(a)$ is non-zero, then the $\rho$-filtration of $a$ is
at most $-r-1$ since $d_r$ increases $\rho$-filtration by $r$.
Then \cref{prop:E-minus-structure}(2) implies that $a$ is $\rho$-divisible.
Let $a = \rho a'$.  Then $d_r(a) = \rho d_r(a')$, so $d_r(a)$
is $\rho$-divisible as well.
\end{proof}

\begin{rmk}
\cref{prop:E-minus-structure}(1) and \cref{prop:E-minus-structure4} are dual to analogous results
about $E_r^+$, as stated in \cite{LowMW}*{Lemma~3.4} and
\cite{BI}*{Proposition~5.1}.  
Namely,
\begin{enumerate}
\item
if $x$ is not divisible by $\rho$ in $E_r^+$, then the $\rho$-filtration
of $x$ is $0$.
\item
if $d_r(x)$ is non-zero in $E_r^+$, then both $x$ and $d_r(x)$
are $\rho$-free.
\end{enumerate}
\end{rmk}

\begin{rmk}
\cref{prop:E-minus-structure3} 
shows that 
unique $\rho$-divisibility is common in $E_r^-$.
If $a$ is $\rho^k$-divisible, then we may unambiguously
use the notation $\frac{1}{\rho^k} a$ for an element
such that $\rho^k \cdot \frac{1}{\rho^k} a = a$.
\end{rmk}

\begin{rmk}
The technical statements of \cref{prop:E-minus-structure}, \cref{prop:E-minus-structure3} and \cref{prop:E-minus-structure4} obscure
the practical significance of the results.  In less formal terms,
the propositions describe the structure of $E_r^-$ as an
$\F_2[\rho]$-module.  There are two types of summands:
\begin{enumerate}
\item
Infinitely divisible summands consisting of elements
$\left\{ a, \frac{1}{\rho} a, \frac{1}{\rho^2} a, \ldots \right\}$
such that the $\rho$-filtration of $a$ is $-1$.
\item
Finite sequences of elements
$\left\{ a, \frac{1}{\rho} a, \ldots, \frac{1}{\rho^s} a \right\}$,
where the $\rho$-filtration of $a$ is $-1$, 
the element $\frac{1}{\rho^s} a$ is not $\rho$-divisible,
and $s < r$.
\end{enumerate}
The $\F_2[\rho]$-modules of the second type are in one-to-one
correspondence with infinitely divisible families of $d_s$ differentials.
The elements $\frac{1}{\rho^k} a$ for $k \geq s$ support non-zero $d_s$
differentials.
\end{rmk}

The following is the analogue of \cite{LowMW}*{Proposition~3.2} for the negative cone.

\begin{prop} 
\label{DiffsGamma}
There are Bockstein differentials
\[ 
d_1 \left( \frac{\ga}{\rho\tau}  \right) = \frac{\ga}{\tau^{2}} \cdot h_0
\]
and
\[ 
d_{2^n} \left( \frac{\ga}{\rho^{2^n}\tau^{2^n}}  \right) = \frac{\ga}{\tau^{2^n+2^{n-1}}} \cdot h_n
\]
for $n\geq 1$.
\end{prop}

\begin{pf}
This follows from the formula for the right unit given in \cref{eq:etaR}. Indeed, 
\[
\eta_R \left( \frac{\ga}{\rho \tau} \right) = \frac{\ga}{\rho\tau} + \frac{\ga}{\tau^2} \cdot \tau_0.
\]
It follows that, in the cobar complex, we have $d \left( \frac{\ga}{\rho\tau}  \right) = \frac{\ga}{\tau^{2}} \otimes  \tau_0$. As $h_0$ is represented by $\tau_0$ in the cobar complex, this establishes the formula for $d_1$. Similarly, for appropriate values of $x$ and $y$ we have
\[\begin{split} \eta_R \left( \frac{\ga}{\rho^{2^n} \tau^{2^n}} \right) 
&= \frac{\ga}{\rho^{2^n} \tau^{2^n}}  \left( 1 + \frac{\rho^{2^n}}{\tau^{2^n}} \tau_0^{2^n} + \rho^{2^{n+1}} \cdot x  \right) \\
&= \frac{\ga}{\rho^{2^n} \tau^{2^n}}  \left( 1 + \frac{\rho^{2^n}}{\tau^{2^{n-1}}} \xi_1^{2^{n-1}} + \rho^{2^n+2^{n-1}} \cdot y  \right) \\
&= \frac{\ga}{\rho^{2^n} \tau^{2^n}}  + \frac{\ga}{\tau^{2^n+2^{n-1}}} \xi_1^{2^{n-1}}.
\end{split}\]
This gives rise to the cobar complex differential $d \left( \frac{\ga}{\rho^{2^n} \tau^{2^n} }  \right) = \frac{\ga}{\tau^{2^n + 2^{n-1}}} \otimes  \xi_1^{2^{n-1}}$. 
Finally, use that $h_n$ is represented by
$\xi_1^{2^{n-1}}$ for $n \geq 1$.
\end{pf}

\begin{rmk}
\label{rmk:DiffsGamma}
The differential
\[ 
d_1 \left( \frac{\ga}{\rho\tau}  \right) = \frac{\ga}{\tau^{2}} \cdot h_0
\]
from \cref{DiffsGamma} implies that
\[ 
d_1 \left( \frac{\ga}{\rho\tau^{2k+1}}  \right) = \frac{\ga}{\tau^{2k+2}} \cdot h_0
\]
for all $k \geq 1$.
This follows from the Leibniz rule applied to the relation
\[
\frac{\gamma}{\rho \tau} = \frac{\gamma}{\rho \tau^{2k+1}} \cdot \tau^{2k}.
\]
Similarly, we have
\[ 
d_{2^n} \left( \frac{\ga}{\rho^{2^n}\tau^{2^{n+1}k + 2^n}}  \right) = \frac{\ga}{\tau^{2^{n+1} k + 2^n+2^{n-1}}} \cdot h_n
\]
for $n\geq 1$.
\end{rmk}

We next show that 
every nonzero element of $\Ext_{C_2}$ is not infinitely divisible by $\rho$.

\begin{prop}
\label{colocalVanishes} 
The inverse limit $\lim\limits_\rho \Ext_{{C_2}}$ vanishes.
\end{prop}

\begin{pf}
The argument is the same as \cite{GHIR}*{Proposition~5.2}. 
It suffices to show that the cobar complex $\mathrm{coB}^*_{\F_2[\tau]}\left(\frac{\F_2[\tau]}{\tau^\infty}\{\gamma\}, \F_2[\tau,x]\right)$ is acyclic, where
\[ 
\eta_R\left(\frac{\gamma}{\tau^k}\right) = 
\frac{\gamma}{\tau^k} \left[ \sum_{i\geq 0} \left( \frac{x}{\tau} \right)^i \right]^k.
\]
Filtering by powers of $x$, we have
\[ E_1 \iso \frac{\F_2[\tau]}{\tau^\infty} \{\gamma\} \otimes_{\F_2} \F_2[v_0,v_1,\dots],\]
where $v_n=[x^{2^n}]$. We have differentials 
\[d_{2^n}\left( \frac{\gamma}{\tau^{2^{n+1}k-2^n} }\right) = \frac{\gamma}{\tau^{2^{n+1}k}} v_n\]
for all $n\geq 0$,
and nothing survives the spectral sequence.
\end{pf}

\cref{colocalVanishes} is a powerful tool for deducing differentials
in $E_1^-$.  It has the following immediate consequence.
Let $x$ be any class in $E_1^-$.
Then either:
\begin{itemize}
\item
$x$ is the target of some Bockstein differential; or
\item
$x$ is of the form $\rho^k \cdot y$, where $y$ supports a Bockstein differential.
\end{itemize}
In informal terms, $x$ supports a Bockstein differential ``up to $\rho$-divisibility" in the latter case.

In practice, we often find elements $x$ in $E_1^-$ that cannot
be hit by Bockstein differentials.  Then we know that $x$ supports a differential
up to $\rho$-divisibility, and it is often possible to determine the exact 
value of this differential by a process of elimination.

Similarly, we often find elements $x$ in $E_1^-$ that cannot
support differentials, even up to $\rho$-divisibility.  Then we know
that $x$ must be hit by a differential, and it is often possible
to determine the exact source of this differential by a process of elimination.

This situation contrasts to the case of $E_1^+$, where many infinite $\rho$-towers survive the Bockstein spectral sequence.  However,
the surviving $\rho$-towers are well-controlled by \cite{LowMW}*{Theorem 4.1}.

%%%%%%%%%%%%%%%%%%%%%%%%%%%%%%%%%%%%%%%%%%%%%%%%%%%%%%%%%%%%%%%
%%%%%%%%%%%%%%%%%%%%%%%%%%%%%%%%%%%%%%%%%%%%%%%%%%%%%%%%%%%%%%%
\section{The negative cone and $\tau$-periodicity}
\label{sec:NCPeriodic}

Periodicity with respect to the element $\tau$ plays an important role
in our computations.
For any $n > 0$, the {homology} of the cofiber $S^{0,0}/\rho^n$ is $\tau$-periodic \cite{CMay}.
Indeed, 
$\rH^{C_2}_{*,*}\left( S^{0,0}/\rho^n \right)$ is isomorphic to
\[ 
\frac{\bF[\rho, \tau^{\pm 1}]}{\rho^n}
 = \frac{\bMR[\tau^{-1}]}{\rho^n}
\]

\begin{lem}
\label{lem:ExtAgreesCofiberRho}
There is an isomorphism
\[
\Ext_{C_2} \left( \frac{\bMC[\tau^{-1}]}{\rho^n} \right) \iso 
\Ext_\R \left( \frac{\bMR[\tau^{-1}]}{\rho^n} \right). \]
\end{lem}

\begin{proof}
The coaction map for the $\cAC_*$-comodule $\rH^{C_2}_{*,*}(S^{0,0}/\rho^n) $ has target
\[ 
\cAC_* \otimes_{\bMC}  \frac{\bMR[\tau^{-1}]}{\rho^n} \iso  
\frac{\cAC_*}{\rho^n} \otimes_{\bMC/\rho^n}  \frac{\bMR[\tau^{-1}]}{\rho^n}.
\]
We have an isomorphism of Hopf algebroids
\[ 
\left( \frac{\bMC}{\rho^n}, \frac{\cAC_*}{\rho^n} \right) \iso 
\left( \frac{\bMR}{\rho^n}, \frac{\cAR_*}{\rho^n} \right)
\]
and therefore an isomorphism of the cobar complexes that compute
the $\Ext$ groups.
\end{proof}

Let us write $ NC_{\rho^n}$ for the $\rho^n$-torsion submodule of $NC$. 
The submodules $NC_{\rho^r}$ are sometimes easier to work with than $NC$ itself.
The following \cref{prop:colim-NCrho} shows that in any given computational
situation, there is no harm in 
restricting to $NC_{\rho^r}$ for $r$ sufficiently large.

\begin{prop}
\label{prop:colim-NCrho}
The inclusion maps $NC_{\rho^r} \xrtarr{} NC$
induce an isomorphism
\[
\Ext_{NC} \cong \colim_r \Ext_\R(NC_{\rho^r}). 
\]
\end{prop}

\begin{proof}
The cobar complex that computes $\Ext_{NC}$ is the colimit of the 
cobar complexes that compute $\Ext_\R(NC_{\rho^r})$,
and homology commutes with filtered colimits.
\end{proof}

We have a short exact sequence 
\begin{equation}
\label{cofiberrho^nSES}
 0 \rtarr \frac{\bMR}{\rho^n} \xrtarr{i} \frac{\bMR[\tau^{-1}]}{\rho^n} \xrtarr{q} \Sigma^{1-n,1}  NC_{\rho^n} \rtarr 0
  \end{equation}
of $\cAR_*$-comodules,
where $q(\tau^{-k}) = \frac{\ga}{\rho^{n-1}\tau^k}$.
See \cref{HomologySmodrho4} for a visualization of the case $n=4$.
An alternative description for the cokernel is
\begin{equation} \label{eq:negcone_altdesccript}
\Sigma^{1-n,1}  NC_{\rho^n} \iso \frac{\bMR}{\tau^\infty,\rho^n}. 
\end{equation}

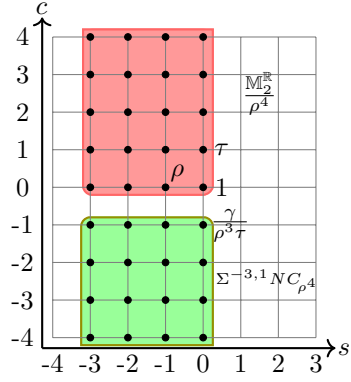
\begin{figure}
\caption{The homology ring $\rH^{C_2}_{*,*}(S^{0,0}/\rho^4) \iso 
\frac{\bF[\rho, \tau^{\pm 1}]}{\rho^4}$, displaying the short exact sequence \eqref{cofiberrho^nSES}}
\label{HomologySmodrho4}
\begin{tikzpicture}[scale=0.5]
\draw[->,thick,xshift={-1.85ex},yshift={-1.85ex}] (-4,-4) -- (-4,4+0.8);
\draw[->,thick,xshift={-1.85ex},yshift={-1.85ex}] (-4,-4) -- (3+0.8,-4);
\filldraw[color=red!40, thick, rounded corners, xshift=1.7ex] (0, -0.2) -- (0, 4.2) -- (-3.45,4.2) -- (-3.45,-0.2) -- cycle;
\filldraw[color=red!40, thick, xshift=1.7ex] (0, 1) -- (0, 4.2) -- (-3.45,4.2) -- (-3.45,1)  -- cycle;
\draw[color=red!60, thick,  xshift=1.7ex] (0, -0) -- (0, 4.2) -- (-3.45,4.2) -- (-3.45,-0) to [bend right=30] (-3.25,-0.2) -- (-0.25,-0.2) to [bend right=30] cycle;
\filldraw[color=green!40, thick, rounded corners, xshift=1.7ex] (-3.5, -0.8) -- (0, -0.8) -- (0,-4.2) -- (-3.5,-4.2) -- cycle;
\filldraw[color=green!40, thick,  xshift=1.7ex] (-3.5, -3) -- (0, -3) -- (0,-4.2) -- (-3.5,-4.2) -- cycle;
\draw[color=olive!100,thick,  xshift=1.7ex] (-3.5,-1) -- (-3.5,-4.2) -- (0,-4.2) -- (0,-1) to [bend right=30] (-0.2,-0.8) -- (-3.3,-0.8) to [bend right=30] (-3.5,-1);
\draw[gray] (-4,-4) grid (3,4);

\foreach \y in {-4,...,4} {
    \node at (-4.8,\y) {\y};
}
\foreach \x in {-4,...,3} {
    \node at (\x,-4.7) {\x};
}
\node at (3.75,-4.3) {$s$};
\node at (-4.3,4.75) {$c$};

\foreach \x in {0,...,-3} {
  \foreach \y in {0,...,4} {
    \filldraw (\x,\y) circle (2.5pt);
  }
}
\node at (1.5,2.5) {$\frac{\bMR}{\rho^4}$};
\node at (1.7,-2.5) {\tiny$\Sigma^{-3,1}NC_{\!\rho^4}$};
\node[xshift={7pt}] at (0,0) {$1$};
\node[xshift={7pt}] at (0,1) {$\tau$};
\node[xshift={4.5pt}, yshift={5.5pt}] at (-1,0) {$\rho$};
\node[xshift={-4pt}] at (1,-1) {$\frac{\ga}{\rho^3\tau}$};

\foreach \x in {-3,...,0} {
  \foreach \y in {-1,...,-4} {
    \filldraw (\x,\y) circle (2.5pt);
  }
}

\end{tikzpicture}
\end{figure}

The short exact sequence {\eqref{cofiberrho^nSES}} of 
$\cAR_*$-comodules
yields a long exact sequence
\begin{equation}
\label{cofiberrho^nLESExt}
\rtarr \Ext_\R^{s,f,c} \left( \frac{\bMR}{\rho^n} \right) \xrtarr{i_*} 
\Ext_\R^{s,f,c} \left( \frac{\bMR[\tau^{-1}]}{\rho^n} \right) \xrtarr {q_*}
\Ext_\R^{s+n-1,f,c-1} \left(  NC_{\rho^n} \right) \xrtarr{\delta}.
 \end{equation}
If we forget to $\bMCC$-modules, we obtain
\begin{equation}
\label{cofiberrho^nLESBockE1}
\to \frac{\Ext_\C^{s,f,c}(\bMCC)[\rho]}{\rho^n} \xrtarr{i_*} 
\frac{\Ext_\C^{s,f,c} ( \bMCC[\tau^{-1}] )[\rho]}{\rho^n} \xrtarr {q_*}
\frac{\Ext_\C^{s,f,c}(  \bMCC/\tau^\infty)[\rho]}{\rho^n} \xrtarr{\delta},
 \end{equation}
in which the objects are Bockstein $E_1$-pages for
the $\Ext$ groups in \cref{cofiberrho^nLESExt}.

Let $E_{1,\rho^n}^{-}$ denote the $\rho^n$-torsion in $E_1^{-}$.

\begin{lem}
\label{E1minus-SES}
There is a short exact sequence
\begin{equation} \label{E1negsplitting}
\coker(i_*) \xrtarr{q_*}
E_{1,\rho^n}^{-}
\xrtarr{\delta}
\ker(i_*).
\end{equation}
\end{lem}

\begin{proof}
The  long exact sequence \cref{cofiberrho^nLESBockE1} gives a short exact sequence
\[
\coker(i_*) \rightarrow
\frac{\Ext_\C(  \bMCC/\tau^\infty)[\rho]}{\rho^n} \rightarrow
\ker(i_*).
\]
The isomorphism \eqref{eq:negcone_altdesccript} gives an identification of $\Ext_\C(  \bMCC/\tau^\infty)[\rho]/\rho^n$ with the $\rho^n$-torsion in $E_1^{-}$.
\end{proof}

\begin{rmk}
The sequence of \cref{E1minus-SES} is 
the restriction to $\rho^n$-torsion of
the short exact sequence for
$E_1^-$ given in \cite{GHIR}*{Proposition 3.1} 
\end{rmk}

The sequence in \cref{E1minus-SES} gives rise to two types of elements in $E_1^-$.
The first type, corresponding to the cokernel of $i_*$,  consists of elements of the form $ \frac{\ga}{\rho^r \tau^{s}}x $, where $x$ belongs to
$\Ext_\C$.
Such classes arise from the $E_1^{+}$-module structure on $E_1^{-}$
via formulas of the form 
$ \frac{\ga}{\rho^r \tau^{s}}x  = x \cdot   \frac{\ga}{\rho^r \tau^{s}}$. 
We will typically write such elements in simplified form, meaning that $x$ is not $\tau$-divisible in $\Ext_\C$. Note that if $x$ is $\tau$-torsion, then the product is zero, as $ \frac{\ga}{\rho^r \tau^{s}}$ is infinitely $\tau$-divisible. Thus, we need only consider the classes $ \frac{\ga}{\rho^r \tau^{s}}x $, where $x$ is not $\tau$-divisible and also not $\tau$-torsion {in $\Ext_\C$}.

The second type of elements in $E_1^-$ corresponds to
the kernel of $i_*$.
As in \cite{GHIR}*{Definition~7.1}, we introduce notation for these classes.

\begin{notn} \label{Notn:Qclasses}
Given a class $y$ in $\Ext_\C$ 
that is annihilated by some power of $\tau$, 
let $Qy$ be the subset of
the Bockstein $E_1$-page
$E^-_{1,\rho^n}$ 
such that $\delta(Qy) = \rho^{n-1}y$ in 
$\Ext_\C(\bMCC)[\rho]/\rho^n$.
By the unique $\rho$-divisibility of \cref{prop:E-minus-structure3},
we define $\frac{Q}{\rho^k} y$ to be the set of elements such that
$\rho^k \cdot \frac{Q}{\rho^k} y = Q y$.
Alternatively, we define $\frac{Q}{\rho^k}y$ so that
$\delta \left(\frac{Q}{\rho^k}y \right) = \rho^{n-1-k} y$.
In particular, $\delta\left(\frac{Q}{\rho^{n-1}} y\right) = y$.

Beware that a class $\frac{Q}{\rho^k}y$ is only well-defined up to the image of the homomorphism $q_*$,
i.e., up to elements of the form $\frac{\ga}{\rho^r \tau^{s}}x$.
\end{notn}

\begin{rmk}
We sometimes use the symbol $Qy$ to refer to the entire set of possible choices,
and we sometimes use the same symbol to refer to a single choice
of element in the set.
See, for example, \cref{rmk:DescriptionQy} for a description of the set $Qy$.
\end{rmk}

\begin{rmk}
It follows from the definition that $\frac{Q}{\rho^k} y$ lies in Bockstein filtration $-k-1$.
\end{rmk}

We have defined $Qy$ to be a coset in $E_{1,\rho^n}^-$.  
Since $E_{1,\rho^n}^-$ is contained in $E_1^-$, 
we can consider $Qy$ to be a coset in $E_1^-$.
This coset in $E_1^-$ is independent of the choice of $n$ because
of the commutative diagram
\[
\begin{tikzcd}
\displaystyle
\frac{\Ext_\C^{s,f,c}(  \bMCC/\tau^\infty)[\rho]}{\rho^n} \ar[r,"\delta"] \ar[d,"\rho"] &
\displaystyle
\frac{\Ext_\C^{s+1,f-1,c+1}(\bMCC)[\rho]}{\rho^n} \ar[d,"\rho"] \\
\displaystyle
\frac{\Ext_\C^{s-1,f,c}(  \bMCC/\tau^\infty)[\rho]}{\rho^{n+1}} \ar[r,"\delta"] &
\displaystyle
\frac{\Ext_\C^{s,f-1,c}(\bMCC)[\rho]}{\rho^{n+1}}. \\
\end{tikzcd}
\]

\begin{rmk}
If $y$ in $\Ext_\C$ is 
annihilated by some power of $\tau$
and also $\tau$-divisible, so that $y = \tau \cdot x$, then the product $\tau \cdot Qx$ 
is contained in $Q(\tau x)$, but the latter set may be larger.
\end{rmk}

\begin{rmk}
\label{rmk:DiffsOnQClasses}
In \cref{sctn:Bockstein-diff}, we study Bockstein differentials, including
differentials on classes of the form $\frac{Q}{\rho^k} y$
(see, for example, \cref{tbl:Bockstein-differentials-Qgamma}). The statement  $d_r \left( \frac{Q}{\rho^r} y \right ) = z$ means that the set $\frac{Q}{\rho^r} y$ contains a class that survives to $E_r^-$ and supports the claimed differential.

For example, \cref{tbl:Bockstein-differentials-Qgamma} implies that $d_4 \left( \frac{Q}{\rho^4} h_1^5 e_0 \right) = \frac{\gamma}{\tau^4} P h_1 e_0$. In $E_1^-$, the set $\frac{Q}{\rho^4} h_1^5 e_0$ in degree (27,8,7) consists of two classes, with indeterminacy given by the element $\frac{\gamma}{\rho^4 \tau^3} h_0 i$. The class $\frac{\gamma}{\rho^4 \tau^3} h_0 i$ supports a Bockstein $d_1$ differential, so only one of the two classes in the set $\frac{Q}{\rho^4} h_1^5 e_0$ survives to $E_4^-$ to support the $d_4$ differential.
\end{rmk}

%%%%%%%%%%%%%%%%%%%%%%%%%%%%%%%%%%%%%%%%%%%%%%%%%%%%%%%%%%%%%%%
%%%%%%%%%%%%%%%%%%%%%%%%%%%%%%%%%%%%%%%%%%%%%%%%%%%%%%%%%%%%%%%

\section{Extensions in $E^-_1$}
\label{sec:E1extensions}

As described at the end of \cref{sec:NCPeriodic},
there are two types of elements in $E_1^-$.
There are elements of the form
$\frac{\gamma}{\rho^i \tau^j} x$, where $x$ is an element of
$\Ext_\C$ that is not divisible by $\tau$ and is $\tau$-free.
Also, there are elements of the form
$\frac{Q}{\rho^i} y$, where $y$ is an element of $\Ext_\C$ that is annihilated
by some power of $\tau$.

\begin{rmk}
We draw particular attention to the element named $\gamma g$
in degree $(20,4,7)$.  At first glance, this element does not appear
to follow our naming conventions.  However, recall that
$\tau g$ is an indecomposable element of $\Ext_\C$.  Consequently,
we have $\gamma g = \frac{\gamma}{\tau} \tau g$ in
$E_1^-$.  We use this notation for its practical convenience.
\end{rmk}

\cref{E1minus-SES} implies that these two types of elements assemble
into a short exact sequence
\begin{equation}
\label{eq:gamma-Q-SES}
\gamma E^-_1 \rtarr E^-_1 \rtarr QE^-_1,
\end{equation}
where 
$\gamma E^-_1$ is the subobject of $E^-_1$ consisting of the
first type of elements, and 
$QE^-_1$ is the quotient of $E^-_1$ that is detected by 
the second type of elements \cite{GHIR}*{Proposition 3.1}.

Because $E^-_1$ is equal to 
$\ds\frac{\Ext_\C \left(  \bMCC/\tau^\infty \right)[\rho]}{\rho^\infty}$, 
it is a module over $E_1^+ = \Ext_\C[\rho]$.
The sequence \cref{eq:gamma-Q-SES} is compatible with $E_1^+$-module
structures.
For the most part, our notation is well-equipped to describe this
module structure.  For example, we have
$\ds z \cdot \frac{\gamma}{\rho^i \tau^j} x = \frac{\gamma}{\rho^i \tau^j} (zx)$.
Also, we have
$z \cdot Qy$ is contained in $Q(zy)$ when $zy$ is non-zero.  Beware that the 
indeterminacy of $Q(zy)$ can be larger than the indeterminacy
of $z \cdot Qy$, so these two expressions are not always equal.

However, when $zy$ does equal zero, it is still possible that the
product $z \cdot Qy$ is nonzero in $E^-_1$.
Such products are hidden by the sequence \cref{eq:gamma-Q-SES}.
In other words, it is possible to have relations
of the form $z \cdot Q y = \frac{\gamma}{\rho^i \tau^j} x$,
where $x$, $y$, and $z$ all belong to $\Ext_\C$ and $zy = 0$.

We make no attempt to exhaustively describe all such products, not even
in a range.  Instead, we will discuss a few specific examples that
are relevant for our later computations.  In all of these examples,
the element $z$ is either $h_0$ or $h_2$.

In order to obtain these relations, we will, as usual, use higher structure
in the form of Massey products.
In fact, $E^-_1$ is not merely a module over $E^+_1$.  It is also a 
``Massey module".  More concretely, there are Massey products of the
form $\langle x, a, b \rangle$, where $a$ and $b$ belong to $E^+_1$;
$x$ belongs to $E^-_1$; and $ab = 0$ and $x a = 0$.
These Massey products satisfy the standard shuffling properties
\[
\langle x, a, b \rangle c = x \langle a, b, c \rangle
\]
and
\[
\langle x, a, b c\rangle \subseteq \langle x, a b, c \rangle.
\]
In these formulas, the first bracket on the right
is a standard Massey product in $E^+_1$, i.e. $\Ext_\C$,
while the other three brackets refer to the Massey module structure of $E^-_1$. 

\begin{prop}
\label{prop:Q-bracket}
Suppose that $y$ is an element of $\Ext_\C$ such that
$\tau^i y$ equals zero.  Then
$Qy$ contains $\left\langle \frac{\gamma}{\tau^i}, \tau^i, y\right\rangle$.
For sufficiently large $N$, $Qy$ is equal to
$\left\langle \frac{\gamma}{\tau^N}, \tau^N, y\right\rangle$.
\end{prop}

\cref{prop:Q-bracket} is a generalization and slight improvement
on \cite{GHIR}*{Lemma 7.3}, which is not entirely precise about indeterminacy.
Its proof only shows that $Q y$ and 
$\left\langle \frac{\gamma}{\tau}, \tau, y \right\rangle$
intersect non-trivially.

\begin{proof}
The argument in \cite{GHIR}*{Lemma 7.3} relies on an inspection
of the cobar complex.  It demonstrates that $Q y$ and
$\left\langle \frac{\gamma}{\tau^i}, \tau^i, y\right\rangle$
intersect.  It remains to analyze the indeterminacies.

The indeterminacy of the bracket equals
$E^-_1 \cdot y + \frac{\gamma}{\tau^i} \cdot E^+_1$ in the appropriate
degree.
The first term is zero because $E^-_1$ is zero in the relevant degree
$(1, -1, 1)$.

The second term is closely related to the indeterminacy 
in the definition of $Qy$ given in \cref{Notn:Qclasses}.
The indeterminacy of $Q y$ consists of all multiples of
$\frac{\gamma}{\tau^k}$ for all $k \geq 1$, but the indeterminacy
of the bracket consists only of multiples of
$\frac{\gamma}{\tau^i}$ for a fixed value of $i$.  Therefore,
the indeterminacy of the bracket is possibly smaller, and the bracket is contained
in $Q y$.

Because $E_1^-$ is finite in each degree, 
there exists a sufficiently large $N$ such that 
the indeterminacy of $Q y$ consists of multiples of $\frac{\gamma}{\tau^N}$.
Then $Q y$ and
$\left\langle \frac{\gamma}{\tau^N}, \tau^N, y\right\rangle$
have the same indeterminacy, and they are equal.
\end{proof}

\begin{rmk}
\label{rmk:DescriptionQy}
Note the inclusion
\[
\left\langle \frac{\gamma}{\tau^i}, \tau^i, y\right\rangle \subseteq
\left\langle \frac{\gamma}{\tau^{i+1}}, \tau^{i+1}, y\right\rangle.
\]
This is a standard shuffling formula for Massey products, using
that $\frac{\gamma}{\tau^i} = \frac{\gamma}{\tau^{i+1}} \cdot \tau$.
Rather than considering brackets of the form
$\left\langle \frac{\gamma}{\tau^i}, \tau^i, y\right\rangle$ for one
value of $i$ at a time, one could study the union
\[
\bigcup_i \left\langle \frac{\gamma}{\tau^i}, \tau^i, y\right\rangle.
\]
According to \cref{prop:Q-bracket}, this union equals $Q y$.
We are not aware of previous work involving this type of structure 
with Massey products combined with families of infinitely
divisible elements.
\end{rmk}

\begin{prop}
\label{prop:E1-minus-hidden}
\cref{tab:E1-minus-extn}
lists some extensions in the $\rho$-Bockstein $E^-_1$-page
that computes $\Ext_{NC}$.
\end{prop}

\begin{proof}
The proofs for every extension are essentially the same.
We give the details for the first one.
\cref{prop:Q-bracket} (and inspection of indeterminacies)
shows that $Q h_1^2 c_0$ equals
$\left\langle \frac{\gamma}{\tau}, \tau, h_1^2 c_0 \right\rangle$.
Consider the shuffle
\[
\left\langle \frac{\gamma}{\tau}, \tau, h_1^2 c_0 \right\rangle h_0 =
\frac{\gamma}{\tau} \left\langle \tau, h_1^2 c_0, h_0 \right\rangle.
\]
The latter $\C$-motivic bracket is equal to $P h_2$,
which follows from the May convergence theorem \cite{MMP}*{Theorem 4.1}
applied to the May differential $d_4(b_{20} h_0(1)) = \tau h_1^2 c_0$.

The other extensions require different $\C$-motivic Massey products,
such as $h_0 d_0 = \langle \tau, h_1^2 c_0, h_2 \rangle$ and
$i = \langle \tau, c_0 d_0, h_0 \rangle$.
All of these brackets can be computed with the May convergence
theorem and appropriate May differentials.
\end{proof}

\renewcommand*{\arraystretch}{1.6}
\begin{longtable}{LLLL}
\caption{Some extensions in $E^-_1$ \label{tab:E1-minus-extn}} \\ 
\toprule
{(s,f,c)} & \text{source} & \text{type} & \text{target} \\
\midrule \endfirsthead
\caption[]{Some extensions in $E^-_1$} \\
\toprule
{(s,f,c)} & \text{source} & \text{type} & \text{target} \\
\midrule \endhead
\bottomrule \endfoot
(11,4,3) & Q h_1^2 c_0 & h_0 & \frac{\gamma}{\tau} P h_2 \\
(11,4,3) & Q h_1^2 c_0 & h_2 & \frac{\gamma}{\tau} h_0 d_0 \\
(19,8,7) & Q P h_1^2 c_0 & h_0 & \frac{\gamma}{\tau} P^2 h_2 \\
(19,8,7) & Q P h_1^2 c_0 & h_2 & \frac{\gamma}{\tau} P h_0 d_0 \\
(23,6,9) & Q c_0 d_0 & h_0 & \frac{\gamma}{\tau} i \\
(23,6,9) & Q c_0 d_0 & h_2 & \frac{\gamma}{\tau} j \\
(26,6,10) & Q c_0 e_0 & h_0 & \frac{\gamma}{\tau} j \\
(26,6,10) & Q c_0 e_0 & h_2 & \frac{\gamma}{\tau} k \\
(27,12,11) & Q P^2 h_1^2 c_0 & h_0 & \frac{\gamma}{\tau} P^3 h_2\\
(27,12,11) & Q P^2 h_1^2 c_0 & h_2 & \frac{\gamma}{\tau} P^2 h_0 d_0 \\
\end{longtable}

\cref{fig:E1-chart} is a chart of the Bockstein $E^-_1$-page.
The extensions in \cref{tab:E1-minus-extn}
appear as dashed lines in that chart.

%%%%%%%%%%%%%%%%%%%%%%%%%%%%%%%%%%%%%%%%%%%%%%%%%%%%%%%%%%%%%%%
%%%%%%%%%%%%%%%%%%%%%%%%%%%%%%%%%%%%%%%%%%%%%%%%%%%%%%%%%%%%%%%

\section{$\R$-motivic and negative cone Bockstein differentials}
\label{sec:PCtoNC}

We now discuss how $\R$-motivic Bockstein differentials 
relate to
Bockstein differentials in the negative cone. 
As discussed in \cref{sec:NCPeriodic}, there are two types of classes in $E_1^-$: the $\gamma$ classes and the $Q$ classes.
We can then sort Bockstein differentials in $E_r^-$ into three types:
\begin{enumerate}
\item 
differentials from $\gamma$ classes to $\gamma$ classes
\item 
differentials from $Q$ classes to $Q$ classes
\item 
differentials from $Q$ classes to $\gamma$ classes.
\end{enumerate}

\begin{rmk}
\label{rmk:gammatoQ}
In principle, there could a fourth type of differential, from $\gamma$ classes to $Q$ classes, although we have not (yet) encountered any examples. Since the $\gamma$ classes are infinitely $\tau$-divisible, while the $Q$ classes are not, a differential $d_r \left(\frac{\ga}{\rho^r \tau^k} x\right) = Qy$ could only occur if $\frac{\ga}{\rho^r \tau^{k+2^n}} x$ supports a shorter Bockstein differential for some $n$.
\end{rmk}

As we will discuss in this section, the three  types of known Bockstein differentials in $E_r^-$ correspond to three types of differentials in $E_r^+$, i.e. $\R$-motivic Bockstein differentials. Suppose that $d_r(x)= \rho^r y$ in the $\R$-motivic Bockstein spectral sequence. The three types are
\begin{enumerate}
\item 
{Free}: $x$ and $y$ are both $\tau$-free classes in $\Ext_\C$.
\item 
{Torsion}: $x$ and $y$ are both $\tau$-power torsion classes in $\Ext_\C$.
\item 
{Mixed}: $x$ is $\tau$-free, while $y$ is a $\tau$-power torsion class in $\Ext_\C$.
\end{enumerate}

\begin{rmk}
\label{rmk:periodic-type}
There is a possible complication that may occur with
free
differentials.
Suppose that $x$, $y$, and $z$ are $\tau$-free in the $E_1$-page.
Also suppose that $d_r(x) = \rho^r \tau^n y$ for some $n > 0$, and
$d_s(z) = \rho^s y$ for some $s > r$.
Then the differential on $z$ is 
free in the sense discussed above,
even though $d_s(\tau^n z)$ is zero.
We do not know if this phenomenon occurs in practical $\Ext$ computations.
It has not (yet) been observed
in the $\R$-motivic Bockstein spectral sequence.
\end{rmk}

\begin{rmk} A priori, there is a fourth possible type of Bockstein differential, as in \cref{rmk:gammatoQ}. Suppose that $x$ and $y$ are both $\tau$-free and that $d_r(x) = \rho^r \tau^n y$ for some $n > 0$. 
It is then possible for there to be
a later differential $d_s(z) = \rho^s y$, where $z$ is a $\tau$-torsion class in $\Ext_\C$. 
This possibility is similar to the possibility discussed in
Remark \ref{rmk:periodic-type}.  The difference is that here $z$ is
$\tau$-torsion, while it is $\tau$-free in the previous remark.

As in Remark \ref{rmk:periodic-type}, we do not know if this phenomenon occurs in practical $\Ext$ computations because
it has not (yet) been observed
in the $\R$-motivic Bockstein spectral sequence.
\end{rmk}

\subsection{Free differentials}
\label{subsctn:periodic-diff}

We first deal with free differentials.
We say that a non-zero differential $d_r(x) = \rho^r y$ in the
($\R$-motivic or $C_2$-equivariant) $\rho$-Bockstein spectral sequence is
{\bf $\tau^{2^n}$-periodic} if there is a non-zero differential
$d_r(\tau^{2^n k} x) = \rho^r \tau^{2^n k} y$ for all $k \geq 0$.
All free differentials that we have encountered are periodic in this sense, although see \cref{rmk:periodic-type}.

\begin{rmk}
\label{rmk:periodic-differential}
Typically, we have that $r < 2^n$.  Then $d_r(\tau^{2^n}) = 0$.
In this case, $d_r(x) = \rho^r y$ is automatically a 
$\tau^{2^n}$-periodic differential as long as $y$
is {$\tau^{2^n}$}-free in the Bockstein $E_r$-page (i.e., {$\tau^{2^n k} y$}
is non-zero for all $k$).
\end{rmk}

We say that a non-zero differential 
$d_r(a) = \rho^r b$ 
in the $C_2$-equivariant Bockstein spectral sequence is
{\bf $\tau^{2^n}$-coperiodic} if there 
exists a sequence of non-zero differentials
$d_r(a_k) = \rho^r b_k$ for $k \geq 0$ such that:
\begin{enumerate}
\item
$a = a_0$ and $b = b_0$, and
\item
$\tau^{2^n} a_{k+1} = a_k$ and
$\tau^{2^n} b_{k+1} = b_k$.
\end{enumerate}
The idea is that the $a_k$'s and the $b_k$'s form infinitely
divisible families of elements in the spectral sequence.
Typically, a $\tau^{2^n}$-coperiodic family of differentials
appears in the form
\[
d_r \left( \frac{\gamma}{\rho^r \tau^{2^n k}} a \right) = 
\frac{\gamma}{\tau^{2^n k}} b.
\]
\cref{lem:coperiodic-differential} gives an easy criterion
for detecting coperiodic differentials.  The point is that
when $r < 2^n$,
a differential is automatically coperiodic if its source is
infinitely divisible by $\tau^{2^n}$.

\begin{lem}
\label{lem:coperiodic-differential}
Let $r < 2^n$.  Suppose that 
$d_r (  a ) = \rho^r b$ is a non-zero differential
and that 
$a$ is infinitely divisible by $\tau^{2^n}$.
Then this differential is $\tau^{2^n}$-coperiodic.
\end{lem}

\begin{proof}
Let $a_0 = a$.  For $k > 0$,
let $a_k$ and be an element such that
$\tau^{2^n} a_k = a_{k-1}$.
Define $b_k$ by the formula $d_r(a_k) = \rho^r b_k$.
Note that $b_k$ is defined only up to $\rho^r$-torsion.
Also note that $b_k$ is non-zero since
\[
\rho^r \tau^{2^n k} b_k = d_r(\tau^{2^n k} a_k) = d_r(a) = \rho^r b.
\]
Since $\rho^r b_k$ is non-zero, we can conclude that the $\rho$-filtration
of $b_k$ is at most $-r-1$.  Therefore, \cref{prop:E-minus-structure} (1)
implies that there is no $\rho^r$-torsion in the same degree
(including the $\rho$-Bockstein filtration degree)
as $b_k$. 
This means that $b_k$ is in fact well-defined.

Since $r < 2^n$, we have that $d_r(\tau^{2^n}) = 0$.
Apply $d_r$ to the relation
$\tau^{2^n} a_k = a_{k-1}$
to obtain that
$\rho^r \tau^{2^n} b_k = \rho^r b_{k-1}$.
This shows that $\tau^{2^n} b_k$ equals $b_{k-1}$
modulo $\rho^r$-torsion.  As above, there is no possible
$\rho^r$-torsion in this degree, so
$\tau^{2^n} b_k$ must equal $b_{k-1}$.
\end{proof}

\cref{lem:coperiodic-differential} applies, for example, to the differentials of \cref{DiffsGamma}, 
as written out in \cref{rmk:DiffsGamma}.

\begin{prop}
\label{prop:tau-periodic-coperiodic-diff}
Let $1 \leq r < 2^n$,
and suppose that $x$ and $y$ in $\Ext_\C$ are both $\tau$-free classes.
There exists a differential
$d_r \left( \frac{\gamma}{\rho^r \tau^{2^n}} x \right) =
\frac{\gamma}{\tau^{2^n}} y$ in $E_r^-$ if and only if there
exists a differential $d_r(x) = \rho^r(y + z)$ in $E_r^+$, where
$\frac{\gamma}{\tau^{2^n}} z$ is zero in $E_r^-$.
\end{prop}

In practice, it is typically the case that the
differentials in Proposition \ref{prop:tau-periodic-coperiodic-diff}
are $\tau^{2^n}$-coperiodic and $\tau^{2^n}$-periodic respectively.
Usually, Remark \ref{rmk:periodic-differential} and 
Lemma \ref{lem:coperiodic-differential} apply.

\begin{proof}
First assume that $d_r(x) = \rho^r (y + z)$ 
and that $\frac{\gamma}{\tau^{2^n}} z$ is zero.
Since $r < 2^n$, we have that
$d_r \left( \frac{\gamma}{\rho^r \tau^{2^n}} \right) = 0$ 
by \cref{DiffsGamma}.
Then the Leibniz rule gives that
\[
d_r \left( \frac{\gamma}{\rho^r \tau^{2^n}} x \right ) =
\frac{\gamma}{\rho^r \tau^{2^n}} d_r(x) =
\frac{\gamma}{\rho^r \tau^{2^n}} \rho^r (y + z ) =
\frac{\gamma}{\tau^{2^n}} y.
\]
 
Now suppose that 
$d_r \left( \frac{\gamma}{\rho^r \tau^{2^n}} x \right) =
\frac{\gamma}{\tau^{2^n}} y$.
The Leibniz rule implies the differential
$\frac{\gamma}{\rho^r \tau^{2^n}} d_r(x) = \frac{\gamma}{\tau^{2^n}} y$.
This implies that $d_r(x)$ 
equals $\rho^r y$ plus a possible error
term of the form $\rho^r z$ that is annihilated by $\frac{\gamma}{\rho^r \tau^{2^n}}$.  In other words, $z$ is annihilated by 
$\frac{\gamma}{\tau^{2^n}}$.
\end{proof}

\begin{rmk}
The possible error terms $z$ in 
\cref{prop:tau-periodic-coperiodic-diff} 
may be difficult to determine.
The elements of $E_1^+$ that annihilate 
$\frac{\gamma}{\tau^{2^n}}$ are of the form
$z + \tau^{2^n} w$, where $z$ is annihilated by some power of $\tau$.
For example, the differential $d_1( \tau h_0^5 h_5) = \rho h_0^6 h_5$ gives $d_1( \frac{\ga}{\rho \tau}  h_0^5 h_5 ) = \frac{\ga}{\tau^2} h_0^6 h_5$. On the other hand, the latter differential only implies that $d_1( \tau h_0^5 h_5 )$ is either $\rho h_0^6 h_5$ or $\rho ( h_0^6 h_5 + \tau^2 d_0 e_0)$, as $\tau^2 d_0 e_0$ annihilates $\frac{\ga}{\tau^2}$.
It can be hard to describe the elements
of $E_r^+$ that annihilate $\frac{\gamma}{\tau^{2^n}}$.
The problem is that earlier differentials can hit
elements of the form $\frac{\gamma}{\tau^{2^n}} z$.
\end{rmk}

\subsection{Torsion differentials}

Next, we consider torsion differentials.
We employ the diagram
\begin{equation}
\label{DoubleExactDiagram}
\begin{tikzcd}
\ar[r]
&
\displaystyle
\Ext_\R \left( \frac{\bMR[\tau^{-1}]}{\rho} \right) \ar[r,"q_*"] \ar[d,"\rho^r"]
&
\Ext_\R(NC_\rho) \ar[d] \ar[r,"\delta"]
&
\displaystyle
\Ext_\R \left( \frac{\bMR}{\rho} \right) \ar[d,"\rho^r"]
\ar[r,"i_*"]
& \ 
\\
\ar[r]
&
\displaystyle
\Ext_\R \left( \frac{\bMR[\tau^{-1}]}{\rho^{r+1}} \right) \ar[r,"q_*"] \ar[d]
&
\Ext_\R(NC_{\rho^{r+1}}) \ar[d,"\rho"] \ar[r,"\delta"]
&
\displaystyle
\Ext_\R \left( \frac{\bMR}{\rho^{r+1}} \right) \ar[d]
\ar[r,"i_*"]
& \ 
\\
\ar[r]
&
\displaystyle
\Ext_\R \left( \frac{\bMR[\tau^{-1}]}{\rho^r} \right) \ar[r,"q_*"]  \ar[d,"\odelta_{r,1}"]
&
\Ext_\R(NC_{\rho^r}) \ar[r,"\delta"] \ar[d,"\odelta_{r,1}"]
&
\displaystyle
\Ext_\R \left( \frac{\bMR}{\rho^r} \right)  \ar[d,"\odelta_{r,1}"]
\ar[r,"i_*"]
& \ 
\\
\ar[r]
&
\displaystyle
\Ext_\R \left( \frac{\bMR[\tau^{-1}]}{\rho} \right)  \ar[r,"q_*"] 
&
\Ext_\R(NC_\rho)  \ar[r,"\delta"] 
&
\displaystyle
\Ext_\R \left( \frac{\bMR}{\rho} \right) 
\ar[r,"i_*"]
& \ 
\end{tikzcd}
\end{equation}
in our analysis.
Each row is exact, as in \eqref{cofiberrho^nLESExt}. Each column is also exact
because of the short exact sequences
\[
\frac{\bM^\R[\tau^{-1}]}{\rho} \xrightarrow{\rho^r}
\frac{\bM^\R[\tau^{-1}]}{\rho^{r+1}} \rightarrow
\frac{\bM^\R[\tau^{-1}]}{\rho^r},
\]
\[
NC_\rho \rightarrow NC_{\rho^{r+1}} \xrightarrow{\rho} NC_{\rho^r},
\]
and
\[
\frac{\bM^\R}{\rho} \xrightarrow{\rho^r}
\frac{\bM^\R}{\rho^{r+1}} \rightarrow
\frac{\bM^\R}{\rho^r}
\]
of $\cAR_*$-comodules.
Typically, a diagram such as \cref{DoubleExactDiagram} commutes only up to a sign. However, since we are working in characteristic 2, the diagram here commutes.

The following lemma is key to our analysis.

\begin{lem} 
\label{TranslateRhoBssODelta}
Let $M$ be an $\cAR_*$-comodule and fix $r \geq 1$
{such that the sequence
\[
\frac{M}{\rho} \xrtarr{\rho^r} \frac{M}{\rho^{r+1}} \rtarr \frac{M}{\rho^r}
\]
is short exact.}
Suppose that $x$ is a $d_{r-1}$-cycle in the Bockstein spectral sequence with coefficients in $M$, and let $[x]$ {be an element of} $\Ext_\R(M/\rho^{r})$ that is detected by $x$. Then there is a Bockstein differential 
$d_{r}(x) = \rho^r y$ if and only if $\odelta_{r,1}([ x ]) = y$,
where $\odelta_{r,1}$ is the connecting homomorphism in the long exact sequence
\[ 
\rtarr \Ext_\R \left( \frac{M}{\rho} \right) \xrtarr{\rho^r} 
\Ext_\R \left( \frac{M}{\rho^{r+1}} \right) \rtarr 
\Ext_\R \left( \frac{M}{\rho^r} \right) \xrtarr{\odelta_{r,1}} 
\Ext_\R \left( \frac{M}{\rho} \right) \rtarr.
\]
\end{lem}

\begin{pf}
The short exact sequence of coefficients yields a short exact sequence of cobar complexes. The assumption that $x$ is a $d_{r-1}$-cycle in the Bockstein spectral sequence means that there exists a lift $\alpha$ of $x$ to the cobar complex such that $d(\alpha)$ is a $\rho^r$-multiple.
The result then follows  from the construction of the connecting
homomorphism. 
\end{pf}

\cref{TranslateRhoBssNegCone} below is an example of the more general
\cref{TranslateRhoBssODelta}.

\begin{lemma}
\label{TranslateRhoBssNegCone}
Suppose that $\frac{\ga}{\rho^r \tau^{2^n}}x $ is a $d_{r-1}$-cycle
in the Bockstein spectral sequence for $\Ext_\R(NC)$, or equivalently for $\Ext_\R(NC_{\rho^{r+1}})$.
Then
\[d_r\left( \frac{\ga}{\rho^r \tau^{2^n}}x \right) = \rho^r b  \quad \text{if and only if} \quad \odelta_{r,1} \left( \left[{\frac{\ga}{\rho^{r-1} \tau^{2^n}}x}\right] \right) = b.\]
Similarly, if $\frac{Q}{\rho^r} x$ is a  $d_{r-1}$-cycle, then 
\[
d_r\left( \frac{Q}{\rho^r} x\right) = \rho^r b \quad \text{if and only if} \quad \odelta_{r,1} \left( \left[ \frac{Q}{\rho^{r-1}} x \right]  \right) =b.
\]
\end{lemma}

\begin{proof}
In the context of \eqref{DoubleExactDiagram}, we apply \cref{TranslateRhoBssODelta} to the $\cAR_*$-module
 $NC_{\rho^N}$, for $N > r$. 
Beware that 
we employ the shift isomorphisms
\[ 
NC_{\rho^k} \iso \Sigma^{k-N,0} \left( \frac{NC_{\rho^N}}{\rho^k} \right),
\]
for $0 < k < N$, to compare \cref{TranslateRhoBssODelta} to the middle column of \eqref{DoubleExactDiagram}.
\end{proof}

With these preliminary results established, we now discuss {torsion} differentials in the Bockstein spectral sequence for the negative cone.
Recall from \cref{rmk:DiffsOnQClasses} that in order to consider a Bockstein differential on $\frac{Q}{\rho^r} x$, we should suppose that this set {\it contains} a $d_{r-1}$-cycle.

\begin{prop} 
\label{Type2diffs}
Suppose that $x$ and $y$ in $\Ext_\C$ are both $\tau$-power torsion classes.
Also assume that $x$ is a $d_{r-1}$-cycle and that the set $\frac{Q}{\rho^r}x$ contains a $d_{r-1}$-cycle in the Bockstein spectral sequence. 
Then 
$d_r(x) = \rho^r y$ in the $\R$-motivic $\rho$-Bockstein spectral sequence if and only if 
$d_r \left(\frac{Q}{\rho^r }x \right) \subseteq Qy$ in the $\rho$-Bockstein spectral sequence for the negative cone.
\end{prop}

In practice, the indeterminacy in the target $Qy$
requires attention.
For example, the differential $d_3(P h_1^3 c_0) = \rho^3 h_1^7 d_0$
implies that 
$d_3 \left( \frac{Q}{\rho^3} P h_1^3 c_0 \right)$ is contained in $Q h_1^7 d_0$.
However, $Q h_1^7 d_0$ has indeterminacy because of the presence of 
$\frac{\gamma}{\tau^3} P h_0^2 d_0$, so \cref{Type2diffs} does not completely determine the value of 
$d_3 \left( \frac{Q}{\rho^3} P h_1^3 c_0 \right)$.
See \cref{lem:d-QPh1^kc0} below for further discussion of this example.

\begin{pf}
We begin with the case $r=1$. 
Consider diagram \cref{DoubleExactDiagram}.
According to \cref{TranslateRhoBssODelta}, we {need} to show that $\odelta_{1,1}(x) = y$ if and only if $\odelta_{1,1}(Qx) = Qy$.

Consider the commuting square
\[
\begin{tikzcd}
Qx \in \Ext_\R(NC_\rho) \ar[r,"\delta"] \ar[d,"\odelta_{1,1}"]
&
\Ext_\R(\bMR/\rho)  \ar[d,"\odelta_{1,1}"] \ni x
\\
Qy \in   \Ext_\R(NC_\rho)  \ar[r,"\delta"] 
  &
   \Ext_\R(\bMR/\rho) \ni y
\end{tikzcd}
\]
{from the lower right corner of diagram \cref{DoubleExactDiagram}.}
Recall that the classes $Qx$ and $Qy$ are defined to satisfy the equations 
$\delta(Qx) = x$ and $\delta(Qy) = y$. Thus the commutativity of the square shows that $\odelta_{1,1}(Qx) = Qy$ if and only if $\odelta_{1,1}(x) = y$.  Here we are using that $Qy$ is only well-defined up to the image of $q_*$.

A similar argument works to establish the general case.
We use the diagram
\begin{equation}
\label{Level_r}
\begin{tikzcd}
\left[ \frac{Q}{\rho^{r-1}} x \right] \in \Ext_\R(NC_{\rho^r}) \ar[r,"\delta"] \ar[d,"\odelta_{r,1}"]
&
\Ext_\R(\bMR/\rho^r)  \ar[d,"\odelta_{r,1}"] \ni [x]
\\
Qy \in   \Ext_\R(NC_\rho)  \ar[r,"\delta"] 
  &
   \Ext_\R(\bMR/\rho)  \ni y.
\end{tikzcd}
\end{equation}
The argument is the same as in the case of $r=1$, relying on \cref{TranslateRhoBssODelta} 
or rather its specialization \cref{TranslateRhoBssNegCone},
provided that we can establish that $\delta\left( \left[ \frac{Q}{\rho^{r-1}}x\right] \right) = [x]$. In $E_r^- = E_\infty^-$, we do have that $\delta\left(  \frac{Q}{\rho^{r-1}}x \right) = x$
because this formula holds already in $E_1^-$.
Therefore,
we may choose lifts
$\left[ \frac{Q}{\rho^{r-1}}x\right]$ and $[x]$ in $\Ext$ groups 
that are represented by
$\frac{Q}{\rho^{r-1}}x$ and $x$ respectively such that
$\delta\left( \left[ \frac{Q}{\rho^{r-1}}x\right] \right) = [x]$.
\end{pf}

\begin{rmk}
In \cref{Type2diffs}, the hypothesis that
$\frac{Q}{\rho^r} x$ survives to the $E_r$-page is necessary.
This is demonstrated by the
$\R$-motivic $\rho$-Bockstein differential
$d_7(P^2 h_1^4) = h_1^9 e_0$.
The proposition suggests that
$d_7 \left( \frac{Q}{\rho^7} P^2 h_1^4 \right)$
ought to equal $Q h_1^9 e_0$.  However, this does not occur.
Rather, there is a differential
$d_3 \left( \frac{Q}{\rho^3} P^2 h_1^4 \right) = \frac{\gamma}{\tau^3} h_0^5 i$,
as shown in \cref{tbl:Bockstein-differentials-Qgamma}.
\end{rmk}

\subsection{Mixed differentials}

Finally, we consider mixed differentials. Here, we state a conjecture relating mixed differentials in the $\R$-motivic Bockstein spectral sequence to differentials in the Bockstein spectral sequence for the negative cone, and we provide some evidence to support the conjecture. 

\begin{conj}
\label{Type3diffs}
Suppose that $x$ and $y$ are classes in $\Ext_\C$ such that $x$ is $\tau$-free
and $y$ is $\tau$-power torsion.
Suppose furthermore that $d_r(x)= \rho^r y$ in the $\R$-motivic $\rho$-Bockstein spectral sequence and that
$r < t < 2^n$.
There exists a $\tau^{2^n}$-periodic differential $d_t(\tau^{2^n}x) = \rho^t z$  if and only if the (nonperiodic) differential
\[ d_{t-r} \left( \frac{Q}{\rho^{t-r}} y \right) = \frac{\ga}{\tau^{2^n}} z
\]
occurs in the $\rho$-Bockstein spectral sequence for the negative cone.

If these differentials occur, then there is a
$\tau^{2^n}$-extension from $y$ to $\rho^{t-r} z$
in $\Ext_\R \left( \frac{\bMR}{\rho^{t+1}} \right)$ that
is hidden by the $\rho$-Bockstein spectral sequence.
\end{conj}

Figure \ref{fig:Type3Conjecture} illustrates
the situation described in Conjecture \ref{Type3diffs}.
The left side of the figure depicts
$\R$-motivic differentials, while the right side depicts
negative cone differentials.
The {solid} differential on the left is the mixed differential
in the hypothesis of the conjecture.
The conjecture states that the 
left-side dashed differential occurs if and only if the right-side dashed
differential occurs.

\begin{figure}
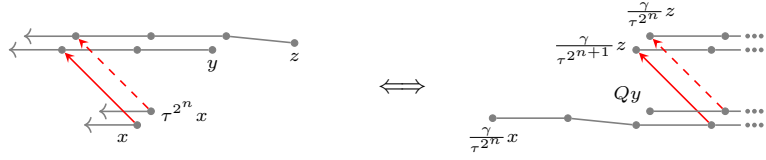

\caption{{An illustration of Conjecture \ref{Type3diffs}}}
\label{fig:Type3Conjecture}
\begin{sseqpage}[no axes, no ticks, x range={0}{5},y range={0}{3},Adams grading,classes=gray,labels=black,differentials=red]
\draw[->,xshift={0.9},yshift={1.91},semithick,color=\rhocolor] (0,0) -- (-0.7,0);
\draw[->,xshift={1.1},yshift={2.09},semithick,color=\rhocolor] (0,0) -- (-0.7,0);
\draw[->,xshift={2.1},yshift={1.09},semithick,color=\rhocolor] (0,0) -- (-0.7,0);
\draw[->,xshift={1.9},yshift={0.91},semithick,color=\rhocolor] (0,0) -- (-0.7,0);
\class["x"{below left=0}] (2,1)
\class["\tau^{2^n} x"{right=0}](2,1)
\class["y"] (3,2)
\rhotower{3}
\d1(2,1)
\class["z"] (4,2)
\rhotower{4}
\begin{scope}[dashed]
\d1(2,1,-1,-1)
\end{scope}
\end{sseqpage}
\hspace{-2em}
\raisebox{5.5em}{$\iff$}
\hspace{-3em}
\begin{sseqpage}[no axes, no ticks, x range={1}{5},y range={0}{3},Adams grading,classes=gray,labels=black,differentials=red]
\class["\frac{\gamma}{\tau^{2^n}} x"{below=0}](2,1)
\rhocotower{6}
\class["\frac{\gamma}{\tau^{2^{n+1}}}z" {left=0}] (4,2)
\rhocotower{1}
\d1(5,1,-1,-1)
\class["Qy" {above left=0}] (4,1)
\rhocotower{4}
\class["\frac{\gamma}{\tau^{2^n}}z" {above=0}] (4,2)
\rhocotower{1}
\begin{scope}[dashed]
\d1(5,1,-1,-1)
\end{scope}
\end{sseqpage}
\end{figure}

\cref{tab:Type3diffs} gives some concrete 
{instances in which \cref{Type3diffs} could potentially be used to
determine differentials.}
In each row, the mixed differential in the left column is known;
it is the hypothesis of the conjecture in each case.  The differentials
in the center column are previously known for reasons that
vary from case to case. 
\cref{Type3diffs} would then imply the differentials in the third column.
For legibility, no powers of $\rho$ appear in the formulas
because they can be inferred from the length of the differential.
For example, the first entry in the left column is more completely stated
as $d_4(\tau h_0^3 h_3) = \rho^4 h_1^2 c_0$.

Note that the conjecture can be used in both directions;  the horizontal
line separates the examples into two families, depending on the direction
of the implications.
In the first family, the four implied differentials are actually already
known for other reasons.
However, in the second family, the three implied $\R$-motivic differentials
are not currently known; they lie outside the range considered in
\cite{BI}.

\renewcommand*{\arraystretch}{1.5}
\begin{longtable}{lll}
\caption{Some applications of \cref{Type3diffs} \label{tab:Type3diffs}} \\
\toprule
mixed differential & known differential & implied differential \\
\midrule \endfirsthead
\caption[]{Some applications of \cref{Type3diffs}} \\
\toprule
mixed differential & known differential & implied differential \\
\midrule \endhead
\bottomrule \endfoot
$d_4(\tau h_0^3 h_3) = h_1^2 c_0$ & $d_5(\tau^9 h_0^3 h_3) = \tau^6 P h_2$ &  $d_1(Q h_1^2 c_0) = \frac{\ga}{\tau^2} P h_2$ \\
$d_3(P h_1) = h_1^3 c_0$ & $d_6(\tau^8 P h_1) = \tau^5 h_0^2 d_0$ & $d_3(Q h_1^3 c_0) = \frac{\ga}{\tau^3} h_0^2 d_0$ \\
$d_3(P h_1^2) = h_1^4 c_0$ & $d_7(\tau^8 P h_1^2) = \tau^4 P c_0$ & $d_4(Q h_1^4 c_0) = \frac{\ga}{\tau^4} P c_0$ \\
$d_3(P c_0) = h_1^4 d_0$ & $d_7(\tau^8 P c_0) = \tau^4 P d_0$ & $d_4( Q h_1^4 d_0) = \frac{\ga}{\tau^4} P d_0$ \\
\hline
$d_{11}(\tau^4 P h_1^2) = h_1^3 e_0$ & $d_1( Q h_1^2 e_0) = \frac{\ga}{\tau^3} h_0 h_2 e_0$ & $d_{12} (\tau^{20} P h_1) = \tau^{13} h_0 h_2 e_0$ \\
$d_{11}(\tau^4 P h_1^2) = h_1^3 e_0$ & $d_3( Q h_1^3 e_0) = \frac{\ga}{\tau^4} i$ & $d_{14}(\tau^{20} P h_1^2) = \tau^{12} i$ \\
$d_{8}(\tau h_0^7 h_4) = h_1^5 e_0$ & $d_4 (Q h_1^5 e_0) = \frac{\ga}{\tau^4} P h_1 e_0$ & $d_{12} (\tau^{17} h_0^7 h_4) = \tau^{12} P h_1 e_0$ \\
\end{longtable}
\renewcommand*{\arraystretch}{1.0}

%%%%%%%%%%%%%%%%%%%%%%%%%%%%%%%%%%%%%%%%%%%%%%%%%%%%%%%%%%%%%%%
%%%%%%%%%%%%%%%%%%%%%%%%%%%%%%%%%%%%%%%%%%%%%%%%%%%%%%%%%%%%%%%
\section{The $\eta$-periodic homotopy groups}
\label{sec:etaperiodic}

In \cite{etaR}, we completely computed the $h_1$-periodic  $\R$-motivic $\rho$-Bockstein spectral sequence and the $h_1$-periodic $\R$-motivic Adams spectral sequence 
that determine
the $\eta$-periodic $\R$-motivic stable homotopy groups
$\piR_{*,*}[\eta^{-1}]$, 
where $\eta\in \piR_{1,0}$ is the motivic Hopf map detected by $h_1$.

We now discuss the analogous 
computation of $\eta$-periodic $C_2$-equivariant stable homotopy groups
$\piC_{*,*}[\eta^{-1}]$.
We begin by showing that these latter groups are not very
complicated.  This is a qualitative difference between
$\R$-motivic and $C_2$-equivariant stable homotopy theory.

\begin{prop}
\label{prop:eta-periodic}
Geometric fixed points induces an isomorphism from the
$\eta$-periodic $C_2$-equivariant stable homotopy groups
to $\Z[\frac12,\eta^{\pm 1}]$. This
isomorphism sends $\rho$ to $-2\eta^{-1}$.
\end{prop}

Consistently with the rest of this manuscript, the 
computation of \cref{prop:eta-periodic} should be interpreted
as 2-complete, so $\Z$ really means the $2$-adic integers.
However, our proof does not use the Adams spectral sequence.
Consequently, our proof also establishes the uncompleted integral computation,
although we do not need it.

\begin{proof}
The forgetful map $U$ takes $\eta^4$ to zero, so  \cref{prop:rho-forget}
implies that $\eta^4$ is divisible by $\rho$.
This means that $\rho$ is invertible in 
$\piC_{*,*}[\eta^{-1}]$, so we may just as well compute
$\piC_{*,*}[\rho^{-1}][\eta^{-1}]$.

The geometric fixed points functor induces an isomorphism
from the $\rho$-periodic homotopy groups
$\piC_{*,*}[\rho^{-1}]$ to 
$\picl_* \otimes \Z[\rho^{\pm 1}]$
\cite{AI}*{Proposition 7.0}.
These latter groups are just the classical stable homotopy groups
with a unit adjoined to make them bigraded.

Geometric fixed points takes $\eta$ to $2$, so
$\piC_{*,*}[\rho^{-1}][\eta^{-1}]$ is isomorphic to
$\picl_*\left[ \frac{1}{2} \right] \otimes \Z[\rho^{\pm 1}]$.
Finally, Serre finiteness \cite{Serre} implies that
$\picl_* \left[ \frac{1}{2} \right]$
is concentrated in degree $0$.

To determine the image of $\rho$ under the isomorphism, recall \cite{Morel} that the element $2+\rho \eta$ is annihilated by $\eta$ in $\piC_{*,*}$. It follows that $2+\rho\eta$ vanishes in the localization $\piC_{*,*}[\eta^{-1}]$.
\end{proof}

The rest of this section gives a detailed 
analysis 
of the $h_1$-periodic $\rho$-Bockstein and Adams spectral sequences.
\cref{prop:eta-periodic} says that 
our analysis computes a simple, already known, answer.
Nevertheless, by carrying out the relatively easy $h_1$-periodic computation, 
we obtain information that can be used to study the
non-periodic parts of the
$\rho$-Bockstein
spectral sequence and the Adams spectral sequence.

\subsection{The $h_1$-periodic Bockstein spectral sequence}
\label{sec:h1invBock}

The $h_1$-periodic $\R$-motivic Bockstein spectral sequence is 
described in \cite{etaR}*{Sections 3--4}.
It takes the form
\[
E^+_1 [h_1^{-1}] = \F_2[\rho, h_1^{\pm 1}] [ v_1^4, v_2, v_3, \ldots ].
\]
All of the differentials are known; however,
in the range under consideration in this article, 
we only need that $d_3(v_1^4) = \rho^3  v_2$ and
$d_7(v_1^8) = \rho^7 v_3$.
These differentials are consequences of the Bockstein
differentials
$d_3( P h_1 ) = \rho^3 h_1^3 c_0$ and
$d_7 ( P^2 h_1 ) = \rho^7 h_1^6 e_0$.

Using the description of $E_1^-$ in \cref{sec:NCPeriodic,sec:E1extensions},
the structure of $E_1^-[h_1^{-1}]$ is
\[
Q \cdot \frac{\F_2[\rho]}{\rho^\infty}[h_1^{\pm 1}] [v_1^4, v_2, v_3, \ldots ].
\]
In more naive terms, $E_1^-[h_1^{-1}]$ is obtained from $E_1^+[h_1^{-1}]$ by adjoining the
symbol $Q$ to each $\F_2[\rho]$-module generator, 
and then replacing each copy of
$\F_2[\rho]$ with the infinitely divisible $\frac{\F_2[\rho]}{\rho^\infty}$.
The Bockstein differentials in $E^-[h_1^{-1}]$ are easily determined
from the multiplicative structure and the differentials in
$E^+[h_1^{-1}]$.
\cref{tab:h1-periodic-Bockstein-differentials} shows
the differentials in both $E^+[h_1^{-1}]$ and $E^-[h_1^{-1}]$
that are relevant for us.
The grading in the tables and figures is chosen such that $h_1$ has degree 
$(0,0)$, so that we can ignore multiples of $h_1$ when considering
degrees.

\renewcommand*{\arraystretch}{1.3}
\begin{longtable}{LLLL} 
\caption{$h_1$-periodic Bockstein differentials } \\
\toprule
{(c,2c+f-s)} & \textrm{source} & d_r & \textrm{target} \\
\midrule \endhead
\bottomrule \endfoot
\label{tab:h1-periodic-Bockstein-differentials}
(4,4) & v_1^4 & d_3 & \rho^3 v_2 \\
(7,5) & v_1^4 v_2 & d_3 & \rho^3 v_2^2 \\
(8,8) & v_1^8 & d_7 & \rho^7 v_3 \\
\hline
(4,2) & \frac{Q}{\rho^3} v_1^4 & d_3 & Q v_2 \\
(7,3) & \frac{Q}{\rho^3} v_1^4 v_2 & d_3 & Q v_2^2 \\
(8,6) & \frac{Q}{\rho^7} v_1^8 & d_7 & Q v_3 \\
\end{longtable}
\renewcommand*{\arraystretch}{1.0}

From the explicit description of the $h_1$-periodic Bockstein
differentials in $E^+[h_1^{-1}]$ and $E^-[h_1^{-1}]$,
it is straightforward to describe
$\Ext_\R[h_1^{-1}] \oplus \Ext_{NC}[h_1^{-1}]$,
i.e., the $h_1$-periodic $C_2$-equivariant Adams $E_2$-page.
The relevant classes
are listed in \cref{tbl:Exth1inv}.
Elements in $E_1^-[h_1^{-1}]$ start off as infinitely
$\rho$-divisible, so we record their $\rho$-divisibility, rather than
their $\rho$-multiples, in the fourth
column of the table.  

The fifth column of \cref{tbl:Exth1inv} describes some values of the
localization map $\Ext_{C_2} \rightarrow \Ext_{C_2} [h_1^{-1}]$.
This information is essential for lifting $h_1$-periodic computations
to the non-periodic setting.

\begin{longtable}{LLLLL}
\caption{$\Ext_{C_2}[h_1^{-1}]$ in  coweights $c \leq 8$} \\
\toprule
{(c,2c+f-s)} & \text{element} & \text{$\rho$-power} & \text{$\rho$-divisibility} & \text{lift to $\Ext_{C_2}$}\\
& & \text{torsion} & \\
\midrule \endhead
\bottomrule \endfoot
\label{tbl:Exth1inv}
(0,0) & 1 & \infty & & 1\\
(3,1) & v_2 & 3 & & c_0 \\
(6,2) & v_2^2 & 3 & & d_0 \\
(7,1) & v_3 & 7 & & e_0 \\
\hline
(0,-2) & Q & & \infty & Q h_1^4 \\
(4,2) & Q v_1^4 & & 3 & Q P h_1^4 \\
(7,3) & Q v_1^4 v_2 & & 3 & Q P h_1^2 c_0 \\
(8,6) & Q v_1^8 & & 7 & Q P^2 h_1^4 \\
\end{longtable}

\newcommand{\Qcolor}{blue!65!white}

\begin{figure}
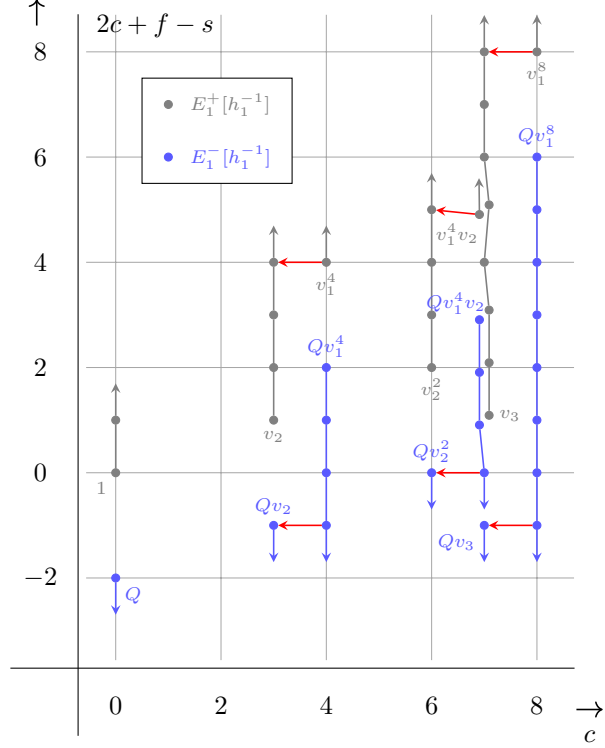

\caption{The $h_1$-periodic $C_2$-equivariant $\rho$-Bockstein spectral sequence}
\label{fig:h1local-C2Bockstein}
\begin{sseqpage}[
classes= {fill, inner sep = 0pt, minimum size = 0.3em, gray,},
differentials=red,
x tick step = 2,
y tick step = 2,
grid = go, 
grid step=2,
grid color = gray!70,
scale=0.7,x range={0}{8},y range={-3}{8}]
\class["1"{below left=0}](0,0)
\class(\lastx,\lasty+1)
\structline[\rhocolor](\lastx,\lasty-1)
\vertrholoc(\lastx,\lasty)
\class["v_2"](3,1)
\class(\lastx,\lasty+1)
\structline[\rhocolor](\lastx,\lasty-1)
\class(\lastx,\lasty+1)
\structline[\rhocolor](\lastx,\lasty-1)
\class(\lastx,\lasty+1)
\structline[\rhocolor](\lastx,\lasty-1)
\vertrholoc(\lastx,\lasty)
\class["v_1^4"](4,4)
\vertrholoc(\lastx,\lasty)
\class["v_2^2"](6,2)
\class(\lastx,\lasty+1)
\structline[\rhocolor](\lastx,\lasty-1)
\class(\lastx,\lasty+1)
\structline[\rhocolor](\lastx,\lasty-1)
\class(\lastx,\lasty+1)
\structline[\rhocolor](\lastx,\lasty-1)
\vertrholoc(\lastx,\lasty)
\class["v_1^4 v_2"{below left=0,xshift=3}](7,5)
\draw[\rhocolor,->,>=stealth,semithick](7,5)--(6.9,5.6);
\begin{scope}[\Qcolor,class labels=\Qcolor]
\renewcommand{\rhocolor}{\Qcolor}
\class[\Qcolor,"Q v_1^4 v_2"{above left=-0.4,yshift=4.5}](7,3)
\class[\Qcolor](\lastx,\lasty-1)
\structline[\rhocolor]
\class[\Qcolor](\lastx,\lasty-1)
\structline[\rhocolor]
\class[\Qcolor](\lastx,\lasty-1)
\structline[\rhocolor]
\vertrhocoloc(\lastx,\lasty)
\end{scope}
\class["v_3"{right=0}](7,1)
\class(\lastx,\lasty+1)
\structline[\rhocolor]
\class(\lastx,\lasty+1)
\structline[\rhocolor]
\class(\lastx,\lasty+1)
\structline[\rhocolor]
\class(\lastx,\lasty+1)
\structline[\rhocolor]
\class(\lastx,\lasty+1)
\structline[\rhocolor]
\class(\lastx,\lasty+1)
\structline[\rhocolor]
\class(\lastx,\lasty+1)
\structline[\rhocolor]
\vertrholoc(\lastx,\lasty)
\class["v_1^8"](8,8)
\vertrholoc(\lastx,\lasty)
\d3(4,4)(3,4)
\d3(7,5)(6,5)
\d7(8,8)(7,8)
%%%%%%%%%%%%
\begin{scope}[\Qcolor,class labels=\Qcolor]
\renewcommand{\rhocolor}{\Qcolor}
\class[\Qcolor,"Q" {below right=0,\Qcolor}](0,-2)
\vertrhocoloc(\lastx,\lasty)
\class[,"Q v_2"{above=0}](3,-1)
\vertrhocoloc(\lastx,\lasty)
\class[\Qcolor,"Q v_1^4"{above=0}](4,2)
\class[\Qcolor](\lastx,\lasty-1)
\structline[\rhocolor]
\class[\Qcolor](\lastx,\lasty-1)
\structline[\rhocolor]
\class[\Qcolor](\lastx,\lasty-1)
\structline[\rhocolor]
\vertrhocoloc(\lastx,\lasty)
\class[\Qcolor,"Q v_2^2"{above=0}](6,0)
\vertrhocoloc(\lastx,\lasty)
\class[\Qcolor,"Q v_3"{below left=0}](7,-1)
\vertrhocoloc(\lastx,\lasty)
\class[\Qcolor,"Q v_1^8"{above=0}](8,6)
\class[\Qcolor](\lastx,\lasty-1)
\structline[\rhocolor]
\class[\Qcolor](\lastx,\lasty-1)
\structline[\rhocolor]
\class[\Qcolor](\lastx,\lasty-1)
\structline[\rhocolor]
\class[\Qcolor](\lastx,\lasty-1)
\structline[\rhocolor]
\class[\Qcolor](\lastx,\lasty-1)
\structline[\rhocolor]
\class[\Qcolor](\lastx,\lasty-1)
\structline[\rhocolor]
\class[\Qcolor](\lastx,\lasty-1)
\structline[\rhocolor]
\vertrhocoloc(\lastx,\lasty)
\end{scope}
\d3(4,-1)(3,-1)
\d3(7,0)(6,0)
\d3(8,-1)(7,-1)
%%%
\filldraw[black,fill=white] (0.5,5.5) rectangle (3.35,7.5);
\class["E_1^+[h_1^{-1}]"{right}](1,7)
\class["E_1^-[h_1^{-1}]"{right},\Qcolor](1,6)
\draw[background,->,semithick] (8.75,-4.5) -- (9.25,-4.5);
\node[background] at (9,-5) {c};
\node[background] at (0.75,8.5) {2c+f-s};
\draw[background,->,semithick] (-1.5,8.5) -- (-1.5,9);
\end{sseqpage}
\end{figure}

\cref{fig:h1local-C2Bockstein} displays the 
$h_1$-periodic $C_2$-equivariant $\rho$-Bockstein spectral sequence.
Black dots depict 
copies of $\F_2[h_1^{\pm 1}]$ from
$E_1^+[h_1^{-1}]$,
while blue dots depict 
copies of $\F_2[h_1^{\pm 1}]$ from
$E_1^-[h_1^{-1}]$.
Vertical lines denote multiplications by $\rho$.
Horizontal arrows depict Bockstein differentials.

\subsection{The $h_1$-periodic Adams spectral sequence}
\label{subsctn:h1-periodic-Adams}

The $C_2$-equivariant $\rho$-Bock\-stein spectral sequence splits into
an $\R$-motivic summand and a negative cone summand.  However, 
the $C_2$-equivariant Adams spectral sequence does not split in this fashion.
In fact, there are interactions between the two terms already in the $h_1$-periodic computations. 
These interactions take the form of 
hidden extensions and Adams differentials that connect elements in
different summands.

\begin{prop}
\label{rho-extn-h1inv}
 There is a hidden $\rho$-extension from $Q$ to $1$ in the $h_1$-periodic $C_2$-equivariant Adams spectral sequence. 
\end{prop}

\begin{pf}
The proof of \cref{prop:eta-periodic} shows that every
element of the $\eta$-periodic $C_2$-equivariant stable homotopy groups
is divisible by $\rho$.  In particular,
the class $1$ detects a $\rho$-divisible class,
so it must be the target of a hidden $\rho$-extension.  There is only
one possible source for this extension.
\end{pf}

\begin{prop} 
\cref{tbl:h1invAdamsDiffs} shows 
all differentials in the $h_1$-periodic $C_2$-equivariant Adams spectral sequence
in coweights at most 8.
The spectral sequence collapses in this range at the $E_3$-page.
\end{prop}

\begin{pf}
The differential $d_2(v_3) = v_2^2$ is given in \cite{etaR}*{Lemma~5.2} and follows from the classical Adams differential $d_2(e_0) = h_1^2 d_0$.
The proof of \cref{prop:eta-periodic} shows that every
element of the $\eta$-periodic $C_2$-equivariant stable homotopy groups
is divisible by $\rho$.
This requirement forces all remaining differentials.
\end{pf}

\renewcommand*{\arraystretch}{1.4}
\begin{table}[ht]
\captionof{table}{$h_1$-periodic Adams differentials up to coweight $8$}
\label{tbl:h1invAdamsDiffs}
\begin{center}
\begin{tabular}{LLL} 
\hline
(c,2c+f-s) & x & d_2(x)\\
\hline
(4,0) & \frac{Q}{\rho^2} v_1^4 & v_2 \\
(7,1) & v_3 & v_2^2 \\
(7,1) & \frac{Q}{\rho^2} v_1^4 v_2 & v_2^2 \\
(8,0) & \frac{Q}{\rho^6} v_1^8 & v_3 + \frac{Q}{\rho^2} v_1^4 v_2 \\
 \hline
\end{tabular}
\end{center}
\end{table}
\renewcommand*{\arraystretch}{1.0}

\begin{rmk}
There is an obvious similarity between the differentials in
\cref{tab:h1-periodic-Bockstein-differentials,tbl:h1invAdamsDiffs}.
This similarity breaks down for the last differential.
Note that $d_2 \left( \frac{Q}{\rho^6} v_1^8 \right)$ cannot equal $v_3$
because $v_3$ is not a cycle.  On the other hand, the formula
$d_2(Q v_1^8) = \rho^6 v_3$ does hold.
However, the latter form does not carry as much information as the
form that we give because of the presence of $\rho$-torsion in coweight $7$.
In other words, $\rho^3 v_3$ is not uniquely divisible by $\rho$.
\end{rmk}

\cref{fig:h1local-C2Adams} displays the $h_1$-periodic $C_2$-equivariant Adams spectral sequence in coweights at most 8.
Black dots depict 
copies of $\F_2[h_1^{\pm 1}]$ from
$\Ext_\R[h_1^{-1}]$,
while blue dots depict 
copies of $\F_2[h_1^{\pm 1}]$ from
$\Ext_{NC}[h_1^{-1}]$.
Vertical lines denote multiplications by $\rho$.
Arrows of slope $-1$ depict Adams $d_2$ differentials.
The dashed line in coweight $0$ represents the hidden $\rho$-extension of \cref{rho-extn-h1inv}.

\begin{figure}
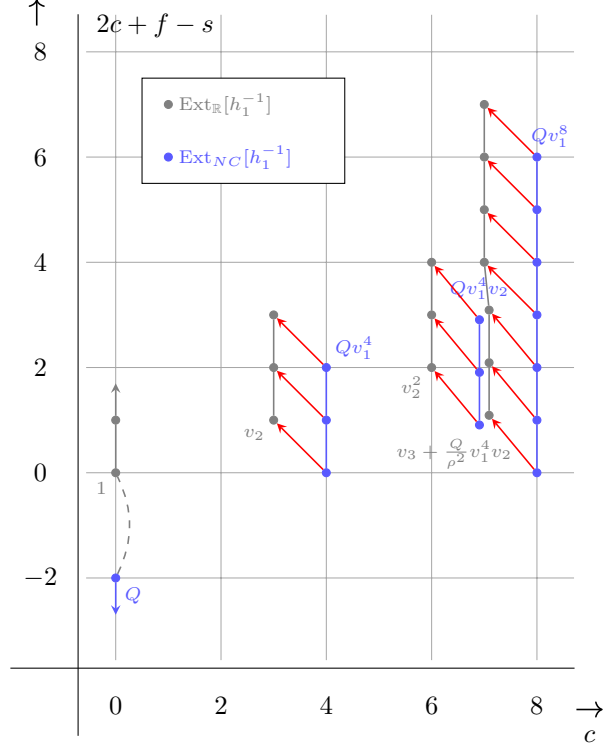

\caption{The $h_1$-periodic $C_2$-equivariant Adams spectral sequence}
\label{fig:h1local-C2Adams}
\begin{sseqpage}[
classes= {fill, inner sep = 0pt, minimum size = 0.3em, gray},
differentials=red,
x tick step = 2,
y tick step = 2,
grid = go, 
grid step=2,
grid color = gray!70,
scale=0.7,x range={0}{8},y range={-3}{8}]
\class["1"{below left=0}](0,0)
\class(\lastx,\lasty+1)
\structline[\rhocolor](\lastx,\lasty-1)
\vertrholoc(\lastx,\lasty)
\class["v_2"{below left=0}](3,1)
\class(\lastx,\lasty+1)
\structline[\rhocolor](\lastx,\lasty-1)
\class(\lastx,\lasty+1)
\structline[\rhocolor](\lastx,\lasty-1)
\class["v_2^2"{below left=0}](6,2)
\class(\lastx,\lasty+1)
\structline[\rhocolor](\lastx,\lasty-1)
\class(\lastx,\lasty+1)
\structline[\rhocolor](\lastx,\lasty-1)
\begin{scope}
\renewcommand{\rhocolor}{\Qcolor}
\class[\Qcolor,"Q v_1^4 v_2"{above=0.2}](7,3)
\class[\Qcolor](\lastx,\lasty-1)
\structline[\rhocolor]
\class[\Qcolor](\lastx,\lasty-1)
\structline[\rhocolor]
\end{scope}
\class["v_3+\frac{Q}{\rho^2} v_1^4 v_2"{below left=0.4,xshift=17.7}](7,1)
\class(\lastx,\lasty+1)
\structline[\rhocolor]
\class(\lastx,\lasty+1)
\structline[\rhocolor]
\class(\lastx,\lasty+1)
\structline[\rhocolor]
\class(\lastx,\lasty+1)
\structline[\rhocolor]
\class(\lastx,\lasty+1)
\structline[\rhocolor]
\class(\lastx,\lasty+1)
\structline[\rhocolor]
%%%%%%%%%%%%
\begin{scope}
\renewcommand{\rhocolor}{\Qcolor}
\class[\Qcolor,"Q"{below right=0}](0,-2)
\vertrhocoloc(\lastx,\lasty)
\class[\Qcolor](4,0)
\class[\Qcolor](\lastx,\lasty+1)
\structline[\rhocolor]
\class[\Qcolor,"Q v_1^4"{above right=0}](\lastx,\lasty+1)
\structline[\rhocolor]
\class[\Qcolor,"Q v_1^8"{above=0,xshift=5}](8,6)
\class[\Qcolor](\lastx,\lasty-1)
\structline[\rhocolor]
\class[\Qcolor](\lastx,\lasty-1)
\structline[\rhocolor]
\class[\Qcolor](\lastx,\lasty-1)
\structline[\rhocolor]
\class[\Qcolor](\lastx,\lasty-1)
\structline[\rhocolor]
\class[\Qcolor](\lastx,\lasty-1)
\structline[\rhocolor]
\class[\Qcolor](\lastx,\lasty-1)
\structline[\rhocolor]
\end{scope}
\structline[\rhocolor,dashed,bend right=25](0,-2)(0,0)
%%%
\d2(4,0)(3,1)
\d2(4,1)(3,2)
\d2(4,2)(3,3)
\d2(7,1)(6,2)
\d2(7,2)(6,3)
\d2(7,3)(6,4)
\d2(8,0)(7,1,2)
\d2(8,1)(7,2,2)
\d2(8,2)(7,3,2)
\d2(8,3)(7,4)
\d2(8,4)(7,5)
\d2(8,5)(7,6)
\d2(8,6)(7,7)
%%%
\filldraw[black,fill=white] (0.5,5.5) rectangle (4.35,7.5);
\class["\Ext_\R[h_1^{-1}]"{right=0}](1,7)
\class["\Ext_{NC}[h_1^{-1}]"{right=0},\Qcolor](1,6)
\draw[background,->,semithick] (8.75,-4.5) -- (9.25,-4.5);
\node[background] at (9,-5) {c};
\node[background] at (0.75,8.5) {2c+f-s};
\draw[background,->,semithick] (-1.5,8.5) -- (-1.5,9);
\end{sseqpage}
\end{figure}

%%%%%%%%%%%%%%%%%%%%%%%%%%%%%%%%%%%%%%%%%%%%%%%%%%%%%%%%%%%%%%%
%%%%%%%%%%%%%%%%%%%%%%%%%%%%%%%%%%%%%%%%%%%%%%%%%%%%%%%%%%%%%%%
\section{Bockstein differentials}
\label{sctn:Bockstein-diff}

In the range under consideration in this article,
the vast majority of Bockstein
differentials in $E^-$ are consequences of 
the Leibniz rule applied to
the differentials on $\frac{\gamma}{\rho^{2^n} \tau^{2^n}}$
given in
\cref{DiffsGamma} together with the $\R$-motivic
differentials in $E^+$.
In other words, most of the Bockstein differentials in $E^-$
are periodic in the sense of 
\cref{subsctn:periodic-diff} and
arise from \cref{prop:tau-periodic-coperiodic-diff}.
The $\R$-motivic differentials are studied extensively
in \cite{LowMW} and \cite{BI}.

This means that the vast majority of Bockstein
differentials in $E^-$ are easy to obtain.  We take all of this
information for granted and will not discuss it in any detail.
Rather, we choose to focus on the handful of
differentials in $E^-$ that are more difficult to obtain.
These more difficult differentials are indicated in \cref{tbl:Bockstein-differentials-Qgamma}.

In the range under consideration in this manuscript
(coweights from $-2$ to $8$, stems up to $30$), there is one
possible Bockstein differential that we have not established.
It is possible that the element $\frac{\gamma}{\tau^7} h_0^3 h_4^2$
in coweight $6$ and stem $30$ is hit by a differential.
More precisely, either there is a Bockstein $d_9$ differential from $\frac{\ga}{\rho^9 \tau} h_2 c_1$ in coweight 7 and stem 31 to $\frac{\gamma}{\tau^7} h_0^3 h_4^2$ or else $\frac{\ga}{\rho^k \tau} h_2 c_1$ supports a longer Bockstein differential. 
The uncertain status of the element $\frac{\gamma}{\tau^7} h_0^3 h_4^2$ is indicated in \cref{fig:Ext} by
an open square.
See \cref{UnknownBockstein} for a discussion of other unknown differentials.

\begin{thm}
\cref{tbl:Bockstein-differentials-Qgamma}
lists some differentials in the $\rho$-Bockstein spectral sequence
that converges to $\Ext_{NC}$.
\end{thm}

See also \cref{DiffsGamma} and \cref{rmk:DiffsGamma} for some
additional Bockstein differentials.

\begin{proof}
Some of the differentials are proved in \cite{C2MW0}.
The remainder are proved in the following lemmas.  The last
column of the table refers to the specific result where each
differential is proved.
\end{proof}

\renewcommand*{\arraystretch}{1.3}
\begin{longtable}{LLLLl} 
\caption{Some Bockstein differentials in $E^-$} \\
\toprule
\mbox{}(s,f,c) & \textrm{source} & d_r & \textrm{target} & proof \\
\midrule \endhead
\bottomrule \endfoot
\label{tbl:Bockstein-differentials-Qgamma}
(8,3,0) & \frac{Q}{\rho^3} h_1^4 & d_3 & \frac{\ga}{\tau^3} h_0^3 h_3 & \cite{C2MW0}*{Lemma 4.4} \\
(10,4,0) & \frac{Q}{\rho^4} h_1^5 & d_4 & \frac{\ga}{\tau^4} P h_1 & \cite{C2MW0}*{Lemma 4.4} \\
(16,7,0) & \frac{Q}{\rho^7} h_1^8 & d_7 & \frac{\ga}{\tau^7} h_0^7 h_4 & \cite{C2MW0}*{Lemma 4.4} \\
(18,8,0) & \frac{Q}{\rho^8} h_1^9 & d_8 & \frac{\ga}{\tau^8} P^2 h_1 & \cite{C2MW0}*{Lemma 4.4} \\
(24,11,0) & \frac{Q}{\rho^{11}} h_1^{12} & d_{11} & \frac{\ga}{\tau^{11}} h_0^5 i & \cite{C2MW0}*{Lemma 4.4} \\
(26,12,0) & \frac{Q}{\rho^{12}} h_1^{13} & d_{12} & \frac{\ga}{\tau^{12}} P^3 h_1 & \cite{C2MW0}*{Lemma 4.4} \\
(32,15,0) & \frac{Q}{\rho^{15}} h_1^{16} & d_{15} & \frac{\ga}{\tau^{15}} h_0^{15} h_5 & \cite{C2MW0}*{Lemma 4.4} \\
(34,16,0) & \frac{Q}{\rho^{16}} h_1^{17} & d_{16} & \frac{\ga}{\tau^{16}} P^4 h_1 & \cite{C2MW0}*{Lemma 4.4} \\
(12,4,3) & \frac{Q}{\rho} h_1^2 c_0 & d_1 & \frac{\ga}{\tau^2} P h_2 & \cref{lem:d-Qh1^kc0}(1) \\
(15,5,3) & \frac{Q}{\rho^3} h_1^3 c_0 & d_3 & \frac{\ga}{\tau^3} h_0^2 d_0 & \cref{lem:d-Qh1^kc0}(2) \\
(17,6,3) & \frac{Q}{\rho^4} h_1^4 c_0 & d_4 & \frac{\ga}{\tau^4} P c_0 & \cref{lem:d-Qh1^kc0}(3) \\
(16,7,4) & \frac{Q}{\rho^3} P h_1^4 & d_3 & h_1^3 \cdot Q h_1^3 c_0 + \frac{\ga}{\tau^3} h_0^7 h_4 & \cref{lem:d3-QPh1^4} \\
(23,7,6) & \frac{Q}{\rho^4} h_1^4 d_0 & d_4 & \frac{\ga}{\tau^4} P d_0 & \cref{lem:d4-Qh1^4d0} \\
(20,8,7) & \frac{Q}{\rho} P h_1^2 c_0 & d_1 & \frac{\ga}{\tau^2} P^2 h_2 & \cref{lem:d-QPh1^kc0}(1) \\
(23,9,7) & \frac{Q}{\rho^3} P h_1^3 c_0 & d_3 & h_1^3 \cdot Q h_1^4 d_0 + \frac{\ga}{\tau^3} P h_0^2 d_0 & \cref{lem:d-QPh1^kc0}(2) \\ 
(21,5,7) & \frac{Q}{\rho} h_1^2 e_0 & d_1 & \frac{\ga}{\tau} h_0 h_2 e_0 & \cref{lem:d-QPh1^ke0}(1) \\
(24,6,7) & \frac{Q}{\rho^3} h_1^3 e_0 & d_3 & \frac{\ga}{\tau^4} i & \cref{lem:d-QPh1^ke0}(2) \\
(26,7,7) & \frac{Q}{\rho^4} h_1^4 e_0 & d_4 & \frac{\ga}{\tau^4} P e_0 & \cref{lem:d-QPh1^ke0}(3) \\
(29,9,7) & \frac{Q}{\rho^5} h_1^6 e_0 & d_5 & \frac{\ga}{\tau^5} h_0^2 d_0^2 & \cref{lem:d-QPh1^ke0}(4) \\
(31,10,7) & \frac{Q}{\rho^6} h_1^7 e_0 & d_6 &\frac{\ga}{\tau^6}P c_0 d_0 & \cref{lem:d-QPh1^ke0}(5) \\
(34,11,7) & \frac{Q}{\rho^8} h_1^8 e_0 & d_8 & \frac{\ga}{\tau^8} P^2 e_0 & \cref{lem:d-QPh1^ke0}(6) \\
(24,11,8) & \frac{Q}{\rho^3} P^2 h_1^4 & d_3 & \frac{\ga}{\tau^3} h_0^5 i & \cref{lem:d-QP^2h1^k}(1) \\
(26,12,8) & \frac{Q}{\rho^4} P^2 h_1^5 & d_4 & \frac{\ga}{\tau^4} P^3 h_1 & \cref{lem:d-QP^2h1^k}(2) \\
(24,6,9) & \frac{Q}{\rho} c_0 d_0 & d_1 & \frac{\ga}{\tau^2} i & \cref{lem:d-Qh1^kc0d0}(1) \\
(26,7,9) & \frac{Q}{\rho^2} h_1 c_0 d_0 & d_2 & \frac{\ga}{\tau^2} P e_0 & \cref{lem:d-Qh1^kc0d0}(2) \\
(29,9,9) & \frac{Q}{\rho^3} h_1^3c_0 d_0 & d_3 & \frac{\ga}{\tau^3}h_0^2 d_0^2 & \cref{lem:d-Qh1^kc0d0}(3) \\
(27,6,10) & \frac{Q}{\rho} c_0 e_0 & d_1 & \frac{\ga}{\tau^2}  j & \cref{lem:d-Qh1^kc0d0}(3) \\
(29,7,10) & \frac{Q}{\rho^2} h_1 c_0 e_0 & d_2 & \frac{\ga}{\tau^2}  d_0^2 & \cref{lem:d-Qh1^kc0d0}(4) \\
\end{longtable}
\renewcommand*{\arraystretch}{1.0}

The differentials in \cref{tbl:Bockstein-differentials-Qgamma} determine additional
differentials (which we do not list) via the Leibniz rule.
 For example, 
the table states that 
$d_4\left(\frac{Q}{\rho^4} h_1^4 e_0\right) = \frac{\ga}{\tau^4}P e_0$. 
Multiplying by $h_1$ gives that $d_4\left(\frac{Q}{\rho^4} h_1^5 e_0\right) = \frac{\ga}{\tau^4}P h_1 e_0$.

Some of the differentials in \cref{tbl:Bockstein-differentials-Qgamma} are predicted by \cref{Type2diffs}.
For example, as
indicated in \cref{sec:h1invBock}, the $h_1$-periodic Bockstein differential $d_3( P h_1 ) = \rho^3 h_1^3 c_0$ gives rise 
 to Bockstein differentials $d_3 \left( \frac{Q}{\rho^3} P h_1^k \right) = Q h_1^{k+2} c_0$,
for $k \geq 4$.
However, we remind the reader that $Q h_1^{k+2} c_0$ is not in general a well-defined class in $E_r^{-}$. For example, while $Q h_1^2 c_0$ is uniquely defined, the symbol $Q h_1^6 c_0$ really denotes the pair of elements
\[ Q h_1^6 c_0 = \{ h_1^4 \cdot Q h_1^2 c_0, \ h_1^4 \cdot Q h_1^2 c_0 + \frac{\ga}{\tau^3} h_0^7 h_4\}. \]
\cref{Type2diffs} indicates that $d_3\left( \frac{Q}{\rho^3} P h_1^k \right)$ is one of these two elements, and a further argument is needed to determine 
the actual value of the differential.

\begin{rmk}
\label{UnknownBockstein}
In the range under consideration in this article,
there are a number of classes that support Bockstein differentials 
in stems beyond our range. For example, the class 
$\frac{\ga}{\tau^8} h_2 h_4$ in 
degree $(18,2,-1)$. This class is not the target of a Bockstein 
differential and therefore must support some differential after 
sufficiently dividing by $\rho$, by \cref{colocalVanishes}. We have 
verified directly that this Bockstein differential must occur after 
the $E_{12}^-$-page. 
Therefore, the classes $\frac{\ga}{\rho^k \tau^8} h_2 h_4$, 
with $k$ up to 12, all survive to the Bockstein $E_\infty$-page. 
We do not bother to list all such similar classes, though they are 
indicated in \cref{fig:Ext} by $\rho$-cotowers that extend beyond 
the 30-stem.
\end{rmk}

The Bockstein differentials of \cref{tbl:Bockstein-differentials-Qgamma},
\cref{DiffsGamma}, and \cref{rmk:DiffsGamma} can be used to determine
the Bockstein $E_\infty$-page in a range.
\cref{fig:Ext} displays the Bockstein $E_\infty$-page in stems less than
31 and coweights between $-2$ and $8$,
while \cref{fig:EinfNegCoweight} displays the
Bockstein $E_\infty$-page 
(which agrees with the Adams $E_\infty$-page in this range)
in stems less than $8$ and coweights
$-9$ through $-2$.

Beware that the dashed lines in the charts are not part of the structure
of the Bockstein $E_\infty$-page; they indicate extensions in the
$C_2$-equivariant Adams $E_2$-page that are hidden by the Bockstein
spectral sequence.  Such hidden extensions are discussed in
\cref{sctn:hidden}.

\begin{lem}
\label{lem:d-Qh1^kc0}
\mbox{}
\begin{enumerate}
\item
$d_1 \left( \frac{Q}{\rho} h_1^2 c_0 \right) = \frac{\gamma}{\tau^2} P h_2$.
\item
$d_3 \left( \frac{Q}{\rho^3} h_1^3 c_0 \right) = \frac{\gamma}{\tau^3} h_0^2 d_0$.
\item
$d_4 \left( \frac{Q}{\rho^4} h_1^4 c_0 \right) = \frac{\gamma}{\tau^4} P c_0$.
\end{enumerate}
\end{lem}

\begin{proof}
We will show that $\frac{Q}{\rho^4} h_1^4 c_0$ supports a $d_4$ differential. 
It follows that
the classes  $\frac{Q}{\rho^4} h_1^2 c_0$ and $\frac{Q}{\rho^4} h_1^3 c_0$ 
must support shorter differentials, and there is only one possible value
(and length) for each of these shorter differentials.

The $E^+$ differential $d_7( \tau^8 P h_1^2) = \rho^7 \tau^4 P c_0$ \cite{BI}*{Table~5} gives
\[ d_7\left( \frac{\ga}{\rho^7 \tau^8} P h_1^2\right) = \frac{\ga}{\tau^{12}} P c_0.\]
The source of this differential is $\tau^8$-torsion, so we conclude that $\frac{\ga}{\tau^4} P c_0$ cannot survive to $E_7^-$.
There is only one possible differential that could take a value of
$\frac{\ga}{\tau^4} P c_0$.
\end{proof}

\begin{lem}
\label{lem:d3-QPh1^4}
$d_3 \left( \frac{Q}{\rho^3} P h_1^4 \right) = h_1^3 \cdot Q h_1^3 c_0 +
\frac{\gamma}{\tau^3} h_0^7 h_4$.
\end{lem}

\begin{pf}
In $\Ext_\R$, there is an $h_1$-extension from $\tau P h_0^2 d_0$ to $ \rho P^2 c_0$
that is hidden in the $\rho$-Bockstein spectral sequence. Multiplying by $\frac{\ga}{\rho \tau^4}$
gives an $h_1$-extension in $\ExtCT$ from $\frac{\ga}{\rho \tau^3} P h_0^2 d_0$ to $ \frac{\ga}{\tau^4} P^2 c_0$.
\cref{lem:d-QPh1^kc0} shows that
$\frac{\ga}{ \tau^3} P h_0^2 d_0$ equals $h_1^3 \cdot Q h_1^4 d_0$ already in the Bockstein $E_4^{-}$-page.
We find that $h_1^4 \cdot \frac{Q}{\rho} h_1^4 d_0$ is 
equal to the nonzero element $\frac{\gamma}{\tau^4} P^2 c_0$
in $\ExtCT$. 
This nonzero element also equals
$h_1^3 c_0 \cdot \frac{Q}{\rho} h_1^3 c_0$.
Therefore, $h_1^3 \cdot \frac{Q}{\rho} h_1^3 c_0$ is also nonzero in $\ExtCT$. The only possibility is that it equals $\frac{\ga}{\rho \tau^3} h_0^7 h_4$.
It follows that
$h_1^3 \cdot \frac{Q}{\rho} h_1^3 c_0 + \frac{\ga}{\rho \tau^3} h_0^7 h_4$
is hit by some Bockstein differential.  
Inspection reveals that there is only one possible source.
\end{pf}

\begin{lem}
\label{lem:d4-Qh1^4d0}
$d_4 \left( \frac{Q}{\rho^4} h_1^4 d_0 \right) = \frac{\gamma}{\tau^4} P d_0$.
\end{lem}

\begin{proof}
The $E^+$ differential $d_7( \tau^8 P c_0) = \rho^7 \tau^4 P d_0$ 
gives
\[ d_7\left( \frac{\ga}{\rho^7 \tau^8} P c_0 \right) = \frac{\ga}{\tau^{12}} P d_0.\]
The source of this differential is $\tau^8$-torsion, so we conclude that $\frac{\ga}{\tau^4} P d_0$ cannot survive to $E_7^-$. 
There is only one possible differential that could take a value of
$\frac{\gamma}{\tau^4} P d_0$.
\end{proof}

\begin{lem}
\label{lem:d-QPh1^kc0}
\mbox{}
\begin{enumerate}
\item
$d_1 \left( \frac{Q}{\rho} P h_1^2 c_0 \right) = \frac{\gamma}{\tau^2} P^2 h_2$.
\item
$d_3 \left( \frac{Q}{\rho^3} P h_1^3 c_0 \right) = h_1^3 \cdot Q h_1^4 d_0 +
\frac{\gamma}{\tau^3} P h_0^2 d_0$.
\end{enumerate}
\end{lem}

\begin{pf}
We start with the second formula.
As discussed in  \cref{sec:h1invBock},
$h_1$-periodic computations show that
$d_3 \left( \frac{Q}{\rho^3} P h_1^k c_0 \right)$ equals
$Q h_1^{k+4} d_0$ for large values of $k$.
Therefore, $d_3 \left( \frac{Q}{\rho^3}P h_1^3 c_0 \right)$ belongs to
\[  
Q h_1^7 d_0 = \{ h_1^3 \cdot Q h_1^4 d_0, \ h_1^3 \cdot Q h_1^4 d_0  + \frac{\ga}{\tau^3} h_0^2 P d_0 \}. 
\]
Alternatively, one can apply \cref{Type2diffs} to obtain the same
formula.

We showed in \cref{lem:d4-Qh1^4d0} that 
$d_4 \left( \frac{Q}{\rho^4} h_1^4 d_0 \right) = \frac{\ga}{\tau^4} P d_0$, 
which implies that the differential
$d_4 \left( h_1^3 \cdot \frac{Q}{\rho^4}  h_1^4 d_0 \right)$ is 
$\frac{\ga}{\tau^4} P h_1^3 d_0$. 
In particular, $h_1^3 \cdot \frac{Q}{\rho^4} h_1^4 d_0$ cannot be the target of a $d_3$ differential.

We have now computed the differential on 
$\frac{Q}{\rho^3} P h_1^3 c_0$.
It follows that
$\frac{Q}{\rho^3} P h_1^2 c_0$ 
must support a shorter differential, and there is only one possible value
(and length) for this shorter differential.
\end{pf}

\begin{lem}
\label{lem:d-QPh1^ke0}
\mbox{}
\begin{enumerate}
\item
$d_1 \left( \frac{Q}{\rho} h_1^2 e_0 \right) = \frac{\gamma}{\tau} h_0 h_2 e_0$.
\item
$d_3 \left( \frac{Q}{\rho^3} h_1^3 e_0 \right) = \frac{\gamma}{\tau^4} i$.
\item
$d_4 \left( \frac{Q}{\rho^4} h_1^4 e_0 \right) = \frac{\gamma}{\tau^4} P e_0$.
\item
$d_5 \left( \frac{Q}{\rho^5} h_1^6 e_0 \right) = \frac{\gamma}{\tau^5} h_0^2 d_0^2$.
\item
$d_6 \left( \frac{Q}{\rho^6} h_1^7 e_0 \right) = \frac{\gamma}{\tau^6} P c_0 d_0$.
\item
$d_8\left(\frac{Q}{\rho^8} h_1^8 e_0\right) = \frac{\ga}{\tau^8} P^2 e_0$.
\end{enumerate}
\end{lem}

\begin{proof}
\cref{tbl:Bockstein-differentials-Qgamma}
gives the differential $d_4 \left( \frac{Q}{\rho^4} h_1^5 \right) = \frac{\gamma}{\tau^4} P h_1$.
Multiply by the permanent cycle $e_0$ to obtain that
$d_4 \left( \frac{Q}{\rho^4} h_1^5 e_0 \right) = \frac{\ga}{\tau^4} P h_1 e_0$.
It follows that 
$d_4 \left( \frac{Q}{\rho^4} h_1^4 e_0 \right) = \frac{\ga}{\tau^4} P e_0$.
This establishes formula (3).

Then $\frac{Q}{\rho^4} h_1^2 e_0$ and $\frac{Q}{\rho^4} h_1^3 e_0$ must
support shorter differentials, and there is only one possibility for each.
This establishes formulas (1) and (2).

Similarly, 
\cref{tbl:Bockstein-differentials-Qgamma}
gives the differential 
$d_8 \left( \frac{Q}{\rho^8} h_1^9 \right) = \frac{\gamma}{\tau^8} P^2 h_1$.
Multiply by $e_0$ to obtain that
$d_8\left(\frac{Q}{\rho^8} h_1^9 e_0\right) = \frac{\ga}{\tau^8} P^2 h_1 e_0$.
This is a nonzero differential, as $\frac{\ga}{\tau^8} P^2 h_1 e_0$ is 
not hit by any earlier differentials.
It follows that $d_8\left(\frac{Q}{\rho^8} h_1^8 e_0\right) = \frac{\ga}{\tau^8} P^2 e_0$.  This establishes formula (6).

Therefore, $Q h_1^6 e_0$ and $Q h_1^7 e_0$ must support earlier
differentials, and there is only one possibility for each.
This establishes formulas (4) and (5).
\end{proof}

\begin{lem}
\label{lem:d-QP^2h1^k}
\mbox{}
\begin{enumerate}
\item
$d_3 \left( \frac{Q}{\rho^3} P^2 h_1^4 \right) = \frac{\gamma}{\tau^3} h_0^5 i$.
\item
$d_4 \left( \frac{Q}{\rho^4} P^2 h_1^5 \right) = \frac{\gamma}{\tau^4} P^3 h_1$.
\end{enumerate}
\end{lem}

\begin{proof}
By \cref{Type2diffs}, the differential $d_7(P^2 h_1) = \rho^7 h_1^6 e_0$ implies that the differential $d_7 \left( \frac{Q}{\rho^7} P^2 h_1^5 \right)$ equals $Q h_1^{10} e_0$, provided that $\frac{Q}{\rho^7} P^2 h_1^5$ survives to the $E_7^-$ page.

However, multiplying the differential $d_8\left( \frac{Q}{\rho^8} h_1^9 \right) = \frac{\ga}{\tau^8} P^2 h_1$ by $h_1 e_0$ gives that
\[ d_8\left( \frac{Q}{\rho^8} h_1^{10} e_0 \right) = \frac{\ga}{\tau^8} P^2 h_1^2 e_0.\]
This target is not hit by any earlier differentials.
As $Q h_1^{10} e_0$ has no indeterminacy,  we conclude that $\frac{Q}{\rho^8} h_1^{10} e_0$ is not a permanent cycle and therefore cannot be the target of a $d_7$ differential on $\frac{Q}{\rho^{15}} P^2 h_1^5$.

It then follows that $\frac{Q}{\rho^7} P^2 h_1^5$, and also $\frac{Q}{\rho^7} P^2 h_1^4$, must support shorter Bockstein differentials. The claimed formulas are the only possibilities.
\end{proof}

\begin{lem}
\label{lem:d-Qh1^kc0d0}
\mbox{}
\begin{enumerate}
\item
$d_1 \left( \frac{Q}{\rho} c_0 d_0 \right) = \frac{\gamma}{\tau^2} i$.
\item
$d_2 \left( \frac{Q}{\rho^2} h_1 c_0 d_0 \right) = \frac{\gamma}{\tau^2} P e_0$.
\item
$d_1 \left( \frac{Q}{\rho} c_0 e_0 \right) = \frac{\gamma}{\tau^2} j$.
\item
$d_2 \left( \frac{Q}{\rho^2} h_1 c_0 e_0 \right) = \frac{\gamma}{\tau^2} d_0^2$.
\end{enumerate}
\end{lem}

\begin{proof}
Formulas (1) and (3) follow from the hidden $h_0$-extensions
from $Q c_0 d_0$ to $\frac{\gamma}{\tau} i$ and from
$Q c_0 e_0$ to $\frac{\gamma}{\tau} j$ (see \cref{tab:E1-minus-extn}),
together with the differentials
$d_1 \left( \frac{\gamma}{\rho \tau} i \right) = \frac{\gamma}{\tau^2} h_0 i$
and 
$d_1 \left( \frac{\gamma}{\rho \tau} j \right) = \frac{\gamma}{\tau^2} h_0 j$.

For formula (2), the $\R$-motivic Bockstein differential 
$d_5(\tau^8 P d_0) = \rho^5 \tau^6 P h_1 e_0$ implies that
$d_5 \left( \frac{\ga}{\rho^5 \tau^8} P d_0 \right) = \frac{\ga}{\tau^{10}} P h_1 e_0$. 
The source of this differential is killed by $\tau^8$, so $\frac{\ga}{\tau^2} P h_1 e_0$ 
must already be hit before the $E_5^-$-page. The differential 
$d_2 \left( \frac{Q}{\rho^2} h_1^2 c_0 d_0 \right) = \frac{\ga}{\tau^2} P h_1 e_0$ is the only possibility.

The argument for formula (4) is nearly identical to the argument for
formula (2), using the $\R$-motivic Bockstein differential
$d_5(\tau^8 P e_0) = \rho^5 \tau^6 h_1 d_0^2$.
\end{proof}

%%%%%%%%%%%%%%%%%%%%%%%%%%%%%%%%%%%%%%%%%%%%%%%%%%%%%%%%%%%%%%%
%%%%%%%%%%%%%%%%%%%%%%%%%%%%%%%%%%%%%%%%%%%%%%%%%%%%%%%%%%%%%%%
\section{Hidden extensions in the $\rho$-Bockstein spectral sequence}
\label{sctn:hidden}

Our next goal is to pass from the $C_2$-equivariant $\rho$-Bockstein
$E_\infty$-page to $\Ext_{C_2}$.  As a general rule, we use the same
notation for an element of $\Ext_{C_2}$ and its representative in the
$\rho$-Bockstein $E_\infty$-page.  In principle, this scheme can be 
ambiguous since $E_\infty$-page elements can detect more than 
one element of $\Ext_{C_2}$ (in the presence of elements in higher
$\rho$-Bockstein filtration).  
In practice, these ambiguities rarely matter.

\begin{eg}
Consider $\frac{\gamma}{\tau} P h_0 h_2$
in degree $(11, 6, 3)$.  This Bockstein $E_\infty$-page element
detects two elements of $\Ext_{C_2}$ because of the presence of
$h_1^3 c_0$ in higher filtration.  However, it is easy to distinguish
these two elements because one is a multiple of $h_0$ and the other is not.
\end{eg}

\begin{eg}
A slightly more troublesome example occurs in degree $(29,8,5)$.
We have both $\frac{\gamma}{\rho \tau^6} d_0^2$ and $\frac{\gamma}{\tau^7} h_0 k$
in the Bockstein $E_\infty$-page, so
$\frac{\gamma}{\rho \tau^6} d_0^2$ detects two elements of $\Ext_{C_2}$.
Both elements have the property that they equal
$\frac{\gamma}{\tau^6} d_0^2$ after multiplication by $\rho$.
We will show later in this section that
one of the elements is a multiple of $h_1$.
However, it turns out that 
the product $\frac{\gamma}{\rho \tau^6} \cdot d_0^2$ equals the
other element, i.e., is not a multiple of $h_1$.
\end{eg}

In passing from the 
$C_2$-equivariant $\rho$-Bockstein $E_\infty$-page to 
$\Ext_{C_2}$, there are a large number of hidden extensions to be
resolved.  We do not attempt an exhaustive study of this
algebraic problem.  However, we will 
establish all hidden extensions by $h_0$ and $h_1$ in the range of 
coweight $-2 \leq c \leq 8$ and stem $0 \leq s \leq 30$.

In practice, we observe that most hidden extensions occur in
$\tau$-periodic families in the following sense.
For each hidden extension from
$\frac{\gamma}{\rho^a \tau^b} x$ to
$\frac{\gamma}{\rho^c \tau^d} y$, there exists some $n$ such that
there are also hidden extensions from
$\frac{\gamma}{\rho^a \tau^{b+2^n e}} x$ to
$\frac{\gamma}{\rho^c \tau^{d+2^n e}} y$ for all $e$.
In the range that we study, 
the values of $2^n$ are typically $4$ or $8$, although larger values of $2^n$ do occasionally occur. 
The first four rows of \cref{tbl:hiddenh0extns} give examples
of this phenomenon when $2^n$ equals $4$, $8$, or $16$.
We will not make use of 
infinite $\tau^{2^n}$-divisibility
in hidden extensions
in this manuscript, but it deserves further study.

Not every hidden extension belongs to a $\tau$-periodic family.
For example, there is a hidden $h_0$-extension from $\frac{Q}{\rho^2} h_1^3 e_0$
to $\frac{\gamma}{\tau^3} i$ in coweight 7 and stem 23.

\subsection{Hidden $h_0$-extensions}

\begin{thm}
\label{Ext-h0-extns} 
The charts in \cref{fig:Ext} show all hidden $h_0$-extensions 
in the $\rho$-Bockstein spectral sequence
in stems less than 31 and coweights from $-2$ to $8$.
The charts in \cref{fig:EinfNegCoweight} show all hidden $h_0$-extensions 
in the $\rho$-Bockstein spectral sequence
in stems less than 8 and coweights from $-9$ to $-2$.
\end{thm}

The hidden $h_0$-extensions appear in the charts as vertical dashed lines.

\begin{proof}
Many of the extensions are detected by $\Ext_\R$.  We do not discuss
these extensions further since they are considered carefully in \cite{BI}.

We use several distinct approaches to study hidden extensions in $\Ext_{NC}$.
To start, 
many potential extensions are ruled out by considering relations in $\Ext_\R$. For instance, the relation $\rho \cdot h_0=0$ shows that 
in order for there to be an $h_0$-extension from $x$ to $y$, the class $x$ cannot be $\rho$-divisible, and $y$ must be $\rho$-torsion. This greatly constrains where $h_0$-extensions can occur. Similarly, the relation $h_0 \cdot h_1=0$ greatly constrains both $h_0$-extensions and $h_1$-extensions.

A great number of the extensions displayed in the charts are obtained by the following method.
This method applies to all hidden $h_0$-extensions from $x$ to $y$
in which both $x$ and $y$ are not divisible by $\rho$.
A similar method is employed in \cite{BI}*{Section~7} to establish hidden extensions.

The short exact sequence $NC_\rho \to NC \xrtarr{\rho} NC$ induces a long exact sequence
\begin{equation} 
(E_{1,\rho}^-)^{s+1,f,c} \to \Ext_{NC}^{s+1,f,c} \xrtarr{\rho} \Ext_{NC}^{s,f,c} \xrtarr{\delta} (E_{1,\rho}^-)^{s,f+1,c-1} \to \dots
\label{eq:FiberRho}
\end{equation}
that is similar to
the middle column of \cref{DoubleExactDiagram}.
Note that we may identify $E_{1,\rho}$ with $\Ext(NC_\rho)$, as the Bockstein spectral sequence for $\Ext(NC_\rho)$ collapses on the $E_1$-page.
Thus, if $x$ and $y$ are classes in $\Ext_{NC}$ that are not $\rho$-divisible, then an $h_0$-extension from $x$ to $y$ in $\Ext_{NC}$ can be detected in $E_1^-$, at least up to $\rho$-multiples.
Recall from \cref{TranslateRhoBssODelta} that
 the connecting homomorphism $\delta$ is specified as follows. 
Suppose that $x$ in $E^-$ is a permanent cycle in the $\rho$-Bockstein spectral sequence and that $x=\rho \cdot w$, where $w$ supports a Bockstein differential $d(w)=y$. Then $\delta([x]) = y$.

For example, consider the elements $\frac{\gamma}{\rho^2 \tau} h_1 c_0$
and $\frac{\gamma}{\tau^3}{P h_2}$ in coweight $1$ and stem $11$.
The differentials
$d_3 \left(\frac{\gamma}{\rho^3 \tau} h_1 c_0 \right) = \frac{\gamma}{\tau^4} P h_2$ and
$d_1 \left(\frac{\gamma}{\rho \tau^3} P h_2 \right) = \frac{\gamma}{\tau^4} P h_0 h_2$
imply that $\delta$ takes
$\frac{\gamma}{\rho^2 \tau} h_1 c_0$
and $\frac{\gamma}{\tau^3}{P h_2}$ to 
$\frac{\gamma}{\tau^4} P h_2$ and 
$\frac{\gamma}{\tau^4} P h_0 h_2$ respectively.
These latter elements are connected by an $h_0$-extension,
so their pre-images under $\delta$ are also connected by an $h_0$-extension.

The same method also applies to show that certain $h_0$-extensions do not 
occur.  For example, consider the elements
$\frac{\gamma}{\tau^3} h_0 h_2 h_4$ and $\frac{\gamma}{\tau^3} f_0$
in coweight $4$ and stem $18$.
The differentials
$d_1 \left( \frac{\gamma}{\rho \tau^3} h_0 h_2 h_4 \right) = 
\frac{\gamma}{\tau^3} h_1^3 h_4$
and
$d_1 \left( \frac{\gamma}{\rho \tau^3} f_0 \right) =
\frac{\gamma}{\tau^3} h_1 e_0$
imply that $\delta$ takes
$\frac{\gamma}{\tau^3} h_0 h_2 h_4$ and $\frac{\gamma}{\tau^3} f_0$
to 
$\frac{\gamma}{\tau^3} h_1^3 h_4$ and $\frac{\gamma}{\tau^3} h_1 e_0$
respectively.
These latter elements are not connected by an $h_0$-extension,
so their pre-images under $\delta$ are also not connected by an $h_0$-extension.

The sequence \cref{eq:FiberRho} works well for establishing
hidden extensions between elements that
are not $\rho$-divisible.  However, there remain many possible
$h_0$-extensions to consider that involve $\rho$-divisible elements.
Most of these extensions are easily implied by other extensions, together
with the multiplicative structure.

For example, consider the elements
$\frac{\gamma}{\rho \tau} h_1 h_3$ and $\frac{\gamma}{\tau} h_1^2 h_3$
in coweight $1$ and stem $19$.
The first element is the product $\frac{\gamma}{\rho \tau^2} \cdot \tau h_1 h_3$.
We already know that there is an $\R$-motivic hidden $h_0$-extension
from $\tau h_1 h_3$ to $\rho \tau h_1^2 h_3$ \cite{BI}.
Then multiplication by $\frac{\gamma}{\rho \tau^2}$ gives the desired
hidden extension.

On the other hand, we also know from \cref{prop:hiddenh0extns} below that
there is a hidden $h_0$-extension from $\frac{\gamma}{\rho \tau^2}$
to $\frac{\gamma}{\tau^2} h_1$ (see \cref{tbl:hiddenh0extns}).
Multiplication by the $\R$-motivic
element $\tau h_1 h_3$ gives another proof of the desired extension.

In practice, there are many hidden extensions that can be established
from a previously known $\R$-motivic extension by multiplication with
an element of $\Ext_{NC}$.  There are also many hidden extensions that
can be established from a previously known extension in $\Ext_{NC}$
by multiplication with an $\R$-motivic element.  (And there are some
extensions, such as the example in the previous paragraph, that can
be established using both approaches.)

Finally, there are several additional extensions that cannot be
proved with the previously described methods.
These more difficult cases appear in \cref{tbl:hiddenh0extns}.
Their proofs are given in \cref{prop:hiddenh0extns}.
\end{proof}

\begin{rmk}
In the range under consideration in \cref{Ext-h0-extns},
one $h_0$-extension deserves slightly more discussion.
Consider the hidden $h_0$-extension in coweight $0$ and stem $26$
from $\frac{\gamma}{\rho^2 \tau^9} h_1 h_4 c_0$ to
$\frac{\gamma}{\tau^9} h_2^2 g$.
This extension follows immediately from $h_4$ multiplication on the
hidden $h_0$-extension in coweight $-7$ from
$\frac{\gamma}{\rho^2 \tau^9} h_1 c_0$ to $\frac{\gamma}{\tau^{11}} P h_2$.
Here we use the relation $P h_2 h_4 = \tau^2 h_2^2 g$
in $\C$-motivic $\Ext$.
The latter extension is deduced immediately from 
the long exact sequence \eqref{eq:FiberRho}.
This argument is entirely straightforward.  We draw attention to it
because it uses an extension in coweight $-7$ that does not appear on our charts,
but it is analogous to an extension in coweight $1$ and stem $11$ that 
does appear on our charts.
\end{rmk}

\begin{landscape}

\begin{longtable}{LLLll}
\caption{Some hidden $h_0$-extensions in the $\rho$-Bockstein $E^-_\infty$-page}
\label{tbl:hiddenh0extns} \\
\toprule \mbox{}
(s,f,c) & \text{source} & \text{target} &   \text{Massey product} & \text{Bockstein differential} \\
\midrule \endfirsthead
\caption[]{Some hidden $h_0$-extensions in the $\rho$-Bockstein $E^-_\infty$-page} \\
\toprule \mbox{}
(s,f,c) & \text{source} & \text{target} & \text{Massey product} & \text{Bockstein differential} \\
\midrule \endhead
\bottomrule \endfoot
(7,0,-16k -  9) & \frac{\ga}{\rho^7 \tau^{16k+8}} & \frac{\ga}{\tau^{16k+11}} h_3 & $\left\langle \frac{\gamma}{\rho^7 \tau^{16k+16}}, \rho^8, \tau^4 h_3 \right\rangle$ & $d_8(\tau^8) = \rho^8 \tau^4 h_3$ \\
(3,0,-8k - 5) & \frac{\ga}{\rho^3 \tau^{8k+4}} & \frac{\ga}{\tau^{8k+5}} h_2 & $\left\langle \frac{\gamma}{\rho^3 \tau^{8k+8}}, \rho^4, \tau^2 h_2 \right\rangle$ & $d_4(\tau^4) = \rho^4 \tau^2 h_2$ \\
(1,0,-4k - 3) & \frac{\ga}{\rho\tau^{4k+2}} & \frac{\ga}{\tau^{4k+2}} h_1  & $\left\langle \frac{\gamma}{\rho \tau^{4k+4}}, \rho^2, \tau h_1 \right\rangle$ & $d_2(\tau^2) = \rho^2 \tau h_1$\\
(8,2,-4k) & \frac{\ga}{\rho^2\tau^{4k+1}} h_2^2 & \frac{\ga}{\tau^{4k+2}}c_0 & $\left\langle \frac{\gamma}{\rho^2 \tau^{4k+4}}, \rho^3, \tau c_0 \right\rangle$ & $d_3(\tau^3 h_2^2) = \rho^3 \tau c_0$ \\
(16,6,0) & \frac{\ga}{\rho^2 \tau^5} h_0^2 d_0 & \frac{\ga}{\tau^6} P c_0 & $\left\langle \frac{\gamma}{\rho^2 \tau^8}, \rho^3, \tau P c_0 \right\rangle$ & $d_3( \tau^3 h_0^2 d_0 ) = \rho^3 \tau P c_0$ \\
(16,6,4) & \frac{\ga}{\rho^2 \tau} h_0^2 d_0 & \frac{\ga}{\tau^{2}} P c_0 & $\left\langle \frac{\gamma}{\rho^2 \tau^4}, \rho^3, \tau P c_0 \right\rangle$ & $d_3( \tau^3 h_0^2 d_0 ) = \rho^3 \tau P c_0$ \\
(24,10,0) & \frac{\ga}{\rho^2 \tau^9} P h_0^2 d_0 & \frac{\ga}{\tau^{10}} P^2 c_0  & $\left\langle \frac{\gamma}{\rho^2 \tau^{12}}, \rho^3, \tau P^2 c_0 \right\rangle$ & $d_3( \tau^3 P h_0^2 d_0 ) = \rho^3 \tau P^2 c_0$ \\
(24,10,4) & \frac{\ga}{\rho^2 \tau^5} P h_0^2 d_0 & \frac{\ga}{\tau^6} P^2 c_0 &  $\left\langle \frac{\gamma}{\rho^2 \tau^8}, \rho^3, \tau P^2 c_0 \right\rangle$ &  $d_3( \tau^3 P h_0^2 d_0 ) = \rho^3 \tau P^2 c_0$ \\
(24,10,8) & \frac{\ga}{\rho^2 \tau} P h_0^2 d_0 & \frac{\ga}{\tau^{2}} P^2 c_0 & $\left\langle \frac{\gamma}{\rho^2 \tau^4}, \rho^3, \tau P^2 c_0 \right\rangle$ &  $d_3( \tau^3 P h_0^2 d_0 ) = \rho^3 \tau P^2 c_0$ \\
(9,4,1) & \frac{\ga}{\rho^2\tau} h_0^3 h_3 & \frac{\ga}{\tau^2} P h_1 & $\left\langle \frac{\gamma}{\rho^2 \tau^4}, \rho^3, \tau P h_1 \right\rangle$ & $d_3 (\tau^3 h_0^3 h_3) = \rho^3 \tau P h_1$ \\
(17,8,1) & \frac{\ga}{\rho^2\tau^5} h_0^7 h_4 & \frac{\ga}{\tau^6} P^2 h_1 & $\left\langle \frac{\gamma}{\rho^2 \tau^8}, \rho^3, \tau P^2 h_1 \right\rangle$ &  $d_3( \tau^3 h_0^7 h_4 ) = \rho^3 \tau P^2 h_1$ \\
(17,8,5) & \frac{\ga}{\rho^2\tau} h_0^7 h_4 & \frac{\ga}{\tau^2} P^2 h_1 & $\left\langle \frac{\gamma}{\rho^2 \tau^4}, \rho^3, \tau P^2 h_1 \right\rangle$ &  $d_3( \tau^3 h_0^7 h_4 ) = \rho^3 \tau P^2 h_1$ \\
(25,12,1) & \frac{\ga}{\rho^2\tau^9} h_0^5 i & \frac{\ga}{\tau^{10}} P^3 h_1  &  $\left\langle \frac{\gamma}{\rho^2 \tau^{12}}, \rho^3, \tau P^3 h_1 \right\rangle$ &  $d_3 ( \tau^3 h_0^5 i ) = \rho^3 \tau P^3 h_1$ \\
(25,12,5) & \frac{\ga}{\rho^2\tau^5} h_0^5 i & \frac{\ga}{\tau^6} P^3 h_1  & $\left\langle \frac{\gamma}{\rho^2 \tau^8}, \rho^3, \tau P^3 h_1 \right\rangle$ &  $d_3 ( \tau^3 h_0^5 i ) = \rho^3 \tau P^3 h_1$ \\
(19, 2,3) & \frac{\ga}{\rho^5 \tau^2} h_3^2 & \frac{\ga}{\tau^4} c_1 & $\left\langle \frac{\gamma}{\rho^5 \tau^8}, \rho^6, \tau^3 c_1 \right\rangle$ & $d_6( \tau^6 h_3^2 ) = \rho^6 \tau^3 c_1$ \\
(26, 5, 0) & \frac{\gamma}{\rho^2 \tau^9} h_1 h_4 c_0 & \frac{\gamma}{\tau^9} h_2^2 g & $\left\langle \frac{\gamma}{\rho^2 \tau^{12}}, \rho^3, \tau^2 h_2^2 g \right\rangle$ & $d_3(\tau^3 h_1 h_4 c_0) = \rho^3 \tau^2 h_2^2 g$ \\
(26, 5, 8) & \frac{\gamma}{\rho^2 \tau} h_1 h_4 c_0 & \frac{\gamma}{\tau} h_2^2 g & $\left\langle \frac{\gamma}{\rho^2 \tau^4}, \rho^3, \tau^2 h_2^2 g \right\rangle$ & $d_3(\tau^3 h_1 h_4 c_0) = \rho^3 \tau^2 h_2^2 g$
\end{longtable}

\end{landscape}

\begin{prop}
\label{prop:hiddenh0extns}
\cref{tbl:hiddenh0extns} lists some hidden $h_0$-extensions
in the $\rho$-Bockstein spectral sequence.
\end{prop}

\begin{proof}
All of these extensions can be established with
shuffles involving Massey products of the form
$\left\langle \frac{\gamma}{\rho^{r-1} \tau^{2k}}, \rho^r, x \right\rangle$,
where $x$ is an element of $\Ext_\R$.
We compute these Massey products using the May convergence theorem \cite{MMP}*{Theorem 4.1} and 
$\R$-motivic Bockstein differentials of the form $d_r(x) = \rho^r y$.
Beware that the May convergence theorem has some technical hypotheses
involving crossing differentials, which are satisfied in all cases that we consider.
Next, we shuffle to obtain
\begin{equation}
\label{eq:h0-extn-shuffle}
h_0 \left\langle \frac{\gamma}{\rho^{r-1} \tau^{2k}}, \rho^r, x \right\rangle =
\left\langle h_0, \frac{\gamma}{\rho^{r-1} \tau^{2k}}, \rho^r\right\rangle x.
\end{equation}
Finally, the second Massey product 
in \eqref{eq:h0-extn-shuffle}
can be computed with the May convergence
theorem and the differentials given in \cref{rmk:DiffsGamma}.
For each extension,
the fourth column of \cref{tbl:hiddenh0extns} displays the relevant
Massey product, 
and the fifth column displays the Bockstein differential $d_r(x) = \rho^r y$
that computes it.

For example,
consider the elements $\frac{\gamma}{\rho^2 \tau} h_2^2$ and
$\frac{\gamma}{\tau^2} c_0$ in stem $8$ and coweight $0$.
We have the Massey product
\[
\frac{\gamma}{\rho^2 \tau} h_2^2 = \left\langle \frac{\gamma}{\rho^2 \tau^4}, \rho^3, \tau c_0 \right\rangle,
\]
which follows from the May convergence theorem using the differential $d_3(\tau^3 h_2^2) = \rho^3 \tau c_0$.
The indeterminacy of the Massey product is zero by inspection.
Similarly, we have the Massey product
\[
\frac{\gamma}{\tau^3} = \left\langle h_0, \frac{\gamma}{\rho^2 \tau^4}, \rho^3 \right\rangle,
\]
which follows from the May convergence theorem using the differential
$d_1 \left( \frac{\gamma}{\rho^3 \tau^3} \right) = \frac{\gamma}{\rho^2 \tau^4} h_0$.
By inspection, this Massey product also has no indeterminacy.
Given these Massey product computations,
the hidden $h_0$-extension from 
$\frac{\gamma}{\rho^2 \tau} h_2^2$ to
$\frac{\gamma}{\tau^2} c_0$ is a case of equation \cref{eq:h0-extn-shuffle}.
\end{proof}

\subsection{Hidden $h_1$-extensions}

\begin{thm}
\label{Ext-h1-extns}
The charts in \cref{fig:Ext} show all hidden $h_1$-extensions 
in the $\rho$-Bockstein spectral sequence
in stems less than 31 and coweights from $-2$ to $8$,
except that there are possible hidden $h_1$-extensions
in degrees $(29,4,-2)$ and $(29,4,6)$.
The charts in \cref{fig:EinfNegCoweight} show all hidden $h_1$-extensions 
in the $\rho$-Bockstein spectral sequence
in stems less than 8 and coweights from $-9$ to $-2$.
\end{thm}

The hidden $h_1$-extensions appear in the chart as dashed lines of slope $1$.
Most of the hidden $h_1$-extensions occur in families that are 
related by $\rho$ multiplications.  
It suffices to establish only one extension in each family because
multiplication (and division) by $\rho$ determines all of the rest.
Unknown $h_1$-extensions are indicated by dotted, rather
than dashed, lines of slope 1.

\begin{rmk}
The classes in degrees $(29,4,-2)$ and $(29,4,6)$ mentioned in 
\cref{Ext-h1-extns} are both $\rho$-divisible, so there are also
potential hidden $h_1$-extensions on classes in degree $(30,4,-2)$ 
and $(30,4,6)$.
\end{rmk}

\begin{proof}
Many of the extensions are detected by $\Ext_\R$.  We do not discuss
these extensions further since they are considered carefully in \cite{BI}.

Many of the extensions in $\Ext_{NC}$ can be verified using 
the long exact sequence \eqref{eq:FiberRho}.  
See the proof of \cref{Ext-h0-extns} for more detail.
For example, consider the elements $\frac{\gamma}{\rho \tau} h_2^2$
and $\frac{\gamma}{\tau^2}{c_0}$ in coweight $0$ and stems $7$ and $8$.
The differentials
$d_3 \left(\frac{\gamma}{\rho^3 \tau} h_2^2 \right) = \frac{\gamma}{\tau^3} c_0$ and
$d_2 \left(\frac{\gamma}{\rho^2 \tau^2} c_0 \right) = \frac{\gamma}{\tau^3} h_1 c_0$
imply that $\delta$ takes
$\frac{\gamma}{\rho^2 \tau} h_2^2$
and $\frac{\gamma}{\rho \tau^2}{c_0}$ to 
$\frac{\gamma}{\tau^3} c_0$ and 
$\frac{\gamma}{\tau^3} h_1 c_0$ respectively.
These latter elements are connected by an $h_1$-extension,
so their pre-images under $\delta$ are also connected by an $h_1$-extension.

Most of the remaining $h_1$-extensions
are easily implied by 
previously known $\R$-motivic extensions multiplied with
elements of $\Ext_{NC}$.  
For example, consider the elements
$\frac{\gamma}{\rho^4 \tau^4} h_3^2$ and $\frac{\gamma}{\tau^6} c_1$
in coweight $1$ and stems $18$ and $19$.
The first element is the product $\frac{\gamma}{\rho^4 \tau^8} \cdot \tau^4 h_3^2$.
We already know that there is an $\R$-motivic hidden $h_1$-extension
from $\tau^4 h_3^2$ to $\rho^4 \tau^2 c_1$ \cite{BI}.
Then multiplication by $\frac{\gamma}{\rho^4 \tau^8}$ gives the desired
hidden extension.

Finally, there are several additional extensions that cannot be proved
with the previously described methods.
These more difficult cases appear in \cref{prop:h1extntable},
\cref{lem:h1extGAta2h32}, and \cref{gata2f0-nonextn}.
\end{proof}

\begin{prop}
\label{prop:h1extntable}
\cref{table:h1extn} lists some hidden $h_1$-extensions in
the $\rho$-Bockstein spectral sequence.
\end{prop}

\renewcommand*{\arraystretch}{1.3}
\begin{longtable}{LLLL}
\caption{Some hidden $h_1$-extensions in the $\rho$-Bockstein spectral sequence}
\label{table:h1extn} \\
\toprule
{(s,f,c)} & \text{source} & \text{target}  & \text{proof} \\
\midrule \endfirsthead
\caption[]{Some hidden $h_1$-extensions in the $\rho$-Bockstein $E^-_\infty$-page} \\
\toprule
{(s,f,c)} & \text{source} & \text{target} & \text{proof}  \\
\midrule \endhead
\bottomrule \endfoot
(6,3,-9) & \frac{\gamma}{\rho^3 \tau^8} h_1^3 & \frac{\gamma}{\tau^{11}} h_0^3 h_3 & d_4(\tau h_0^3 h_3) = \rho^4 h_1^2 c_0 \\
(6,3,-5) & \frac{\gamma}{\rho^3 \tau^4} h_1^3 & \frac{\gamma}{\tau^{7}} h_0^3 h_3 & d_4(\tau h_0^3 h_3) = \rho^4 h_1^2 c_0 \\
(22, 6, -2) & \frac{\gamma}{\rho^2 \tau^9} h_0^2 g & \frac{\gamma}{\tau^{12}} i	& d_4(i) = \rho^4 h_1 c_0 e_0 \\
(29, 13, -2) & \frac{\gamma}{\rho^2 \tau^{14}} P^3 h_2 & \frac{\gamma}{\tau^{15}} P^2 h_0^2 d_0 & d_4(\tau P^2 h_0^2 d_0) = \rho^4 P^2 h_1^3 d_0 \\
(29, 11, 1)	& \frac{\gamma}{\rho^4 \tau^8} P h_1^3 d_0 & \frac{\gamma}{\tau^{12}} P^2 d_0 & d_5(\tau^4 P^2 d_0) = \rho^5 \tau^2 P^2 h_1 e_0 \\
(22,4,2) & \frac{\ga}{\rho^4 \tau^4} h_1^3 h_4 & \frac{\ga}{\tau^6} h_2 g & d_3(\tau^2 h_2 g) = \rho^3 h_1^2 h_4 c_0 \\
(22,6,2) & \frac{\ga}{\rho^2 \tau^5} h_0^2 g & \frac{\ga}{\tau^8} i & d_4(i) = \rho^4 h_1 c_0 e_0 \\
(16,8,3) & \frac{\ga}{\rho \tau^3} h_0^7 h_4 & \frac{\ga}{\tau^4} P^2 h_1 &  d_7(P^2 h_1) = \rho^7 h_1^6 e_0 \\
(22, 11, 3) & \frac{\gamma}{\rho^3 \tau^4} P^2 h_1^3 & \frac{\gamma}{\tau^7} h_0^5 i & d_4(\tau h_0^5 i) = \rho^4 P^2 h_1^2 c_0 \\
(24, 12, 3) & \frac{\gamma}{\rho \tau^7} h_0^5 i & \frac{\gamma}{\tau^8} P^3 h_1 & d_3(P^3 h_1) = \rho^3 P^2 h_1^3 c_0 \\
(29, 13, 6) & \frac{\gamma}{\rho^2 \tau^6} P^3 h_2 & \frac{\gamma}{\tau^7} P^2 h_0^2 d_0 & d_4(\tau P^2 h_0^2 d_0) = \rho^4 P^2 h_1^3 d_0 \\
\end{longtable}
\renewcommand*{\arraystretch}{1.0}

\begin{proof}
Each of the extensions is established with the same style of proof.
Starting from a previously known $\R$-motivic differential
$d_r(x) = \rho^r h_1 y$, the May convergence theorem \cite{MMP}*{Theorem 4.1} implies that
the Massey product $\left\langle \frac{\gamma}{\tau^k}, \rho^r, h_1 y \right\rangle$ contains $\frac{\gamma}{\tau^k} x$.  Then we have
\[
\left\langle \frac{\gamma}{\tau^k}, \rho^r, h_1 y \right\rangle =
\left\langle \frac{\gamma}{\rho^j \tau^k}, \rho^{r+j}, h_1 y \right\rangle =
\left\langle \frac{\gamma}{\rho^j \tau^k}, \rho^{r+j}, y \right\rangle h_1.
\]
This shows that $\frac{\gamma}{\tau^k} x$ is the target of a 
hidden $h_1$-extension, and there is only one possible value for the 
source of the extension.

Several technical hypotheses must be satisfied in order to carry out
this argument.  These hypotheses can be verified manually for each
extension in \cref{table:h1extn}.  The hypotheses are:
\begin{itemize}
\item
The number $j$ must satisfy the properties that $\rho^{r+j} y$ is zero
in $\Ext_\R$ and that $\frac{\gamma}{\rho^j \tau^k}$ survives the
$\rho$-Bockstein spectral sequence.  This ensures that the 
above displayed Massey products are well-defined.
\item
The Massey product
$\left\langle \frac{\gamma}{\rho^j \tau^k}, \rho^{r+j}, h_1 y \right\rangle$
must have no indeterminacy.  This ensures that the three Massey products
displayed above are in fact equal, rather than related by containment.
\item
There must be no crossing differentials for the Massey product
$\left\langle \frac{\gamma}{\tau^k}, \rho^r, h_1 y \right\rangle$.
This ensures that the May convergence theorem applies.
\item
There is only one possible source for the hidden $h_1$-extension.
\end{itemize}

The last column of \cref{table:h1extn} gives the relevant
$\R$-motivic differential for each case.
\end{proof}

\begin{rmk}
The published version of \cite{BI} does not include the
$\R$-motivic Bockstein differentials
$d_4(\tau P^2 h_0^2 d_0)$ and $d_5(\tau^4 P^2 d_0)$.
\end{rmk}

\begin{lem}
\label{lem:h1extGAta2h32} $(18,2,3)$.
There is a hidden $h_1$-extension in $\Ext_{NC}$ from $\frac{\ga}{\rho^4\tau^2} h_3^2$ to $\frac{\ga}{\tau^4} c_1$.
\end{lem}

\begin{pf}
We have the Massey product
$\frac{\gamma}{\tau^4} c_1 = 
\left\langle \frac{\gamma}{\rho^4 \tau^8}, \rho^6, \rho^2 h_2 \cdot \tau^2 c_1 \right\rangle$.
This bracket can be computed with the May convergence theorem \cite{MMP}*{Theorem 4.1} and the
$\R$-motivic
Bockstein differential $d_4(\rho^4 \tau^4 c_1) = \rho^8 \tau^2 h_2 c_1$.
By inspection, the bracket has no indeterminacy.

We also have a relation $h_1 \cdot \tau^3 c_1 = \rho^2 h_2 \cdot \tau^2 c_1$
in $\Ext_\R$ \cite{BI}.  Therefore,
\[
\frac{\gamma}{\tau^4} c_1 = 
\left\langle \frac{\gamma}{\rho^4 \tau^8}, \rho^6, \rho^2 h_2 \cdot \tau^2 c_1 \right\rangle =
\left\langle \frac{\gamma}{\rho^4 \tau^8}, \rho^6, h_1 \cdot \tau^3 c_1 \right\rangle =
\left\langle \frac{\gamma}{\rho^4 \tau^8}, \rho^6, \tau^3 c_1 \right\rangle h_1.
\]
The last equality holds because there is no indeterminacy.
We conclude that $\frac{\gamma}{\tau^4} c_1$ is the target of a hidden
$h_1$-extension, and there is only one possibility.
\end{pf}

\begin{lem}
\label{gata2f0-nonextn}
$(22,4,5)$.
The element $\frac{\ga}{\rho^4 \tau^2}f_0$ does not support a hidden 
$h_1$-extension.
\end{lem}

\begin{pf}
We have the Massey product
$\frac{\ga}{\rho^4 \tau^2}f_0 = 
\bracket{\frac{\ga}{\rho^4 \tau^8}, \rho^6, h_2 \cdot \tau^4 g}$.
This bracket can be computed with the May convergence theorem \cite{MMP}*{Theorem 4.1}
and the $\R$-motivic 
Bockstein differential $d_6(\tau^6 f_0) = \rho^6 \tau^4 h_2 g$.
By inspection, the bracket has no indeterminacy.

Now shuffle to obtain
\[ 
\frac{\gamma}{\rho^4 \tau^2} f_0 \cdot h_1 = 
\bracket{\frac{\ga}{\rho^4 \tau^8}, \rho^6, h_2 \cdot \tau^4 g} h_1 = 
\frac{\ga}{\rho^4 \tau^8} \bracket{\rho^6, h_2 \cdot \tau^4 g, h_1}.
\]
By inspection, the latter $\R$-motivic bracket equals 
$\{0, \rho^5 h_1 \cdot \tau^4 h_4 c_0\}$,
which is annihilated by $\frac{\ga}{\rho^4 \tau^8}$. 
\end{pf}

\begin{rmk}
The $\R$-motivic Bockstein differential $d_6(\tau^6 f_0) = \rho^6 \tau^4 h_2 g$
used in the proof of \cref{gata2f0-nonextn} does not appear in the published version of \cite{BI}.
\end{rmk}

\begin{rmk}
Note that $\frac{\gamma}{\rho^2 \tau^2} f_0$ detects the product
$\frac{\gamma}{\rho^2 \tau^4} \cdot \tau^2 f_0$,
and $\tau^2 f_0$ supports a hidden $h_1$-extension to $\rho^2 \tau^2 h_1 g$.
However, this does not imply the existence of a hidden extension
from $\frac{\gamma}{\rho^2 \tau^2} f_0$ to $\frac{\gamma}{\tau^2} h_1 g$.  The presence of the element $\frac{\gamma}{\tau^2} g$ in higher $\rho$-filtration interferes.
In fact, the product 
$\frac{\gamma}{\rho^2 \tau^4} \cdot \tau^2 f_0$ is not divisible by $\rho^2$.
\end{rmk}

We provide one more result about hidden $h_1$-extensions.
These extensions fall outside of the range considered in this article.
Nevertheless, we include them because they lie in very low stems and
are potentially of further interest.

\begin{lem}
\label{lem:h1extGAta2h1} $(2,1,-4k-3)$. 
There is a hidden $h_1$-extension  in $\Ext_{NC}$ from $ \frac{\ga}{\rho\tau^{4k+2}} h_1 $ to $ \frac{\ga}{\tau^{4k+3}} h_0 h_2$.
\end{lem}

\begin{pf}
We have the Massey product
\[
\frac{\ga}{\rho \tau^{4k+2}} h_1 = \bracket{\frac{\ga}{\rho \tau^{4k+4}},\rho^2, \tau h_1 \cdot h_1}.
\]
This bracket can be computed with the May convergence theorem \cite{MMP}*{Theorem 4.1}
and the $\R$-motivic 
Bockstein differential $d_2(\tau^{2} h_1) = \rho^2 \tau h_1^2$.
By inspection, there is no indeterminacy.

We also have the Massey product
\[
\frac{\gamma}{\tau^{4k+3}} = \bracket{ \frac{\gamma}{\rho \tau^{4k+4}}, \rho^2, h_0 }.
\]
This bracket can be computed with the May convergence theorem
and the $\R$-motivic Bockstein differential $d_1(\rho \tau) = \rho^2 h_0$.
By inspection, there is no indeterminacy.

Now we can compute that
\begin{align*}
\frac{\ga}{\rho \tau^{4k+2}} h_1 \cdot h_1 &= 
\bracket{\frac{\ga}{\rho \tau^{4k+4}},\rho^2, \tau h_1 \cdot h_1} h_1 = 
\bracket{\frac{\ga}{\rho \tau^{4k+4}},\rho^2, \tau h_1 \cdot h_1^2} \\
&= \bracket{\frac{\ga}{\rho \tau^{4k+4}},\rho^2, h_0^2 h_2}
= \bracket{\frac{\ga}{\rho \tau^{4k+4}},\rho^2, h_0} h_0 h_2 =
\frac{\ga}{\tau^{4k+3}} h_0 h_2.
\end{align*}

The first and fifth equalities are the Massey products computed in the previous paragraphs.  
The third equality is the $\R$-motivic relation
$h_0^2 h_2 = \tau h_1 \cdot h_1^2$.
The second and fourth equalities hold because there are no
indeterminacies.  Here we need that
$\frac{\gamma}{\rho \tau^{4k+4}} (\tau h_1)^2$ is zero because
of the Bockstein differential
$d_2 \left( \frac{\gamma}{\rho^2 \tau^{4k+1}} h_1 \right) =
\frac{\gamma}{\tau^{4k+2}} h_1^2$.
\end{pf}

%%%%%%%%%%%%%%%%%%%%%%%%%%%%%%%%%%%%%%%%%%%%%%%%%%%%%%%%%%%%%%%
%%%%%%%%%%%%%%%%%%%%%%%%%%%%%%%%%%%%%%%%%%%%%%%%%%%%%%%%%%%%%%%
\section{Adams differentials}
\label{sec:Adams}

In \cref{sctn:Bockstein-diff,sctn:hidden},
we computed the $C_2$-equivariant Adams $E_2$-page in a range,
including all extensions by $h_0$ and $h_1$.
Our next task is to compute Adams differentials.

Our Adams charts (in \cref{sec:charts}) are organized by coweight.  Since Adams differentials
decrease the coweight by $1$, we cannot display these differentials
graphically as lines connecting elements.  On the other hand, our charts
do show multiplications by $\rho$, $h_0$, and $h_1$.
Consequently, it is convenient to specify the Adams differentials
on all $\F_2[\rho, h_0, h_1]$-module generators of the 
Adams $E_2$-page.

Betti realization from $\R$-motivic homotopy theory to $C_2$-equivariant
homotopy theory induces a map of Adams spectral sequences.  This map allows
us to deduce information about the $C_2$-equivariant Adams spectral
sequence from information about the $\R$-motivic Adams spectral sequence.
The latter spectral sequence is thoroughly understood in a range \cite{BI}.
In fact, we will need some additional $\R$-motivic Adams differentials
that do not appear in \cite{BI}.  

\begin{prop}
\label{prop:R-Adams-diff}
\cref{tbl:R-Adams-diff} lists some $d_2$ differentials in the
$\R$-motivic Adams spectral sequence.
\end{prop}

\begin{longtable}{LLL}
\caption{Some $\R$-motivic Adams differentials} \\
\toprule
\mbox{}(s,f,c) & x & d_2(x) \\
\midrule \endhead
\bottomrule \endfoot
\label{tbl:R-Adams-diff}
(29, 7, 13) & k & h_0 d_0^2 + \rho h_1 d_0^2 \\
(23, 7, 17) & \tau^6 i & \tau^6 P h_0 d_0 \\
(25, 8, 17) & \tau^6 P e_0 & \tau^6 P h_1 d_0 \\
\end{longtable}

\begin{proof}
Most of these calculations follow by comparison to the
$\C$-motivic Adams spectral sequence along the extension-of-scalars
functor from $\R$-motivic stable homotopy theory to
$\C$-motivic stable homotopy theory.
However, this comparison only shows that
$d_2(k)$ equals either $h_0 d_0^2$ or $h_0 d_0^2 + \rho h_1 d_0^2$
because $\rho h_1 d_0^2$ maps to zero under extension-of-scalars.
In order to settle this uncertainty, use the
$\R$-motivic relation $h_1 k = \rho d_0 e_0$ together with the
$\R$-motivic Adams differential $d_2(d_0 e_0) = h_1^2 d_0^2$.
\end{proof}

\subsection{Some permanent cycles}
\label{sctn:AdamsPermCycles}

We begin by establishing some permanent cycles.

\begin{prop} 
\label{prop:GArhotauPermCycles}
The classes $\frac{\ga}{\tau}$, $\frac{\ga}{\rho\tau^2}$, $\frac{\ga}{\rho^3\tau^4}$, and $\frac{\ga}{\rho^7\tau^8}$ are permanent cycles in the $C_2$-equivariant Adams spectral sequence.
\end{prop}

\begin{proof}
The element $\frac{\gamma}{\tau}$ is a permanent cycle because there are no possible non-zero values for differentials.

The element $\frac{\gamma}{\rho \tau^2}$ is a permanent cycle because it is annihilated by $h_0^2$, but all possible values for differentials support $h_0$ multiplications

The elements $\frac{\gamma}{\rho^3 \tau^4}$ and $\frac{\gamma}{\rho^7 \tau^8}$
are more difficult.  
We know from the periodicity isomorphism of 
\cref{prop:gapbridging} that
$\piC_{3,-5}$ is isomorphic to the $\rho$-power
torsion subgroup of $\piC_{3,3}$.
The latter group is known because it is isomorphic to the
corresponding $\R$-motivic stable homotopy group.  It follows that
$\frac{\gamma}{\rho^3 \tau^4}$ is a permanent cycle.

Similarly, the $7$-stem is $\tau^{16}$-periodic (with some exceptions in specific coweights).  In particular, $\piC_{7,-9}$ is isomorphic to the 
$\rho$-power torsion subgroup of  $\piC_{7,7}$ by \cref{prop:gapbridging}.
The latter group is known by comparison to the $\R$-motivic stable homotopy
groups, and 
$\frac{\gamma}{\rho^7 \tau^8}$ is a permanent cycle.
\end{proof}

\begin{rmk}
The proof of \cref{prop:GArhotauPermCycles} lies somewhat outside of the 
spirit of the rest of this manuscript.  We would prefer a more ``algebraic"
proof, such as an argument that uses Toda brackets.  However, such a proof
has eluded us.  The particular elements in the proposition are related to the
Hopf maps.  There are many situations in which the indecomposable nature
of Hopf maps makes them exceptional.
\end{rmk}

\begin{prop}
\label{prop:Adams-diff-lowstem}
In stems less than 8 and coweights between $-9$ and $-2$, all elements
of the $C_2$-equivariant Adams $E_2$-page are permanent cycles.
\end{prop}

\cref{prop:Adams-diff-lowstem} establishes that all of the elements appearing
in \cref{fig:EinfNegCoweight} are permanent cycles.  In other words,
the $C_2$-equivariant Adams $E_2$-page is equal to the
$E_\infty$-page in this range.

\begin{proof}
For degree reasons, the only possible differentials
are:
\begin{itemize}
\item
$d_2 \left( \frac{\gamma}{\rho^7 \tau^8} \right)$ might equal
$\frac{\gamma}{\tau^{11}} h_2^2$.
\item
$d_2 \left( \frac{\gamma}{\rho^3 \tau^4} \right)$ might equal
$\frac{\gamma}{\tau^5} h_1^2$.
\end{itemize}
These possibilities are ruled out by 
\cref{prop:GArhotauPermCycles}.
\end{proof}

\begin{prop}
\label{prop:perm-cycles}
\cref{tbl:perm-cycles} lists some permanent cycles in the
$C_2$-equivariant Adams spectral sequence.
\end{prop}

\begin{longtable}{LLLL}
\caption{Some permanent cycles} \\
\toprule
\mbox{}(s,f,c) & \text{element} & \text{Massey product} & \text{Bockstein differential} \\
\midrule \endhead
\bottomrule \endfoot
\label{tbl:perm-cycles}
(26, 5, 0) & \frac{\gamma}{\rho^2 \tau^9} h_1 h_4 c_0 & \left\langle \frac{\gamma}{\rho^2 \tau^{12}}, \rho^3, \tau^2 h_2^2 g \right\rangle & d_3(\tau^3 h_1 h_4 c_0) = \rho^3 \cdot \tau^2 h_2^2 g \\
(15, 3, 1) & \frac{\gamma}{\rho^6 \tau} h_1^2 h_3 & \left\langle \frac{\gamma}{\rho^6 \tau^8}, \rho^7, \rho^2 \tau^2 e_0 \right\rangle & d_9(\tau^7 h_1^2 h_3) = \rho^9 \cdot \tau^2 e_0 \\
(19, 2, 3) & \frac{\gamma}{\rho^5 \tau^2} h_3^2 & \left\langle \frac{\gamma}{\rho^5 \tau^8}, \rho^6, \tau^3 c_1 \right\rangle & d_6(\tau^6 h_3^2) = \rho^6 \cdot \tau^3 c_1 \\
(19, 8, 3) & \frac{\gamma}{\rho^4 \tau^3} h_0^7 h_4 & \left\langle \frac{\gamma}{\rho^4 \tau^8}, \rho^5, \tau^2 P^2 h_2 \right\rangle & d_5(\tau^5 h_0^7 h_4) = \rho^5 \cdot \tau^2 P^2 h_2 \\
(27, 12, 3) & \frac{\gamma}{\rho^4 \tau^7} h_0^5 i & \left\langle \frac{\gamma}{\rho^4 \tau^8}, \rho^4, P^2 h_1^2 c_0 \right\rangle & d_4(\tau h_0^5 i) = \rho^4 \cdot P^2 h_1^2 c_0 \\
(16, 6, 4) & \frac{\gamma}{\rho^2 \tau} h_0^2 d_0 & \left\langle \frac{\gamma}{\rho^2 \tau^4}, \rho^3, \tau P c_0 \right\rangle & d_3(\tau^3 h_0^2 d_0) = \rho^3 \cdot \tau P c_0 \\
(24, 10, 4) & \frac{\gamma}{\rho^2 \tau^5} P h_0^2 d_0 & \left\langle \frac{\gamma}{\rho^2 \tau^8}, \rho^3, \tau P^2 c_0 \right\rangle & d_3(\tau^3 P h_0^2 d_0) = \rho^3 \cdot \tau P^2 c_0 \\
(24, 10, 8) & \frac{\gamma}{\rho^2 \tau} P h_0^2 d_0 & \left\langle \frac{\gamma}{\rho^2 \tau^4}, \rho^3, \tau P^2 c_0 \right\rangle & d_3(\tau^3 P h_0^2 d_0) = \rho^3 \cdot \tau P^2 c_0 \\
(26, 5, 8) & \frac{\gamma}{\rho^2 \tau} h_1 h_4 c_0 & \left\langle \frac{\gamma}{\rho^2 \tau^4}, \rho^3, \tau^2 h_2^2 g \right\rangle & d_3(\tau^3 h_1 h_4 c_0) = \rho^3 \cdot \tau^2 h_2^2 g \\

(18, 1, 2) & \frac{\gamma}{\rho^3 \tau^4} h_4 & \left\langle \frac{\gamma}{\rho^4 \tau^8}, \rho^5, \tau^2 h_2 h_4 \right\rangle & d_4(\rho \tau^4 h_4) = \rho^5 \tau^2 h_2 h_4 \\
(23, 3, 3) & \frac{\gamma}{\rho^2 \tau^5} h_2^2 h_4 & \left\langle \frac{\gamma}{\rho^2 \tau^8}, \rho^3, \tau h_4 c_0 \right\rangle & d_3(\tau^3 h_2^2 h_4) = \rho^3 \cdot \tau h_4 c_0 \\
(20, 3, 4) & \frac{\gamma}{\rho^6 \tau} h_0 h_3^2 & \left\langle \frac{\gamma}{\rho^6 \tau^8}, \rho^7, \tau^4 g \right\rangle & d_7(\tau^7 h_0 h_3^2) = \rho^7 \cdot \tau^4 g \\
(28, 7, 4) & \frac{\gamma}{\rho^5 \tau^6} i & \left\langle \frac{\gamma}{\rho^5 \tau^8}, \rho^6, d_0^2 \right\rangle & d_6(\tau^2 i) = \rho^6 \cdot d_0^2 \\
(23, 3, 7) & \frac{\gamma}{\rho^2 \tau} h_2^2 h_4 & \left\langle \frac{\gamma}{\rho^2 \tau^4}, \rho^3, \tau h_4 c_0 \right\rangle & d_3(\tau^3 h_2^2 h_4) = \rho^3 \cdot \tau h_4 c_0 \\
\end{longtable}

\begin{proof}
The proof for each permanent cycle is essentially the same.
The first step is to establish a Massey product for the element.
These Massey products are listed in the third column of the table.
Each Massey product is an application of the May convergence theorem 
\cite{MMP}*{Theorem 4.1}
to the Bockstein spectral sequence; the relevant Bockstein differential
appears in the fourth column of the table.
In all cases, the technical condition involving crossing differentials
is satisfied, and there is no indeterminacy.

Then we use the Moss convergence theorem \cite{Moss} \cite{BK}
in order to establish that the elements are permanent cycles.
In all cases, the technical condition involving crossing differentials 
is satisfied.

The argument for $\frac{\gamma}{\rho^5 \tau^2} h_3^2$ requires one
additional technical argument.
In order to apply the Moss convergence theorem, we need to
know that $\tau^3 c_1$ does not support a hidden $\rho^6$-extension
in the $\R$-motivic Adams spectral sequence.
This follows by comparison to the spectrum $L$ of \cite{BIK} that
detects $\R$-motivic $v_1$-periodic homotopy.
The unit map $S \rightarrow L$
takes $\tau^3 c_1$ to zero, but it takes the possible values of the
hidden $\rho^6$-extension to non-zero elements in $L$.
\end{proof}

The classes in the second column of \cref{tbl:perm-cycles} are listed 
according to their names in the Bockstein spectral sequence. In
general, this only specifies an element of $\Ext_{C_2}$ up to terms
in higher $\rho$-filtration. For the most part, this causes no difficulties.
For example, the element in degree (19,2,3) detected by the Massey product
in \cref{tbl:perm-cycles} is only specified up to the element 
$\frac{\gamma}{\rho \tau^4} h_2 h_4$ in higher $\rho$-filtration. 
However, the latter is a permanent cycle, so we conclude that both
classes detected by $\frac{\gamma}{\rho^5 \tau^2} h_3^2$ are 
permanent cycles. 

The one exception is the Bockstein $E_\infty$-page element detected by 
$\frac{\gamma}{\rho^6 \tau} h_1^2 h_3$ in degree $(15,3,1)$. 
We show in \cref{thm:Adams-d3} that
the element $\frac{\gamma}{\tau^5} h_0^2 h_4$ in higher $\rho$-filtration
supports an Adams $d_3$ differential. 
Therefore it is important to determine which element of $\Ext_{C_2}$ 
in degree $(15,3,1)$ is detected by the bracket in \cref{tbl:perm-cycles}. 
Of the two relevant elements in $\Ext_{C_2}$, one is
annihilated by $h_0$, while the other supports an $h_0$-extension to 
$\frac{\gamma}{\tau^5} h_0^3 h_4$.
We settle this ambiguity in \cref{lem:fifteenthreeone}.

\begin{lemma}
\label{lem:fifteenthreeone}
$(15, 3, 1)$
In $\Ext_{C_2}$, the Massey product 
$\left\langle 
\frac{\gamma}{\rho^6 \tau^8}, \rho^7, \rho^2 \tau^2 e_0 
\right\rangle$
is annihilated by $h_0$.
\end{lemma}

\begin{pf}
We multiply the bracket by $h_0$ and shuffle:
\[
\left\langle 
\frac{\gamma}{\rho^6 \tau^8}, \rho^7, \rho^2 \tau^2 e_0 
\right\rangle
h_0 = 
\frac{\gamma}{\rho^6 \tau^8}
\left\langle 
\rho^7, \rho^2 \tau^2 e_0 , h_0
\right\rangle.
\]
The $\R$-motivic bracket on the right is in degree (9,4,10). By \cite{BI},
the only nonzero element in this degree is divisible by $\rho^{13}$. 
Therefore it is annihilated by $\frac{\gamma}{\rho^6 \tau^8}$.
\end{pf}

\subsection{Adams $d_2$ differentials}
\label{sctn:Adamsd2}

Next we study Adams $d_2$ differentials.  We consider higher
differentials in later sections.

\begin{thm}
\label{thm:Adams-d2}
\cref{tbl:Adams-d2} lists some $d_2$ differentials in the
$C_2$-equivariant Adams spectral sequence.
If a $\F_2[\rho, h_0, h_1]$-module generator of the Adams $E_2$-page
in stems less than $31$ and coweights from $-1$ to $8$
does not belong to $\Ext_\R$ and does not appear in the table, then it does not support an Adams $d_2$ differential.
\end{thm}

\cref{tbl:Adams-d2} does not include the differentials on elements in $\Ext_\R$
since those differentials are carefully considered in \cite{BI}.

\renewcommand*{\arraystretch}{1.25}
\begin{longtable}{LLLL}
\caption{Some Adams $d_2$ differentials} \\
\toprule
\mbox{}(s,f,c) & \textrm{generator} & d_2 & \textrm{proof} \\
\midrule \endhead
\bottomrule \endfoot
\label{tbl:Adams-d2}
( 15 , 1 , -1 ) &  \frac{\gamma}{\tau^7} h_4  &  \frac{\gamma}{\tau^7} h_0 h_3^2  &  \frac{\gamma}{\tau^7} \cdot h_4  \\
( 23 , 7 , -1 ) &  \frac{\gamma}{\tau^{11}} i  &  \frac{\gamma}{\tau^{11}} P h_0 d_0  &  \frac{\gamma}{\tau^{17}} \cdot \tau^6 i  \\
( 29 , 6 , -1 ) &  \frac{\gamma}{\rho^6 \tau^9} h_0 h_2 g  &  \frac{\gamma}{\tau^{13}} d_0^2  &  h_0 \cdot \frac{\gamma}{\rho^6 \tau^9} h_0 h_2 g = \frac{\gamma}{\tau^{13}} \cdot k  \\
       &    &    &  \rho^2 \cdot \frac{\gamma}{\rho^6 \tau^9} h_2 g = h_1 \cdot \frac{\gamma}{\rho^5 \tau^8} h_1 g \\
( 18 , 3 , 0 ) &  \frac{\gamma}{\rho^4 \tau^5} h_0 h_3^2  &  \frac{\gamma}{\tau^7} h_0 e_0  &  h_0 \cdot \frac{\gamma}{\rho^4 \tau^5} h_0 h_3^2 = \frac{\gamma}{\tau^7} \cdot f_0  \\
( 26 , 7 , 0 ) &  \frac{\gamma}{\tau^{11}} j  &  \frac{\gamma}{\tau^{11}} P h_0 e_0  &  \frac{\gamma}{\tau^{11}} \cdot j  \\
( 30 , 6 , 0 ) &  \frac{\gamma}{\rho^4 \tau^9} h_2^2 g  &  \frac{\gamma}{\rho \tau^{12}} d_0^2  &  h_0 \cdot \frac{\gamma}{\rho^4 \tau^9} h_2^2 g = 0  \\
       &    &    &  \rho h_1 \cdot \frac{\gamma}{\rho^4 \tau^9} h_2^2 g = \frac{\gamma}{\rho \tau^{12}} \cdot k \\
( 15 , 1 , 1 ) &  \frac{\gamma}{\tau^5} h_4  &  \frac{\gamma}{\tau^5} h_0 h_3^2  &  \frac{\gamma}{\tau^5} \cdot h_4  \\
( 17 , 3 , 1 ) &  \frac{\gamma}{\rho^8 \tau} h_1^2 h_3  &  \frac{\gamma}{\rho \tau^5} h_1 d_0  &  h_0 \cdot \frac{\gamma}{\rho^8 \tau} h_1^2 h_3 = \frac{\gamma}{\tau^5} \cdot e_0  \\
( 23 , 6 , 1 ) &  \frac{\gamma}{\rho^7 \tau^4} h_1^2 d_0  &  \frac{\gamma}{\tau^9} P d_0  &  h_0 \cdot \frac{\gamma}{\rho^7 \tau^4} h_1^2 d_0 = \frac{\gamma}{\tau^{15}} \cdot \tau^6 i  \\
( 25 , 7 , 1 ) &  \frac{\gamma}{\rho^8 \tau^4} h_1^3 d_0  &  \frac{\gamma}{\rho \tau^9} P h_1 d_0  &  h_0 \cdot \frac{\gamma}{\rho^8 \tau^4} h_1^3 d_0 = \frac{\gamma}{\tau^{15}} \cdot \tau^6 P e_0  \\
( 29 , 7 , 1 ) &  \frac{\gamma}{\tau^{11}} k  &  \frac{\gamma}{\tau^{11}} h_0 d_0^2  &  \frac{\gamma}{\tau^{11}} \cdot k  \\
( 18 , 3 , 2 ) &  \frac{\gamma}{\rho^4 \tau^3} h_0 h_3^2  &  \frac{\gamma}{\rho^2 \tau^4} h_1 d_0 + \frac{\gamma}{\tau^5} h_0 e_0  &  h_0 \cdot \frac{\gamma}{\rho^4 \tau^3} h_0 h_3^2 = \frac{\gamma}{\tau^5} \cdot f_0  \\
       &    &    &  h_1 \cdot \frac{\gamma}{\rho^4 \tau^3} h_0 h_3^2 = \frac{\gamma}{\rho^2 \tau^4} \cdot e_0 \\ 
( 21 , 4 , 2 ) &  \frac{\gamma}{\rho^4 \tau^4} e_0  &  \frac{\gamma}{\rho^4 \tau^4} h_1^2 d_0  &  \rho^4 \cdot \frac{\gamma}{\rho^4 \tau^4} e_0 = \frac{\gamma}{\tau^4} \cdot e_0  \\
( 23 , 5 , 2 ) &  \frac{\gamma}{\rho^5 \tau^4} h_1 e_0  &  \frac{\gamma}{\rho^5 \tau^4} h_1^3 d_0  &  \rho^5 \cdot \frac{\gamma}{\rho^5 \tau^4} h_1 e_0 = \frac{\gamma}{\tau^4} \cdot h_1 e_0  \\
( 26 , 6 , 2 ) &  \frac{\gamma}{\rho^6 \tau^5} h_0^2 g  &  \frac{\gamma}{\tau^9} P e_0 + \frac{\gamma}{\rho^3 \tau^8} P d_0  &  h_0 \cdot \frac{\gamma}{\rho^6 \tau^5} h_0^2 g = \frac{\gamma}{\tau^9} \cdot j  \\
       &    &    &  \rho^2 \cdot \frac{\gamma}{\rho^6 \tau^5} h_0^2 g = h_1 \cdot \frac{\gamma}{\rho^5 \tau^{4}} h_1 e_0 \\
( 28 , 7 , 2 ) &  \frac{\gamma}{\rho^5 \tau^8} i  &  \frac{\gamma}{\rho^4 \tau^8} P h_1 d_0  &  \rho \cdot \frac{\gamma}{\rho^5 \tau^8} i = h_1 \cdot  \frac{\gamma}{\rho^6 \tau^5} h_0^2 g  \\
( 15 , 1 , 3 ) &  \frac{\gamma}{\tau^3} h_4  &  \frac{\gamma}{\tau^3} h_0 h_3^2  &  \frac{\gamma}{\tau^3} \cdot h_4  \\
( 21 , 4 , 3 ) &  \frac{\gamma}{\rho^3 \tau^4} f_0  &  \frac{\gamma}{\tau^5} h_0^2 g  &  \frac{\gamma}{\rho^3 \tau^4} \cdot f_0  \\
       &    &    &  \frac{\gamma}{\rho^3 \tau^4} \cdot h_0^2 e_0 = \frac{\gamma}{\tau^5} h_0 h_2 \cdot e_0 \\
( 23 , 7 , 3 ) &  \frac{\gamma}{\tau^7} i  &  \frac{\gamma}{\tau^7} P h_0 d_0  &  \frac{\gamma}{\tau^{13}} \cdot \tau^6 i  \\
( 29 , 7 , 3 ) &  \frac{\gamma}{\tau^9} k  &  \frac{\gamma}{\tau^9} h_0 d_0^2  &  \frac{\gamma}{\tau^9} \cdot k  \\
( 15 , 7 , 4 ) &  \frac{Q}{\rho^2} P h_1^4  &  h_1^6 c_0  &  h_1 \text{-periodic}  \\
( 18 , 4 , 4 ) &  \frac{\gamma}{\tau^3} f_0  &  \frac{\gamma}{\tau^3} h_0^2 e_0  &  \frac{\gamma}{\tau^3} \cdot f_0  \\
( 26 , 6 , 4 ) &  \frac{\gamma}{\rho^6 \tau^3} h_0^2 g  &  \frac{\gamma}{\tau^7} P e_0  &  h_0 \cdot \frac{\gamma}{\rho^6 \tau^3} h_0^2 g = \frac{\gamma}{\tau^7} \cdot j  \\
( 15 , 1 , 5 ) &  \frac{\gamma}{\tau} h_4  &  \frac{\gamma}{\tau} h_0 h_3^2  &  \frac{\gamma}{\tau} \cdot h_4  \\
( 17 , 4 , 5 ) &  \frac{\gamma}{\tau} e_0  &  \frac{\gamma}{\tau} h_1^2 d_0  &  \frac{\gamma}{\tau} \cdot e_0  \\
( 23 , 7 , 5 ) &  \frac{\gamma}{\tau^5} i  &  \frac{\gamma}{\tau^5} P h_0 d_0  &  \frac{\gamma}{\tau^{11}} \cdot \tau^6 i  \\
( 25 , 8 , 5 ) &  \frac{\gamma}{\tau^5} P e_0  &  \frac{\gamma}{\tau^5} P h_1^2 d_0  &  \frac{\gamma}{\tau^{11}} \cdot \tau^6 P e_0  \\
( 26 , 5 , 5 ) &  \frac{\gamma}{\rho^5 \tau^2} h_1 g  &  \frac{\gamma}{\rho^2 \tau^6} i  &  \rho h_1^2 \cdot \frac{\gamma}{\rho^5 \tau^2} h_1 g = \frac{\gamma}{\rho \tau^6} \cdot j  \\
( 29 , 6 , 5 ) &  \frac{\gamma}{\rho^6 \tau^3} h_0 h_2 g  &  \frac{\gamma}{\rho^3 \tau^6} P e_0 + \frac{\gamma}{\tau^7} d_0^2  &  h_0 \cdot \frac{\gamma}{\rho^6 \tau^3} h_0 h_2 g = \frac{\gamma}{\tau^7} \cdot k  \\
       &    &    &  \rho^2 \cdot \frac{\gamma}{\rho^6 \tau^3} h_0 h_2 g = h_1 \cdot \frac{\gamma}{\rho^5 \tau^2} h_1 g \\
( 18 , 4 , 6 ) &  \frac{\gamma}{\tau} f_0  &  \frac{\gamma}{\tau} h_0^2 e_0  &  \frac{\gamma}{\tau} \cdot f_0  \\
( 26 , 7 , 6 ) &  \frac{\gamma}{\tau^5} j  &  \frac{\gamma}{\tau^5} P h_0 e_0  &  \frac{\gamma}{\tau^5} \cdot j  \\
( 19 , 8 , 7 ) &  Q P h_1^2 c_0  &  \rho^2 h_1^6 d_0  &  \text{\cref{lem:d2QPh12c0}} \\
( 20 , 5 , 7 ) &  Q h_1^2 e_0  &  Q h_1^4 d_0  &  h_1^2 \cdot Q h_1^2 e_0 =  Q h_1^4 \cdot e_0  \\
( 22 , 9 , 7 ) &  \frac{Q}{\rho^2} P h_1^3 c_0  &  h_1^7 d_0  &  \rho^2 \cdot \frac{Q}{\rho^2} P h_1^3 c_0 =  h_1 \cdot Q P h_1^2 c_0  \\
( 23 , 6 , 7 ) &  \frac{Q}{\rho^2} h_1^3 e_0  &  \frac{\gamma}{\tau^3} P d_0 + \frac{Q}{\rho^2} h_1^5 d_0  &  h_0 \cdot \frac{Q}{\rho^2} h_1^3 e_0 = \frac{\gamma}{\tau^9} \cdot \tau^6 i  \\
       &    &    &  \rho^2 \cdot \frac{Q}{\rho^2} h_1^3 e_0 = h_1 \cdot Q h_1^2 e_0 \\
( 25 , 7 , 7 ) &  \frac{Q}{\rho^3} h_1^4 e_0  &  \frac{Q}{\rho^3} h_1^6 d_0  &  \rho \cdot \frac{Q}{\rho^3} h_1^4 e_0 = h_1 \cdot \frac{Q}{\rho^2} h_1^3 e_0  \\
( 28 , 9 , 7 ) &  \frac{Q}{\rho^4} h_1^6 e_0  &  \frac{\gamma}{\rho^3 \tau^4} P^2 c_0  &  \rho \cdot \frac{Q}{\rho^4} h_1^6 e_0 = h_1^2 \cdot \frac{Q}{\rho^3} h_1^4 e_0  \\
( 29 , 6 , 7 ) &  \frac{\gamma}{\rho^6 \tau} h_0 h_2 g  &  \frac{\gamma}{\tau^5} d_0^2  &  h_0 \cdot \frac{\gamma}{\rho^6 \tau} h_0 h_2 g = \frac{\gamma}{\tau^5} \cdot k  \\
( 30 , 10 , 7 ) &  \frac{Q}{\rho^5} h_1^7 e_0  &  \frac{\gamma}{\rho^4 \tau^4} P^2 h_1 c_0  &  \rho \cdot \frac{Q}{\rho^5} h_1^7 e_0 = h_1 \cdot \frac{Q}{\rho^4} h_1^6 e_0  \\
( 23 , 11 , 8 ) &  \frac{Q}{\rho^2} P^2 h_1^4  &  \rho^4 h_1^9 e_0  &  h_1 \text{-periodic}  \\
( 25 , 12 , 8 ) &  \frac{Q}{\rho^3} P^2 h_1^5  &  \rho^3 h_1^{10} e_0  &  \rho \cdot \frac{Q}{\rho^3} P^2 h_1^5 = h_1 \cdot \frac{Q}{\rho^2} P^2 h_1^4  \\
( 26 , 7 , 8 ) &  \frac{\gamma}{\tau^3} j  &  \frac{\gamma}{\tau^3} P h_0 e_0  &  \frac{\gamma}{\tau^3} \cdot j  \\
( 30 , 6 , 8 ) &  \frac{\gamma}{\rho^4 \tau} h_2^2 g  &  \frac{\gamma}{\rho \tau^4} d_0^2  &  h_0 \cdot \frac{\gamma}{\rho^4 \tau} h_2^2 g = 0  \\
       &    &    &  \rho h_1 \cdot \frac{\gamma}{\rho^4 \tau} h_2^2 g = \frac{\gamma}{\rho \tau^4} \cdot k \\
( 30 , 15 , 8 ) &  \frac{Q}{\rho^5} P^2 h_1^8  &  \rho h_1^3 e_0 + \frac{Q}{\rho} h_1^{11} c_0  &  h_1 \text{-periodic}  \\
\end{longtable}
\renewcommand*{\arraystretch}{1.0}

\begin{proof}
Most of the differentials can be deduced using multiplicative relations
that relate the elements under consideration to elements in the
$\R$-motivic Adams spectral sequence.  The relevant relations appear 
in the last column of the table.
For example, 
consider the element $\frac{\gamma}{\rho^4 \tau^3} h_0 h_3^2$ in 
degree $(18, 3, 2)$.
The relation $h_0 \cdot \frac{\gamma}{\rho^4 \tau^3} h_0 h_3^2 = \frac{\gamma}{\tau^5} \cdot f_0$ implies that
\[
h_0 \cdot d_2 \left( \frac{\gamma}{\rho^4 \tau^3} h_0 h_3^2 \right) = 
\frac{\gamma}{\tau^5} \cdot d_2(f_0) =
\frac{\gamma}{\tau^5} h_0^2 e_0.
\]
Therefore, $d_2 \left( \frac{\gamma}{\rho^4 \tau^3} h_0 h_3^2 \right)$ equals
either $\frac{\gamma}{\tau^5} h_0 e_0$ or 
$\frac{\gamma}{\tau^5} h_0 e_0 + \frac{\gamma}{\rho^2 \tau^4} h_1 d_0$.
On the other hand, the relation
$h_1 \cdot \frac{\gamma}{\rho^4 \tau^3} h_0 h_3^2 = \frac{\gamma}{\rho^2 \tau^4} \cdot e_0$ implies that
\[
h_1 \cdot d_2 \left( \frac{\gamma}{\rho^4 \tau^3} h_0 h_3^2 \right) = 
\frac{\gamma}{\rho^2 \tau^4} \cdot d_2(e_0) = 
\frac{\gamma}{\rho^2 \tau^4} \cdot h_1^2 d_0.
\]
Therefore, 
$d_2 \left( \frac{\gamma}{\rho^4 \tau^3} h_0 h_3^2 \right)$ must equal
$\frac{\gamma}{\tau^5} h_0 e_0 + \frac{\gamma}{\rho^2 \tau^4} h_1 d_0$.

In other cases, we use multiplicative relations to relate
the elements under consideration to other elements whose
$d_2$ differentials have already been established.
For example, consider the element $\frac{Q}{\rho^2} P h_1^3 c_0$ 
in degree $(22,9,7)$.
The relation
$\rho^2 \cdot \frac{Q}{\rho^2} P h_1^3 c_0 = h_1 \cdot Q P h_1^2 c_0$
implies that
\[
\rho^2 \cdot d_2 \left( \frac{Q}{\rho^2} P h_1^3 c_0 \right) = h_1 \cdot 
d_2\left( Q P h_1^2 c_0 \right).
\]
Since 
$d_2\left( Q P h_1^2 c_0 \right)$ is already known to equal $\rho^2 h_1^6 d_0$
by \cref{lem:d2QPh12c0},
it follows that 
$d_2 \left( \frac{Q}{\rho^2} P h_1^3 c_0 \right)$ must equal $h_1^7 d_0$.

Some differentials are determined by the $h_1$-periodic
computations of \cref{subsctn:h1-periodic-Adams};
see especially \cref{tbl:h1invAdamsDiffs}.

Many possibilities are ruled out by considering multiplication
by $\rho$, $h_0$, or $h_1$.  For example,
the element $\frac{\gamma}{\tau^7} e_0$ in degree $(17,4,-1)$
is annihilated by $h_1$, but $\frac{\gamma}{\rho^2 \tau^7} h_0^2 d_0$
supports an $h_1$ multiplication.  Therefore,
$d_2 \left( \frac{\gamma}{\tau^7} e_0 \right)$ cannot equal
$\frac{\gamma}{\rho^2 \tau^7} h_0^2 d_0$.

A few more difficult cases remain.  Most of them are settled in
\cref{prop:perm-cycles}.  See  
\cref{lem:Qh1^2c0-perm} and \cref{lem:d2QPh12c0} below
for two additional cases.
\end{proof}

\begin{lemma} $(11,4,3)$ 
\label{lem:Qh1^2c0-perm}
The element $Q h_1^2 c_0$ is a permanent cycle in the $C_2$-equivariant Adams spectral sequence.
\end{lemma}

\begin{proof}
The element $h_1^2 c_0$ is a permanent cycle in the $\R$-motivic Adams 
$E_\infty$-page that maps to the element of the same name in 
the $\C$-motivic Adams $E_\infty$-page.  In turn,
Betti realization takes $h_1^2 c_0$ to zero.  Therefore, the underlying
homomorphism $U$ 
(see \cref{subsctn:underlying})
takes the $C_2$-equivariant homotopy classes
detected by $h_1^2 c_0$ to zero.
From \cref{prop:rho-forget}, 
it follows that $h_1^2 c_0$ must detect homotopy elements that are 
divisible by $\rho$.
The only possibility is that $Q h_1^2 c_0$ is a permanent cycle
that supports a hidden $\rho$-extension to $h_1^2 c_0$.
\end{proof}

\begin{lemma}
\label{lem:d2QPh12c0}
$(19,8,7)$ 
$d_2 \left( Q P h_1^2 c_0 \right) = \rho^2 h_1^6 d_0$.
\end{lemma}

\begin{pf}
We know from \cite{BI}*{Table 17}
that $\rho^3 h_1^4 e_0$ is a permanent cycle in the 
$\R$-motivic Adams $E_\infty$-page that maps to $P h_1^2 c_0$
in the $\C$-motivic Adams $E_\infty$-page.  In turn, Betti realization
takes $P h_1^2 c_0$ to zero.  Therefore,
the underlying homomorphism $U$ 
(see \cref{subsctn:underlying})
takes
the $C_2$-equivariant homotopy classes detected by $\rho^3 h_1^4 e_0$ to zero.
From \cref{prop:rho-forget}, it follows that $\rho^3 h_1^4 e_0$ must detect homotopy elements
that are divisible by $\rho$.
There are no possible hidden $\rho$-extensions.
The only remaining possibility is that 
$Q P h_1^2 c_0 + \rho^2 h_1^4 e_0$
is a permanent cycle.
We already know from \cite{BI} that
$d_2(\rho^2 h_1^4 e_0) = \rho^2 h_1^6 d_0$, so
$d_2(Q P h_1^2 c_0)$ also equals $\rho^2 h_1^6 d_0$.
\end{pf}

\subsection{Adams $d_3$ differentials}
\label{sctn:Adamsd3}

Having settled all of the Adams $d_2$ differentials in a range,
our next task is to compute Adams $d_3$ differentials.
Adams $E_3$ charts are displayed in \cref{fig:Ethree}.

\begin{thm}
\label{thm:Adams-d3}
\cref{tbl:Adams-d3} lists some $d_3$ differentials in the
$C_2$-equivariant Adams spectral sequence.
If a $\F_2[\rho, h_0, h_1]$-module generator of the Adams $E_3$-page
in stems less than $30$ and coweights from $-1$ to $8$
does not appear in the table, then it does not support an Adams $d_3$ differential.
\end{thm}

\renewcommand*{\arraystretch}{1.25}
\begin{longtable}{LLLL}
\caption{Some Adams $d_3$ differentials} \\
\toprule
\mbox{}(s,f,c) & \textrm{generator} & d_3 & \textrm{proof} \\
\midrule \endhead
\bottomrule \endfoot
\label{tbl:Adams-d3}
(15, 2, -1) & \frac{\gamma}{\tau^7} h_0 h_4 & \frac{\gamma}{\tau^7} h_0 d_0 & \frac{\gamma}{\tau^7} \cdot h_0 h_4 \\
(15, 2, 1) & \frac{\gamma}{\tau^5} h_0 h_4 & \frac{\gamma}{\tau^5} h_0 d_0 & \frac{\gamma}{\tau^5} \cdot h_0 h_4 \\
(15, 2, 3) & \frac{\gamma}{\tau^3} h_0 h_4 & \frac{\gamma}{\tau^3} h_0 d_0 & \frac{\gamma}{\tau^3} \cdot h_0 h_4 \\
(23, 5, 4) & \frac{\gamma}{\rho^5 \tau^2} h_1 e_0 & \frac{\gamma}{\tau^6} P d_0 & h_1 \cdot \frac{\gamma}{\rho^5 \tau^2} h_1 e_0 = \rho \cdot \frac{\gamma}{\rho^5 \tau^3} h_0^2 g \\
(25, 6, 4) & \frac{\gamma}{\rho^5 \tau^3} h_0^2 g &	\frac{\gamma}{\rho \tau^6} P h_1 d_0 & \textrm{\cref{lem:d3-g/p^5t^3-h0^2g}} \\
(15, 2, 5) & \frac{\gamma}{\tau} h_0 h_4 & \frac{\gamma}{\tau} h_0 d_0 & \frac{\gamma}{\tau} \cdot h_0 h_4 \\
(23, 5, 5) & \frac{\gamma}{\rho^2 \tau^2} h_1 g & \frac{\gamma}{\tau^5} P d_0 & h_1 \cdot \frac{\gamma}{\rho^2 \tau^2} h_1 g = \rho \cdot \frac{\gamma}{\rho^2 \tau^3} h_0 h_2 g \\
(25, 6, 5) & \frac{\gamma}{\rho^2 \tau^3} h_0 h_2 g & \frac{\gamma}{\rho \tau^5} P h_1 d_0 & \rho \cdot \frac{\gamma}{\rho^2 \tau^3} h_0 h_2 g = \frac{\gamma}{\rho \tau^6} \cdot \tau^3 h_0 h_2 g \\
(30, 6, 6) & \frac{\gamma}{\tau^7} r & \frac{\gamma}{\tau^6} h_1 d_0^2 & 
\frac{\gamma}{\tau^7} \cdot r \\
(15, 2, 7) & h_0 h_4 & h_0 d_0 + \rho h_1 d_0 & \text{\cite{BI}} \\
(29, 8, 7) & \frac{\gamma}{\tau^5} h_0 k & \frac{\gamma}{\rho^4 \tau^4} P^2 c_0 & \textrm{\cref{lem:d3-g/t^5-h0k}}  
\end{longtable}
\renewcommand*{\arraystretch}{1.0}

\begin{proof}
As in the proof of \cref{thm:Adams-d2}, most of the differentials
can be deduced using multiplicative relations.  The relevant
relations appear in the last column of the table.
For example, consider the element 
$\frac{\gamma}{\rho^2 \tau^3} h_0 h_2 g$ in degree $(25, 6, 5)$.
The relation
$\rho \cdot \frac{\gamma}{\rho^2 \tau^3} h_0 h_2 g = \frac{\gamma}{\rho \tau^6} \cdot \tau^3 h_0 h_2 g$ implies that
\[
\rho \cdot d_3 \left( \frac{\gamma}{\rho^2 \tau^3} h_0 h_2 g \right) =
\frac{\gamma}{\rho \tau^6} \cdot d_3 \left( \tau^3 h_0 h_2 g \right) =
\frac{\gamma}{\rho \tau^6} \cdot \rho \tau P h_1 d_0 =
\frac{\gamma}{\tau^5} P h_1 d_0.
\]
The second equality is an $\R$-motivic differential established in \cite{BI}.

Many possibilities are ruled out by recognizing that 
$\F_2[\rho, h_0, h_1]$-module generators are products of elements
that are already known to be permanent cycles.  For example,
the element $\frac{\gamma}{\rho^5 \tau^8} h_1 g$
in degree $(26,5,-1)$ is the product
$\frac{\gamma}{\rho^5 \tau^{12}} h_1 \cdot \tau^4 g$.

More possibilities are ruled out by considering multiplication
by $\rho$, $h_0$, or $h_1$.  For example,
the element $\frac{\gamma}{\tau} h_3^2$ in degree $(14,2,4)$
is annihilated by $h_1$, but $\frac{Q}{\rho} h_1^3 c_0$
supports an $h_1$ multiplication.  Therefore,
$d_3 \left( \frac{\gamma}{\tau} h_3^2 \right)$ cannot equal
$\frac{Q}{\rho} h_1^3 c_0$.

A few more difficult cases remain.  Most of them are settled in
\cref{prop:perm-cycles}.  See  
\cref{lem:d3-g/p^5t^3-h0^2g}, \cref{lem:d3-Q/p-h1^8e0}, and
\cref{lem:d3-g/t^5-h0k} 
below for three additional cases.
\end{proof}

\begin{lemma}
\label{lem:d3-g/p^5t^3-h0^2g}
$(25, 6, 4)$
$d_3\left( \frac{\gamma}{\rho^5 \tau^3} h_0^2 g \right) = 
\frac{\gamma}{\rho \tau^6} P h_1 d_0$.
\end{lemma}

\begin{proof}
The classical 24-stem has order four.  
\cref{prop:rho-SES} then implies that the 
quotient of $\piC_{24,3}$ by the image of $\rho$
has order four at most.
Of the three elements $\frac{\gamma}{\rho \tau^6} h_4 c_0$, $\frac{\gamma}{\rho \tau^6} P h_1 d_0$, and $\frac{\gamma}{\rho^6 \tau^4} P^2 h_1^2$, at least
one must either support a differential or receive a hidden $\rho$-extension.
There are no possible hidden $\rho$-extensions, 
so at least one of these
three elements must support a differential.
By inspection, there is only one possible differential.
\end{proof}

\begin{lemma}
\label{lem:d3-Q/p-h1^8e0}
$(27, 11, 7)$ 
The element $\frac{Q}{\rho} h_1^8 e_0$ is a permanent cycle in the
$C_2$-equivariant Adams spectral sequence.
\end{lemma}

\begin{proof}
Suppose that $\frac{Q}{\rho} h_1^8 e_0$ supported a differential.
Then, in coweight 7, the elements
$Q h_1^7 e_0$ and $Q h_1^8 e_0$ would both detect
homotopy elements in the cokernel of $\rho$ since there are no possible
hidden $\rho$-extensions.  
Moreover, these two elements would be related by an $\eta$-extension.
This contradicts \cref{prop:rho-SES} since there is no $\eta$-extension in the classical 24-stem.
\end{proof}

\begin{lemma}
\label{lem:d3-g/t^5-h0k}
$(29, 8, 7)$
$d_3 \left( \frac{\gamma}{\tau^5} h_0 k \right) = \frac{\gamma}{\rho^4 \tau^4} P^2 c_0$.
\end{lemma}

\begin{proof}
If $\frac{\gamma}{\tau^5} h_0 k$ were non-zero in the
Adams $E_\infty$-page, then it would
detect an element of $\piC_{29,7}$ that is annihilated by $\rho$.
Since the classical 29-stem is zero, 
\cref{prop:rho-SES} implies that there are no non-zero elements in
$\piC_{29,7}$ that are annihilated by $\rho$.
Therefore, $\frac{\gamma}{\tau^5} h_0 k$ must be hit by an Adams differential
or support an Adams differential.  The only possibility is that 
$d_3 \left( \frac{\gamma}{\tau^5} h_0 k \right)$ equals $\frac{\gamma}{\rho^4 \tau^4} P^2 c_0$.
\end{proof}

\subsection{Higher differentials}
\label{sctn:Adamsdhigher}

Having settled all of the Adams $d_3$ differentials in a range,
our next goal is to study higher Adams differentials in that range.
\cref{thm:Adams-d-higher} establishes that
there are no higher differentials in a large range.
In \cref{lem:d4-g/t^11-Dh2^2} and \cref{lem:d4-g/t^10-Dh2^2},
we also establish some non-zero $d_4$ differentials that lie slightly
outside of the range that we have completely analyzed.
As there are only a handful of remaining differentials in the 
range under consideration, we do not display $E_4$ charts. 
See \cref{fig:Einf} for Adams $E_\infty$ charts.

\begin{thm}
\label{thm:Adams-d-higher}
In stems less than $27$ and coweights from $-1$ to $8$,
the Adams differentials $d_r$ vanish for all $r \geq 4$.
\end{thm}

\begin{proof}
As in the proof of \cref{thm:Adams-d3},
many possibilities are ruled out by recognizing that
$\F_2[\rho, h_0, h_1]$-module generators are products of elements
that are already known to be permanent cycles. 
More possibilities are ruled out by consideration of multiplication
by $\rho$, $h_0$, or $h_1$. 
A few more difficult cases are settled in \cref{prop:perm-cycles}.
See \cref{lem:g/pt^6-h4} and \cref{lem:g/p^2t^4-f0}
below for some slightly different cases.
\end{proof}

\begin{lemma}
\label{lem:g/pt^6-h4}
$(16, 1, 0)$, $(16, 1, 4)$
The elements $\frac{\gamma}{\rho \tau^6} h_4$ 
and $\frac{\gamma}{\rho \tau^2} h_4$ are
permanent cycles.
\end{lemma}

\begin{proof}
We use the Moss convergence theorem \cite{Moss} \cite{BK}.
However, unlike in the proof of \cref{prop:perm-cycles}, we do not use
Massey products in the $E_2$-page.  Rather, 
the Adams differential $d_2(h_4) = (h_0 + \rho h_1) h_3^2$
implies that 
$\frac{\gamma}{\rho \tau^6} h_4$ equals
the Massey product
$\left\langle \frac{\gamma}{\rho \tau^6}, h_0 + \rho h_1, h_3^2 \right\rangle$
in the Adams $E_3$-page.  The technical hypotheses of the Moss convergence
theorem are satisfied, so the Massey product converges to a Toda bracket,
and $\frac{\gamma}{\rho \tau^6} h_4$
is a permanent cycle.

The argument for $\frac{\gamma}{\rho \tau^6} h_4$ is identical
since it equals the analogous Massey product
$\left\langle \frac{\gamma}{\rho \tau^2}, h_0 + \rho h_1, h_3^2 \right\rangle$
in the $E_3$-page.
\end{proof}

\begin{rmk}
In the proof of \cref{lem:g/pt^6-h4}, we have used the element
$h_0 + \rho h_1$ rather than the seemingly more obvious choice $h_0$.
This is necessary since $\frac{\gamma}{\rho \tau^6} \cdot (h_0 + \rho h_1)$
is zero, while $\frac{\gamma}{\rho \tau^6} \cdot h_0$ is non-zero.
\end{rmk}

\begin{lemma}
\label{lem:d4-g/t^11-Dh2^2}
$(30, 6, 2)$  $d_4 \left( \frac{\gamma}{\tau^{11}} r \right) = \frac{\gamma}{\rho^5 \tau^8} P h_1^2 d_0$.
\end{lemma}

\begin{proof}
If $\frac{\gamma}{\rho^5 \tau^8} P h_1^2 d_0$ were non-zero in the
Adams $E_\infty$-page, then it would
detect an element of $\piC_{29,1}$ that is annihilated by $\rho$.
Since the classical 29-stem is zero, 
\cref{prop:rho-SES} implies that there are no non-zero elements in
$\piC_{29,1}$ that are annihilated by $\rho$.
Therefore, $\frac{\gamma}{\rho^5 \tau^8} P h_1^2 d_0$ must be hit by an Adams differential
or support an Adams differential. 

There are now two possibilities.
Either
$d_4 \left( \frac{\gamma}{\tau^{11}} r \right)$
or $d_4 \left( \frac{\gamma}{\rho^4 \tau^7} h_2^2 g \right)$ equals
the element
$\frac{\gamma}{\rho^5 \tau^8} P h_1^2 d_0$.
In the latter case, 
$\frac{\gamma}{\rho^3 \tau^7} h_2^2 g$ would detect an element
in $\piC_{29,2}$ that is non-zero in the cokernel of $\rho$.
This is also ruled out by \cref{prop:rho-SES} since the classical
29-stem is zero.
\end{proof}

\begin{lemma}
\label{lem:g/p^2t^4-f0}
$(20, 4, 3)$
The element $\frac{\gamma}{\rho^2 \tau^4} f_0$ is a permanent cycle.
\end{lemma}

\begin{proof}
Similarly to the proof of \cref{lem:g/pt^6-h4},
we use the Moss convergence theorem \cite{Moss} \cite{BK}.
The Adams differential $d_2(f_0) = h_0 h_2 d_0$ implies that 
$\frac{\gamma}{\rho^2 \tau^4} f_0$
equals the Massey product
$\left\langle \frac{\gamma}{\rho^2 \tau^4}, h_0, h_2 d_0 \right\rangle$
in the Adams $E_3$-page.  The technical hypotheses of the Moss convergence
theorem are satisfied, so the Massey product converges to a Toda bracket,
and 
$\frac{\gamma}{\rho^2 \tau^4} f_0$
is a permanent cycle.
\end{proof}

\begin{lemma}
\label{lem:d4-g/t^10-Dh2^2}
$(30, 6, 3)$ $d_4 \left( \frac{\gamma}{\tau^{10}} r \right) = \frac{\gamma}{\rho \tau^9} h_0^2 d_0^2$.
\end{lemma}

\begin{proof}
If $\frac{\gamma}{\rho \tau^9} h_0^2 d_0^2$ were non-zero in the
Adams $E_\infty$-page, then it would
detect an element of $\piC_{29,2}$ that is annihilated by $\rho$.
Since the classical 29-stem is zero, 
\cref{prop:rho-SES} implies that there are no non-zero elements in
$\piC_{29,2}$ that are annihilated by $\rho$.
Therefore, $\frac{\gamma}{\rho \tau^9} h_0^2 d_0^2$ must be hit by an Adams differential or support an Adams differential. 
The only possibility is that 
$d_4 \left( \frac{\gamma}{\tau^{10}} r \right)$ equals $\frac{\gamma}{\rho \tau^9} h_0^2 d_0^2$.
\end{proof}

\begin{rmk}
We have used Massey products and Toda brackets to rule out a number of
possible long differentials.
Another possible approach is to use $v_1$-periodic $C_2$-equivariant 
computations \cite{Balderrama}.
For example, 
consider the element $\frac{\gamma}{\rho^6 \tau} h_0 h_3^2$ in
degree $(20, 3, 4)$.
For degree reasons, there is a possible
$d_8$ differential with value $\frac{\gamma}{\tau^4} P^2 h_1^3$.
However, the latter element cannot be hit by a differential
because it is detected in $v_1$-periodic homotopy.
\end{rmk}

%%%%%%%%%%%%%
%%%%%%%%%%%%%  Cofiber rho  %%%%%%%%%%%%%%%%%%%%%
%%%%%%%%%%%%%

\section{Hidden $\rho$-extensions and the cofiber of $\rho$ sequence}
\label{sec:cofibrho}

Our next goal is to establish products in $\piC_{*,*}$ that
are hidden in the $C_2$-equi\-var\-iant Adams spectral sequence.
We concentrate
exclusively on extensions by the elements $\rho$, 
$\eta$ (detected by $h_1$), and $\hsf=2+\rho\eta$ (detected by $h_0$).
These extensions preserve coweight, 
so they are easy to visualize on our charts that are indexed by
coweight.
There are numerous other hidden extensions in the range under study
that we have not attempted to resolve.  Many such extensions are 
probably easy to obtain, while others would require significant work.
Hidden extensions by $\hsf$ and $\eta$ are handled in \cref{sec:hiddenAdams}.
We refer the careful reader to
\cite{IWX23}*{Section~2.1.2} for the precise meaning of
``hidden extension".

We deal with the hidden $\rho$-extensions in this section, while also 
analyzing the short exact sequence
\begin{equation}
\label{eq:SEScofibrho}
	0 \rightarrow (\coker \rho)_{*,*} \xrightarrow{U} \picl_*[\tau^{\pm 1}] 
\xrightarrow{p} (\ker \rho)_{*,*-1} \rightarrow 0
\end{equation}
of \cref{prop:rho-SES}. 
This analysis provides an abundance of computational information about the
values of the underlying homomorphism.
In principle, one should first determine all hidden $\rho$-extensions and then 
study the sequence \eqref{eq:SEScofibrho}. In practice, these two analyses are
best undertaken simultaneously. 
We are aided significantly by the knowledge that the underlying homomorphism
$(\coker \rho)_{*,*} \xrtarr{U} \picl_*[\tau^{\pm 1}]$ is a ring homomorphism and preserves
the Adams filtration (\cref{prop:rho-SES}).

We describe elements of stable homotopy groups by their names on the 
$E_\infty$-page of the Adams spectral sequence. Thus, homotopy elements 
are only specified up to elements of higher Adams filtration, and we only 
determine the sequence \cref{eq:SEScofibrho} up to such ambiguity.
See \cite{IWX23}*{Section~2.1} for a related discussion of maps between filtered 
abelian groups.

\begin{thm}
\label{thm:cofibrhoseqn}
In stems up to 26 and coweights from 0 to 7,
the values of the exact sequence \cref{eq:SEScofibrho} 
are given in the charts of \cref{fig:cofiberrho}.
\end{thm}

The charts of \cref{fig:cofiberrho} tabulate the values of $U$ and $p$
in a concise graphical format.  The dots on the chart represent
elements of $\picl_* \cdot \tau^k$.  The blue labels
indicate pre-images of the map $U$, while the orange labels
indicate values of the map $p$.  See also \cref{subsctn:chart-cofiber-rho} for additional information about reading the charts.

\begin{rmk}
\label{rmk:cofiber-rho-ambigious}
Our graphical calculus breaks down in one instance in each coweight of \cref{fig:cofiberrho} in stem 23 and filtration 9.
There are three non-zero elements in $\picl_{23}$ that
are detected in filtration 9.  Exactly one of those elements is a multiple of 
$\eta$, and that element is indicated in the charts by the $h_1$-extensions.
However, the other dot ambiguously represents one of the two
remaining non-zero elements.
The Massey product 
$\langle h_3, h_0^4, h_0^4 h_4 \rangle$ can be used to 
to resolve some, but not all, of the ambiguities.
\end{rmk}

We  typically suppress powers of $\tau$ in the target of $U$ or source of $p$.
For example, the element $c_0$ is in coweight 3 in the $C_2$-equivariant
Adams spectral sequence, but we write $U(c_0) = c_0$ rather than 
$U (c_0) = \tau^3 c_0$.

\begin{pf}
In the range under consideration, many values of the homomorphisms 
$U$ and $p$ are detected by the maps 
of Adams $E_\infty$-pages, i.e., with no filtration shifts.
For example, the 
homomorphism 
$p\colon \picl_* \cdot \tau^0 \to (\ker \rho)_{*,-1}$ is entirely 
detected by Adams $E_\infty$-pages in stems less than 26.

Hidden values of $U$ and $p$  involve a strict increase in Adams filtration. 
An early example of this phenomenon occurs in the homomorphism
$U\colon (\coker \rho)_{7,0} \to \picl_7 \cdot \tau^0$. 
The group $(\coker \rho)_{7,0}$ contains a single element, detected by 
$\frac{Q}{\rho^2} h_1^4$ in filtration 3. This element must map to an element of $\picl_7$ 
of filtration at least 4, and $h_0^3 h_3$ is the only possibility.

Many hidden values are implied by multiplicative structure. 
For instance, consider $U\colon (\coker \rho)_{8,0} \to \picl_8 \cdot \tau^0$. 
Here $\frac{\gamma}{\rho^2 \tau} h_2^2$ in filtration 2 must map to a 
class detected in higher filtration, of which the only possibility is $c_0$. 
It follows that $h_1 \cdot \frac{\gamma}{\rho^2 \tau} h_2^2$ must map to 
$h_1 c_0$. Incidentally, we also know that 
$h_1 \cdot \frac{\gamma}{\rho^2 \tau} h_2^2$ equals 
$ \frac{\gamma}{\rho \tau^2} \cdot c_0$ in the $C_2$-equivariant Adams 
spectral sequence, so this discussion agrees with the facts that 
$U\left(\frac{\gamma}{\rho \tau^2}\right) = h_1$ and $U(c_0) = c_0$.
The value $U\left(\frac{\gamma}{\tau^{2k+1}}\right) = h_0$ also determines many 
other values of the homomorphism $U$ under multiplicative decompositions.

Another use of multiplicative structure occurs in
the homomorphism
\[
U\colon (\coker \rho)_{15,1} \to \picl_{15} \cdot \tau^1.
\]
The element 
$\frac{\gamma}{\rho^6 \tau} h_1^2 h_3$ in filtration 3 is 
annihilated by $h_0$.
It follows that $U \left( \frac{\gamma}{\rho^6 \tau} h_1^2 h_3 \right)$
must be annihilated by $h_0$ in $\picl_{15}$.
The group $\picl_{15}$ contains three non-zero elements that 
are annihilated by $h_0$:
one detected by $h_0^7 h_4$ in 
filtration 8 and the other two detected by $h_1 d_0$ in filtration 5. But 
$U \left( \frac{\gamma}{\tau^5} h_0^6 h_4 \right)$ is  equal to $h_0^7 h_4$, so it follows that 
$U\left( \frac{\gamma}{\rho^6 \tau} h_1^2 h_3 \right)$ must be detected in filtration 5.
\end{pf}

\begin{rmk}
In the language of \cite{IWX23}*{Section~2.1}, 
there are no crossing values for the map $U$
in the range described in \cref{thm:cofibrhoseqn}.
On the other hand, there is exactly one crossing value for $p$
in this range, occurring in the homomorphism
$p\colon \picl_{15} \cdot \tau^5 \to (\ker \rho)_{15,4}$.
The homomorphism takes $h_1 d_0$ in filtration 5 to
$\frac{\gamma}{\tau} h_1 d_0$ in filtration 5.
It also takes $h_0^3 h_4$ in filtration 4 to $\frac{\gamma}{\rho\tau} h_0^2 d_0$ in filtration 6.
\end{rmk}

\begin{rmk}
Our use of homotopy elements up to higher filtration can lead to some
subtle complications.  We illustrate this problem with the
map $p\colon \picl_{21} \cdot \tau^2 \to (\ker \rho)_{21,1}$ and
its value $p(h_1 g)$.
In the 20 stem, the value $p(g) = \frac{\gamma}{\tau^6} g$ is not hidden.  
Then
\[
p(h_1 g) = h_1 \cdot p(g) = h_1 \cdot \frac{\gamma}{\tau^6} g.
\]
This last expression is zero in the $C_2$-equivariant Adams $E_\infty$-page.
Therefore, $p(h_1 g)$ is in fact detected in higher filtration
by $\frac{\gamma}{\rho^5 \tau^4} h_1^2 d_0$.

Note that $\frac{\gamma}{\tau^6} g$ detects two elements in
$\piC_{20,1}$ because of the presence of
$\frac{\gamma}{\rho^5 \tau^4} h_1 d_0$ in higher filtration.
One of these two homotopy elements supports an $\eta$-extension,
and the other does not.
In terms of homotopy elements, the formula
$p(g) = \frac{\gamma}{\tau^6} g$ 
means that $p(g)$ is the homotopy element that is detected by
$\frac{\gamma}{\tau^6} g$ and that also supports an $\eta$-extension.
\end{rmk}

\begin{thm}
\label{thm:hidden-rho}
In stems less than 26 and coweights from -1 to 7, 
\cref{tbl:hidden-rho} lists all hidden $\rho$-extensions in
the $C_2$-equivariant Adams spectral sequence.
\end{thm}

\renewcommand*{\arraystretch}{1.25}
\begin{longtable}{LLL}
\caption{Some hidden $\rho$-extensions in the
$C_2$-equivariant Adams spectral sequence} \\
\toprule
\mbox{}(s,f,c) & \textrm{source} & \textrm{target} \\
\midrule \endhead
\bottomrule \endfoot
\label{tbl:hidden-rho}
(17, 4, -1) & \frac{\gamma}{\tau^7} e_0 & \frac{\gamma}{\rho \tau^6} h_1 d_0 \\
(25, 8, -1) & \frac{\gamma}{\tau^{11}} P e_0 & \frac{\gamma}{\rho \tau^{10}} P h_1 d_0 \\
(26, 7, -1) & \frac{\gamma}{\tau^{12}} j & \frac{\gamma}{\tau^{11}} P e_0 \\
(k+5, k+3, 0) & Q h_1^{k+4} & h_1^{k+4} \\
(15, 1, 0) & \frac{\gamma}{\tau^6} h_4 & \frac{\gamma}{\tau^5} h_3^2 \\
(15, 3, 0) & \frac{\gamma}{\rho \tau^5} h_0 h_3^2 & \frac{\gamma}{\tau^5} d_0 \\
(29, 7, 0) & \frac{\gamma}{\tau^{12}} k & \frac{\gamma}{\tau^{11}} d_0^2 \\
(30, 6, 0) & \frac{\gamma}{\tau^{13}} r & \frac{\gamma}{\tau^{12}} k \\
(29, 8, 1) & \frac{\gamma}{\tau^{11}} h_0 k & \frac{\gamma}{\rho^5 \tau^8} P h_1 d_0 \\
(30, 6, 1) & \frac{\gamma}{\tau^{12}} r & \frac{\gamma}{\tau^{11}} h_0 k \\
(15, 1, 2) & \frac{\gamma}{\tau^4} h_4 & \frac{\gamma}{\tau^3} h_3^2 \\
(15, 3, 2) & \frac{\gamma}{\rho \tau^3} h_0 h_3^2 & \frac{\gamma}{\tau^3} d_0 \\
(23, 7, 2) & \frac{\gamma}{\tau^8} i & \frac{\gamma}{\tau^7} P d_0 \\
(29, 7, 2) & \frac{\gamma}{\tau^{10}} k & \frac{\gamma}{\tau^9} d_0^2 \\
(11, 4, 3) & Q h_1^2 c_0 & h_1^2 c_0 \\
(12, 5, 3) & Q h_1^3 c_0 & h_1^3 c_0 \\
(13, 6, 3) & Q h_1^4 c_0 & h_1^4 c_0 \\
(14, 7, 3) & Q h_1^5 c_0 & h_1^5 c_0 \\
(15, 3, 3) & \frac{\gamma}{\tau^3} h_0^2 h_4 & \frac{Q}{\rho^2} h_1^2 c_0 \\
(17, 4, 3) & \frac{\gamma}{\tau^3} e_0 & \frac{\gamma}{\rho \tau^2} h_1 d_0 \\
(18, 4, 3) & \frac{\gamma}{\tau^4} f_0 & \frac{\gamma}{\tau^3} h_0 e_0 \\
(23, 8, 3) & \frac{\gamma}{\tau^7} h_0 i & \frac{\gamma}{\rho^5 \tau^4} P^2 h_1 \\
(15, 1, 4) & \frac{\gamma}{\tau^2} h_4 & \frac{\gamma}{\tau} h_3^2 \\
(15, 3, 4) & \frac{\gamma}{\rho \tau} h_0 h_3^2 & \frac{\gamma}{\tau} d_0 \\
(17, 4, 4) & \frac{\gamma}{\tau^2} e_0 & \frac{\gamma}{\rho \tau} h_1 d_0 \\
(18, 5, 4) & \frac{\gamma}{\tau^2} h_1 e_0 & \frac{\gamma}{\rho \tau} h_1^2 d_0 \\
(18, 4, 5) & \frac{\gamma}{\tau^2} f_0 & \frac{\gamma}{\tau} h_0 e_0 \\
(15, 4, 7) & h_0^3 h_4 & \rho^4 h_1 e_0 \\
(17, 5, 7) & h_2 d_0 & \tau h_1^2 d_0 \\
(23, 8, 7) & Q h_1^5 e_0 + \frac{\gamma}{\tau^3} h_0 i & \frac{Q}{\rho^2} P h_1^3 c_0 + h_1^5 e_0 \\
(24, 9, 7) & Q h_1^6 e_0 & \frac{Q}{\rho^2} P h_1^4 c_0 + h_1^6 e_0 \\
(25, 8, 7) & \frac{\gamma}{\tau^3} P e_0 & \frac{\gamma}{\rho \tau^2} P h_1 d_0 \\
(25, 10, 7) & Q h_1^7 e_0 & \frac{Q}{\rho^2} P h_1^5 c_0 + h_1^7 e_0 \\
(26, 7, 7) & \frac{\gamma}{\tau^4} j & \frac{\gamma}{\tau^3} P e_0 \\
(26, 11, 7) & Q h_1^8 e_0 & \frac{Q}{\rho^2} P h_1^6 c_0 + h_1^8 e_0 \\
(27, 12, 7) & Q h_1^9 e_0 & \frac{Q}{\rho^2} P h_1^7 c_0 + h_1^9 e_0 \\
(28, 13, 7) & Q h_1^{10} e_0 & \frac{Q}{\rho^2} P h_1^8 c_0 + h_1^{10} e_0 \\
(29, 14, 7) & Q h_1^{11} e_0 & \frac{Q}{\rho^2} P h_1^9 c_0 + h_1^{11} e_0 \\
(30, 15, 7) & Q h_1^{12} e_0 & \frac{Q}{\rho^2} P h_1^{10} c_0 + h_1^{12} e_0 \\
\end{longtable}
\renewcommand*{\arraystretch}{1.0}

The hidden $\rho$-extensions in \cref{tbl:hidden-rho}
are displayed in \cref{fig:Einf} by dashed lines.

\begin{rmk}
In coweight $7$, the hidden $\rho$-extensions on
$h_0^3 h_4$ and $h_2 d_0$ are detected in $\R$-motivic homotopy
and are studied in \cite{BI}.  We include them
in \cref{tbl:hidden-rho} for completeness.
\end{rmk}

\begin{rmk}
\cref{tbl:hidden-rho} lists some extensions that lie beyond the $25$-stem.
Even though our analysis of the Adams $E_\infty$-page is incomplete
in this range, the uncertainties in Adams differentials do not affect
the specific elements involved in these hidden extensions.
\end{rmk}

\begin{rmk}
In stems less than 8 and coweights between $-9$ and $-2$, there
are no possible hidden $\rho$-extensions.
In other words, there are no hidden $\rho$-extensions
in \cref{fig:EinfNegCoweight}.
\end{rmk}

\begin{proof}
The multiplicative structure rules out many possible hidden $\rho$-extensions.
For example, consider the element $\frac{\gamma}{\tau^7} h_2 h_4$
in stem $18$ and coweight $0$.  This element is the product
$\frac{\gamma}{\tau^7} \cdot h_2 h_4$ of two permanent cycles.
The first factor is annihilated by $\rho$, so the product
cannot support a hidden $\rho$-extension.

Information about \eqref{eq:SEScofibrho},
as stated in \cref{thm:cofibrhoseqn} and displayed in \cref{fig:cofiberrho},
offers a powerful technique for resolving hidden $\rho$-extensions. 
For example, consider the possible hidden $\rho$-extensions 
on $\frac{\gamma}{\tau^6} h_4$ 
and $\frac{\gamma}{\rho \tau^5} h_0 h_3^2$
in stem 15 and coweight 0.  
By inspection of coweight $-1$, the map
$p: \picl_{14}\cdot \tau^0 \xrtarr{} (\ker \rho)_{14,-1}$
is an isomorphism.
Therefore,
$(\coker \rho)_{14,0}$ vanishes, and  both
$\frac{\gamma}{\tau^5} h_3^2$ and $\frac{\gamma}{\tau^5} d_0$
must be the targets of hidden $\rho$-extensions.
The sources of these extensions must be
$\frac{\gamma}{\tau^6} h_4$ 
and $\frac{\gamma}{\rho \tau^5} h_0 h_3^2$
respectively.

In some cases, multiplicative structure is also needed to 
establish hidden $\rho$-extensions.  For example,
consider the element $\frac{\gamma}{\tau^3} e_0$ in degree $(17,4,3)$.
In the short exact sequence
\[
0 \rightarrow (\coker \rho)_{16,3} \rightarrow \picl_{16} \cdot \tau^3 \rightarrow 
(\ker \rho)_{16,2} \rightarrow 0,
\]
the classical group $\picl_{16}$ has order four.
By inspection of coweight $2$,
the group
$(\ker \rho)_{16,2}$ has order two, detected by $\frac{\gamma}{\tau^4} h_1 h_4$.
Therefore, $(\coker \rho)_{16,3}$ has order two, which means that 
either $\frac{\gamma}{\rho \tau^2} h_1 d_0$ or $\frac{Q}{\rho^3} h_1^4 c_0$
is the target of a hidden extension on $\frac{\gamma}{\tau^3} e_0$.
Consideration of $h_1$-extensions rules out the latter possibility,
since $\frac{\gamma}{\tau^3} e_0$ cannot support a hidden $\eta$-extension.

Multiplicative structure can also be used in a more straightforward way
to obtain hidden $\rho$-extensions.
For example, an analysis of the short exact sequence
\[
0 \rightarrow (\coker \rho)_{16,4} \rightarrow \picl_{16} \cdot \tau^4 \rightarrow 
(\ker \rho)_{16,3} \rightarrow 0
\]
shows that there is a hidden $\rho$-extension from $\frac{\gamma}{\tau^2} e_0$
to $\frac{\gamma}{\rho \tau} h_1 d_0$.  Then multiplication by $h_1$
implies that there is a hidden $\rho$-extension from
$\frac{\gamma}{\tau^2} h_1 e_0$ to $\frac{\gamma}{\rho \tau} h_1^2 d_0$.

One slightly more difficult case is established below in
\cref{lem:rho-g/t^7-h0i}.
\end{proof}

\begin{rmk}
There are two ways to use 
the sequence \eqref{eq:SEScofibrho}
to establish a given hidden $\rho$-extension.
One approach is to show that a possible target of an extension
cannot be detected in $\coker \rho$.  This is illustrated in the second paragraph
of the proof of \cref{thm:hidden-rho}.

A second option is to show that a possible source of an extension 
cannot be detected in $\ker \rho$.
For example, consider the possible hidden extension
from $Q h_1^3 c_0$ to $h_1^3 c_0$ in coweight $3$.
In the short exact sequence
\[
0 \rightarrow (\coker \rho)_{12,4} \rightarrow \picl_{12} \cdot \tau^4 \rightarrow 
(\ker \rho)_{12,3} \rightarrow 0,
\]
the middle object $\picl_{12}$ is zero, so
$(\ker \rho)_{12,3}$ must also be zero.  Therefore,
$Q h_1^3 c_0$ cannot detect an element of 
$(\ker \rho)_{12,3}$, and it must support a hidden $\rho$-extension.

In many cases, a given possible hidden $\rho$-extension can be determined
using either $\coker \rho$ or $\ker \rho$.
\end{rmk}

\begin{lemma}
\label{lem:rho-g/t^7-h0i}
$(23,8,3)$
There is a hidden $\rho$-extension from $\frac{\gamma}{\tau^7} h_0 i$
to $\frac{\gamma}{\rho^5 \tau^4} P^2 h_1$.
\end{lemma}

\begin{proof}
From \eqref{eq:SEScofibrho}, we have the
short exact sequence
\[
0 \rightarrow (\coker \rho)_{23,4} \rightarrow \picl_{23} \cdot \tau^4 \rightarrow 
(\ker \rho)_{23,3} \rightarrow 0.
\]
If $\frac{\gamma}{\tau^7} h_0 i$ did not support a hidden $\rho$-extension,
then $\frac{\gamma}{\tau^7} h_0^{k+1} i$ would detect an element of
order $32$ in $(\ker \rho)_{23,3}$. 
However, $\picl_{23}$ does not contain an element 
whose order is a multiple of $32$,
so $(\ker \rho)_{23,3}$ does not contain an element
of order 32.
\end{proof}

%%%%%%%%%%%%%
%%%%%%%%%%%%%  Hidden extensions  %%%%%%%%%%%%%%%%%%%%%
%%%%%%%%%%%%%

\section{Hidden extensions by $\hsf$ and $\eta$}
\label{sec:hiddenAdams}

We now turn to the hidden extensions by $\hsf$ and $\eta$.

\subsection{Hidden $\hsf$-extensions}

\begin{thm}
In stems less than 26 and coweights from -1 to 7, 
\cref{tbl:hidden-h} lists all hidden $\hsf$-extensions in
the $C_2$-equivariant Adams spectral sequence.
\end{thm}

\renewcommand*{\arraystretch}{1.25}
\begin{longtable}{LLLl}
\caption{Some hidden $\hsf$-extensions in the
$C_2$-equivariant Adams spectral sequence} \\
\toprule
\mbox{}(s,f,c) & \textrm{source} & \textrm{target} & proof \\
\midrule \endhead
\bottomrule \endfoot
\label{tbl:hidden-h}
(28, 6, 0) & \frac{\gamma}{\rho^2 \tau^9} h_2^2 g & \frac{\gamma}{\tau^{11}} d_0^2 & \cref{lem:h-g/p^2t^9-h2^2g} \\
(21, 4, 1) & \frac{\gamma}{\rho \tau^6} g & \frac{\gamma}{\rho^5 \tau^4} h_1^2 d_0 & \cref{lem:h-g/pt^6-g} \\
(23, 5, 1) & \frac{\gamma}{\tau^7} h_2 g & \frac{\gamma}{\rho^6 \tau^4} h_1^3 d_0 & $\coker \rho$ \\
(22, 6, 2) & \frac{\gamma}{\rho^2 \tau^5} h_0 h_2 e_0 & \frac{\gamma}{\tau^7} P d_0 & \cref{lem:h-22-6-2} \\
(11, 6, 3) & \frac{\gamma}{\tau} P h_0 h_2 & \rho^2 h_1^5 c_0 & $\ker \rho$ \\
(23, 5, 3) & \frac{\gamma}{\tau^5} h_2 g & \frac{\gamma}{\rho \tau^6} P d_0 & $\coker \rho$ or $\ker \rho$ \\
(26, 5, 4) & \frac{\gamma}{\rho^2 \tau^5} h_1 h_4 c_0 & \frac{\gamma}{\rho^3 \tau^6} i & \cref{lem:h-26-5-4} \\
(20, 5, 6) & \frac{\gamma}{\tau} h_0 g & \frac{Q}{\rho} h_1^4 d_0 & $\ker \rho$ \\
(11, 4, 7) & \rho^6 e_0 & \tau^2 P h_0 h_2 & $\coker \rho$ \\
(19, 10, 7) & \frac{\gamma}{\tau} P^2 h_0 h_2 & \rho^6 h_1^8 e_0 & $\ker \rho$ \\
(23, 11, 7) & \frac{\gamma}{\tau^3} h_0^4 i & \rho^3 h_1^9 e_0 & $\ker \rho$ \\
(27, 11, 7) & \frac{Q}{\rho} h_1^8 e_0 & \frac{\gamma}{\tau^5} P^3 h_2 & $\coker \rho$ \\
(27, 14, 7) & \frac{\gamma}{\tau^5} P^3 h_0 h_2 & Q P h_1^{10} c_0 + \rho^2 h_1^{12} e_0 & $\ker \rho$
\end{longtable}
\renewcommand*{\arraystretch}{1.0}

\begin{proof}
Many cases are ruled out by the relations
$\rho \cdot \hsf = 0$ and $\hsf \cdot \eta = 0$.
In particular, if an element of the Adams $E_\infty$-page
supports a multiplication by $\rho$ or $h_1$,
then it cannot be the target of a hidden $\hsf$-extension.
Similarly, if an element of the Adams $E_\infty$-page is a
multiple of $\rho$ or $h_1$, then it cannot support a hidden
$\hsf$-extension.

Many extensions are established using the homomorphism
$U: (\coker \rho)_{*,*} \xrtarr{} \picl_* [ \tau^{\pm 1}]$
whose values were established in \cref{thm:cofibrhoseqn}
and are displayed in \cref{fig:cofiberrho}.
These extensions are denoted by ``$\coker \rho$'' in the fourth
column of \cref{tbl:hidden-h}.
For example,
consider the elements $\frac{\gamma}{\tau^7} h_2 g$ and
$\frac{\gamma}{\rho^6 \tau^4} h_1^3 d_0$ in stem $23$ and coweight $1$.
The map $U$ takes these elements to $h_0 h_2 g$ and $P h_1 d_0$
respectively.  There is a classical (hidden) $2$-extension from
$h_0 h_2 g$ to $P h_1 d_0$, so there must also be an $\hsf$-extension
from $\frac{\gamma}{\tau^7} h_2 g$ and
$\frac{\gamma}{\rho^6 \tau^4} h_1^3 d_0$.
The same method can also be used to rule out hidden $\hsf$-extensions.

Additional extensions are established using the homomorphism
$p: \picl_* [ \tau^{\pm 1}] \xrtarr{} (\ker \rho)_{*, *-1}$
whose values were established in \cref{thm:cofibrhoseqn}
and are displayed in \cref{fig:cofiberrho}.
These extensions are denoted by ``$\ker \rho$'' in the fourth
column of \cref{tbl:hidden-h}.
For example,
consider the elements $\frac{\gamma}{\tau} P h_0 h_2$ and
$\rho^2 h_1^5 c_0$ in stem $11$ and coweight $3$.
These elements are the images of $P h_0 h_2$ and $P h_1^3$
respectively under the map $p$.
There is a classical (not hidden) $2$-extension from
$P h_0 h_2$ to $P h_1^3$, so there must also be an $\hsf$-extension
from $\frac{\gamma}{\tau} P h_0 h_2$ to
$\rho^2 h_1^5 c_0$.
The same method can also be used to rule out hidden $\hsf$-extensions.

Several cases remain to be studied.  
The fourth column of \cref{tbl:hidden-h} shows the specific lemmas
in which each is established.
\end{proof}

\begin{rmk}
In coweight $7$, the hidden $\hsf$-extension on
$\rho^6 e_0$ is detected in $\R$-motivic homotopy
and is studied in \cite{BI}.  
For completeness, we include it in \cref{tbl:hidden-h}.
\end{rmk}

\begin{rmk}
\cref{tbl:hidden-h} lists some extensions that lie beyond the $25$-stem.
Even though our analysis of the Adams $E_\infty$-page is incomplete
in this range, the uncertainties in Adams differentials do not affect
the specific elements involved in these hidden extensions.
\end{rmk}

\begin{rmk}
In stems less than 8 and coweights between $-9$ and $-2$, there
are no possible hidden $\hsf$-extensions.
In other words, there are no hidden $\hsf$-extensions
in \cref{fig:EinfNegCoweight}.
\end{rmk}

\begin{lemma}
\label{lem:h-g/p^2t^9-h2^2g}
$(28,6,0)$ There is a hidden $\hsf$-extension from
$\frac{\gamma}{\rho^2 \tau^9} h_2^2 g$ to $\frac{\gamma}{\tau^{11}} d_0^2$.
\end{lemma}

\begin{proof}
We write $\frac{\gamma}{\tau^{12}}$ for a homotopy class in $\piC_{0,-13}$ that is detected by
$\frac{\gamma}{\tau^{12}}$, and we write $\kappa^2$ for a homotopy class
in $\piC_{28,12}$ that is detected by $d_0^2$.
Using the Adams differential $d_2 \left( \frac{\gamma}{\rho^3 \tau^9} h_2^2 g \right) = \frac{\gamma}{\tau^{12}} d_0^2$ and the Moss convergence
theorem \cite{Moss} \cite{BK}, the element
$\frac{\gamma}{\rho^2 \tau^9} h_2^2 g$ detects the Toda bracket
$\left\langle \rho, \frac{\gamma}{\tau^{12}}, \kappa^2 \right\rangle$.

The $\rho$-Bockstein differential $d_1(\tau) = \rho h_0$ and the
May convergence theorem \cite{MMP}*{Theorem 4.1} implies that
$\frac{\gamma}{\tau^{11}}$ detects the Massey product
$\left\langle h_0, \rho, \frac{\gamma}{\tau^{12}} \right\rangle$.  Then the
Moss convergence theorem implies that
$\frac{\gamma}{\tau^{11}}$ detects the Toda bracket
$\left\langle \hsf, \rho, \frac{\gamma}{\tau^{12}} \right\rangle$.

Now shuffle to obtain
\[
\hsf \left\langle \rho, \frac{\gamma}{\tau^{12}}, \kappa^2 \right\rangle =
\left\langle \hsf, \rho, \frac{\gamma}{\tau^{12}} \right\rangle  \kappa^2.
\]
\end{proof}

\begin{lemma}
\label{lem:h-g/pt^6-g}
$(21, 4, 1)$ There is a hidden $\hsf$-extension from
$\frac{\gamma}{\rho \tau^6} g$ to $\frac{\gamma}{\rho^5 \tau^4} h_1^2 d_0$.
\end{lemma}

\begin{proof}
We will show below in \cref{thm:hidden-eta} that there is a hidden
$\eta$-extension from $\frac{\gamma}{\rho \tau^6} g$ to 
$\frac{\gamma}{\rho^6 \tau^4} h_1^2 d_0$.
Therefore, there is a hidden $\rho \eta$-extension
from 
$\frac{\gamma}{\rho \tau^6} g$ to $\frac{\gamma}{\rho^5 \tau^4} h_1^2 d_0$.

Also, $\frac{\gamma}{\rho \tau^6} g$ detects the product
$\frac{\gamma}{\rho \tau^{10}} \cdot \overline{\kappa}$, where
$\frac{\gamma}{\rho \tau^{10}}$ and $\overline{\kappa}$ are homotopy classes
that are detected by 
$\frac{\gamma}{\rho \tau^{10}}$ and $\tau^4 g$, respectively.
We know from \cref{tbl:hiddenh0extns} that $(\hsf + \rho \eta) \frac{\gamma}{\rho \tau^{10}}$ 
equals zero.  Consequently,
there must also be a hidden $\hsf$-extension from
$\frac{\gamma}{\rho \tau^6} g$ to $\frac{\gamma}{\rho^5 \tau^4} h_1^2 d_0$.
\end{proof}

\begin{lemma}
$(21,2,2)$
There is no hidden $\hsf$-extension on $\frac{\gamma}{\rho^5 \tau^4} h_1 h_4$.
\end{lemma}

\begin{proof}
The element
$\frac{\gamma}{\rho \tau^5} h_0 h_2 e_0$ supports a hidden $\eta$-extension
because of the hidden $\rho$-extension from $\frac{\gamma}{\tau^8} i$
to $\frac{\gamma}{\tau^7} P d_0$.
Therefore, it cannot be the target of a hidden $\hsf$-extension.
\end{proof}

\begin{lemma}
\label{lem:h-22-6-2}
$(22,6,2)$ There is a hidden $\hsf$-extension from
$\frac{\gamma}{\rho^2 \tau^5} h_0 h_2 e_0$ to $\frac{\gamma}{\tau^7} P d_0$.
\end{lemma}

\begin{proof}
We write $\frac{\gamma}{\tau^{12}}$ for a homotopy class in $\piC_{0,-13}$ that is detected by
$\frac{\gamma}{\tau^{12}}$, and we write $\alpha$ for a homotopy class
in $\piC_{22,14}$ that is detected by $\tau^4 P d_0$.
Using the Adams differential $d_2 \left( \frac{\gamma}{\rho^3 \tau^5} h_0 h_2 e_0 \right) = \frac{\gamma}{\tau^8} P d_0$ and the Moss convergence
theorem \cite{Moss} \cite{BK}, the element
$\frac{\gamma}{\rho^2 \tau^5} h_0 h_2 e_0$ detects the Toda bracket
$\left\langle \rho, \frac{\gamma}{\tau^{12}}, \alpha \right\rangle$.

As in the proof of \cref{lem:h-g/p^2t^9-h2^2g}, the element
$\frac{\gamma}{\tau^{11}}$ detects the Toda bracket
$\left\langle \hsf, \rho, \frac{\gamma}{\tau^{12}} \right\rangle$.
Now shuffle to obtain
\[
\hsf \left\langle \rho, \frac{\gamma}{\tau^{12}}, \alpha \right\rangle =
\left\langle \hsf, \rho, \frac{\gamma}{\tau^{12}} \right\rangle  \alpha.
\]
\end{proof}

\begin{lemma}
$(23, 3, 2)$
There is no hidden $\hsf$-extension on $\frac{\gamma}{\rho^6 \tau^4} h_1^2 h_4$.
\end{lemma}

\begin{proof}
The Bockstein differential
$d_6 \left( \frac{\gamma}{\rho^6 \tau^4} h_1 \right) =
\frac{\gamma}{\tau^7} h_2^2$ implies that
$\frac{\gamma}{\rho^6 \tau^4} h_1^2 h_4$
detects the Massey product
$\left\langle \frac{\gamma}{\tau^7} h_2, h_2, h_1 h_4 \right\rangle$
in the $C_2$-equivariant Adams $E_2$-page.
The Moss Convergence Theorem \cite{Moss} \cite{BK}
then implies that
$\frac{\gamma}{\rho^6 \tau^4} h_1^2 h_4$
detects the Toda bracket 
$\left\langle \frac{\gamma}{\tau^7} \nu, \nu, \eta_4 \right\rangle$.
Here we write $\frac{\gamma}{\tau^7}$ for a homotopy element that is 
detected by $\frac{\gamma}{\tau^7}$, and $\eta_4$ is a homotopy element
that is detected by $h_1 h_4$.

Now shuffle to obtain
\[
\left\langle \frac{\gamma}{\tau^7} \nu, \nu, \eta_4 \right\rangle \hsf =
\frac{\gamma}{\tau^7} \nu \left\langle \nu, \eta_4, \hsf \right\rangle.
\]
Next we must analyze the Toda bracket $\langle \nu, \eta_4, \hsf \rangle$.
By inspection of the $\R$-motivic Adams charts in \cite{BI},
there are several possible values.  Most of these values are $\rho$-multiples
and are annihilated by $\frac{\gamma}{\tau^7}$.

The one remaining possibility is that 
the Toda bracket $\langle \nu, \eta_4, \hsf \rangle$ could be detected
by $\tau h_2^2 \cdot d_0$.  Write $\alpha$ and $\kappa$ for homotopy elements
detected by $\tau h_2^2$ and $d_0$ respectively.  
By inspection of the coweight $-5$ part of \cref{fig:EinfNegCoweight},
the product $\frac{\gamma}{\tau^7} \cdot \alpha$ must be a multiple of
$\eta$.  Therefore, $\frac{\gamma}{\tau^7} \nu \cdot \alpha \cdot \kappa$
is zero.

We have now concluded that 
$\frac{\gamma}{\tau^7} \nu \left\langle \nu, \eta_4, \hsf \right\rangle$
is zero.  Therefore,
$\left\langle \frac{\gamma}{\tau^7} \nu, \nu, \eta_4 \right\rangle \hsf$
is also zero, and there is no hidden $\hsf$-extension.
\end{proof}

\begin{lemma}
\label{lem:h-26-5-4}
$(26, 5, 4)$
There is a hidden $\hsf$-extension from
$\frac{\gamma}{\rho^2 \tau^5} h_1 h_4 c_0$ to
$\frac{\gamma}{\rho^3 \tau^6} i$.
\end{lemma}

\begin{proof}
\cref{fig:cofiberrho} shows that 
$p: \picl_{26} \cdot \tau^5 \xrtarr{} (\ker \rho)_{26,4}$
takes $h_2^2 g$ to $\frac{\gamma}{\rho^3 \tau^6} i$.
The classical element $h_2^2 g$ detects a $\nu$-multiple,
so $\frac{\gamma}{\rho^3 \tau^6} i$ also detects a $\nu$-multiple.
The only possibility is that there is a hidden $\nu$-extension
from $\frac{\gamma}{\tau^5} h_2 g$ to $\frac{\gamma}{\rho^3 \tau^6} i$.

Let $\alpha$ be a homotopy class in $\piC_{23,9}$ that is 
detected by $h_2 g$, and write $\frac{\gamma}{\tau^5}$ for a 
homotopy class in $\piC_{0,-6}$ that is detected by 
$\frac{\gamma}{\tau^5}$.
Then $\frac{\gamma}{\tau^5} h_2 g$ detects 
$\frac{\gamma}{\tau^5} \cdot \alpha$.
Moreover, the hidden $\nu$-extension discussed in the previous
paragraph implies that
$\frac{\gamma}{\rho^3 \tau^6} i$ detects
$\frac{\gamma}{\tau^5} \cdot \nu \cdot \alpha$.

We know from 
\cref{tbl:hiddenh0extns} (see also \cref{fig:EinfNegCoweight}) that
$\frac{\gamma}{\tau^5} \cdot \nu$ is divisible by $\hsf$.
Therefore,
$\frac{\gamma}{\rho^3 \tau^6} i$ must be the target of a hidden $\hsf$-extension, and there is only one possible source.
\end{proof}

\subsection{Hidden $\eta$-extensions}

\begin{thm} 
\label{thm:hidden-eta}
\cref{tbl:hidden-eta} lists all hidden $\eta$-extensions in
the $C_2$-equivariant Adams spectral sequence whose sources
are in stems less than 25 and coweights from -1 to 7.
\end{thm}

\renewcommand*{\arraystretch}{1.25}
\begin{longtable}{LLLl}
\caption{Some hidden $\eta$-extensions in the
$C_2$-equivariant Adams spectral sequence} \\
\toprule
\mbox{}(s,f,c) & \textrm{source} & \textrm{target} & proof \\
\midrule \endhead
\bottomrule \endfoot
\label{tbl:hidden-eta}
(21, 5, -1) & \frac{\gamma}{\tau^8} h_1 g & \frac{\gamma}{\tau^{10}} P d_0 & $\ker \rho$ \\
(24, 6, -1) & \frac{\gamma}{\rho \tau^9} h_0 h_2 g & \frac{\gamma}{\tau^{11}} P e_0 & $\rho$-extension \\
(27, 6, 0) & \frac{\gamma}{\rho \tau^9} h_2^2 g & \frac{\gamma}{\tau^{11}} d_0^2 & $\rho$-extension \\
(15, 4, 1) & \frac{\gamma}{\tau^5} h_0^3 h_4 & \frac{\gamma}{\tau^5} P c_0 & $\ker \rho$ \\
(21, 4, 1) & \frac{\gamma}{\rho \tau^6} g & \frac{\gamma}{\rho^6 \tau^4} h_1^2 d_0 & $\coker \rho$ \\
(23, 9, 1) & \frac{\gamma}{\tau^9} h_0^2 i & \frac{\gamma}{\tau^9} P^2 c_0 & $\ker \rho$ \\
(20, 4, 2) & \frac{\gamma}{\tau^5} g & \frac{\gamma}{\rho \tau^5} h_0 h_2 e_0 & $\ker \rho$ \\
(21, 6, 2) & \frac{\gamma}{\rho \tau^5} h_0 h_2 e_0 & \frac{\gamma}{\tau^7} P d_0 & $\rho$-extension \\
(27, 6, 2) & \frac{\gamma}{\rho \tau^7} h_2^2 g & \frac{\gamma}{\tau^9} d_0^2 & $\rho$-extension \\
(15, 3, 3) & \frac{\gamma}{\tau^3} h_0^2 h_4 & \frac{Q}{\rho^3} h_1^4 c_0 & $\coker \rho$ \\
(23, 8, 3) & \frac{\gamma}{\tau^7} h_0 i & \frac{\gamma}{\rho^6 \tau^4} P^2 h_1^2 & $\coker \rho$ \\
(15, 3, 4) & \frac{\gamma}{\rho \tau} h_0 h_3^2 & \frac{\gamma}{\rho \tau} h_1 d_0 & $\rho$-extension \\
(15, 4, 5) & \frac{\gamma}{\tau} h_0^3 h_4 & \frac{\gamma}{\tau} P c_0 & $\ker \rho$ \\
(23, 9, 5) & \frac{\gamma}{\tau^5} h_0^2 i & \frac{\gamma}{\tau^5} P^2 c_0 & $\ker \rho$ \\
(15, 4, 7) & h_0^3 h_4 & \rho^3 h_1^2 e_0 & $\coker \rho$ \\
(21, 5, 7) & \gamma h_1 g & \frac{\gamma}{\tau^2} P d_0 & $\ker \rho$ \\
(24, 6, 7) & \frac{\gamma}{\rho \tau} h_0 h_2 g & \frac{\gamma}{\tau^3} P e_0 & $\rho$-extension \\
\end{longtable}
\renewcommand*{\arraystretch}{1.0}

\begin{rmk}
The hidden $\eta$-extension on $\frac{\gamma}{\rho \tau^6} g$
was used earlier in
\cref{lem:h-g/pt^6-g} to establish a hidden
$\hsf$-extension on $\frac{\gamma}{\rho \tau^6} g$.
The proof of 
the hidden $\eta$-extension on 
$\frac{\gamma}{\rho \tau^6} g$
does not use information about
$\hsf$-extensions, so there is no danger of circularity.
\end{rmk}

\begin{rmk}
We warn the careful reader about a subtlety in $\eta$-multiplications
that arises because of our precise definition of hidden extensions.
Consider the $C_2$-equivariant Adams $E_\infty$-page element
$\frac{\gamma}{\tau^9} h_0 h_2 g$ in degree $(23,6,-1)$.
This $E_\infty$-page element detects more than one
element in homotopy because of the presence of elements in higher filtration.
Some of these elements support $\eta$-extensions, and some of them
are annihilated by $\eta$.  The coweight $-1$ part of \cref{fig:Einf}
suggests that $\frac{\gamma}{\tau^9} h_0 h_2 g$ detects a homotopy
element that is simultaneously a multiple of $\hsf$ and a multiple
of $\rho$.  In fact, this is not the case.  The element
$\frac{\gamma}{\tau^9} h_0 h_2 g$ does detect a multiple of $\hsf$,
and it also detects a multiple of $\rho$, but those two multiples
differ by an element in higher filtration.
This difference can be established by considering the $\rho \eta$-extension
on $\frac{\gamma}{\rho \tau^9} h_0 h_2 g$.

A similar phenomenon occurs for the elements
$\frac{\gamma}{\tau^6} g$ in degree $(20,4,1)$
and $\frac{\gamma}{\tau} h_0 h_2 g$ in degree $(23,6,7)$.
\end{rmk}

\begin{rmk}
In stems less than 8 and coweights between $-9$ and $-2$, there
are no possible hidden $\eta$-extensions.
In other words, there are no hidden $\eta$-extensions
in \cref{fig:EinfNegCoweight}.
\end{rmk}

\begin{proof}
Many cases are ruled out by the relation $\hsf \cdot \eta = 0$.
In particular, if an element of the Adams $E_\infty$-page
supports a (hidden or not hidden) multiplication by $\hsf$, 
then it cannot be the target
of a hidden $\eta$-extension.  Similarly, if an element of the
Adams $E_\infty$-page is a (hidden or not hidden) multiple of $\hsf$, then it cannot support
a hidden $\eta$-extension.

Multiplicative structure in the Adams $E_\infty$-page also
rules out many cases.  For example, consider the element
$\frac{\gamma}{\rho^7 \tau^6} h_3^2$ in degree $(21,2,-1)$.
This element decomposes as a product $\frac{\gamma}{\rho^7 \tau^8} \cdot \tau^2 h_3^2$ of permanent cycles.  The second factor does not support a hidden
$\eta$-extension from the analysis of the $\R$-motivic Adams spectral
sequence \cite{BI}, so $\frac{\gamma}{\rho^7 \tau^6} h_3^2$ also
does not support a hidden $\eta$-extension.

The $\rho$ multiplications in the Adams $E_\infty$-page rule out
additional possibilities.  For example, consider the element
$\frac{\gamma}{\rho \tau^5} h_1^2 h_4$ in degree $(18, 3, 1)$.
It detects elements 
of $\piC_{18,1}$, all of which are annihilated by $\rho^2$.
Therefore, it cannot support a hidden $\eta$-extension to
$\frac{\gamma}{\rho^4 \tau^4} h_1 d_0$, since the latter element
supports a $\rho^2$-extension.

Some hidden $\eta$-extensions can be deduced immediately from
the hidden $\rho$-ex\-ten\-sions that 
were previously analyzed in \cref{sec:cofibrho}.
These extensions are denoted by ``$\rho$-extension'' in the fourth
column of \cref{tbl:hidden-eta}.
For example,
consider the element $\frac{\gamma}{\rho \tau^5} h_0 h_2 e_0$
in degree $(21, 6, 2)$.
There is a hidden $\rho$-extension from $\frac{\gamma}{\tau^8} i$
to $\frac{\gamma}{\tau^7} P d_0$, so there must also be a hidden
$\eta$-extension from $\frac{\gamma}{\rho \tau^5} h_0 h_2 e_0$
to $\frac{\gamma}{\tau^7} P d_0$.

The presence or absence of hidden $\rho$-extensions
can also rule out possible hidden $\eta$-extensions.
For example, consider the element $\frac{\gamma}{\tau^5} d_0$
in degree $(14, 4, 0)$. This element is the target of a hidden
$\rho$-extension, so it detects a multiple of $\rho^3$.
But $\frac{\gamma}{\rho \tau^5} h_0^2 d_0$ detects elements that are
not multiples of $\rho^3$, so there cannot be a hidden
$\eta$-extension from $\frac{\gamma}{\tau^5} d_0$ to
$\frac{\gamma}{\rho \tau^5} h_0^2 d_0$.

Another example concerns
the element $\frac{\gamma}{\tau^7} g$ in degree 
$(20, 4, 0)$. We already know that it does not support a hidden
$\rho$-extension, so it cannot support a hidden $\eta$-extension
to $\frac{Q}{\rho^8} h_1^{12}$.

Many extensions are established using the homomorphism
$U: (\coker \rho)_{*,*} \xrtarr{} \picl_* [ \tau^{\pm 1}]$
whose values were established in \cref{thm:cofibrhoseqn}
and are displayed in \cref{fig:cofiberrho}.
These extensions are denoted by ``$\coker \rho$'' in the fourth
column of \cref{tbl:hidden-eta}.
For example,
consider the elements $\frac{\gamma}{\rho \tau^6} g$ and
$\frac{\gamma}{\rho^6 \tau^4} h_1^2 d_0$ in stems $21$ and $22$ and coweight $1$.
The map $U$ takes these elements to $h_1 g$ and $P d_0$
respectively.  There is a classical (hidden) $\eta$-extension from
$h_1 g$ to $P d_0$, so there must also be an $\eta$-extension
from $\frac{\gamma}{\rho \tau^6} g$ to $\frac{\gamma}{\rho^6 \tau^4} h_1^2 d_0$.
The same method can also be used to rule out hidden $\eta$-extensions.

Additional extensions are established using the homomorphism
$p: \picl_* [ \tau^{\pm 1}] \xrtarr{} (\ker \rho)_{*, *-1}$
whose values were established in \cref{thm:cofibrhoseqn}
and are displayed in \cref{fig:cofiberrho}.
These extensions are denoted by ``$\ker \rho$'' in the fourth
column of \cref{tbl:hidden-eta}.
For example,
consider the elements $\frac{\gamma}{\tau^8} h_1 g$ and
$\frac{\gamma}{\tau^{10}} P d_0$ in stems $21$ and $22$ and coweight $-1$.
These elements are the images of $h_1 g$ and $P d_0$
respectively under the map $p$.
There is a classical (hidden) $\eta$-extension from
$h_1 g$ to $P d_0$, so there must also be an $\eta$-extension
from $\frac{\gamma}{\tau^8} h_1 g$ to $\frac{\gamma}{\tau^{10}} P d_0$.
The same method can also be used to rule out hidden $\eta$-extensions.

\cref{lem:eta-15-4-1}, \cref{lem:eta-22-3-5}, and \cref{lem:eta-23-4-5}
handle three additional cases.
\end{proof}

\begin{lemma}
\label{lem:eta-15-4-1}
$(15, 4, 1)$ There is no hidden $\eta$-extension
on $\frac{\gamma}{\rho \tau^4} d_0$.
\end{lemma}

\begin{proof}
Let $\alpha$ be an element of $\piC_{17,1}$ that is detected by
$\frac{\gamma}{\rho^3 \tau^4} d_0$.  
We may choose $\alpha$ such that $\eta^2 \alpha$ is $0$, since
$\frac{\gamma}{\rho^3 \tau^4} h_1 d_0$ does not support a hidden
$\eta$-extension.
Then $\frac{\gamma}{\rho \tau^4} d_0$
detects $\rho^2 \alpha$.

If $\frac{\gamma}{\rho \tau^4} d_0$ supported a hidden
$\eta$-extension, then $\eta^2 \cdot \rho^2 \alpha$ would be non-zero.
This is inconsistent with the choice of $\alpha$ in the previous
paragraph.
\end{proof}

\begin{lemma}
\label{lem:eta-22-3-5}
$(22, 3, 5)$
There is no hidden $\eta$-extension on $\frac{\gamma}{\rho^3 \tau^2} c_1$.
\end{lemma}

\begin{proof}
The element $\frac{\gamma}{\rho^3 \tau^2} c_1$ equals the product
$\frac{\gamma}{\rho^3 \tau^4} \cdot \tau^2 c_1$ on the $C_2$-equivariant
Adams $E_\infty$-page.
Write $\frac{\gamma}{\rho^3 \tau^4}$ for a homotopy element that
is detected by $\frac{\gamma}{\rho^3 \tau^4}$, and write
$\alpha$ for a homotopy element that is detected by $\tau^2 c_1$.
We want to show that $\eta \cdot \frac{\gamma}{\rho^3 \tau^4} \cdot \alpha$
is zero.

According to the coweight $-5$ part of \cref{fig:EinfNegCoweight},
the product 
$\eta \cdot \frac{\gamma}{\rho^3 \tau^4}$ equals $\rho^2 \beta$,
where $\beta$ is detected by $\frac{\gamma}{\rho^5 \tau^4} h_1$.
Therefore, we want to show that $\rho^2 \beta \cdot \alpha$ is zero.

The underlying map takes $\alpha$ and $\beta$ to elements of 
$\picl_*$ that are detected by $c_1$ and $h_2^2$ respectively.
In $\picl_*$, the product of these images is zero since there is no
hidden $\nu$-extension on $h_2 c_1$.
Therefore, the underlying map takes $\alpha \cdot \beta$ to zero.

This means that $\alpha \cdot \beta$ 
in $\piC_{25,5}$ is $\rho$-divisible.
There are several such $\rho$-divisible elements in 
$\piC_{25,5}$, but all 
such classes detected in Adams filtration at least 5
are annihilated by $\rho^2$.
\end{proof}

\begin{lemma}
\label{lem:eta-23-4-5}
$(23, 4, 5)$ 
There is no hidden $\eta$-extension on $\frac{\gamma}{\rho^5 \tau^2} f_0$.
\end{lemma}

\begin{proof}
\cref{Ext-h0-extns} shows that there are $h_0$-extensions
in $\Ext^{C_2}$
from $\frac{\gamma}{\rho^6 \tau} h_0 h_3^2$ and $\frac{\gamma}{\rho^5 \tau^2} f_0$
to $\frac{\gamma}{\tau^3} g$ and $\frac{\gamma}{\tau^3} h_2 g$ respectively.
This implies that
there is an $h_2$-extension from
$\frac{\gamma}{\rho^6 \tau} h_0 h_3^2$ to $\frac{\gamma}{\rho^5 \tau^2} f_0$
in $\Ext^{C_2}$ and therefore also in the Adams $E_\infty$-page.
In particular,
$\frac{\gamma}{\rho^5 \tau^2} f_0$ detects a multiple of $\nu$,
so it cannot support an $\eta$-extension \cite{I}*{Lemma~4.4}.
\end{proof}

%%%%%%%%%%%%%%%%%%%%%%%%%%%%%%%%%%%%%%%%%%%%%%%
%%%%%%%%%%%%%%%%%%% Charts %%%%%%%%%%%%%%%%%%%%%%%%

\section{Charts}
\label{sec:charts}

\subsection{Bockstein $E^-_1$-page}
\cref{fig:E1-chart} on page \pageref{E1end} depicts the Bockstein $E^-_1$-page
that converges to $\Ext_{NC}$ in stems less than 31.
This data arises from 
\cref{sec:NCPeriodic} and \cref{sec:E1extensions}.

Here is a key for reading the  charts, which are separated by coweight:
\begin{enumerate}
\item
Solid gray dots indicate copies of $\displaystyle \frac{\F_2[\tau, \rho]}{\tau^\infty, \rho^\infty}$,
i.e., elements that are infinitely divisible by both $\tau$ and by $\rho$.
Beware that dividing by $\rho$ increases the stem, but this degree shift
is not displayed on the chart.  These elements are precisely the subobject
$\gamma E^-_1$ of $E^-_1$ (see \cref{eq:gamma-Q-SES}).
\item
Hollow purple dots indicate copies of $\displaystyle \frac{\F_2[\rho]}{\rho^\infty}$,
i.e., elements that are infinitely divisible by $\rho$.
Beware that dividing by $\rho$ increases the stem, but this degree shift
is not displayed on the chart.  These elements detect the quotient
$Q E^-_1$ of $E^-_1$ (see \cref{eq:gamma-Q-SES}).
\item
Vertical lines indicate $h_0$-multiplications.
\item
Lines of slope 1 indicate $h_1$-multiplications.
\item
Lines of slope $\frac{1}{3}$ indicate $h_2$-multiplications.
\item
Dashed lines indicate extensions that are hidden by
the sequence \cref{eq:gamma-Q-SES}, i.e., that connect elements
in $\gamma E^-_1$ to elements in $Q E^-_1$.
\item
Orange lines indicate extensions whose target is the $\tau$-multiple of the labelled element. 
For example, $h_2 \cdot \frac{\gamma}{\tau} h_0^2$ equals
$\frac{\gamma}{\tau} \cdot \tau h_1^3=0$, while
$h_2 \cdot \frac{\gamma}{\tau^2} h_0^2$ equals
$\frac{\gamma}{\tau} h_1^3$, which is nonzero.
\end{enumerate}

\subsection{Adams $E_2$-charts}
\label{subsctn:E2-key}

\cref{fig:Ext}
on pages \pageref{fig:Ext}--\pageref{E2end}
depicts $\Ext_{C_2}$, i.e., the  $E_2$-page of the $C_2$-equivariant Adams spectral
sequence, in coweights -2 to 8 and stems less than or equal to 30.
Each coweight appears on a separate grid.
The details of this calculation are described in 
 \cref{sctn:Bockstein-diff,sctn:hidden}.

Here is a key for reading the charts:
\begin{enumerate}
\item
Filled dots and hollow dots indicate copies of $\F_2$.
\item
Green dots indicate classes in $\Ext_\R$.
\item
Gray dots represent 
elements of $\Ext_{NC}$ that lie in $\gamma E^-_1$.
See also \cref{eq:gamma-Q-SES}.
\item
Purple hollow dots represent
elements of $\Ext_{NC}$ that are detected by the quotient $QE^-_1$.
See also \cref{eq:gamma-Q-SES}.
\item
Horizontal lines indicate $\rho$-multiplications.
\item
Vertical lines indicate $h_0$-multiplications.
\item
Diagonal lines indicate $h_1$-multiplications.
\item
Horizontal arrows indicate infinite sequences of $\rho$-multiplications,
i.e., $\rho$-periodic elements.
\item
Dashed lines indicate extensions that are hidden in the Bockstein
spectral sequence.
\item 
Dotted lines indicate potential multiplications that have neither been 
established nor ruled out. These uncertainties occur only 
in the 29-stem in coweights -2 and 6.
\item
The hollow square in degree (30,5,6) indicates an element that could possibly be the value of a Bockstein differential.
\end{enumerate}

\subsection{Adams $E_3$-charts}

\cref{fig:Ethree}  on pages \pageref{fig:Ethree}--\pageref{page:E3end}
depicts the $E_3$-page of the $C_2$-equivariant Adams spectral sequence,
in coweights -2 to 8 and stems less than or equal to 30.
Each coweight appears on a separate grid.
The details of this calculation are described in 
 \cref{sctn:Adamsd2}.

The key for the Adams $E_3$-chart is essentially the same as the key
for the Adams $E_2$-chart given in \cref{subsctn:E2-key}, with the 
following exceptions:
\begin{enumerate}
\item
Green dots indicate classes in the image of the $\R$-motivic Adams $E_3$-page, as computed in \cite{BI}.
\item
Gray dots represent classes that are not in the image of the 
$\R$-motivic Adams $E_3$-page.
\item
There is no distinction between extensions that are hidden or not hidden in the
Bockstein spectral sequence.
\end{enumerate}

\subsection{Adams $E_\infty$-charts in negative coweight and low stems}

\cref{fig:EinfNegCoweight} on pages \pageref{fig:EinfNegCoweight}--\pageref{EinfLowCoweightEnd}
depicts the $E_\infty$-page of the $C_2$-equivariant
Adams spectral sequence in stems less than 8 and in 
coweights $-9$ through $-2$.
In this range, the spectral sequence collapses at the $E_2$-page.
The details of this calculation are described in 
\cref{sctn:Bockstein-diff}, \cref{sctn:hidden}, and
\cref{sctn:AdamsPermCycles}.

Here is a key for reading the charts, which are separated by coweight:
\begin{enumerate}
\item
Dots indicate copies of $\F_2$.
\item
Horizontal lines indicate multiplications by $\rho$.
\item
Vertical lines indicate multiplications by $h_0$.
\item
Lines of slope $1$ indicate multiplications by $h_1$.
\item
Dashed lines indicate $h_0$-multiplications and
$h_1$-multiplications that are hidden in the Bockstein spectral
sequence. (There are no extensions that are hidden by the
Adams spectral sequence in this range.)
\end{enumerate}

\subsection{Adams $E_\infty$-charts}
\label{subsctn:Einfty-chart-key}

\cref{fig:Einf} on pages \pageref{fig:Einf}--\pageref{Einfend}
depicts the $E_\infty$-page of the $C_2$-equivariant Adams spectral sequence,
in coweights -1 to 7 and stems less than or equal to 30.
The details of this calculation are described in 
 \cref{sctn:Adamsd3,sctn:Adamsdhigher,}, with some of the multiplicative 
 structure determined in \cref{sec:cofibrho,sec:hiddenAdams}.

Here is a key for reading the Adams $E_\infty$-charts, which are separated by coweight:
\begin{enumerate}
\item
Dots indicate copies of $\F_2$.
\item
Green dots indicate classes in the image of the $\R$-motivic Adams $E_\infty$-page, as computed in \cite{BI}.
\item
Gray dots represent classes that are not in the image of the 
$\R$-motivic Adams $E_\infty$-page.
\item
Horizontal lines indicate multiplications by $\rho$.
\item
Vertical lines indicate multiplications by $h_0$.
\item
Lines of slope $1$ indicate multiplications by $h_1$.
\item
Horizontal arrows indicate infinite sequences of multiplications
by $\rho$.
\item
Dashed lines of negative slope indicate $\rho$-multiplications that
are hidden in the Adams spectral sequence.
\item
Dashed vertical lines show $\hsf$-multiplications that
are hidden in the Adams spectral sequence.
\item
Dashed lines of positive slope indicate $\eta$-multiplications that
are hidden in the Adams spectral sequence.
\item 
The dotted line on the element in degree $(29,4,6)$ indicates a potential $h_1$-multiplication that has neither been established nor ruled out. 
\item
The hollow square in degree (30,5,6) indicates an element that could
possibly be zero in the Adams $E_\infty$-page.
\end{enumerate}

\subsection{Charts for the cofiber of $\rho$ sequence}
\label{subsctn:chart-cofiber-rho}

\cref{fig:cofiberrho} on pages \pageref{fig:cofiberrho}--\pageref{cofibrhoEnd}
depicts the short exact sequence \eqref{eq:SEScofibrho}, in coweights 0 to 7 
and in stems less than or equal to 26.
This calculation is described in \cref{sec:cofibrho}.

Here is a key for reading these charts, which are separated by coweight:
\begin{enumerate}
\item
Ignoring the color and labels, each chart is the $E_\infty$-page of the 
classical Adams spectral sequence for the sphere, in stems up to 26.
In other words, it is an associated graded object for the central
object $\picl_*$ of the short exact sequence.
\item
Dots indicate copies of $\F_2$.
\item
Blue dots indicate classes in the image of $(\coker \rho)_{*,*} \to \picl_*$.
The labels indicate the pre-image in $(\coker \rho)_{*,*}$.
\item
Orange dots represent classes that are not in the image of $(\coker \rho)_{*,*} \to \picl_*$.
The labels indicate their images in $(\ker \rho)_{*,*}$.
\item
Vertical lines indicate $h_0$-multiplications.
\item
Lines of slope 1 indicate $h_1$-multiplications.
\item
Dashed vertical lines indicate $2$-extensions that
are hidden in the Adams spectral sequence.
\item
Dashed lines of positive slope indicate $\eta$-extensions that
are hidden in the Adams spectral sequence.
\end{enumerate}

See \cref{rmk:cofiber-rho-ambigious} for a minor amibiguity in the notation
in these charts.

\newpage

\vfill

\clearpage

\begin{landscape}

\begin{figure}[h]
\caption{The $E_1^-$-page of the $\rho$-Bockstein spectral sequence}
\label{fig:E1-chart}

\includegraphics[height=.92\textwidth]{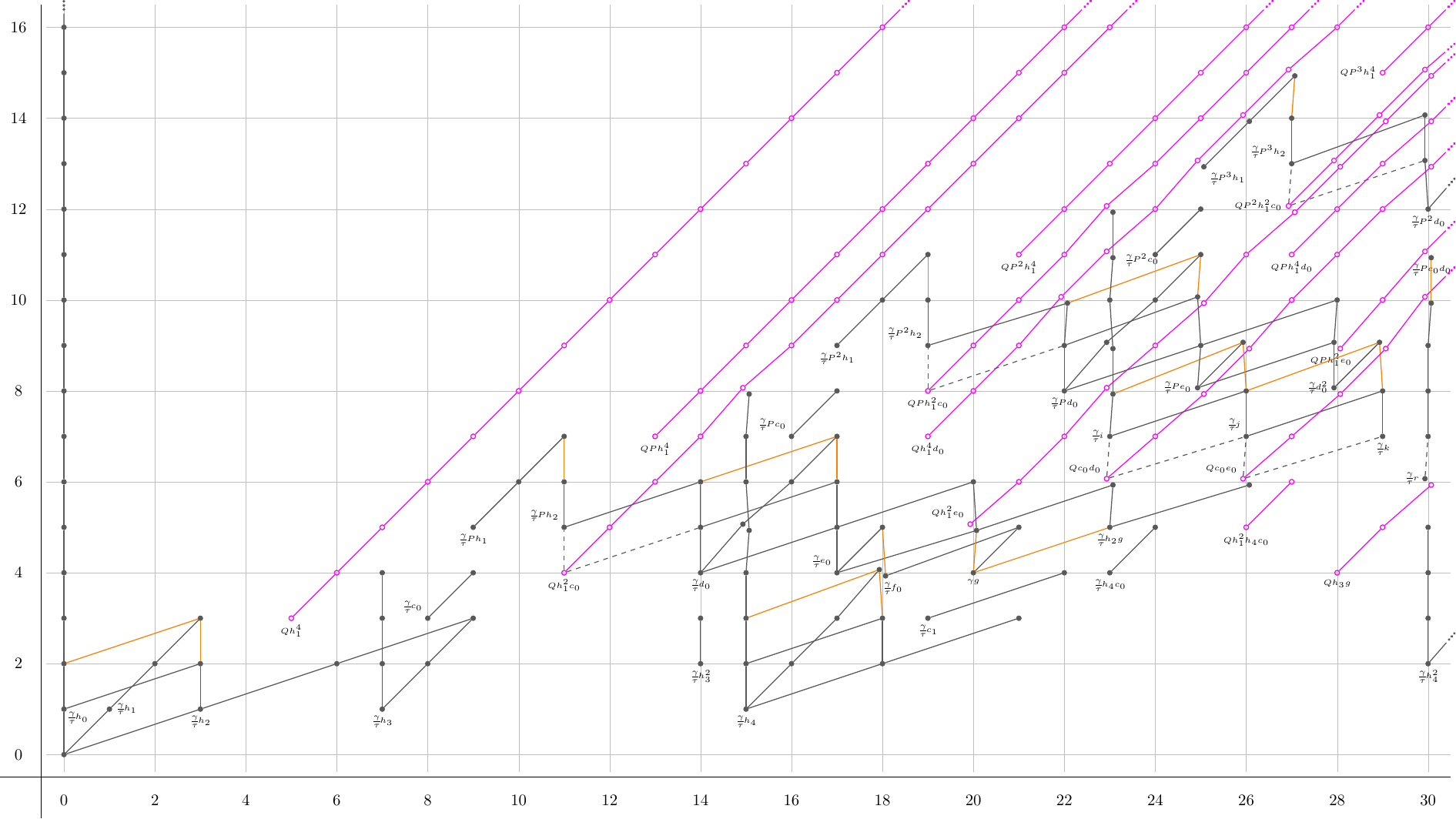}

\end{figure}

\label{E1end}

\label{E2start}

\begin{figure}[h]
\caption{The $E_2$ page of the $C_2$-equivariant Adams spectral sequence}
\label{fig:Ext}
\includegraphics[height=.93\textwidth]{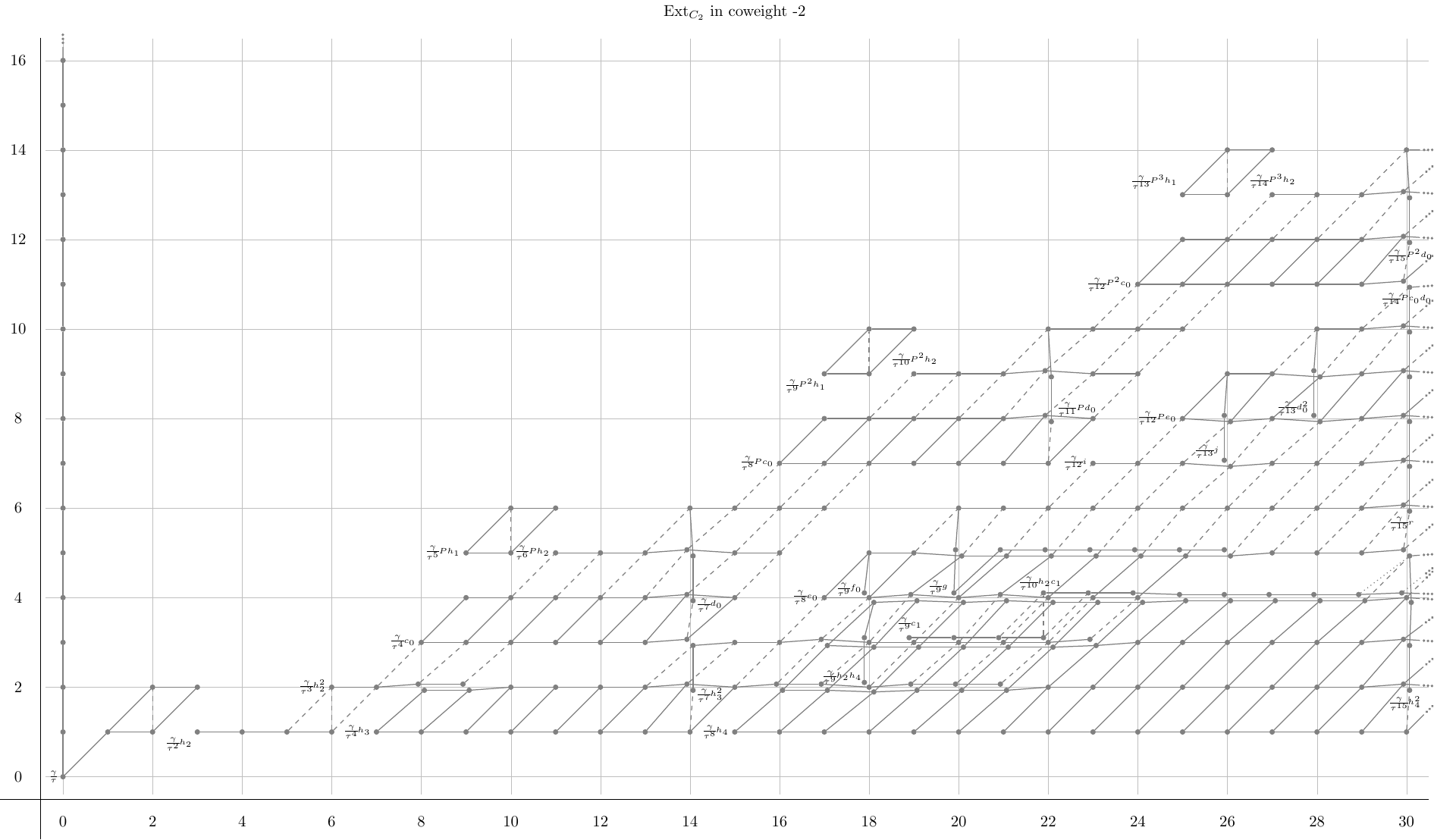}
\end{figure}

\clearpage

\foreach \c in {-1,...,8} {
	\includegraphics[height=.93\textwidth]{AdamsCharts/Ext/C2ExtCharts-coweight\c.pdf}
	
}

\label{E2end}

\clearpage

\label{E3start}

\begin{figure}[h]
\caption{The $E_3$ page of the $C_2$-equivariant Adams spectral sequence}
\label{fig:Ethree}
\includegraphics[height=.93\textwidth]{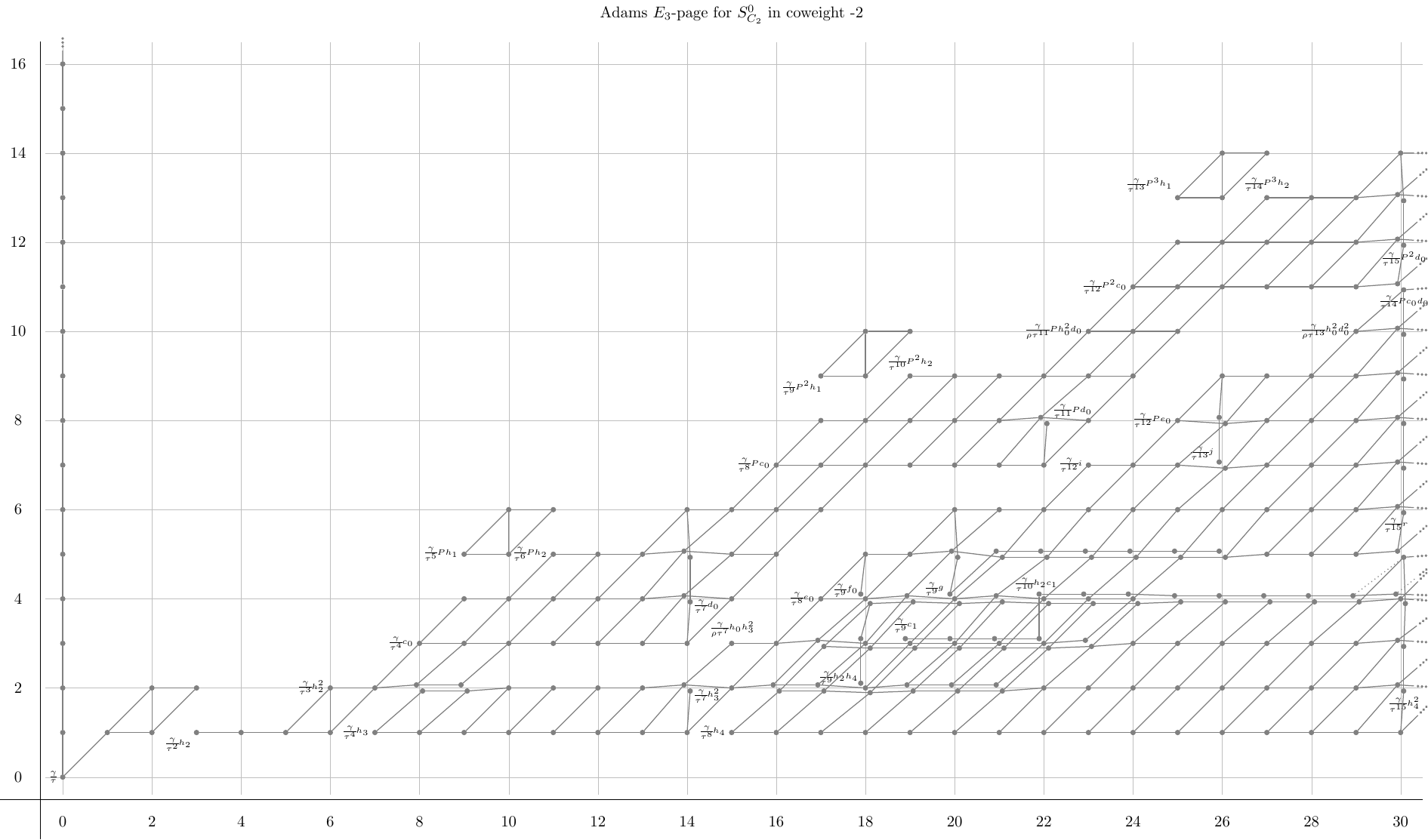}
\end{figure}

\clearpage

\includegraphics[height=.93\textwidth]{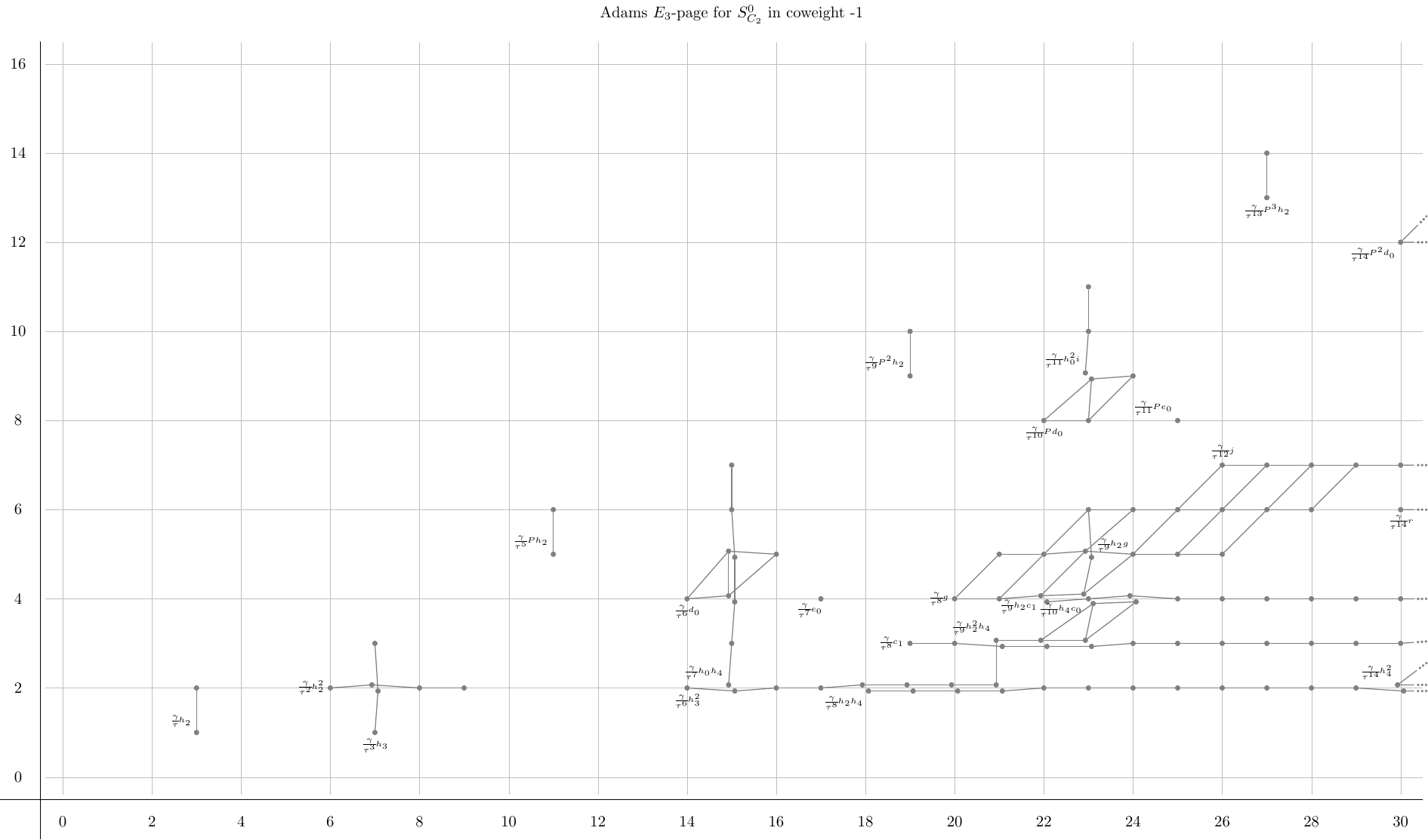}

\includegraphics[height=.93\textwidth]{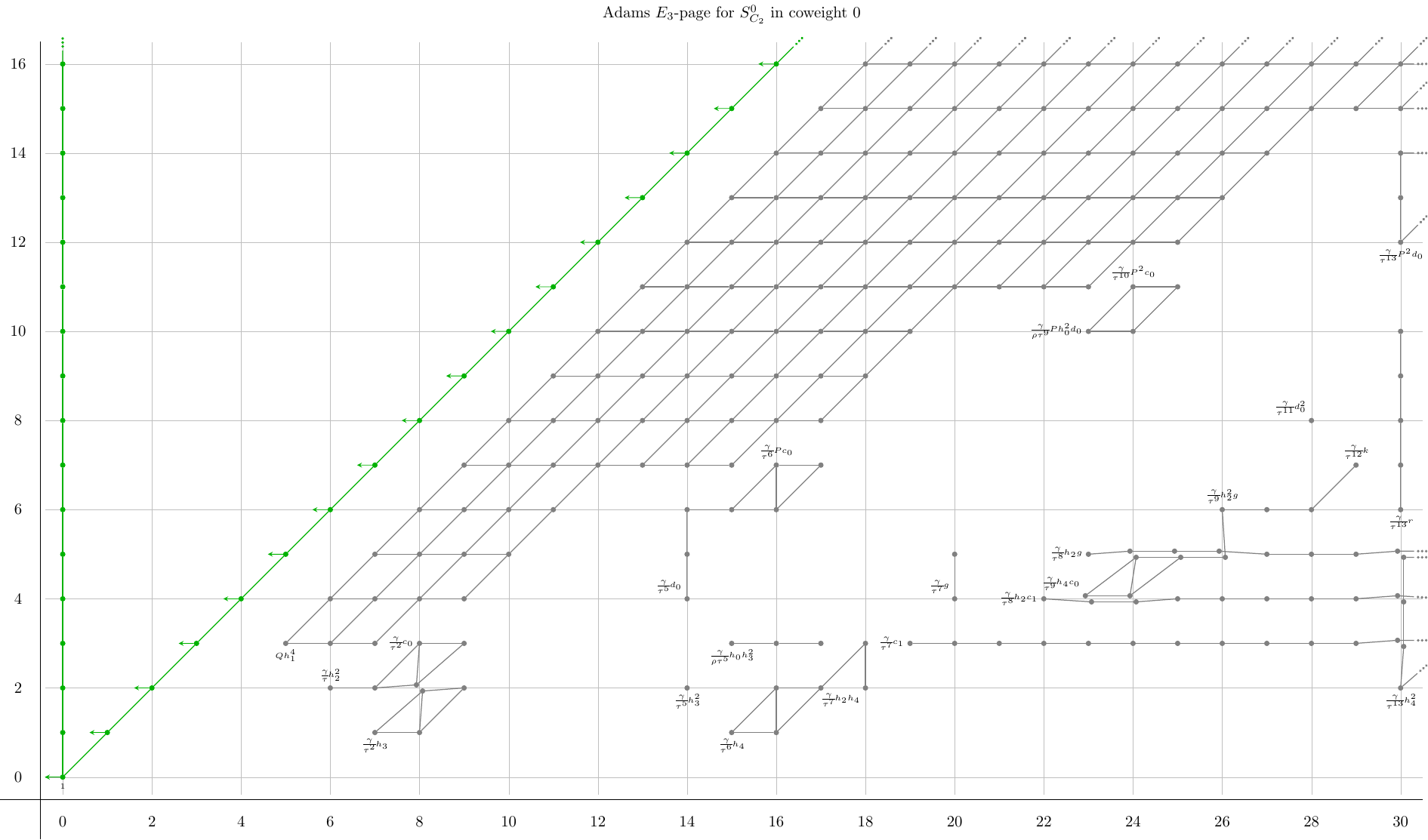}

\includegraphics[height=.93\textwidth]{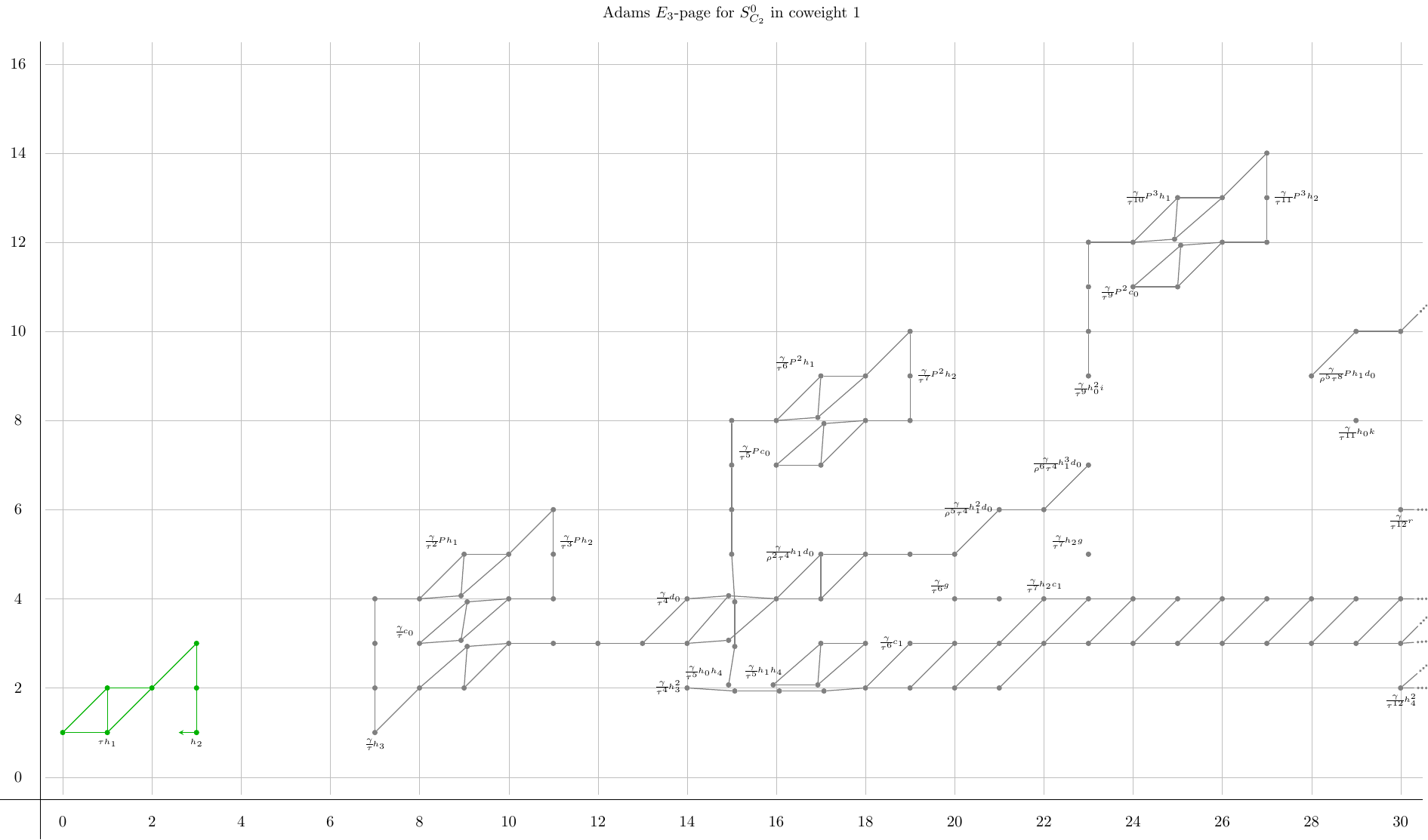}

\includegraphics[height=.93\textwidth]{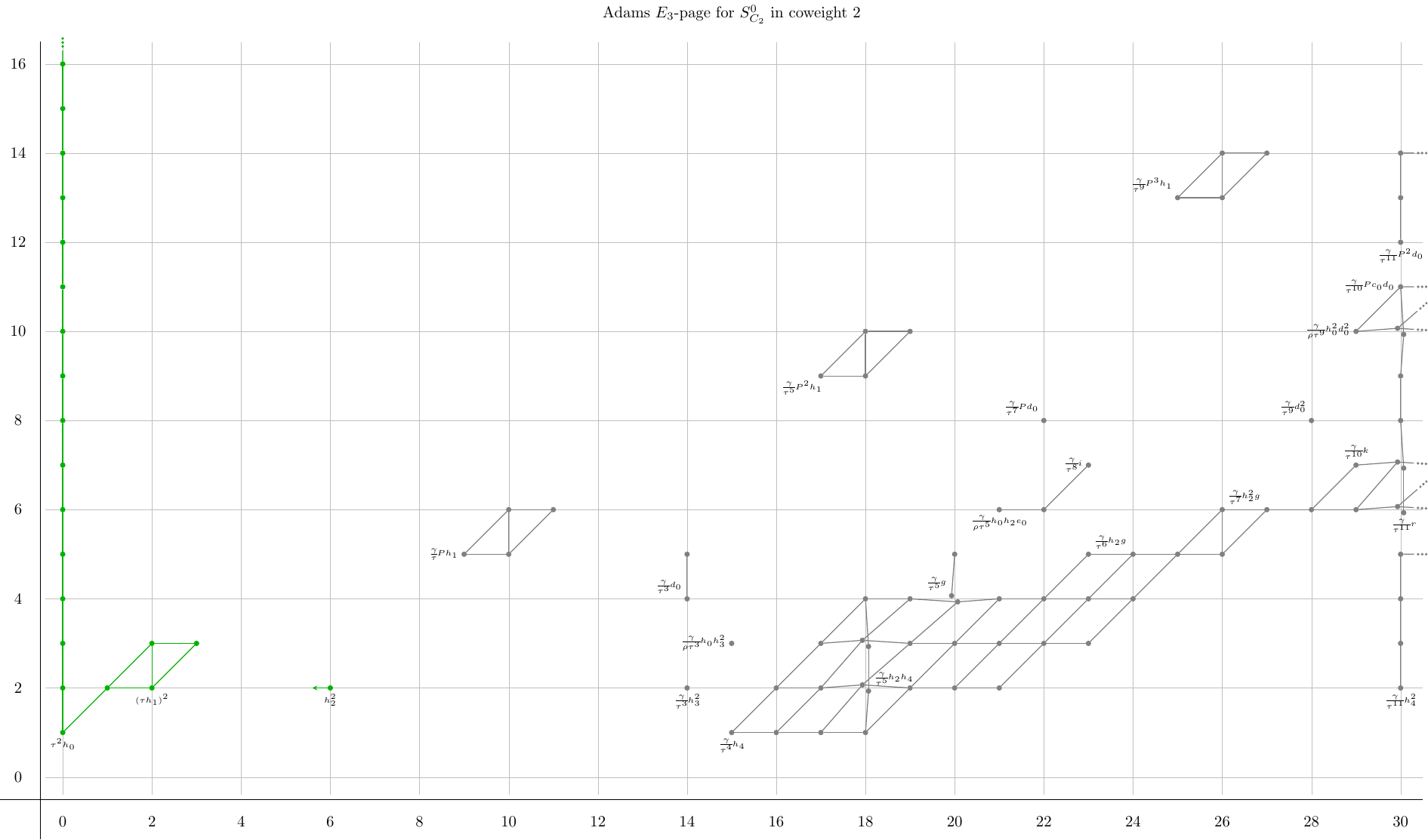}

\includegraphics[height=.93\textwidth]{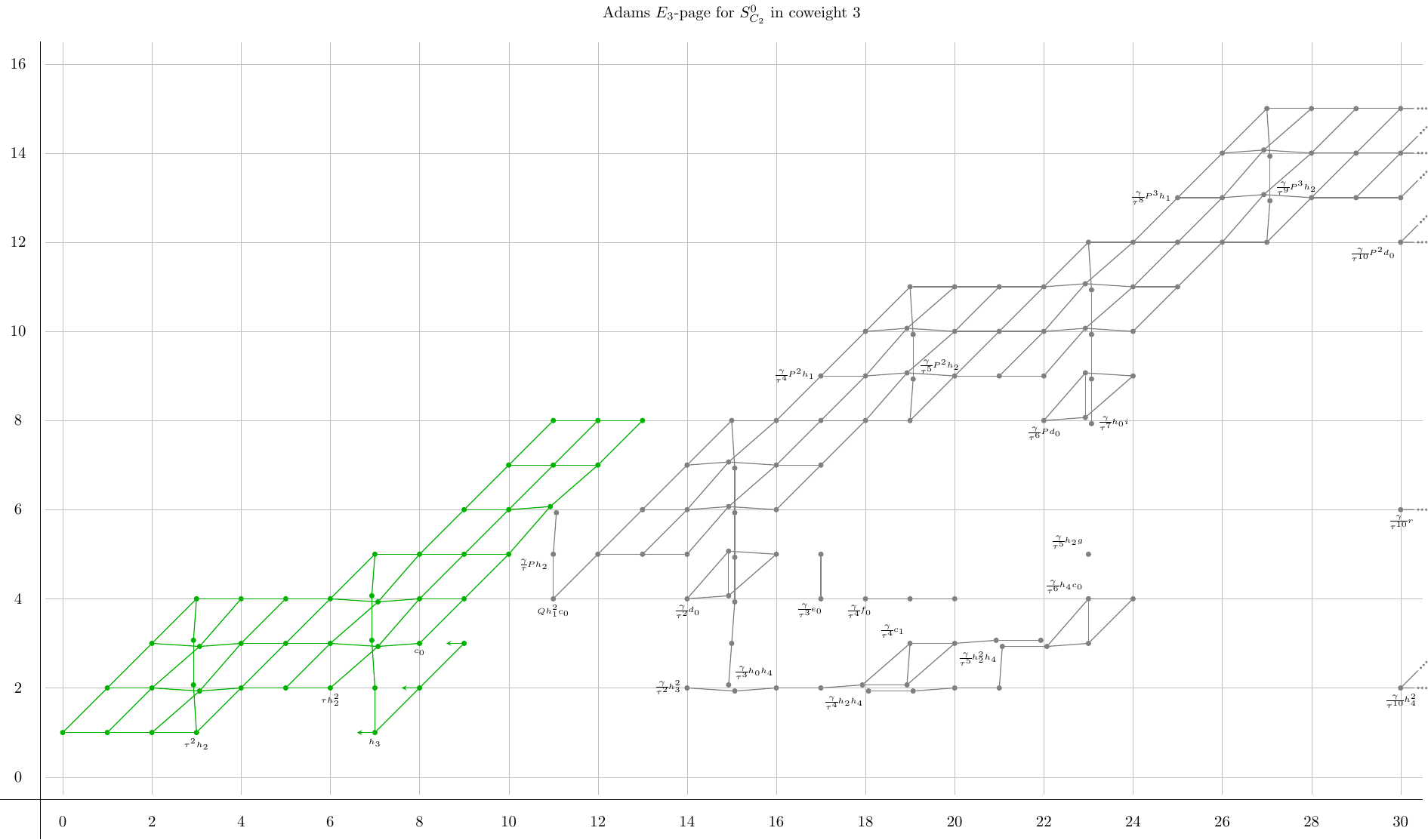}

\includegraphics[height=.93\textwidth]{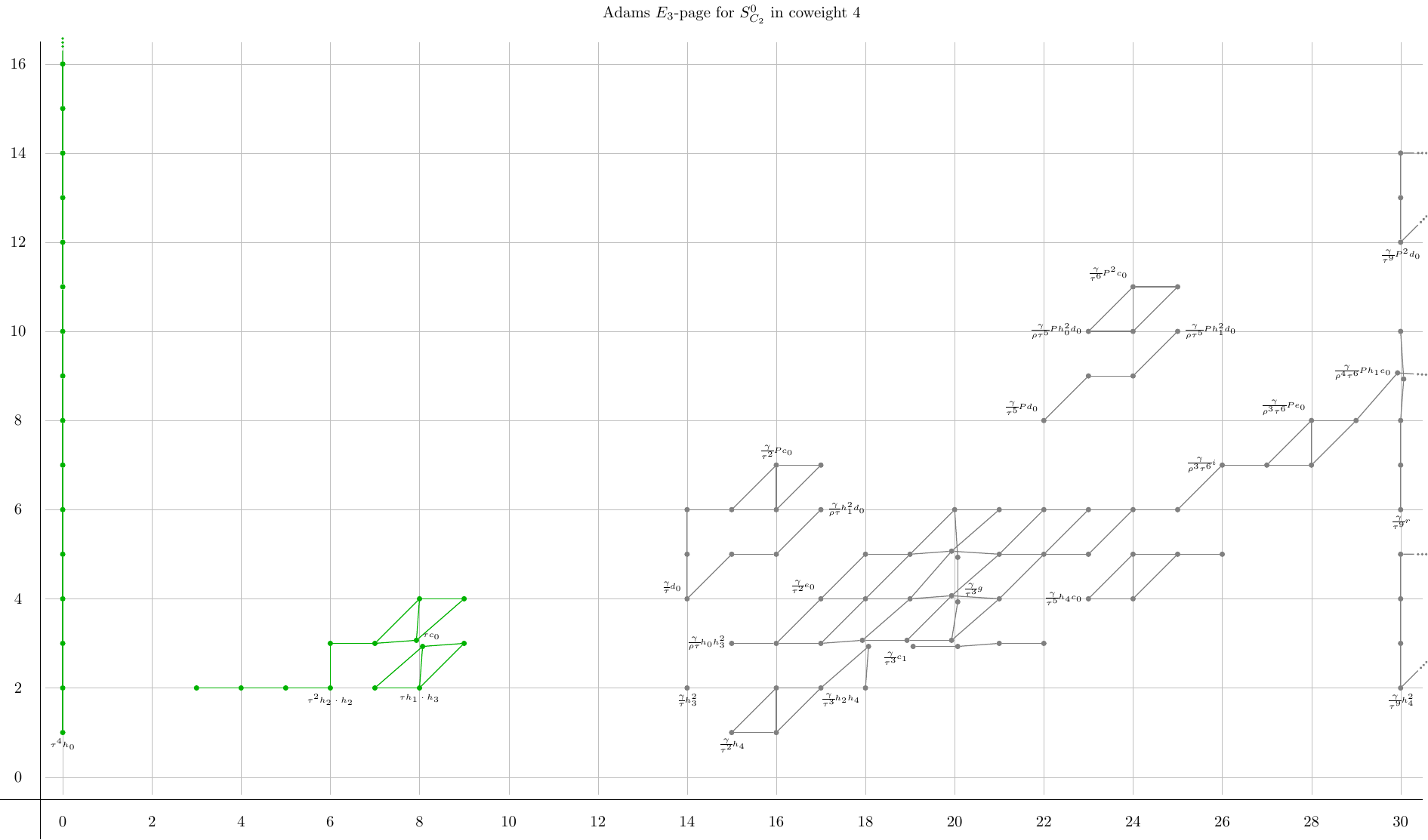}

\includegraphics[height=.93\textwidth]{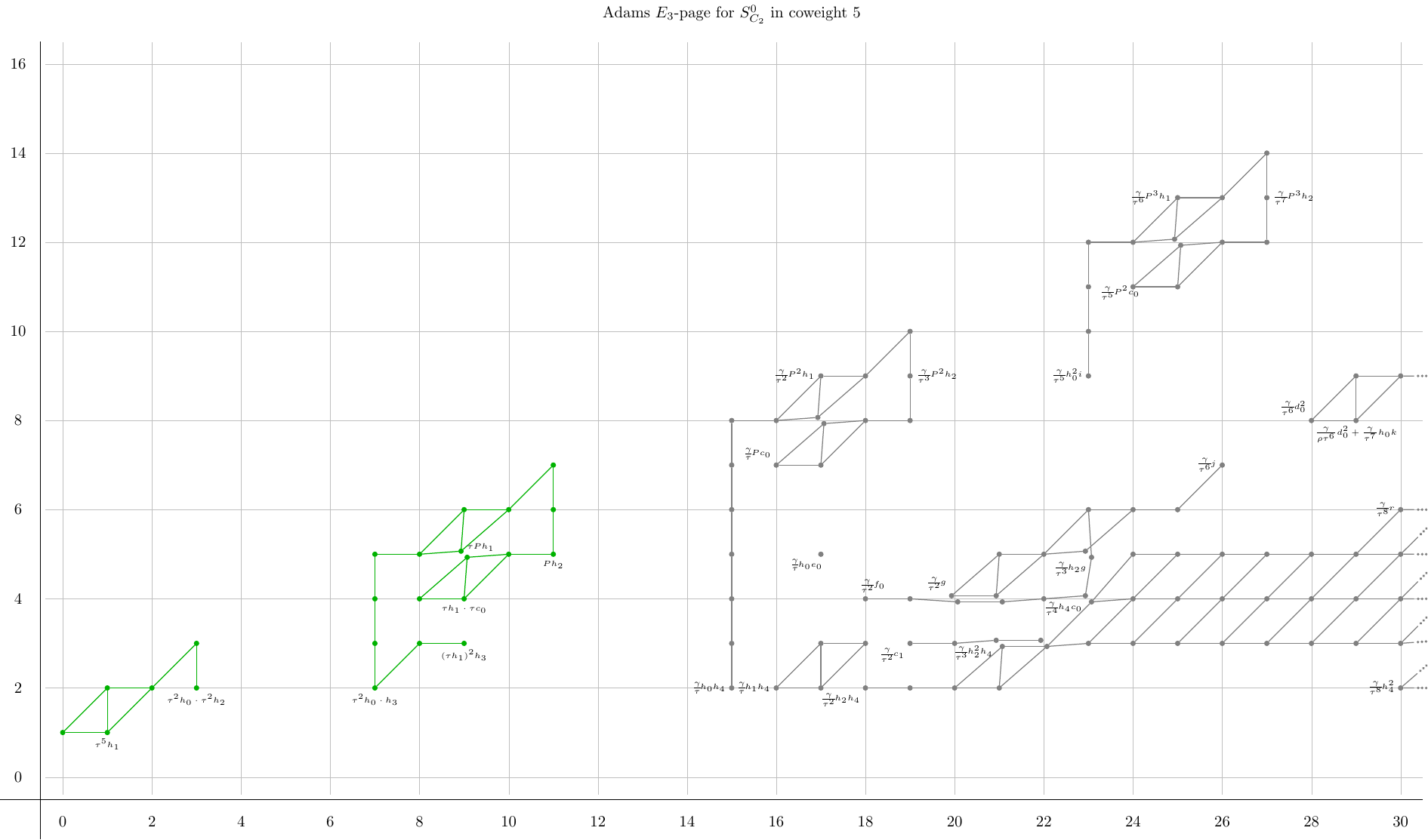}

\includegraphics[height=.93\textwidth]{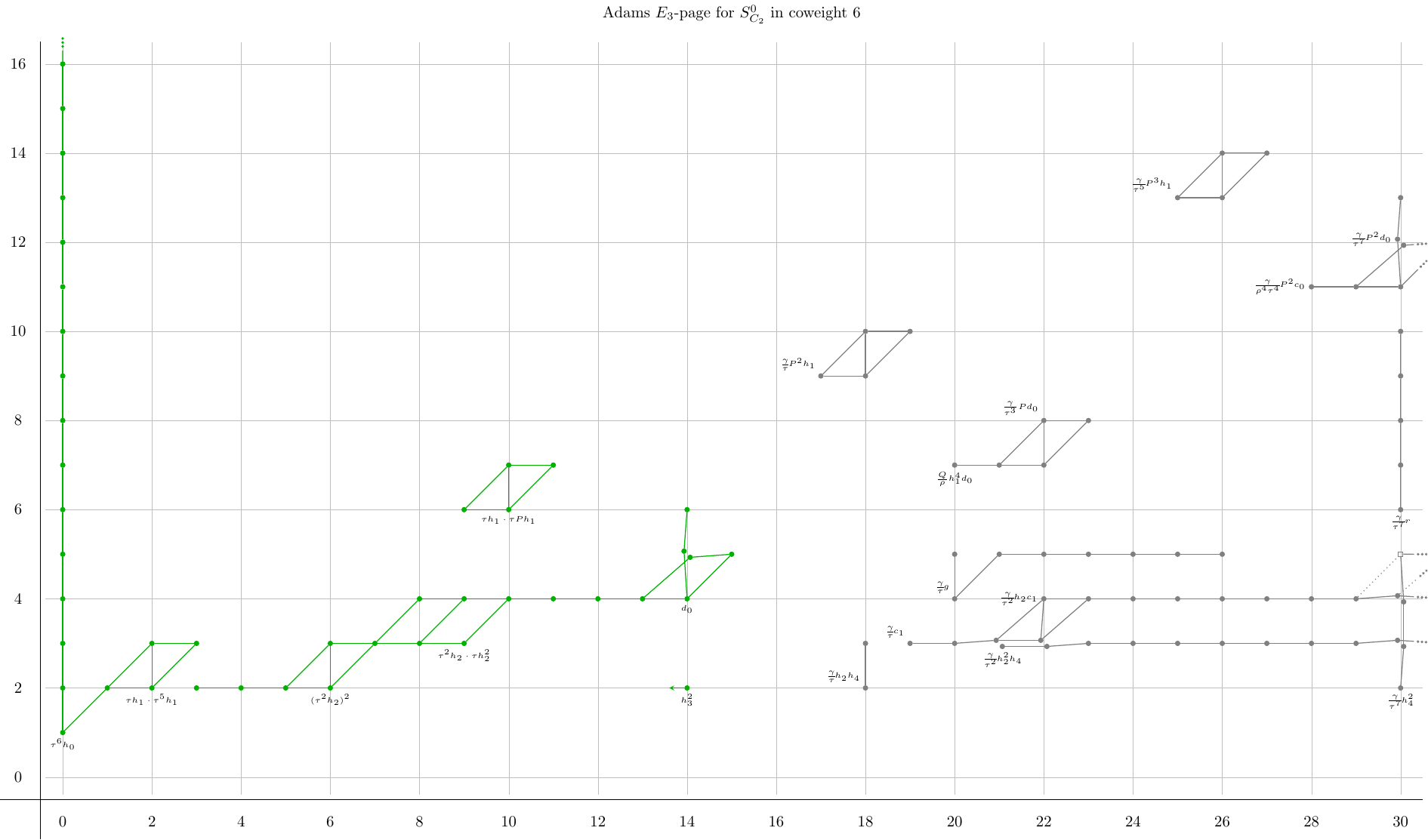}

\includegraphics[height=.93\textwidth]{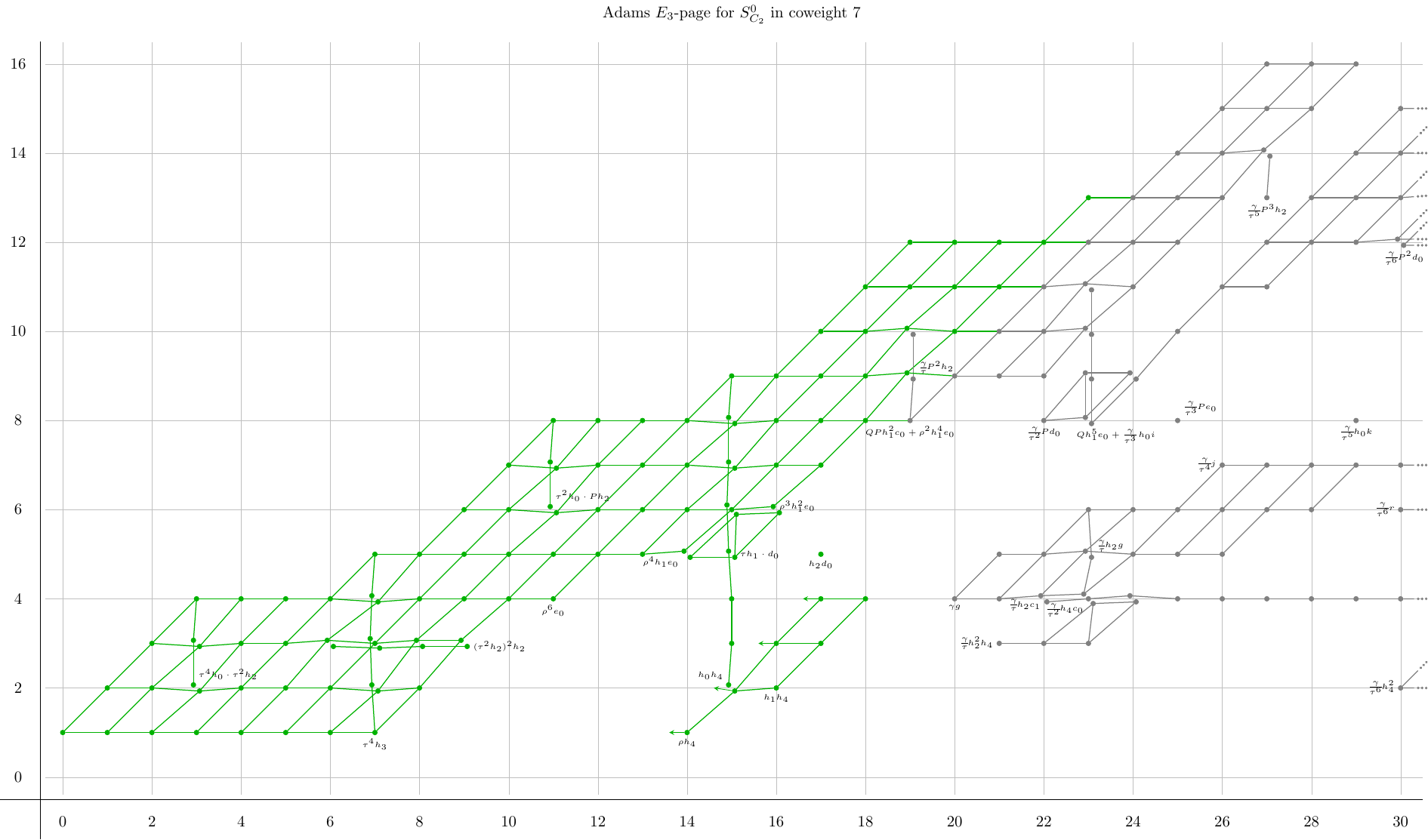}

\includegraphics[height=.93\textwidth]{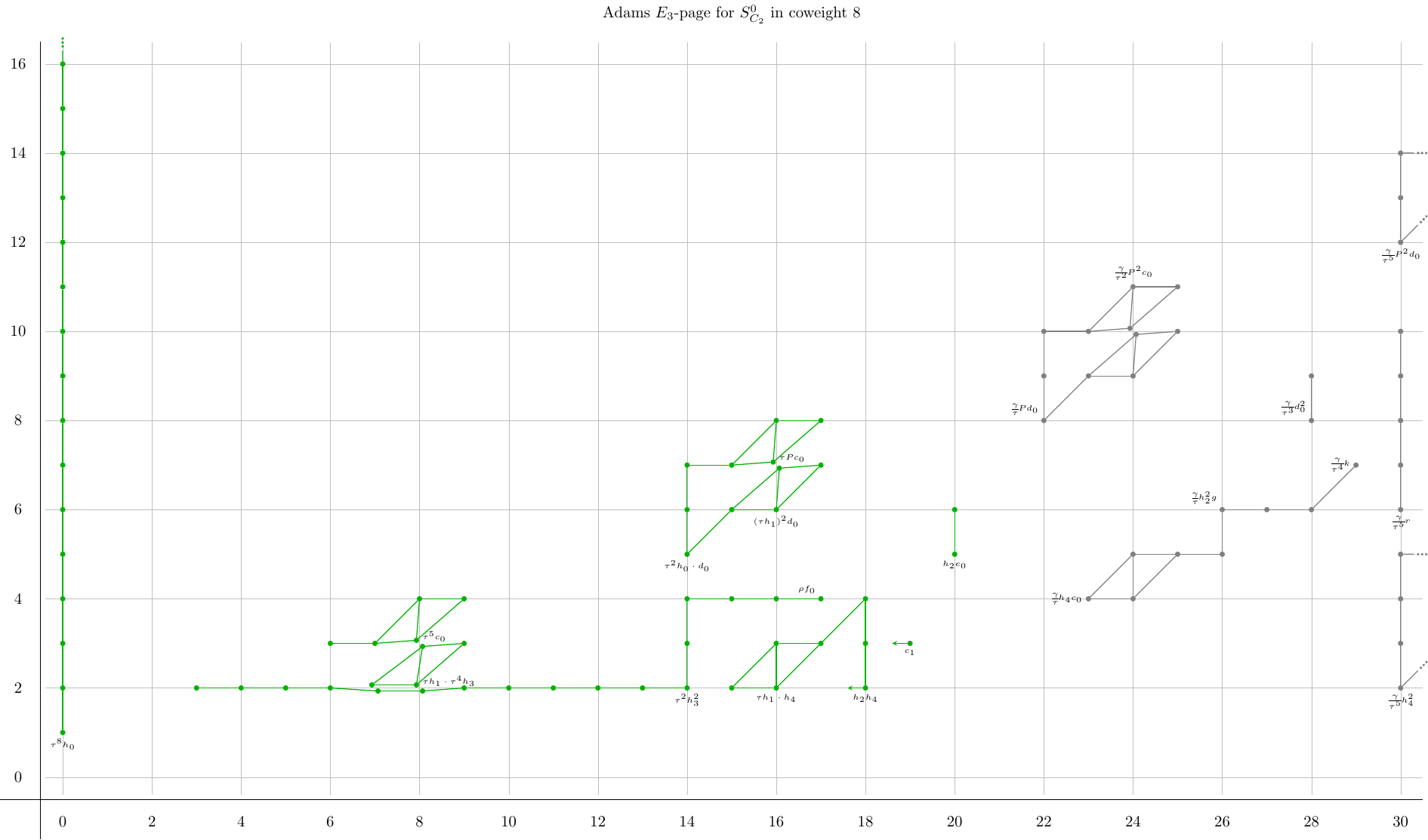}

\label{page:E3end}

\clearpage

\begin{figure}[h]
\caption{The $E_\infty$-page of the $C_2$-equivariant Adams spectral sequence for $S^{0,0}$ in negative coweight and low stems}
\label{fig:EinfNegCoweight}
\end{figure}

\centerline{
\includegraphics[width=.55\textheight]{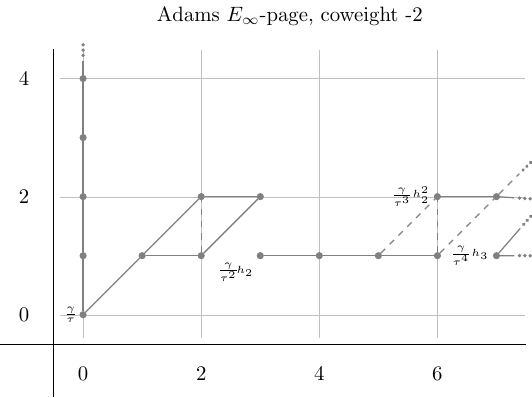}
\qquad\qquad
\includegraphics[width=.55\textheight]{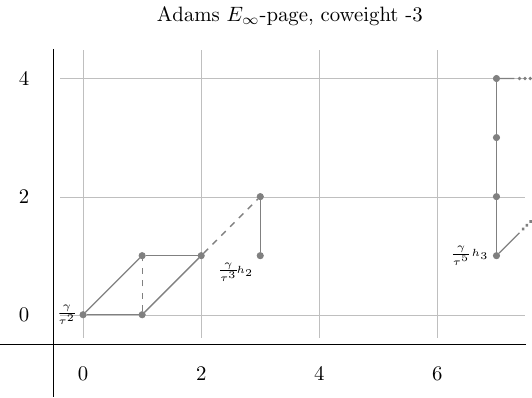}
}

\centerline{
\includegraphics[width=.55\textheight]{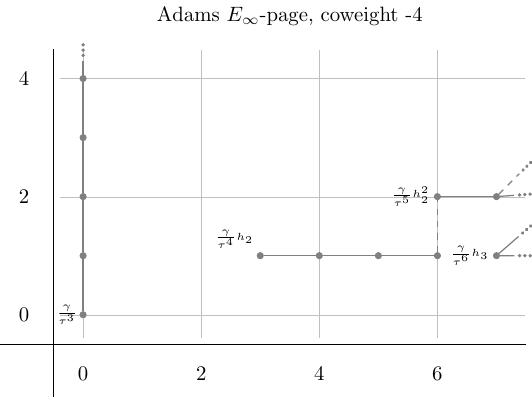}
\qquad\qquad
\includegraphics[width=.55\textheight]{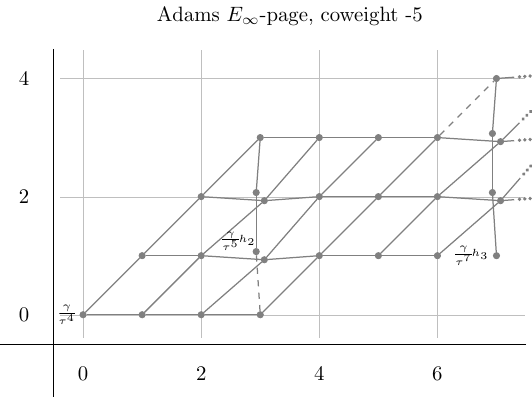}
}

\centerline{
\includegraphics[width=.55\textheight]{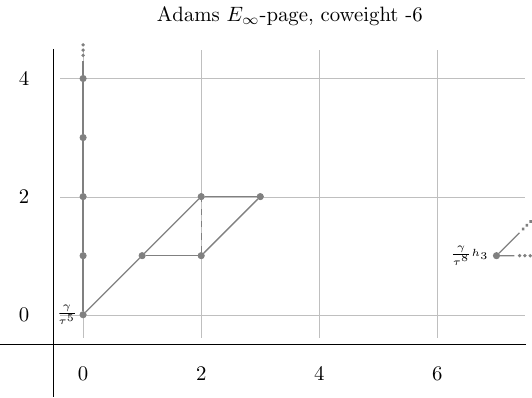}
\qquad\qquad
\includegraphics[width=.55\textheight]{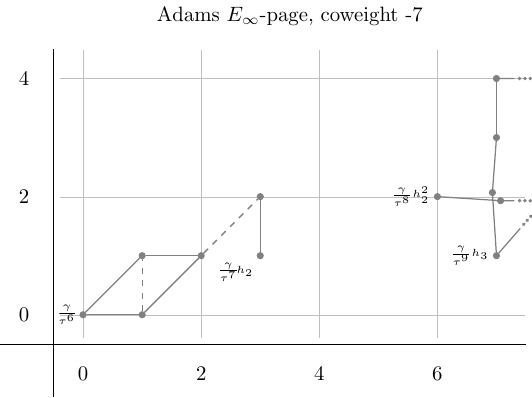}
}

\centerline{
\includegraphics[width=.55\textheight]{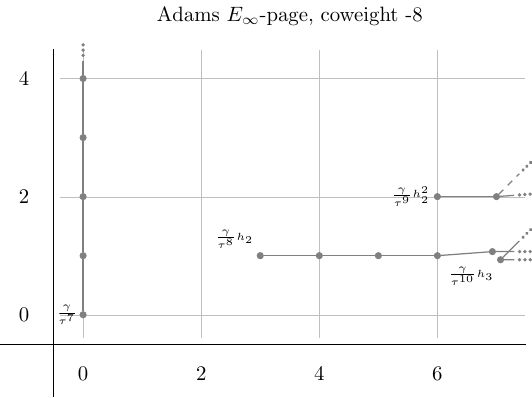}
\qquad\qquad
\includegraphics[width=.55\textheight]{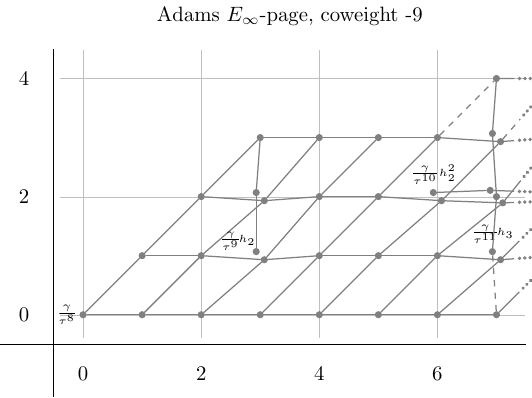}
}

\label{EinfLowCoweightEnd}

\clearpage

\begin{figure}[h]
\caption{The $E_\infty$ page of the $C_2$-equivariant Adams spectral sequence}
\label{fig:Einf}
\includegraphics[height=.93\textwidth]{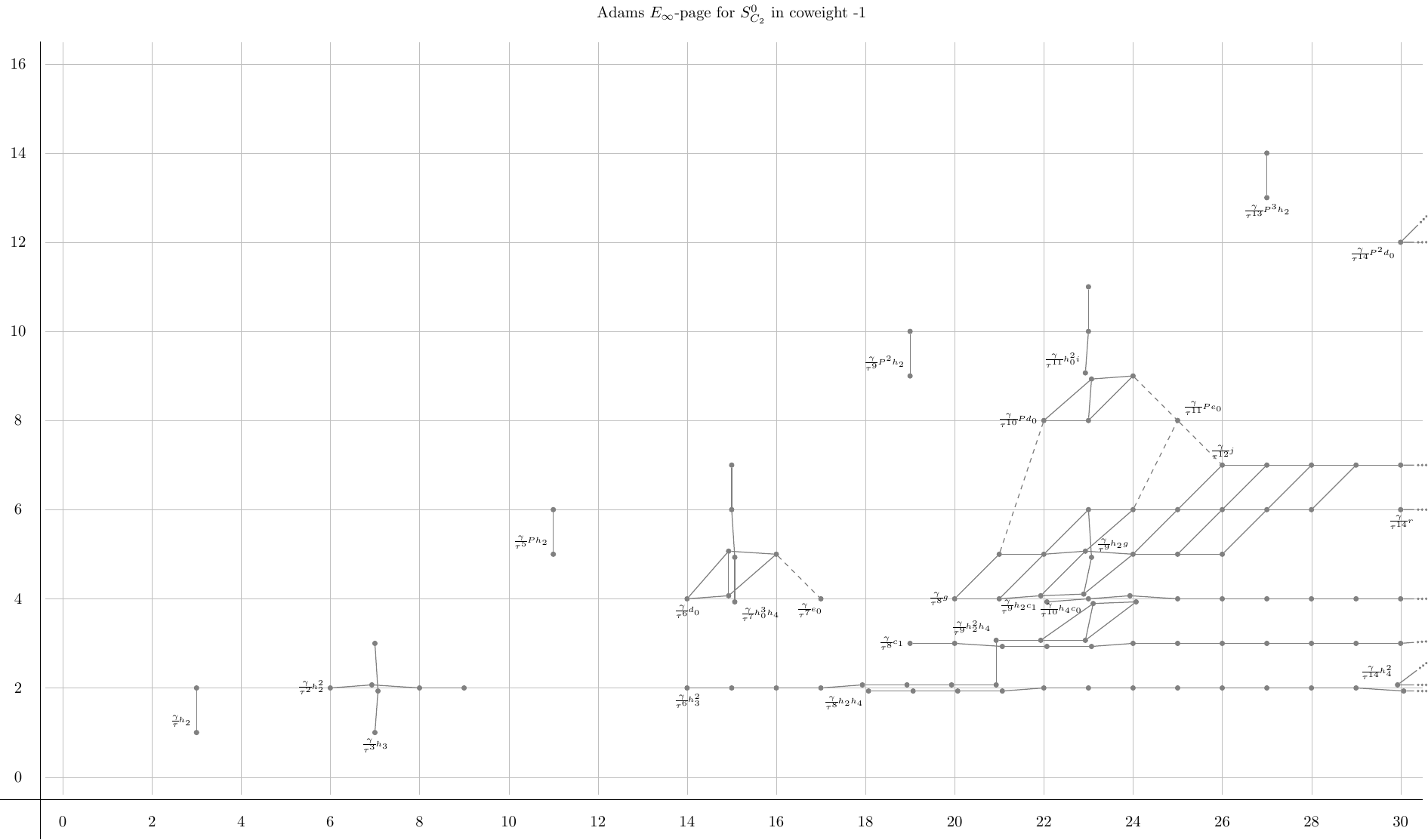}
\end{figure}

\label{Einfstart}

\clearpage

\includegraphics[height=.93\textwidth]{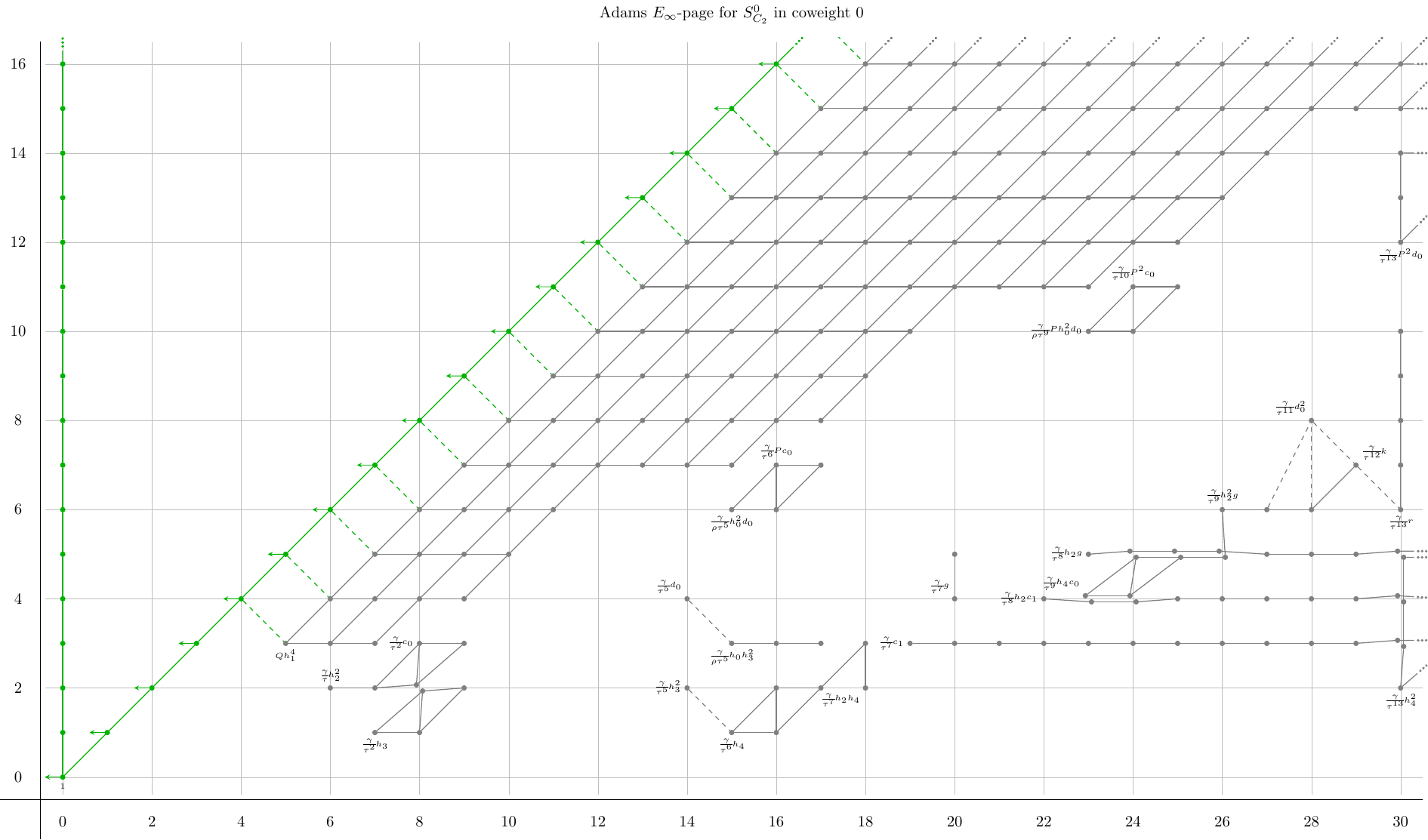}

\includegraphics[height=.93\textwidth]{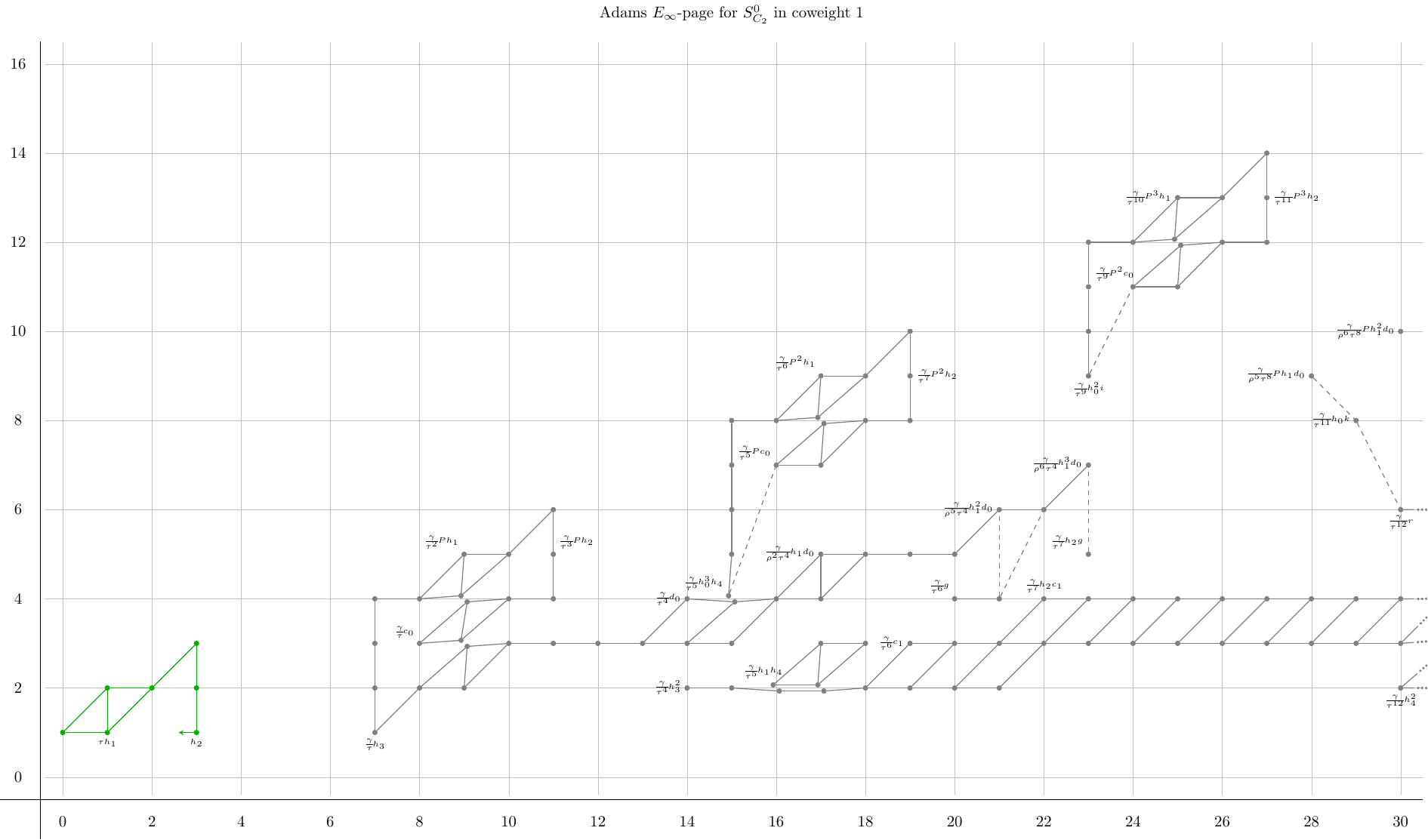}

\includegraphics[height=.93\textwidth]{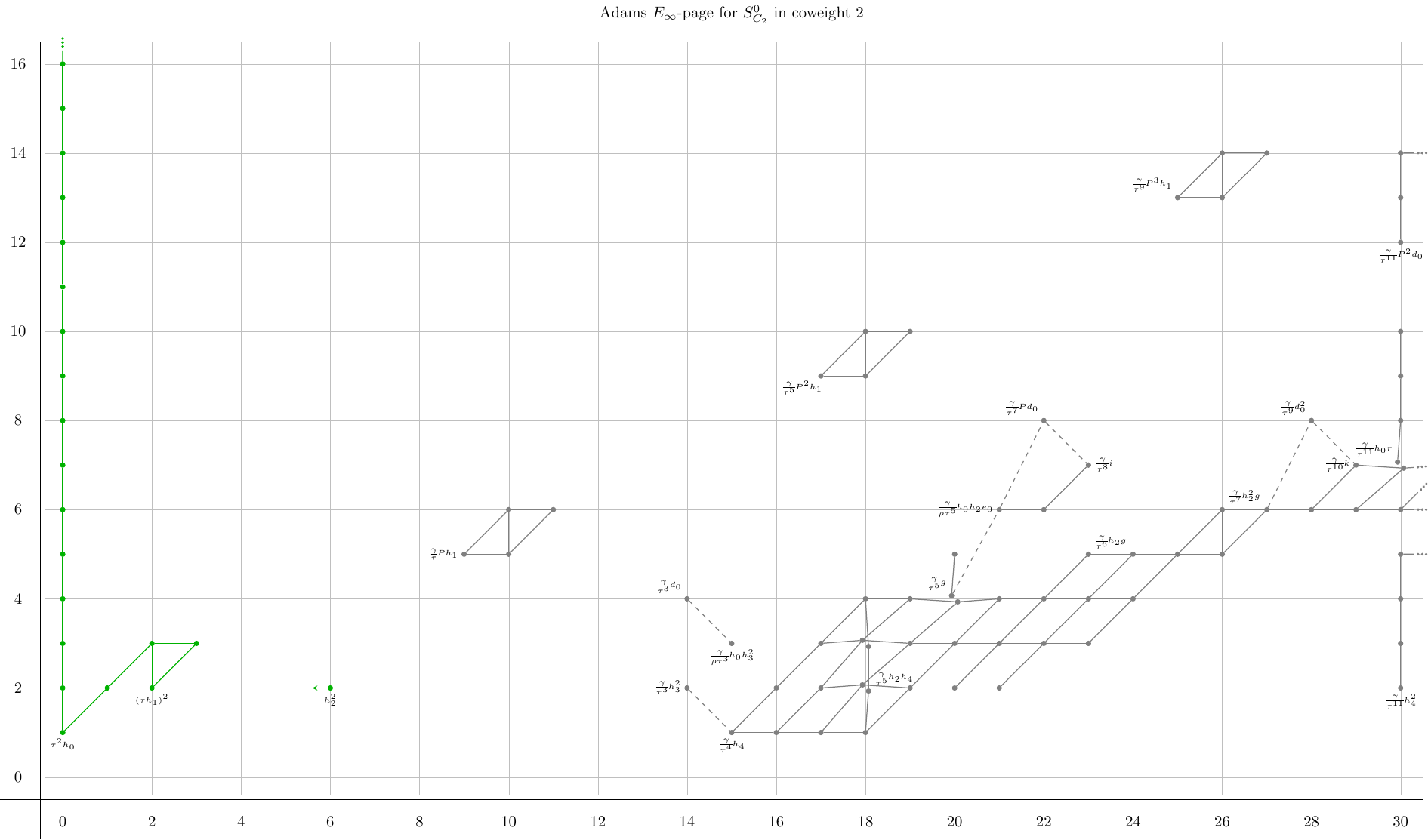}

\includegraphics[height=.93\textwidth]{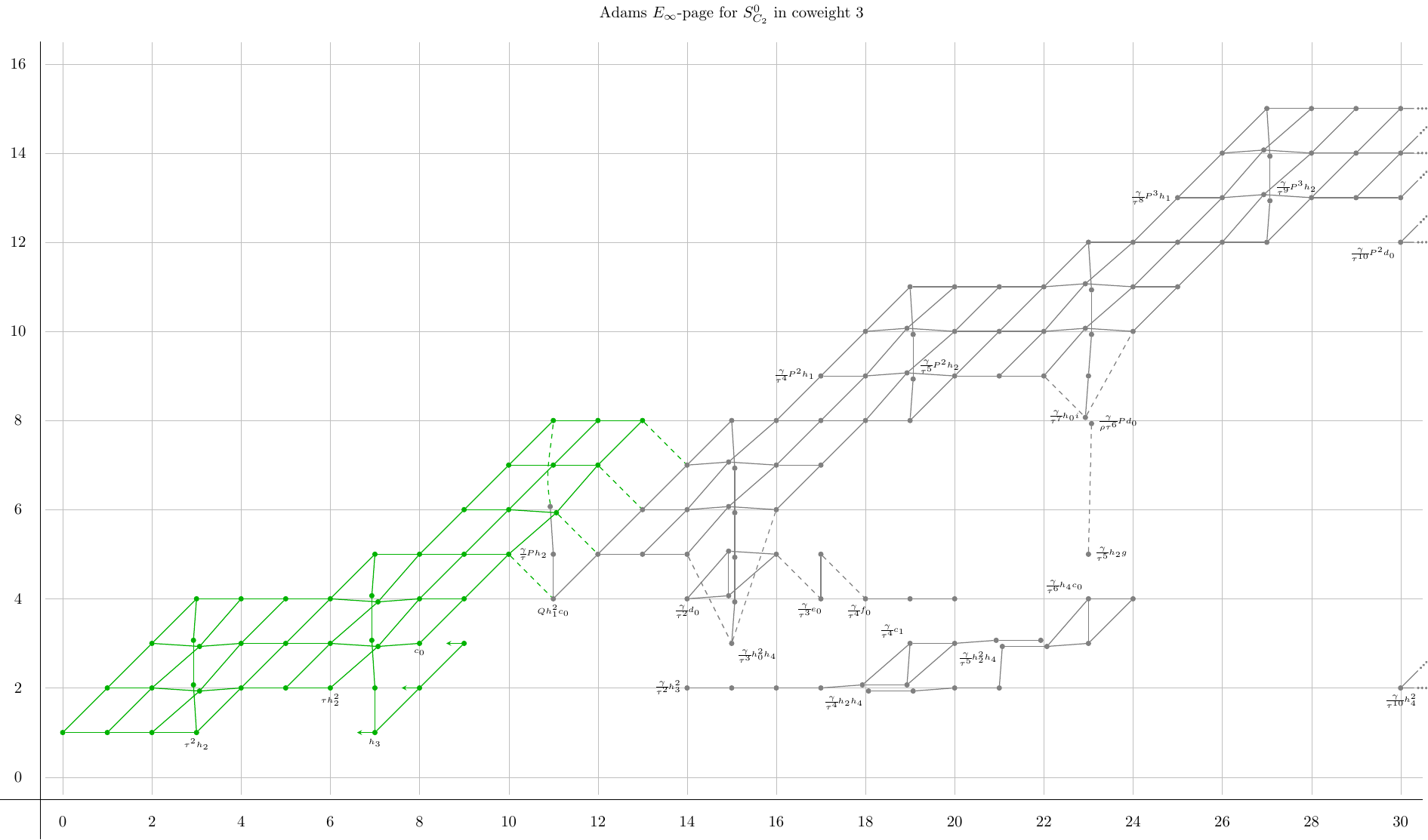}

\includegraphics[height=.93\textwidth]{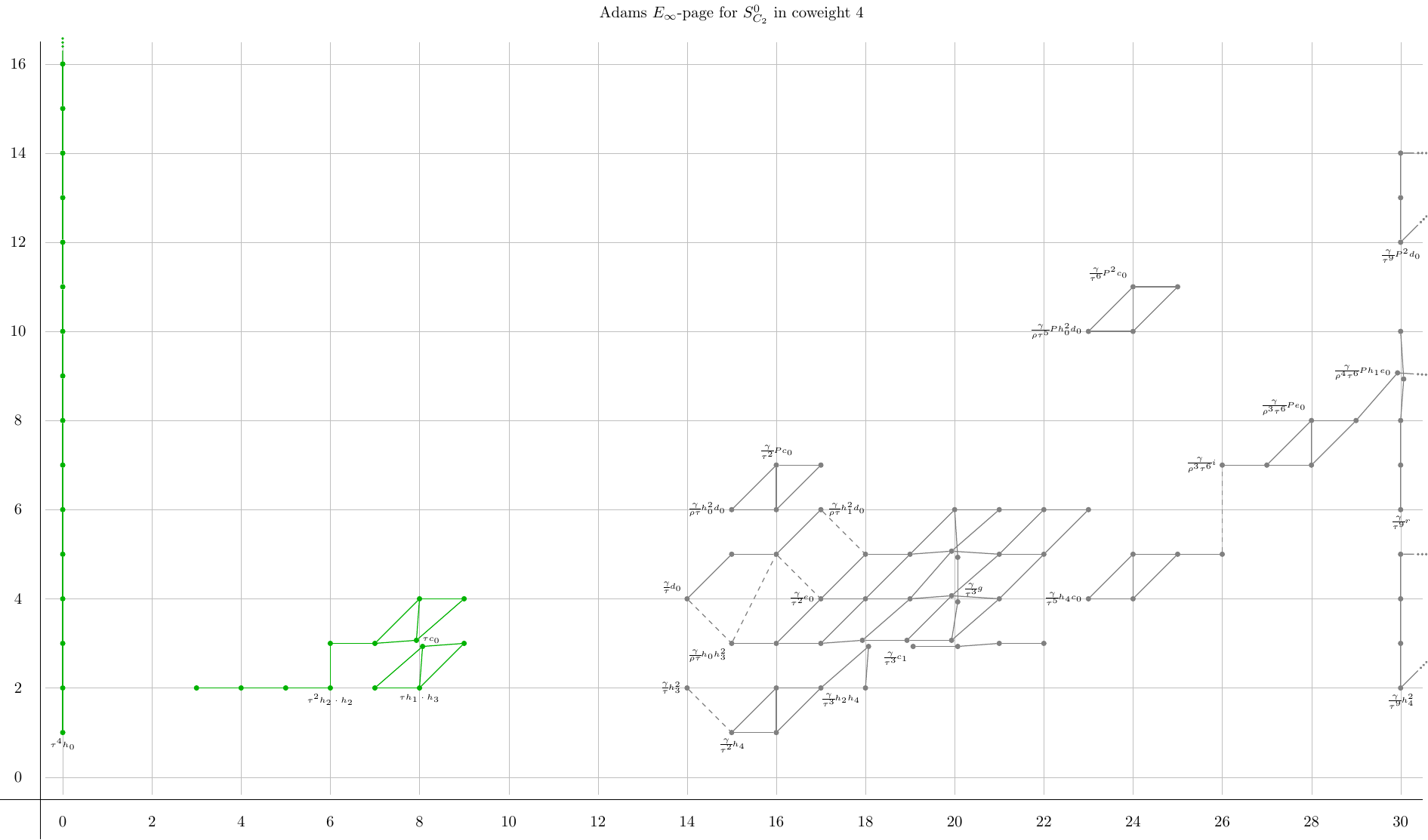}

\includegraphics[height=.93\textwidth]{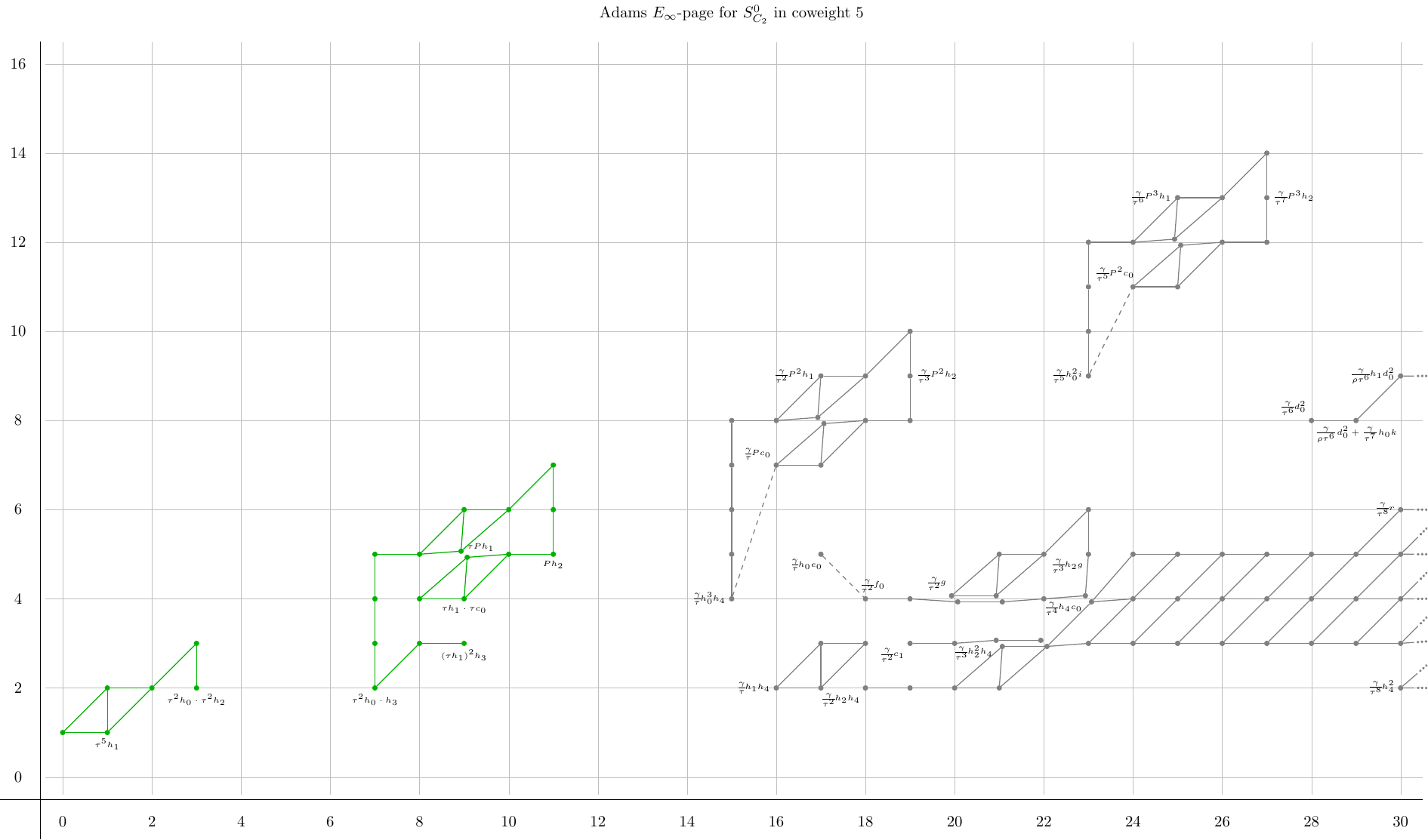}

\includegraphics[height=.93\textwidth]{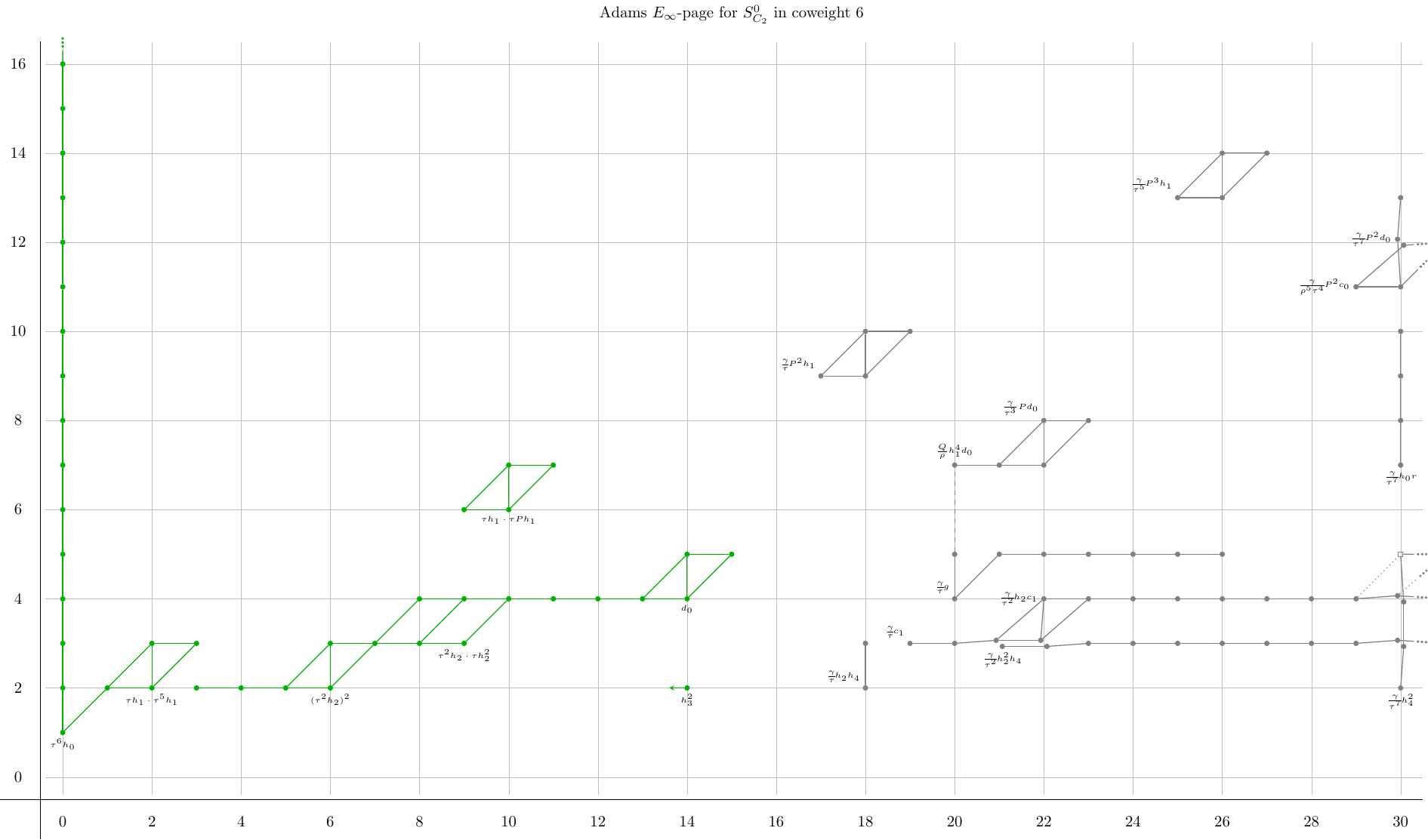}

\includegraphics[height=.93\textwidth]{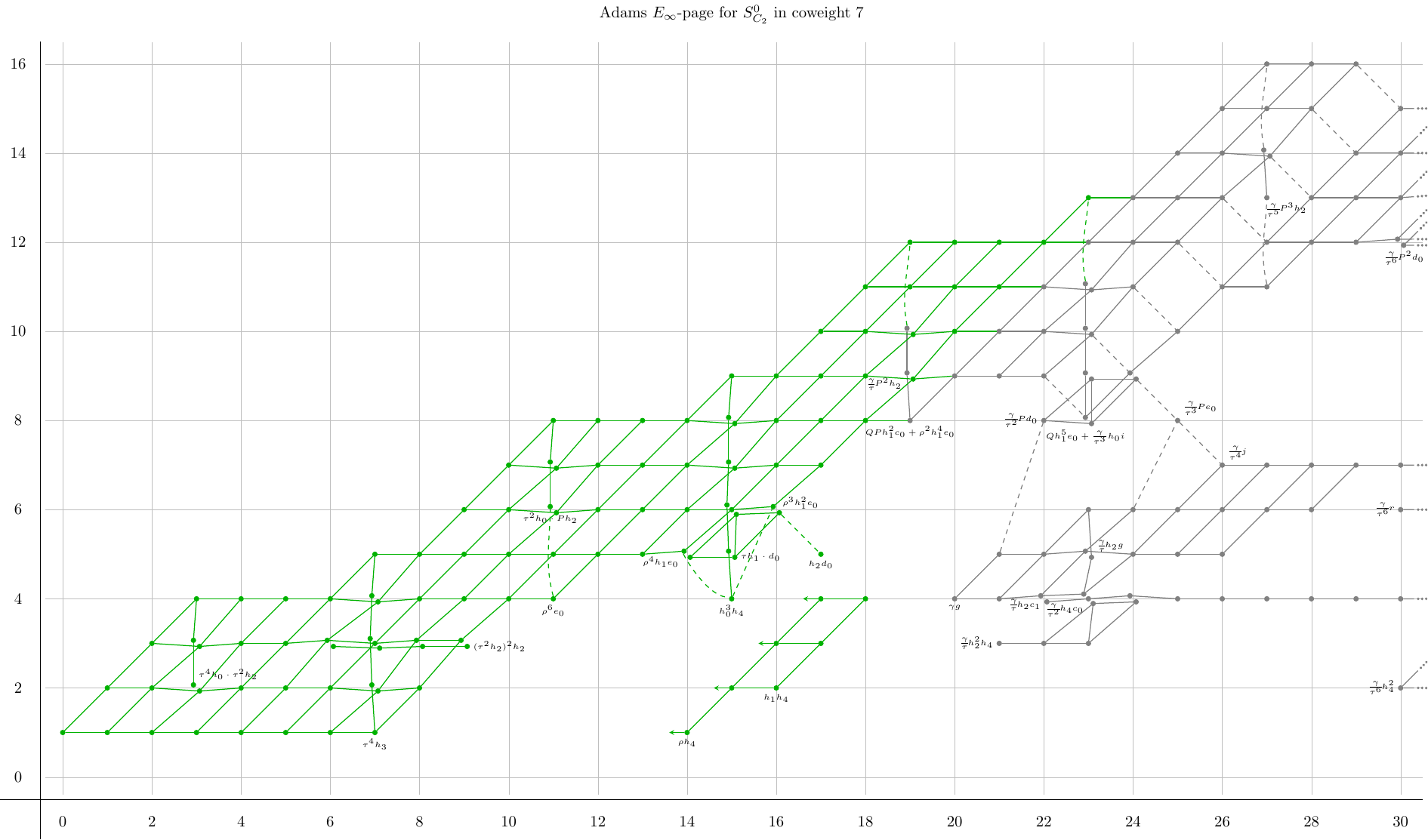}

\label{Einfend}

\begin{figure}[h]
\caption{The exact sequence $\coker \rho \into \picl_* \onto \ker \rho$}
\label{fig:cofiberrho}
\includegraphics[height=\textwidth]{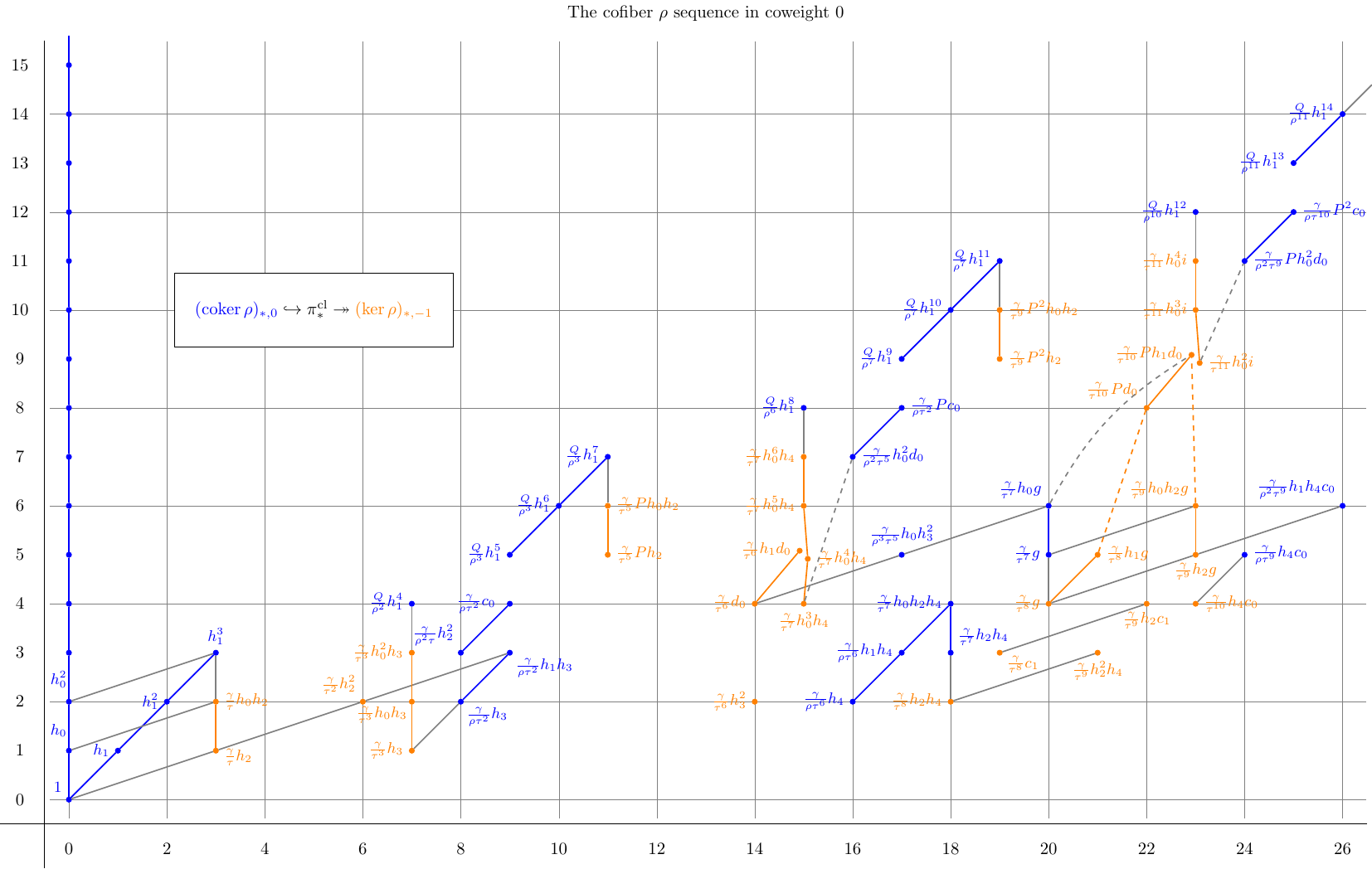}
\end{figure}

\clearpage

\includegraphics[height=\textwidth]{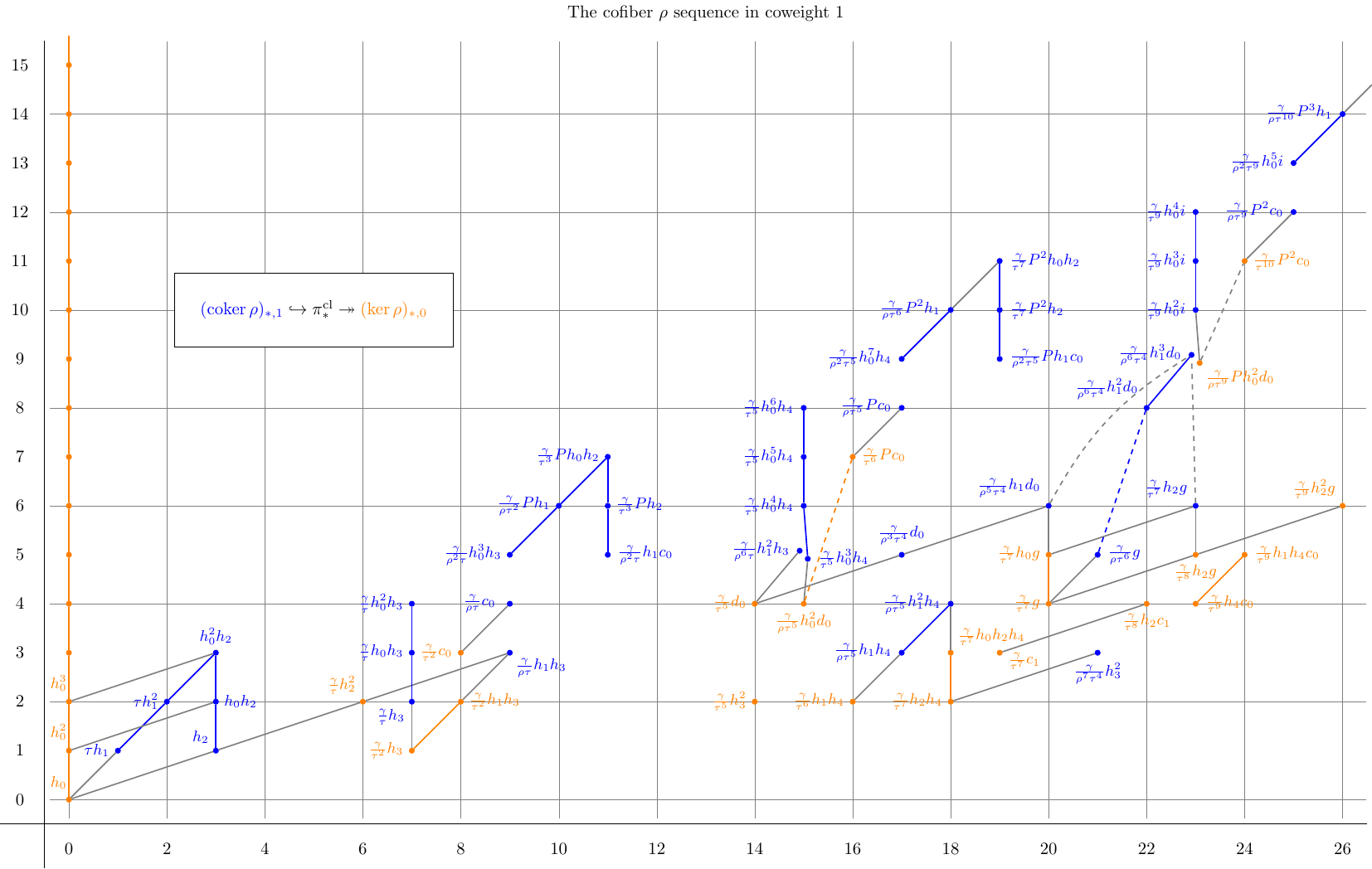}

\includegraphics[height=\textwidth]{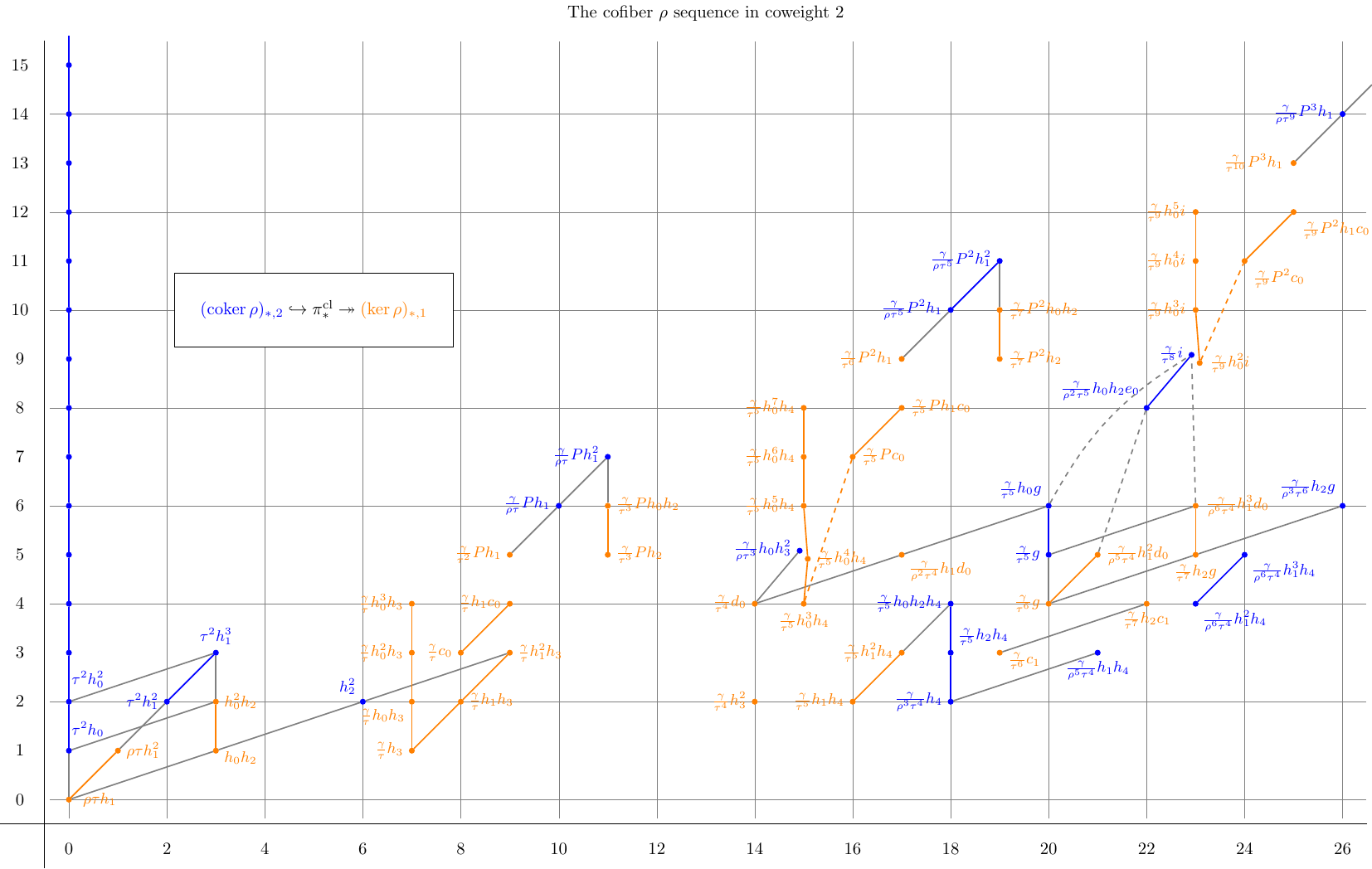}

\includegraphics[height=\textwidth]{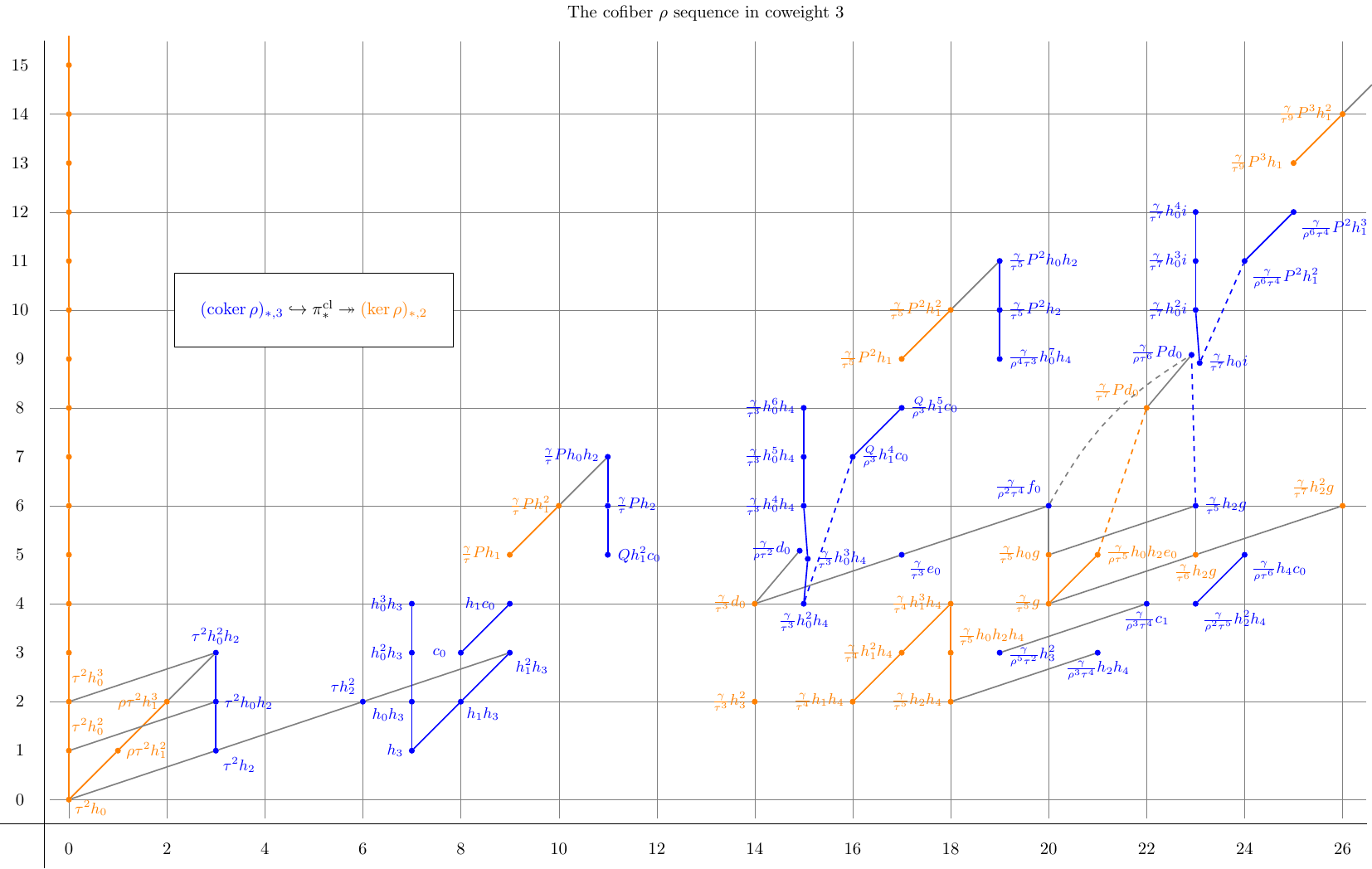}

\includegraphics[height=\textwidth]{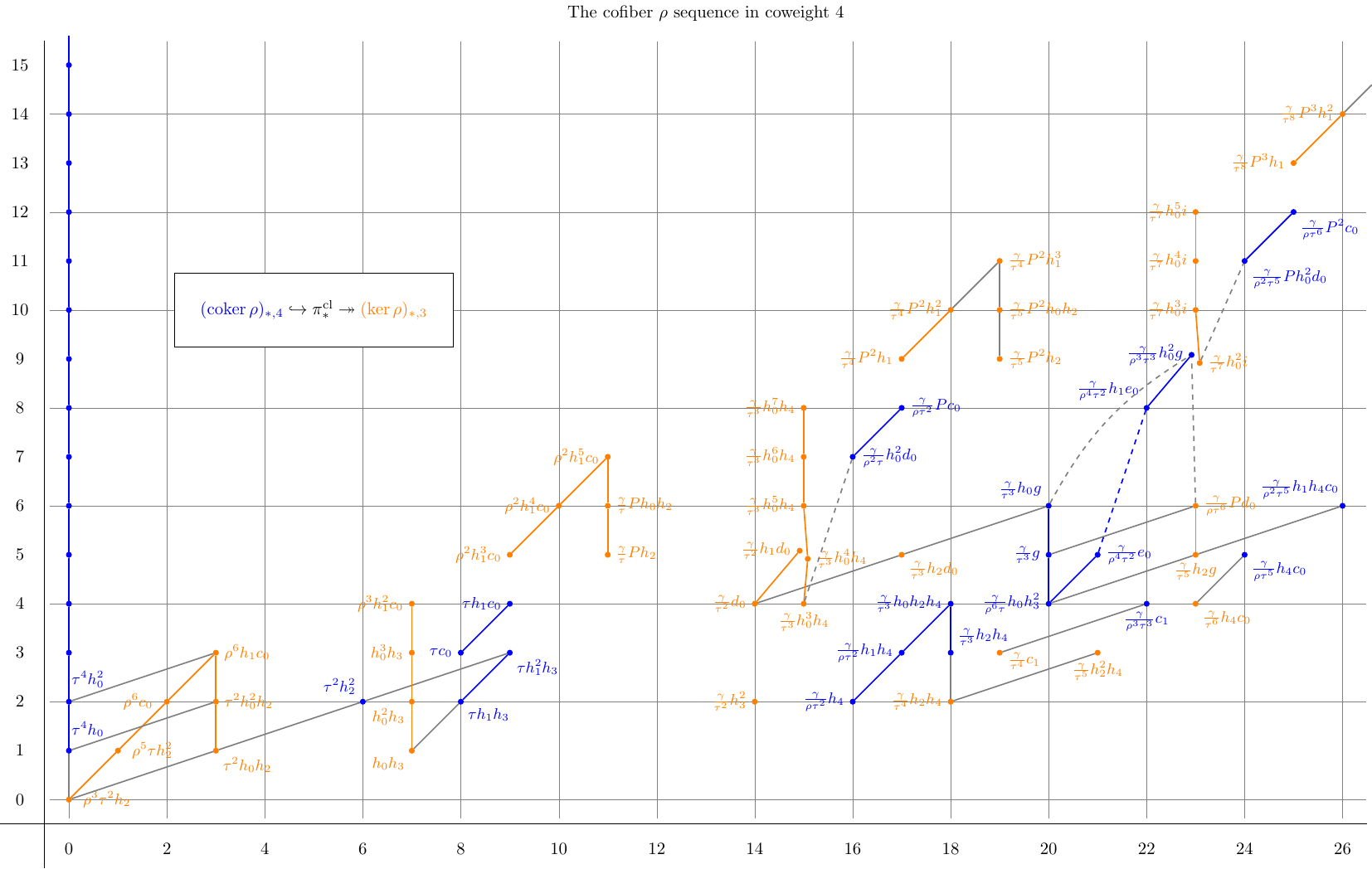}

\includegraphics[height=\textwidth]{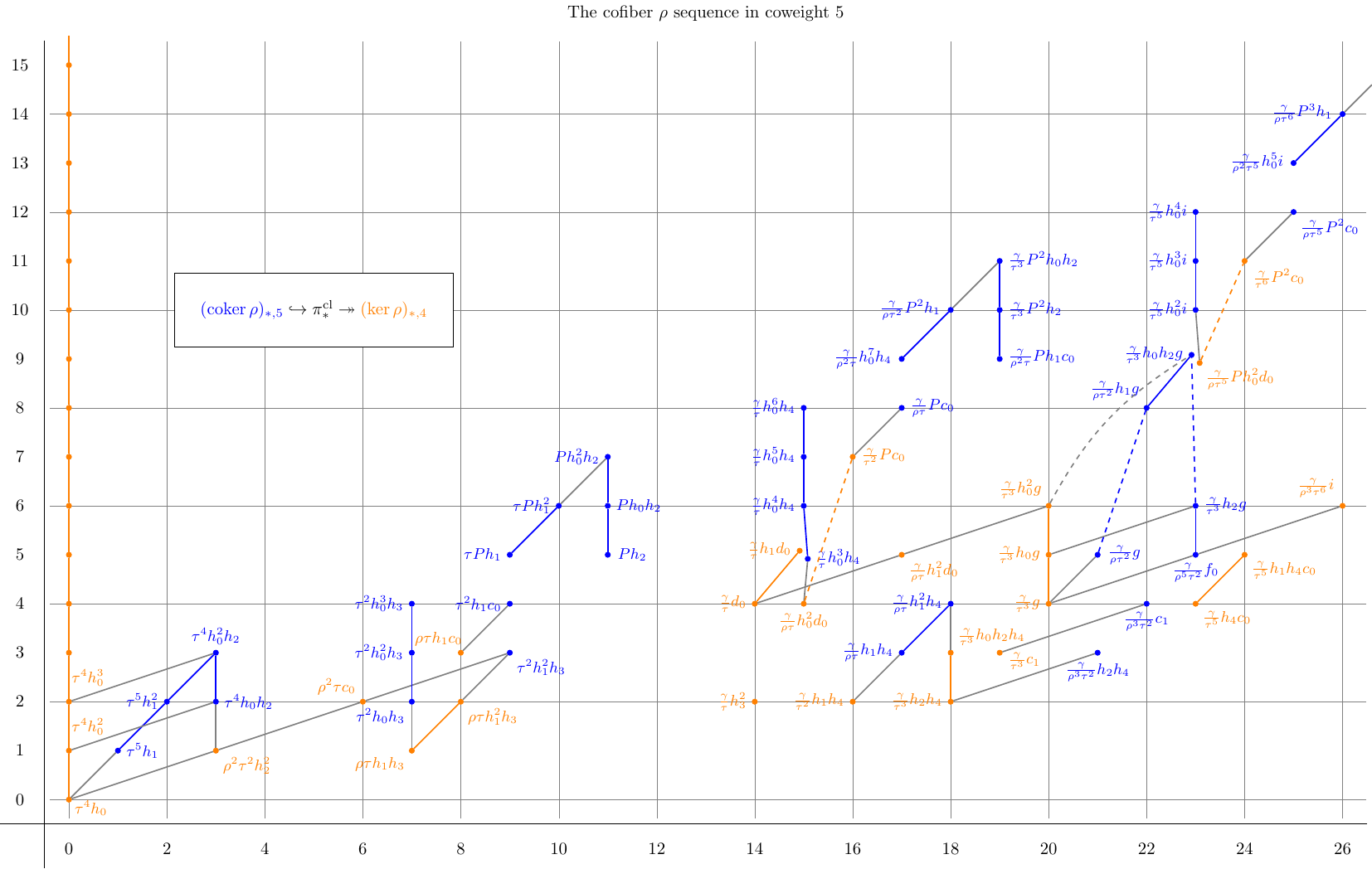}

\includegraphics[height=\textwidth]{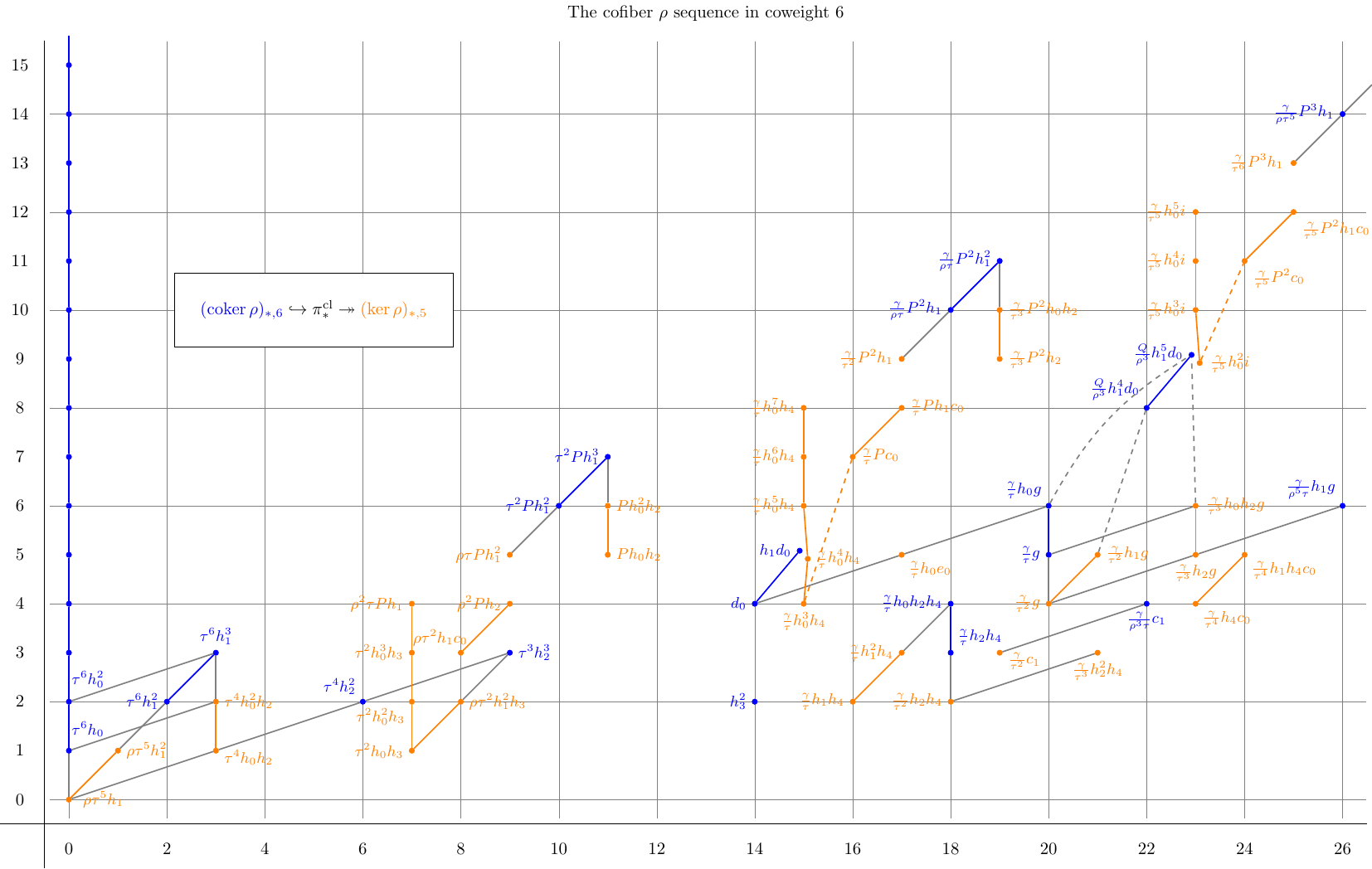}

\includegraphics[height=\textwidth]{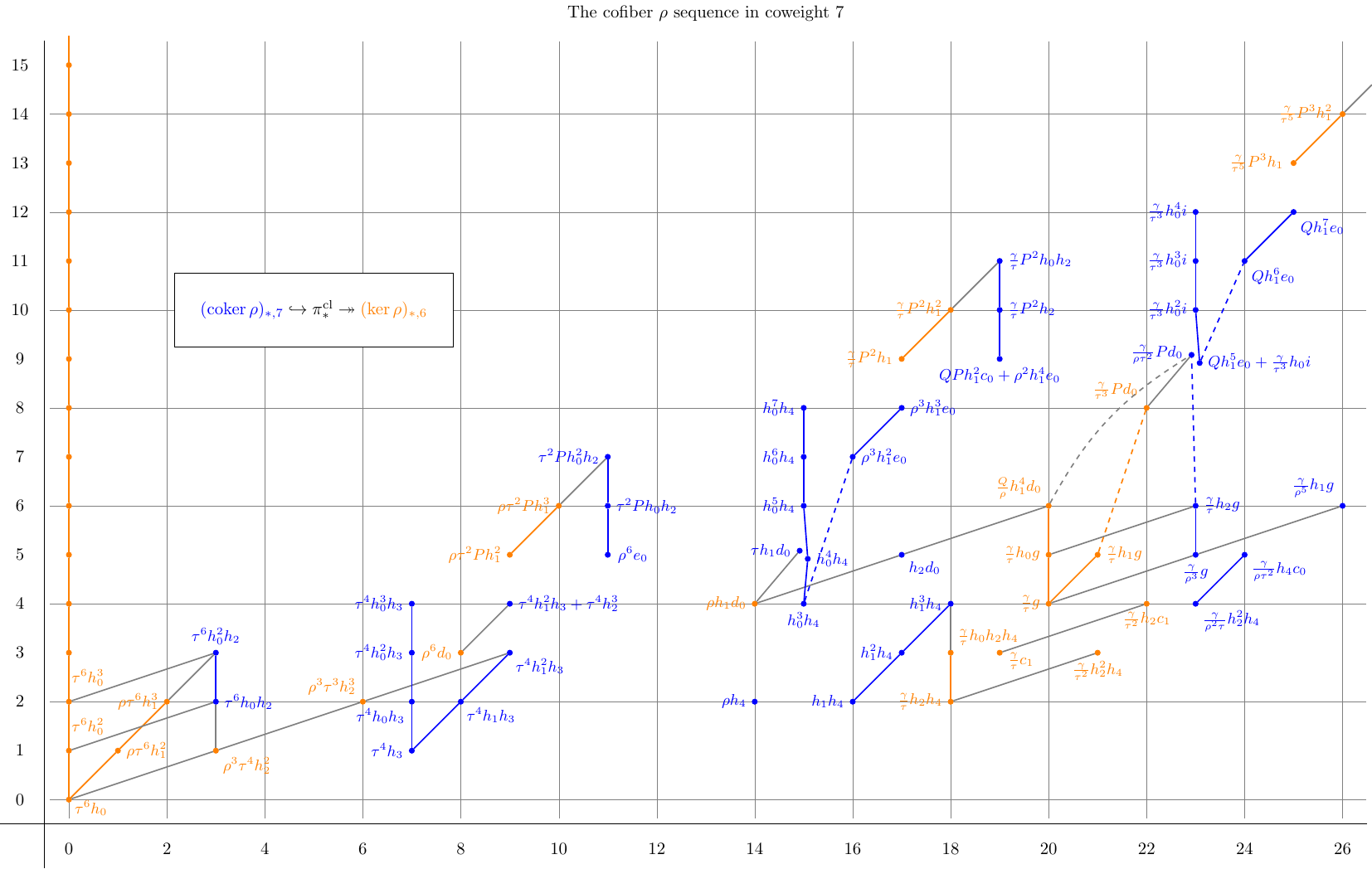}
\label{cofibrhoEnd}
\end{landscape}

\bibliographystyle{amsalpha}

\begin{bibdiv}
\begin{biblist}

\bib{AI}{article}{
   author={Araki, Sh\^{o}r\^{o}},
   author={Iriye, Kouyemon},
   title={Equivariant stable homotopy groups of spheres with involutions. I},
   journal={Osaka Math. J.},
   volume={19},
   date={1982},
   number={1},
   pages={1--55},
   issn={0388-0699},
   review={\MR{656233}},
}

\bib{Balderrama}{article}{
	author={Balderrama, William},
	label={Ba},
	title={The $C_2$-equivariant $K(1)$-local sphere},
    date = {25 March 2021},
    doi = {\href{https://doi.org/10.48550/arXiv.2103.13895}{10.48550/arXiv.2103.13895}},
}

\bib{Behrens07}{article}{
   author={Behrens, Mark},
   title={Some root invariants at the prime 2},
   conference={
      title={Proceedings of the Nishida Fest (Kinosaki 2003)},
   },
   book={
      series={Geom. Topol. Monogr.},
      volume={10},
      publisher={Geom. Topol. Publ., Coventry},
   },
   date={2007},
   pages={1--40},
   review={\MR{2402775}},
   doi={10.2140/gtm.2007.10.1},
}

\bib{BS}{article}{
   author={Behrens, Mark},
   author={Shah, Jay},
   title={$C_2$-equivariant stable homotopy from real motivic stable
   homotopy},
   journal={Ann. K-Theory},
   volume={5},
   date={2020},
   number={3},
   pages={411--464},
   issn={2379-1683},
   review={\MR{4132743}},
   doi={10.2140/akt.2020.5.411},
}

\bib{BGI}{article}{
   author={Belmont, Eva},
   author={Guillou, Bertrand J.},
   author={Isaksen, Daniel C.},
   title={$C_2$-equivariant and $\mathbb{R}$-motivic stable stems II},
   journal={Proc. Amer. Math. Soc.},
   volume={149},
   date={2021},
   number={1},
   pages={53--61},
   issn={0002-9939},
   review={\MR{4172585}},
   doi={10.1090/proc/15167},
}
		
\bib{BI}{article}{
   author={Belmont, Eva},
   author={Isaksen, Daniel C.},
   title={$\Bbb R$-motivic stable stems},
   journal={J. Topol.},
   volume={15},
   date={2022},
   number={4},
   pages={1755--1793},
   issn={1753-8416},
   review={\MR{4461846}},
   doi={10.1112/topo.12256},
}

\bib{BIK}{article}{
	author={Belmont, Eva},
	author={Isaksen, Daniel C.},
	author={Kong, Hana Jia},
	title={$\R$-motivic $v_1$-periodic homotopy},
	year={12 April 2022},
	doi={\href{https://doi.org/10.48550/arXiv.2204.05937}{10.48550/arXiv.2204.05937}},    
}

\bib{BK}{article}{
	author={Belmont, Eva},
	author={Kong, Hana Jia},
	title={A Toda bracket convergence theorem for multiplicative spectral sequences},
	year={16 December 2021}
	doi={\href{https://doi.org/10.48550/arXiv.2112.08689}{10.48550/arXiv.2112.08689}},
}


\bib{B}{article}{
   author={Bredon, Glen E.},
   label={Br1},
   title={Equivariant stable stems},
   journal={Bull. Amer. Math. Soc.},
   volume={73},
   date={1967},
   pages={269--273},
   issn={0002-9904},
   review={\MR{206947}},
   doi={10.1090/S0002-9904-1967-11713-0},
}

\bib{B68}{article}{
   author={Bredon, Glen E.},
   label={Br2},
   title={Equivariant homotopy},
   conference={
      title={Proc. Conf. on Transformation Groups},
      address={New Orleans, La.},
      date={1967},
   },
   book={
      publisher={Springer, New York},
   },
   date={1968},
   pages={281--292},
   review={\MR{0250303}},
}
	
	
\bib{BG}{article}{
   author={Bruner, Robert},
   author={Greenlees, John},
   title={The Bredon-L\"{o}ffler conjecture},
   journal={Experiment. Math.},
   volume={4},
   date={1995},
   number={4},
   pages={289--297},
   issn={1058-6458},
   review={\MR{1387694}},
}

\bib{LowMW}{article}{
   author={Dugger, Daniel},
   author={Isaksen, Daniel C.},
   title={Low-dimensional Milnor-Witt stems over $\mathbb{R}$},
   journal={Ann. K-Theory},
   volume={2},
   date={2017},
   number={2},
   pages={175--210},
   issn={2379-1683},
   review={\MR{3590344}},
   doi={10.2140/akt.2017.2.175},
}

\bib{DI-comparison}{article}{
   author={Dugger, Daniel},
   author={Isaksen, Daniel C.},
   title={$\mathbb{Z}/2$-equivariant and $\mathbb{R}$-motivic stable stems},
   journal={Proc. Amer. Math. Soc.},
   volume={145},
   date={2017},
   number={8},
   pages={3617--3627},
   issn={0002-9939},
   review={\MR{3652813}},
   doi={10.1090/proc/13505},
}	

\bib{GHIR}{article}{
   author={Guillou, Bertrand J.},
   author={Hill, Michael A.},
   author={Isaksen, Daniel C.},
   author={Ravenel, Douglas Conner},
   title={The cohomology of $C_2$-equivariant $\mathcal A(1)$ and the homotopy
   of ${\rm ko}_{C_2}$},
   journal={Tunis. J. Math.},
   volume={2},
   date={2020},
   number={3},
   pages={567--632},
   issn={2576-7658},
   review={\MR{4041284}},
   doi={10.2140/tunis.2020.2.567},
}

\bib{etaR}{article}{
   author={Guillou, Bertrand J.},
   author={Isaksen, Daniel C.},
   title={The $\eta$-inverted $\mathbb R$-motivic sphere},
   journal={Algebr. Geom. Topol.},
   volume={16},
   date={2016},
   number={5},
   pages={3005--3027},
   issn={1472-2747},
   review={\MR{3572357}},
   doi={10.2140/agt.2016.16.3005},
}
		



\bib{C2MW0}{article}{
   author={Guillou, Bertrand J.},
   author={Isaksen, Daniel C.},
   title={The Bredon-Landweber region in $C_2$-equivariant stable homotopy
   groups},
   journal={Doc. Math.},
   volume={25},
   date={2020},
   pages={1865--1880},
   issn={1431-0635},
   review={\MR{4184454}},
}

\bib{HHR}{article}{
   author={Hill, M. A.},
   author={Hopkins, M. J.},
   author={Ravenel, D. C.},
   title={On the nonexistence of elements of Kervaire invariant one},
   journal={Ann. of Math. (2)},
   volume={184},
   date={2016},
   number={1},
   pages={1--262},
   issn={0003-486X},
   review={\MR{3505179}},
   doi={10.4007/annals.2016.184.1.1},
}

\bib{HK}{article}{
   author={Hu, Po},
   author={Kriz, Igor},
   title={Real-oriented homotopy theory and an analogue of the Adams-Novikov
   spectral sequence},
   journal={Topology},
   volume={40},
   date={2001},
   number={2},
   pages={317--399},
   issn={0040-9383},
   review={\MR{1808224}},
   doi={10.1016/S0040-9383(99)00065-8},
}

\bib{HKSZ}{article}{
      author={Hu, Po},
      author={Kriz, Igor},
      author={Somberg, Petr},
      author={Zou, Foling},
      title={The $\mathbb{Z}/p$-equivariant dual Steenrod algebra for an odd prime $p$}, 
      year={26 May 2022},
      doi={\href{https://doi.org/10.48550/arXiv.2205.13427}{10.48550/arXiv.2205.13427}},

}

\bib{Iriye82}{article}{
   author={Iriye, Kouyemon},
   label={Ir},
   title={Equivariant stable homotopy groups of spheres with involutions.
   II},
   journal={Osaka J. Math.},
   volume={19},
   date={1982},
   number={4},
   pages={733--743},
   issn={0030-6126},
   review={\MR{687770}},
}	
	
	
\bib{I}{article}{
   author={Isaksen, Daniel C.},
   label={Is},
   title={Stable stems},
   journal={Mem. Amer. Math. Soc.},
   volume={262},
   date={2019},
   number={1269},
   pages={viii+159},
   issn={0065-9266},
   isbn={978-1-4704-3788-6},
   isbn={978-1-4704-5511-8},
   review={\MR{4046815}},
   doi={10.1090/memo/1269},
}

\bib{IWX23}{article}{
   author={Isaksen, Daniel C.},
   author={Wang, Guozhen},
   author={Xu, Zhouli},
   title={Stable homotopy groups of spheres: from dimension 0 to 90},
   journal={Publ. Math. Inst. Hautes \'{E}tudes Sci.},
   volume={137},
   date={2023},
   pages={107--243},
   issn={0073-8301},
   review={\MR{4588596}},
   doi={10.1007/s10240-023-00139-1},
}


\bib{L}{article}{
   author={Landweber, Peter S.},
   title={On equivariant maps between spheres with involutions},
   journal={Ann. of Math. (2)},
   volume={89},
   date={1969},
   pages={125--137},
   issn={0003-486X},
   review={\MR{238313}},
   doi={10.2307/1970812},
}

\bib{Ma}{article}{
    author={Ma, Sihao},
    label={Ma},
    title={The Borel and Genuine $C_2$-equivariant Adams Spectral Sequences},
    year={26 August 2022}
    doi={\href{https://doi.org/10.48550/arXiv.2208.12883}{10.48550/arXiv.2208.12883}},
}

\bib{Mahowald-Ravenel93}{article}{
   author={Mahowald, Mark E.},
   author={Ravenel, Douglas C.},
   title={The root invariant in homotopy theory},
   journal={Topology},
   volume={32},
   date={1993},
   number={4},
   pages={865--898},
   issn={0040-9383},
   review={\MR{1241877}},
   doi={10.1016/0040-9383(93)90055-Z},
}

\bib{MM}{article}{
   author={Mandell, M. A.},
   author={May, J. P.},
   title={Equivariant orthogonal spectra and $S$-modules},
   journal={Mem. Amer. Math. Soc.},
   volume={159},
   date={2002},
   number={755},
   pages={x+108},
   issn={0065-9266},
   review={\MR{1922205}},
   doi={10.1090/memo/0755},
}

\bib{CMay}{article}{
   author={May, Clover},
   label={MC},
   title={A structure theorem for $RO(C_2)$-graded Bredon cohomology},
   journal={Algebr. Geom. Topol.},
   volume={20},
   date={2020},
   number={4},
   pages={1691--1728},
   issn={1472-2747},
   review={\MR{4127082}},
   doi={10.2140/agt.2020.20.1691},
}


\bib{MMP}{article}{
   author={May, J. Peter},
   label={MJP},
   title={Matric Massey products},
   journal={J. Algebra},
   volume={12},
   date={1969},
   pages={533--568},
   issn={0021-8693},
   review={\MR{238929}},
   doi={10.1016/0021-8693(69)90027-1},
}

\bib{Morel}{article}{
   author={Morel, Fabien},
   label={Mor},
   title={On the motivic $\pi_0$ of the sphere spectrum},
   conference={
      title={Axiomatic, enriched and motivic homotopy theory},
   },
   book={
      series={NATO Sci. Ser. II Math. Phys. Chem.},
      volume={131},
      publisher={Kluwer Acad. Publ., Dordrecht},
   },
   isbn={1-4020-1834-7},
   date={2004},
   pages={219--260},
   review={\MR{2061856}},
   doi={10.1007/978-94-007-0948-5\_7},
}

\bib{Moss}{article}{
   author={Moss, R. Michael F.},
   label={Mo},
   title={Secondary compositions and the Adams spectral sequence},
   journal={Math. Z.},
   volume={115},
   date={1970},
   pages={283--310},
   issn={0025-5874},
   review={\MR{266216}},
   doi={10.1007/BF01129978},
}

\bib{SW}{article}{
   author={Sankar, Krishanu},
   author={Wilson, Dylan},
   title={On the $C_p$-equivariant dual Steenrod algebra},
   journal={Proc. Amer. Math. Soc.},
   volume={150},
   date={2022},
   number={8},
   pages={3635--3647},
   issn={0002-9939},
   review={\MR{4439482}},
   doi={10.1090/proc/15846},
}


\bib{Serre}{article}{
   author={Serre, Jean-Pierre},
   title={Groupes d'homotopie et classes de groupes ab\'{e}liens},
   language={French},
   journal={Ann. of Math. (2)},
   volume={58},
   date={1953},
   pages={258--294},
   issn={0003-486X},
   review={\MR{59548}},
   doi={10.2307/1969789},
}


\end{biblist}
\end{bibdiv}

\end{document}